\documentclass{article}
\usepackage{amsmath}
\usepackage{amsthm}
\usepackage{amssymb}
\usepackage{comment}
\usepackage{hyperref}
\usepackage{mathtools}
\usepackage{bbm}
\usepackage{cite}
\usepackage{xcolor}

\usepackage{stmaryrd}
\usepackage{enumitem}

\usepackage{tikz-cd}
\SetSymbolFont{stmry}{bold}{U}{stmry}{m}{n}
\usepackage{euscript}
\usepackage[arrow,curve,matrix,arc,2cell]{xy}
\UseAllTwocells
\usepackage[disable]{todonotes}
\usepackage{slashed}
\DeclareFontFamily{U}{rsfs}{} 
\DeclareFontShape{U}{rsfs}{n}{it}{<->
	rsfs10}{} \DeclareSymbolFont{mscr}{U}{rsfs}{n}{it}
\DeclareSymbolFontAlphabet{\scr}{mscr}
\def\mathscr{\scr}

\def\C{\mathbb{C}}

\def\Q{\mathbb{Q}}
\def\Z{\mathbb{Z}}

\DeclareMathOperator{\At}{At}
\DeclareMathOperator{\Hom}{Hom}

\DeclareMathOperator{\Eu}{Eu}
\DeclareMathOperator{\Ext}{Ext}
\DeclareMathOperator{\Ker}{Ker}
\DeclareMathOperator{\Per}{Per}
\DeclareMathOperator{\Quot}{Quot}

\DeclareMathOperator{\Spec}{Spec}
\DeclareMathOperator{\Pic}{Pic}
\DeclareMathOperator{\ch}{ch}

\DeclareMathOperator{\stabS}{\mathbf{S}}
\DeclareMathOperator{\stabT}{\mathbf{T}}

\DeclareMathOperator{\at}{at}
\DeclareMathOperator{\vir}{vir}
\DeclareMathOperator{\vdim}{vdim}
\DeclareMathOperator{\aone}{\mathbb{A}^1}
\DeclareMathOperator{\pone}{\mathbb{P}^1}
\DeclareMathOperator{\Ob}{Ob}
\DeclareMathOperator{\ob}{ob}
\DeclareMathOperator{\teq}{\mathfrak{t}}

\DeclarePairedDelimiter\abs{\lvert}{\rvert}

\newcommand{\chexc}{e}  
\newcommand{\globm}[1]{\Gamma_m(#1)} 
\newcommand{\Xhat}{\widehat{X}}
\newcommand{\chat}{\widehat{c}}
\newcommand{\dbcoh}{D^b_{\operatorname{coh}}}
\newcommand{\Poin}{\mathcal{P}oin}
\newcommand{\Sheafhom}{\mathcal{H}om}

\newcommand{\Sheafext}{\mathcal{E}xt}
\newcommand{\Picstack}{\mathcal{P}ic}

\newcommand{\detL}{\mathscr{L}}
\newcommand{\bchi}{\boldsymbol\chi}
\newcommand{\vnor}[2]{\mathfrak{N}(#1,#2)}
\newcommand{\vmor}[2]{\mathfrak{M}(#1,#2)}

\newcommand{\Homflag}[1]{F(#1)} 

\newcommand{\zerostable}[1][]{\mathcal{M}^0_{#1}}
\DeclareMathOperator{\onestable}{\mathcal{M}^1}

\DeclareMathOperator{\zeronestablechat}{\mathcal{M}^{0,1}_{\widehat{c}}}
\newcommand{\zeronestable}[1][]{\mathcal{M}^{0,1}_{#1}}

\newcommand{\flag}[2][]{\mathcal{N}^{#2}_{#1}}
\DeclareMathOperator{\flagzero}{\mathcal{N}^{0}}
\DeclareMathOperator{\flagell}{\mathcal{N}^ \ell}
\DeclareMathOperator{\flagellplusone}{\mathcal{N}^ {\ell+1}}
\DeclareMathOperator{\flagN}{\mathcal{N}^ N}
\DeclareMathOperator{\flagellss}{\mathcal{N}^{\ell,\ell+1}}

\DeclareGraphicsExtensions{ext-list}

\theoremstyle{plain}
\newtheorem{theorem}{Theorem}[section]
\newtheorem{proposition}[theorem]{Proposition}

\newtheorem{lemma}[theorem]{Lemma}
\newtheorem{assumption}[theorem]{Assumption}
\newtheorem{corollary}[theorem]{Corollary}
\newtheorem{convention}[theorem]{Convention}

\theoremstyle{definition}
\newtheorem{definition}[theorem]{Definition}
\newtheorem{claim}[theorem]{Claim}
\newtheorem{remark}[theorem]{Remark}
\newtheorem{construction}[theorem]{Construction}

\newtheorem{example}[theorem]{Example}
\newtheorem{examples}[theorem]{Examples}

\newtheorem{conjecture}[theorem]{Conjecture}

\makeatletter
\@addtoreset{equation}{section}
\makeatother

\begin{document}


	\title{A blowup formula for virtual enumerative invariants on projective surfaces}

	\author{Nikolas Kuhn and Yuuji Tanaka}
	\date{}


	\maketitle



	
	
	

\begin{abstract}
    We apply virtual localization to the problem of finding blowup formulae for virtual sheaf-theoretic invariants on a smooth projective surface. This leads to a general procedure that can be used to express virtual enumerative invariants on the blowup in terms of those on the original surface. We use an enhanced master space construction over the moduli spaces of $m$-stable sheaves introduced by Nakajima and Yoshioka. Our work extends their analogous results for the equivariant moduli spaces of framed sheaves on $\mathbb{P}^2$. In contrast to their work, we make no use of GIT methods and work with an arbitrary smooth complex projective surface, assuming only the absence of strictly semistable sheaves.  The main examples to keep in mind are Mochizuki's virtual analogue of the Donaldson invariant and the virtual $\chi_y$-genus of the moduli space of Gieseker semistable sheaves on the surface.
\end{abstract}

	\setcounter{tocdepth}{2}
	\tableofcontents

	\section{Introduction}

 Enumerative invariant theories on a smooth projective surface $X$ that are defined via its moduli spaces of semistable sheaves often have a surprisingly rich structure (at least conjecturally). In particular, an a priori infinite number of invariants coming from the infinitely many choices of Chern classes on $X$ can in many cases be expressed by just a finite collection of invariants on $X$. 
 Examples of this phenomenon include the Witten conjecture for the Donaldson invariants  (proven in \cite{GoNaYo3} for projective surfaces X with $H^1(X)=0$ and $p_g (X) >0$ and with preceding work in \cite{KrMr}, \cite{FiSt}) as well as the predicted modularity for virtual $\chi_y$-genera \cite{GoKo} and for virtual Segre and Verlinde numbers (see e.g. \cite{GoKo2} for recent progress).
    In this paper, we establish a general procedure to compute virtual sheaf-theoretic invariants on the blowup of $X$ in terms of those on $X$ itself, under the assumption that there are no strictly semistable sheaves in the moduli spaces. This extends Nakajima and Yoshioka's results in \cite{NaYo3} for the equivariant spaces of framed sheaves on $\mathbb{P}^2$ to the case of virtual enumerative invariants. We expect that our result can be useful in proving blowup formulae that match with the conjectural forms for several classes of enumerative invariants, such as the ones mentioned above. This can possibly be used to reveal modular properties behind those enumerative invariant theories, as has been demonstrated in e.g. \cite{KrMr}, \cite{FiSt}, \cite{Gott}, \cite{GoNaYo1}, and \cite{GoNaYo3} for the Donaldson theory. As a direct application of our result, we deduce a virtual analogue of the Friedman--Morgan blowup formula which allows us to conclude that the virtual and the classical invariants are equivalent in our setting.

	\paragraph{The virtual Donaldson invariants.}
	Let $X$ be a smooth projective surface and $H$ a fixed ample line bundle (or \emph{polarization}) on $X$. 
	Consider the graded algebra $\operatorname{Sym}(H^*(X,\mathbb{Z}))$, where the grading is given by taking $\sigma\in H^i(X,\mathbb{Z})$ to be in degree $i$. For the moment, let $c_1\in H^2(X,\Z)$ and $c_2\in H^4(X,\Z)$ be arbitrary algebraic classes. We write $\vdim:=4c_2-c_1^2-3\chi(\mathcal{O}_X)$. Then the {\it virtual Donaldson invariant}, or the {\it Donaldson--Mochizuki invariant} \cite{Moch}
	associated to $c_1,c_2$ is a map 
	\[D_{X,H,c_1,c_2}:\operatorname{Sym}(H^*(X,\mathbb{Z}))_{\vdim} \to \mathbb{Q}.\]
	Equivalently, one can bundle the different operators $D_{X,H,c_1,c_2}$ into a single map
	\[D_{X,H,c_1}:\operatorname{Sym}(H^*(X,\mathbb{Z})) \to \mathbb{Q}.\]
	Here, the restriction of $D_{X,H,c_1}$ to the degree $v$ part is equal to $D_{X,H,c_1,c_2}$ if $v=c_1^2-4c_2-3\chi(\mathcal{O}_X)$ for some integer value of $c_2$; and zero otherwise. 

	The invariants are defined by using the virtual fundamental class of the moduli space of Gieseker semistable sheaves of rank $2$ with fixed-determinant and with Chern classes $c_1$ and $c_2$. A precise definition for $\deg_H c_1$ odd will be given in \S \ref{subsec:donmoch}. For the general case, see Mochizuki \cite[\S 7.3]{Moch}.
	
	The classical Donaldson invariants \cite{Dona}, \cite{DoKr} are defined for a closed, oriented, simply-connected, smooth $4$-manifold with $b_2^+\geq 3$ and $b_2^+$ odd. They are defined by integration over a suitably compactified moduli space of anti-self-dual connections on a principal $SO(3)$-bundle with the place of the first and second Chern class taken by the second Stiefel--Whitney $w_2$ and the first Pontrjagin class $p_1$,  respectively. When $w_2=0$, the invariants are at first only defined for large enough value of $-p_1/4$ (the so-called \emph{stable range}), but then extended to the \emph{unstable range} by using a certain blowup formula (cf. Friedman and Morgan \cite[\S 3.8, Theorem 8.1]{FrMo}; see also Theorem \ref{th:blowupdon} for our analogue for the virtual invariants). When $b_2^+=1$, one can still define invariants, but they depend on a choice of a metric and there is an interesting wall-crossing behavior. 
	
	Consider now the underlying smooth 4-manifold of the complex projective surface $X$ and that the topological invariants are induced by algebraic classes (in the way of formulas \eqref{eq:chtotop1} and \eqref{eq:chtotop2}). We allow $b^+_2\geq 1$ but still assume $X$ to be simply-connected. 
	Then one can show that the classical invariants agree with the virtual ones whenever $c_2$ is large enough and if the polarization is generic in a precise sense, as described by G\"ottsche--Nakajima--Yoshioka  \cite[\S 1.1 and \S A]{GoNaYo1} and \cite[\S 2.2]{GoNaYo3}. 
In addition, they pointed out that the invariants agree for all values of $c_2$, if one is able to show that Friedmann--Morgan's blowup formula holds for Mochizuki's invariants. It is remarked in \cite{GoNaYo3} that the equivalence of classical and virtual invariants follows from the course of their proof of Witten's conjecture on $D=SW$ 
 for simply-connected complex projective surfaces with $p_g>0$ in \cite{GoNaYo3}. 
We give a new, more direct, proof for this equivalence by establishing the blowup formula for the virtual invariants. 
We obtain it, assuming that $\deg_H c_1$ is not divisible by two, but under no assumptions on the surface $X$, see Theorem \ref{th:blowupdon}.

	\paragraph{Virtual Euler characteristics and $\chi_y$-genera.}
	
	In \cite{Gott2}, G\"ottsche studied the generating series of Poincar\'e polynomials $p(X^{[n]},z):=\sum_i(-1)^i b_i(X^{[n]})z^n$ for the Hilbert schemes of points of the smooth projective surface $X$. He found the following formula:
	\[\sum_{n\geq 0} p(X^{[n]},z) t^n = \prod_{k\geq 1} \prod_{i=1}^4 (1-z^{2k-2+i}t^k)^{(-1)^{i+1} b_i(X)}. \]
	In particular, the generating series only depends on the Betti numbers of $X$ itself.  
	Specializing to $z=1$, we get the generating series for topological Euler characteristics
	\[\sum_{n\geq 0} \chi(X^{[n]}) q^n = \prod_{k\geq 1} (1-q^k)^{-\chi(X)}.\] 
	
	For higher ranks, one can consider the moduli space $M_{X,H}(r,c_1,n)$ of $H$-Gieseker semistable sheaves of a given rank and first and second Chern classes. Explicit results for Betti numbers of moduli spaces of semistable sheaves of rank $2$ on $\mathbb{P}^2$ and on ruled surfaces were obtained by Yoshioka \cite{Yosh1}, \cite{Yosh2}. A first result valid for rank $2$ moduli spaces on arbitrary surfaces was the blowup formula for the \emph{virtual Hodge polynomial}, a refinement of the Poincar\'e polynomial for complex algebraic schemes, proven by Li and Qin \cite{LiQi1},\cite{LiQi2}. 
	Let $h(Y,x_1,x_2)$ denote the virtual Hodge polynomial of a complex algebraic variety $Y$. Supppose that $c_1$ is chosen such that $\deg_H c_1$ is odd. Then consider the following generating series 
	
	\[Z_{X,H,c_1}(x_1,x_2,q):= \sum_{n} h(M_{X,H}(2,c_1,n),x_1,x_2) q^{n-c_1^2/4-3/4\chi(\mathcal{O}_X)}.\]
	
	Let $p:\Xhat\to X$ be the blowup of $X$ with exceptional divisor $C$. Let $\widehat{c}_1$ a fixed divisor class on $\Xhat$. It is well-known that for a small enough rational number $0<\epsilon_0\ll 1$, the moduli space $M_{\Xhat, H-\epsilon C}(2,\widehat{c}_1,n)$ is independent of $0<\epsilon\leq \epsilon_0$ (where $\epsilon_0$ may depend on $\widehat{c}_1,n$). We denote the moduli space obtained in this way for small enough choices of $\epsilon$ by  $M_{\Xhat, H_{\infty}}(2,\widehat{c}_1,n)$. Then we have the generating function $Z_{\Xhat,H_{\infty},\widehat{c}_1}(x_1,x_2,q)$. Li--Qin's formula now states that
	\[Z_{\Xhat,H_{\infty},p^*c_1-aC}(x_1,x_2,q) = Z_a(x_1x_2,q) Z_{X,H,c_1}(x_1,x_2,q)\]
	for $a=0,1$. Here the function $Z_a$ is given by 
	\begin{equation}\label{eq:chiyblowup}
		Z_a(x,q) = \frac{\sum_{n\in \Z} x^{\frac{(2n+a)^2-(2n+a)}{2}} q^{\frac{(2n+a)^2}{4}} }{\prod_{n\geq 1} (1-x^{2n}q^n )^2}. 
	\end{equation}
	
	More recently, G\"ottsche and Kool \cite{GoKo} studied a different generalization of the Euler characteristic, the {\it virtual $\chi_y$-genus} introduced by Ciocan-Fontanine--Kapranov \cite{CiKa} and Fantechi--G\"ottsche \cite{FaGo}, which is an invariant of a proper Deligne--Mumford stack $Y$ with perfect obstruction theory, defined as 
	\[\chi_y^{\vir}(Y):= \sum_{p\geq 0} y^p\chi^{\vir}(Y,\Omega_{Y}^{p,\vir}).\]
	For a smooth projective scheme with the trivial obstruction theory, it agrees with the usual $\chi_y$-genus, which can be obtained from the Hodge polynomial by setting $x_1=y,x_2=1$. Specializing to $y=-1$, the expression for the virtual $\chi_y$-genus recovers the \emph{virtual Euler characteristic} of $Y$, denoted $\chi^{\vir}(Y)$. 
	
	Whenever $\deg_H c_1$ is odd, the moduli space $M_{X,H}(2,c_1,n)$ carries a perfect obstruction theory. 
	In this case, we may consider the generating series 
	\[Z^{\vir}_{X,H,c_1}(y,q):=\sum_n \chi^{\vir}_{-y}(M_{X,H}(2,c_1,n)) q^{n-c_1^2/4-3/4\chi(\mathcal{O}_X)}.\]
	G\"ottsche and Kool conjecture an explicit expression for this generating series in terms of modular forms and finitely many \emph{Seiberg--Witten invariants} of the surface $X$ \cite[Conjecture 5.7]{GoKo} (note that they work with $\overline{\chi}^{\vir}_{-y}(Y) = y^{-\vdim(Y)/2} \chi^{\vir}_{-y}(Y)$). Assuming their conjecture, one gets the same blowup formula as for the virtual Hodge polynomial specialized to $x_1=y,x_2=1$:
	\[Z^{\vir}_{\Xhat,H_{\infty},p^*c_1-aC}(y,q) = Z_a(y,q)Z^{\vir}_{X,H,c_1}(y,q).\] 
	In particular, we have a relatively simple formula for the generating series of virtual Euler characteristics:
	\begin{equation}
		Z^{\vir}_{\Xhat,H_{\infty},p^*c_1-aC}(1,q) = \frac{\sum_{n\in \Z}q^{(n+\frac{a}{2})^2} }{\prod_{n\geq 1}(1-q^n)^2} Z^{\vir}_{X,H,c_1}(1,q). \label{eq:blowupeulerchar}
	\end{equation}
	
	\paragraph{Moduli spaces of sheaves on the blowup.}
	
	In a series of papers \cite{NaYo1},\cite{NaYo2},\cite{NaYo3}, Nakajima and Yoshioka used moduli spaces $M^0_{\chat}$ of (semi-)stable \emph{perverse coherent sheaves} on the blowup of a smooth projective surface in order to study a blowup formula for Donaldson invariants with fundamental matters on $\mathbb{P}^2$. 
	Here, we use $\chat\in \Z \oplus H^2(X,\Z)\oplus H^4(X,\Q)$ to denote the Chern character of the sheaves being parametrized. If $\chat=p^*c$ is the pull-back of a cohomology class $c=r+c_1+\ch_2$ with $\deg_H c_1$ odd, then the space $M^0_{\chat}$ is a moduli space for the class of sheaves on $\Xhat$ which are pull-backs of $H$-Gieseker stable sheaves on $X$. In this example, and more generally for any $\chat$ satisfying $\deg_{p^*H} \chat_1$ odd, the space $M^0_{\chat}$ differs from the moduli space $M_{X,H_{\infty}}(r,\chat_1,\chat_2)$ (where $\chat_2=\chat_1^2/2-\widehat{\ch}_2$). 
	Sheaves $E$ in $ M^0_{\chat}$ are allowed to have some torsion along the exceptional curve $C$, but satisfy the additional condition $\Hom_{\Xhat}(E,\mathcal{O}_C(-1))=0$. Roughly speaking, in $M_{X,H_{\infty}}(r,p^*\chat_1,\chat_2)$, all torsion sheaves are declared to be destabilizing subsheaves, while in $M^0_{\chat}$  subsheaves of the form $\mathcal{O}_C(-m'), m'\geq 1$ are not destabilizing, but so are those subsheaves $E'\subset E$ for which $E/E'\simeq \mathcal{O}_C(-m'), m'\geq 1$. 
	
	By varying which of the sheaves $\mathcal{O}_C(-m)$ are considered for destabilizing, one obtains a sequence of moduli spaces $M_{\chat}^m$ of $\emph{m}$-stable sheaves for $m\geq 0$ with the following properties:
	\begin{enumerate}[label=\arabic*)]
		\item One has $M^m_{\chat} = M^0_{\chat(-mC)}$, where we set $\chat(-mC):=\chat\cdot\ch(\mathcal{O}_{\Xhat}(-mC))$.  
		\item One has some control over the difference between $M^{m-1}_{\chat}$ and $M^{m}_{\chat}$. 
		\item Writing $\chat = p_*c-j\ch(\mathcal{O}_C(-1))$, one has a map $p_*:M^0_{\chat}\to M_X(r,c_1,c_2)$ given by pushforward of sheaves along $p:\Xhat\to X$. 
		\item For sufficiently large $m$ one has $M^m_{\chat}\simeq M_{\Xhat}(r,\chat_1,\chat_2)$.
	\end{enumerate}
	We may summarize this graphically as
	\begin{equation*}
		\begin{tikzcd}
			M^0_{\chat}\ar[r,dashed, leftrightarrow]\ar[d,"{p_*}"]&[-.8 em] M^1_{\chat}\ar[r,dashed,leftrightarrow]&\cdots\ar[r,dashed,leftrightarrow] & M^{m-1}_{\chat}\ar[r,dashed,leftrightarrow]&[-.5 em] M^{m}_{\chat}\ar[d,equal,shorten <= .5 ex,shorten >= .5ex] \ar[r,equal,shorten <= 1 em,shorten >= 1em] &[-.8 em] \cdots\\
			M_X(r,c_1,c_2)& & & & M_{\Xhat}(r,\chat_1,\chat_2) &
		\end{tikzcd}
	\end{equation*}
	All of the spaces $M^m_{\chat}$ have natural virtual classes and we may consider virtual integrations on them.

	\paragraph{Main results.} 
	We state the main results. For technical reasons, we work with the moduli \emph{stacks} of $\mathcal{M}_X(r,c_1,c_2)$ (also denoted $\mathcal{M}_X(c)$, where $c=r+c_1+c_1^2/2-c_2$) of $H$-Gieseker stable oriented sheaves on $X$  rather than with the coarse spaces $M_{X}(r,c_1,c_2)$. Similarly we consider the moduli stacks $\mathcal{M}^m_{\chat}$ (also denoted $\mathcal{M}^m(\chat)$, or $\mathcal{M}^m$ when $\chat$ is understood) of oriented $m$-stable sheaves on $\Xhat$. The same picture described in the above 
	holds for the moduli stacks in place of the coarse moduli spaces. We assume throughout that $r$ and $\deg_H c_1$ are coprime.
	  
	Our Theorems \ref{th:mainthwallcrossing}, \ref{th:compthm} and \ref{thm:structhm} are analogous to Nakajima--Yoshioka's results about moduli spaces of framed sheaves on $\mathbb{P}^2$, specifically to Theorem 1.5, Proposition 1.2 and Theorem 2.6 in their paper \cite{NaYo3} respectively. We largely follow their notation in stating these results.
	
		The main technical result in this paper is a {\it weak wall-crossing formula} giving an expression for the difference of virtual integrals on the spaces $\mathcal{M}^{m-1}_{\chat}$ and $\mathcal{M}^m_{\chat}$.
		To state it, let $\Phi=\Phi(-)$ be a Chow cohomology class on the moduli stack of coherent sheaves on $\Xhat$. More concretely, $\Phi$ should be any rule of the following form: For a finite type Artin stack $U$  and a $U$-flat coherent sheaf $\mathcal{E}$ on $\Xhat\times U$, it associates a Chow cohomology class $\Phi(\mathcal{E})$ on $U$, in a way that is compatible with pullback of sheaves along $U'\to U$. Let $C_m:=\mathcal{O}_C(-m-1)$ denote the coherent sheaf on $\Xhat$ which is the pushforward of the unique degree $-m-1$ line bundle on $C$ and set $e_m:=\ch(C_m)$. For an equivariant parameter $t$, let $e^t$ denote the trivial line bundle with first Chern class $t$ and let $\operatorname{Eu}_t$ denote the equivariant Euler class with respect to $t$. 
		\begin{theorem}[Weak wall-crossing formula for the blowup]\label{th:mainthwallcrossing}
			We have the following equality:
			\begin{gather*}
				\int_{[\mathcal{M}^{m+1}(\chat)]^{\vir}} \Phi(\mathcal{E}^{m+1}_{\chat})-\int_{[\mathcal{M}^{m} (\chat)]^{\vir}} \Phi(\mathcal{E}^m_{\chat})\\
				= \sum_{j=1}^{\infty}\frac{1}{j!}\int_{[\mathcal{M}^m(\chat-je_m)]^{\vir}}\underset{t_1=0}{\operatorname{Res}}\cdots \underset{t_j=0}{\operatorname{Res}}\, \Phi\left(\mathcal{E}^m_{\chat-je}\oplus \bigoplus_{i=1}^j C_m\otimes e^{-t_j} \right) \Psi^j_m(\mathcal{E}),
			\end{gather*}
			where
			\[\Psi^j_m(\mathcal{E}) =\frac{1}{j!}\frac{\prod_{1\leq i_1\neq i_2 \leq j} (t_{i_1}-t_{i_2}) }{\prod_{i=1}^j\left(\operatorname{Eu}_{t_i}(\mathfrak{N}(\mathcal{E},C_m) \otimes e^{-t_i})\cdot \operatorname{Eu}_{t_i}(\mathfrak{N}(C_m,\mathcal{E}) \otimes e^{t_i}) \right)},\]
			and where 
			\[\mathfrak{N}(\mathcal{E},\mathcal{F}):=R\Hom_{\pi_2}(\mathcal{E},\mathcal{F})[1]\]
			for any coherent sheaves $\mathcal{E},\mathcal{F}$ on $\Xhat\times \mathcal{M}^m(\chat-je_m)$.
		\end{theorem}

		In particular, by a repeated application of Theorem \ref{th:mainthwallcrossing}, one may express a virtual integral over $\mathcal{M}_{\Xhat}(\chat)$ in terms of integrals over the spaces $\mathcal{M}^0(\chat-j_1e_1-\cdots-j_me_m)$ for some $m>0$ depending on $\chat$, and where $j_1,\ldots,j_n$ vary over the non negative integers. 
		In order to obtain blowup formulae, one needs a way to go from a space of the form $\mathcal{M}^0_{\chat}$ for varying $\chat$ back to the space $\mathcal{M}_X(r,c_1,n)$ for varying $n$. It turns out that the morphism $p_*$ is not as useful for this purpose, since pushing forward the virtual class along this morphism will introduce higher Chern classes of the universal sheaf. In the case $\chat=p^*c$ the morphism $p_*$ is an isomorphism and one directly gets an integral over a space $\mathcal{M}_X(r,c_1,n)$. In general, Nakajima and Yoshioka discovered the following method to get around this issue: Given a Chern character of the form $\chat=p^*c-j[C]$, we have an isomorphism $\mathcal{M}^0(\chat)=\mathcal{M}^1(\chat(C))$, where now $\chat(C)=p^*c+(r-j)[C]+j-r/2$. In particular, if $j\geq r$, one may apply Theorem \ref{th:mainthwallcrossing} again and proceed by induction on $j$. This introduces new wall-crossing terms, which one takes care of by a descending induction over the virtual dimension of the moduli spaces. The interesting case occurs when $0<j<r$. Then one needs to consider the following situation: 
		
		Set $\chat:= p^*c+j\chexc$ for $0<j<r$. Then there exists a proper moduli stack $\mathcal{Q}(\chat,j)$ defined in \S \ref{subsec:cohsys} and morphisms:  
		\begin{equation*}
					\begin{tikzcd}
						& \mathcal{Q}(\chat,j)\ar[dr,"q_2"]\ar[dl,"q_1"']& \\
						\mathcal{M}^1_{\chat}& &\mathcal{M}^1_{p^*c} 
					\end{tikzcd}
				\end{equation*}
		The point is that the moduli stack $\mathcal{M}^1(p^*c)$, after wall-crossing, can be directly related to a moduli space of sheaves on $X$. Nakajima and Yoshioka show that $q_2$ is a certain relative Grassmann bundle and that, when the moduli spaces are unobstructed, the map $q_1$ is birational. We show that one can recover part of this picture for virtual integrals. Let $\Phi$ be as in the setup for Theorem \ref{th:mainthwallcrossing}. Let $\mathcal{V}$ be the universal subbundle coming from the Grassmann bundle structure on $q_2$. Denote by $\mathcal{E}$ the universal bundle on $\mathcal{M}^1(\chat)$ and by $\mathcal{F}$ the one on $\mathcal{M}^1(p^*c)$.
		
		\begin{theorem}[Push-down formula]\label{th:compthm}
			The stack $\mathcal{Q}(\chat,j)$ has a natural virtual fundamental class $[\mathcal{Q}(\chat,j)]^{\vir}$ with the following properties
			\begin{enumerate}[label=\arabic*)]
				\item We have an equality of virtual integrals 
				\[\int_{ [\mathcal{M}^1(\chat)]^{\vir}} \Phi(\mathcal{E}) =\int_{[\mathcal{Q}(\chat,j)]^{\vir}}\Phi(q_2^*\mathcal{F}\oplus \mathcal{V}\otimes \mathcal{O}_C(-1)).\]
				\item  The class $[\mathcal{Q}(\chat,j)]^{\vir}$ agrees with the flat pullback of $[\mathcal{M}^1(p^*c)]^{\vir}$ along $q_2$.	
		\end{enumerate}
		\end{theorem}

		\paragraph{Blowup formula for virtual Donaldson invariants.}
		
		Using a slight variation of Theorems \ref{th:mainthwallcrossing} and \ref{th:compthm}, we obtain the following blowup formula for the Donaldson--Mochizuki invariants, which is a direct analogue of the one given by Friedman and Morgan \cite[\S 3.8, Theorem 8.1]{FrMo}. 
		\begin{theorem} \label{th:blowupdon} Let $c_1\in H^2(X,\Z)$ be an algebraic class with $\deg_H c_1$ odd. Let $\sigma_1,\ldots \sigma_n\in H^*(X,\mathbb{Q})$. We have the following relations
			\begin{enumerate}[label=(\arabic*)]  
				\item $D_{\widehat{X},p^*c_1}(\sigma_1\ldots\sigma_n) = D_{X,c_1}(\sigma_1,\ldots,\sigma_n)$,
				\item $D_{\widehat{X},p^*c_1}(\sigma_1\ldots\sigma_n [C]^{i}) = 0$, for $1\leq i\leq 3$,
				\item $D_{\widehat{X},p^*c_1}(\sigma_1\ldots\sigma_n[C]^{4}) = -2 D_{X,c_1}(\sigma_1,\ldots,\sigma_n)$. \label{th:blowupdon3}
			\end{enumerate}
		\end{theorem}
		The proof will be given in \S \ref{subsec:donmoch}. 
		It follows from our assumption on $c_1$ that $H$ does not lie on a wall of type $c_1$ in the notation of \cite[\S 1.1]{GoNaYo1}. Therefore, the classical invariants agree up to a sign with the algebraic invariants as outlined there and we have as a consequence of Theorem \ref{th:blowupdon}:
		\begin{corollary}
			For any simply-connected projective surface and $\deg_H {c_1}$ odd, the virtual Donaldson invariant agrees with the classical one up to a global constant. 
		\end{corollary}
		
		\paragraph{Weak general structure theorem.}
		Using Theorem \ref{th:mainthwallcrossing}, one can express integrals over the moduli space of stable sheaves on $\Xhat$ in terms of integrals over moduli space of stable sheaves on $X$ with lower virtual dimension. Ideally, we would hope to use Theorem \ref{th:mainthwallcrossing} to prove formulas such as \eqref{eq:blowupeulerchar} for the virtual Euler characteristic and its generalization to the virtual $\chi_y$-genus. It is not clear how to do this for essentially two reasons: First, the terms appearing in our wall-crossing formula are complicated and it is unclear whether one can structure them in a neat way. Second, if one is only interested in invariants of a certain type, like the virtual Euler characteristic, the wall-crossing terms will usually be more general integrals over the moduli space that require various insertions. 
		However, for a wide range of choices of invariants one can restrict the possible insertions that one has to consider. This includes in particular the virtual $\chi_y$-genus, see Example \ref{ex:insertions} \ref{ex:insertions3}. 
		Let $\Phi$ be as in the setup of Theorem \ref{th:mainthwallcrossing} and suppose additionally that $\Phi$ can be evaluated for families of torsion free sheaves $X$ over a base $U$ in a way that is compatible with pullback along $p$. For a power series $P\in \Q[[\nu_2,\ldots,\nu_r]]$, let $P(\mathcal{E})$ be the expression obtained by replacing each occurence of $\nu_i$ in $P$ with $\ch_i(\mathcal{E})/[pt]$. Then we have
		
		 \begin{theorem}\label{thm:structhm} 
		 	Fix an admissible Chern class $\chat = p^*c-j\chexc$ for some $j\geq 0$ and where $c=r+c_1+c_2$. Suppose that $\Phi$ satisfies Assumption \ref{assum:goodphi}. Then there exist universal power series $\Omega_n\in \mathbb{Q}[[\nu_2,\ldots,\nu_r]]$ that depend only on $r,\Phi$ and $j$ which satisfy:
		 	
			\begin{equation}\label{eq:structeq}
				\int_{[\mathcal{M}_{\Xhat}(\chat)]^{\vir}} \Phi(\mathcal{E}) = \sum_{n=0}^{\infty} \int_{[\mathcal{M}_X(c+n[pt])]^{\vir}} \Phi(\mathcal{E})\Omega_n(\mathcal{E}).
			\end{equation}
		This holds for the fixed and non-fixed determinant spaces. 
		 \end{theorem}
		The proof will be given in \S \ref{subsec:structhm}.
		\begin{remark}
		In \cite[Theorem 2.6]{NaYo3}, Nakajima and Yoshioka prove the analogous result for multiplicative classes for moduli spaces of framed sheaves, but their proof should go through for more general $\Phi$ as considered here. Since in their case the integral takes values in the equivariant cohomology ring, it turns out that for many choices of $\Phi$ the $\Omega_n$ are uniquely determined by the equations \eqref{eq:structeq} when one lets $\chat_2$ vary over the integers. But at the same time, their $\Omega_n$ should agree with ours after forgetting the equivariant parameters. Therefore, explicit determination of the $\Omega_n$ in the framed setting should imply an explicit general formula.  
		\end{remark} 
		\begin{remark}
			\begin{enumerate}[label=\arabic*)]
				\item Suppose $r=1$ and for fixed $\Phi$ consider the generating series for the Hilbert schemes of points
				\[Z_{X}(z):=\sum_{n}z^n\int_{X^{[n]}}\Phi(\mathcal{I}_n),\]
				where $\mathcal{I}_n$ denotes the universal ideal sheaf. Note that the virtual and the usual fundamental class agree in this situation. 
				Then Theorem \ref{thm:structhm} gives the following relation:
				\[ Z_{\Xhat}(z)= A(z) Z_{X}(z)\]
				for some power series $A(z)$.
				This recovers a consequence of the stronger universality results for integrals over Hilbert schemes of points by Ellingsrud--G\"ottsche--Lehn \cite{ElGoLe}.
				\item For $r=2$, it follows from the proof that the $\Omega_n$ are unchanged when replacing $\Phi(\mathcal{E})$ by $\Phi(\mathcal{E})(\ch_2(\mathcal{E})/[pt])^k$. For the fixed determinant spaces, we can then consider the generating series 
				\[Y_{X,c_1}(z,w) :=\sum_{n\in \Z}\sum_{m=0}^{v(n)/2} z^{v(n)} w^{-2m} \int_{[\mathcal{M}_{X}(2,c_1,n)]^{\vir}} \Phi(\mathcal{E})(\ch_2(\mathcal{E})/[pt])^m,\]
				where $v(n)=4n-c_1^2-3\chi(\mathcal{O}_X)$.
				Then we have for $a=0,1$:
				\[Y_{\Xhat,p^*c_1-a[C]}(z,w) =Y_a(z,w)Y_{X,c_1}(z,w) \]
				for some power series $Y_0(z,w)$ and $Y_1(z,w)$ depending only on $\Phi$. 
			\end{enumerate}
		\end{remark}

		In view of the predicted formula for the virtual $\chi_y$-genus, we have the following natural
		\begin{conjecture}
			Let $\Phi$ be as in Example \ref{ex:insertions} \ref{ex:insertions3} and let $Z_a$ be as in \eqref{eq:chiyblowup}. Then we have 
			\[Y_a(z,w)=Y_a(z,0)=Z_a(y,z^4).\]
			Here the coefficients of $Y_a$ are implicitly taken to be polynomials in $y$. 
		\end{conjecture}

		\paragraph{Strategy of the proof.} 
		The basic fact that will allow us to prove the wall crossing Theorem \ref{th:mainthwallcrossing} is that the stability conditions defining $\mathcal{M}^m$ and $\mathcal{M}^{m+1}$ are related by an intermediate stability condition ``on the wall'', already considered in \cite{NaYo2}, which we call $(m,m+1)$-semistability. This intermediate notion of stability defines a moduli stack $\mathcal{M}^{m,m+1}$ of semistable objects so that we have open embeddings
		\[\mathcal{M}^m\hookrightarrow \mathcal{M}^{m,m+1}\hookleftarrow \mathcal{M}^m.\]
		Unless $\mathcal{M}^m=\mathcal{M}^{m+1}$, the stack $\mathcal{M}^{m,m+1}$ will have points with positive-dimensional automorphism groups. The difference of integrals over $\mathcal{M}^m$ and $\mathcal{M}^{m+1}$ will eventually be expressed in terms of exactly those stacky loci. The basic idea is to use the intrinsic geometry of the moduli problem defining $\mathcal{M}^{m,m+1}$ to create a \emph{master space} from which a wall-crossing formula can be obtained by $\C^*$-localization. Recall that a master space is a proper space with a $\C^*$-action that contains the moduli spaces in question as components of its $\C^*$-fixed locus. We proceed in several steps:
		\begin{enumerate}[label=\arabic*.]
			\item \emph{Reduce to the case of one-dimensional automorphism groups.} The stabilizer groups of $\mathcal{M}^{m,m+1}$ can be of arbitrarily high dimension, which makes it hard to work with this stack directly. In order to simplify the problem we work with an \emph{enhanced moduli} stack $\mathcal{N}\to \mathcal{M}^{m,m+1}$ whose objects are pairs of an $m,m+1$-semistable oriented sheaf $E$ on $\Xhat$ in $\mathcal{M}^{m,m+1}$ together with a full flag in the space of global sections of a sufficiently ample twist of $E$. Then we define a series of new stability and semistability conditions incorporating the data of the flag which cut out proper Deligne--Mumford stacks $\mathcal{N}^{\ell}$ and intermediate Artin stacks $\mathcal{N}^{\ell,\ell+1}$ resulting in the following picture:
		\begin{equation*}
				\begin{tikzcd}[column sep = small]
					 & \mathcal{N}^{0,1} & &\cdots & &\mathcal{N}^{N-1,N}& \\
					\mathcal{N}^0\ar[d]\ar[ur,hook]\ar[rr,dashed, leftrightarrow]& &\mathcal{N}^1\ar[r, dashed, leftrightarrow]\ar[ul, hook]&\cdots & \mathcal{N}^{N-1}\ar[l, dashed, leftrightarrow]\ar[ur,hook]\ar[rr,dashed,leftrightarrow]& &\mathcal{N}^{N} \ar[ul, hook]\ar[d] 	\\
					\mathcal{M}^{m}& & & & & & \mathcal{M}^{m+1}
				\end{tikzcd}
			\end{equation*}
		Here $N$ is the length of the flag, and the vertical maps are relative flag bundles. This allows us to break the original wall-crossing problem into a series of simpler steps, as now each of the spaces $\mathcal{N}^{\ell,\ell+1}$ has at most one-dimensional automorphism groups at each point. 
		\item \emph{Create a master space using the intrinsic geometry of the moduli spaces.}
		For fixed $\ell$, we prove a wall-crossing result for $\mathcal{N}^{\ell}$ and $\mathcal{N}^{\ell+1}$ by using the universal sheaf on $\mathcal{N}^{\ell,\ell+1}$ to construct a master space. One needs to take care to establish the desired properties, such as properness,  of the master space and to identify the contributions of the $\C^*$-fixed loci.
				Once this is done, we can sum the resulting wall-crossing terms to obtain the desired formula for $\mathcal{M}^m$ and $\mathcal{M}^{m+1}$.  
		\item \emph{Construct perfect obstruction theories.} A technical but important point in our work is the construction of perfect obstruction theories on all involved moduli spaces and to show that they are all mutually compatible.  Some of the more delicate points involving obstruction theories are to show that the fixed-loci of the $\C^*$-action on the master space have the expected obstruction theories as well as showing  the necessary compatibilities needed to establish Theorem \ref{th:compthm}.		We make repeated use of the Atiyah class on algebraic stacks and its basic properties. Unfortunately, there appear to be many folklore results for which there is no suitable reference. This will be addressed in \cite{Kuhn}, whose main results are summarized in Appendix \ref{app:atiyah}.

		\end{enumerate}

		We learnt the strategy to constuct the master space used in the second step from work of Kiem and Li. 
		The method of using a flag structure and the creation of the spaces $\mathcal{N}^{\ell}$ is parallel to \cite{NaYo2} and based on Mochizuki's work in \cite{Moch}. Our contribution is that we work entirely with the moduli stack of sheaves and do not rely on a connection to quiver representations or on GIT methods. We believe that our method is conceptually simpler and has a greater potential to generalize beyond the given setting.

	\paragraph{Construction of the master space.}
		We give some motivation and explain the strategy of Kiem--Li's construction of the master space that we adapt. Since this works very generally, we keep the discussion somewhat abstract. For each object in $E \in \mathcal{N}^{\ell,\ell+1}(\C)$, there are two associated filtrations, one for each of the two stability conditions defining $\mathcal{N}^{\ell}$ and $\mathcal{N}^{\ell+1}$ respectively. By our assumptions on stabilizer groups, each of these filtrations is of length at most two and of length one if and only if the object is stable with respect to the stability condition. The object $E$ is called \emph{properly polystable} if and only if one of these filtrations splits $E$ as a direct sum. In that case, the other filtration is the opposite one induced by the splitting. It turns out that $E$ is polystable if and only if it has positive-dimensional stabilizer group. In particular, any object $E$ that is not in $\mathcal{N}^{\ell+1}$ (and therefore has a nontrivial corresponding filtration) can be degenerated to a polystable object which is the associated graded. This is illustrated in Figure \ref{fig:1}.
		\begin{figure}[h]
			\includegraphics[width=.75\textwidth]{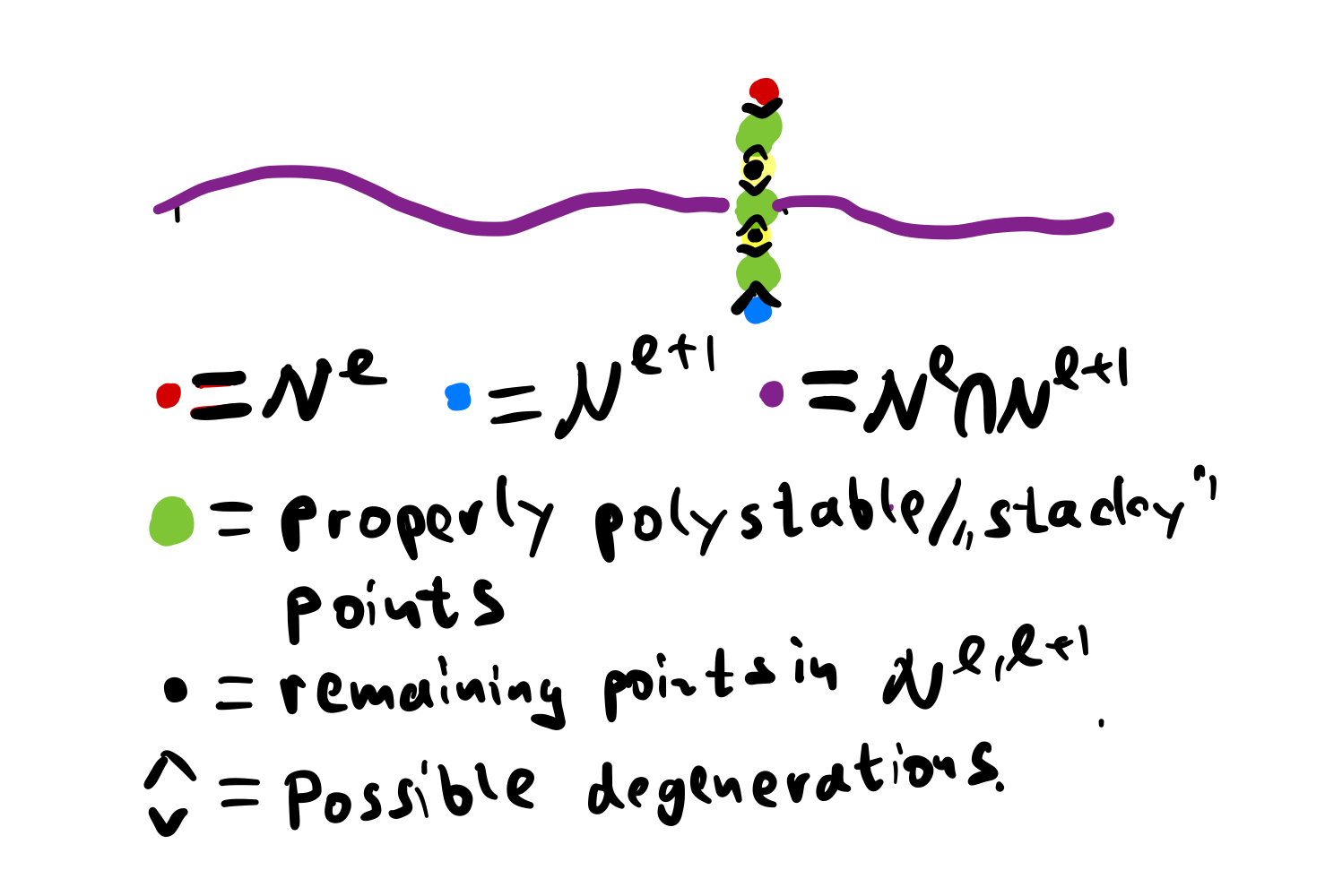}
			\centering 
			\caption{Illustration of $\mathcal{N}^{\ell,\ell+1}$.}
			\label{fig:1}
		\end{figure}
		
		Now suppose that $E$ is a polystable object of $\mathcal{N}^{\ell,\ell+1}$, so that $E$ decomposes as a direct sum $S\oplus T$. The $\C^*$-action on $E$ given by $(t,t^{-1})$ induces an action on the deformation space of $E$. The deformation space is stratified according to the limiting behavior of points under the $\C^*$-action, as shown in Figure \ref{fig:2}. 
		\begin{figure}[h]
  			\includegraphics[width=.9\textwidth]{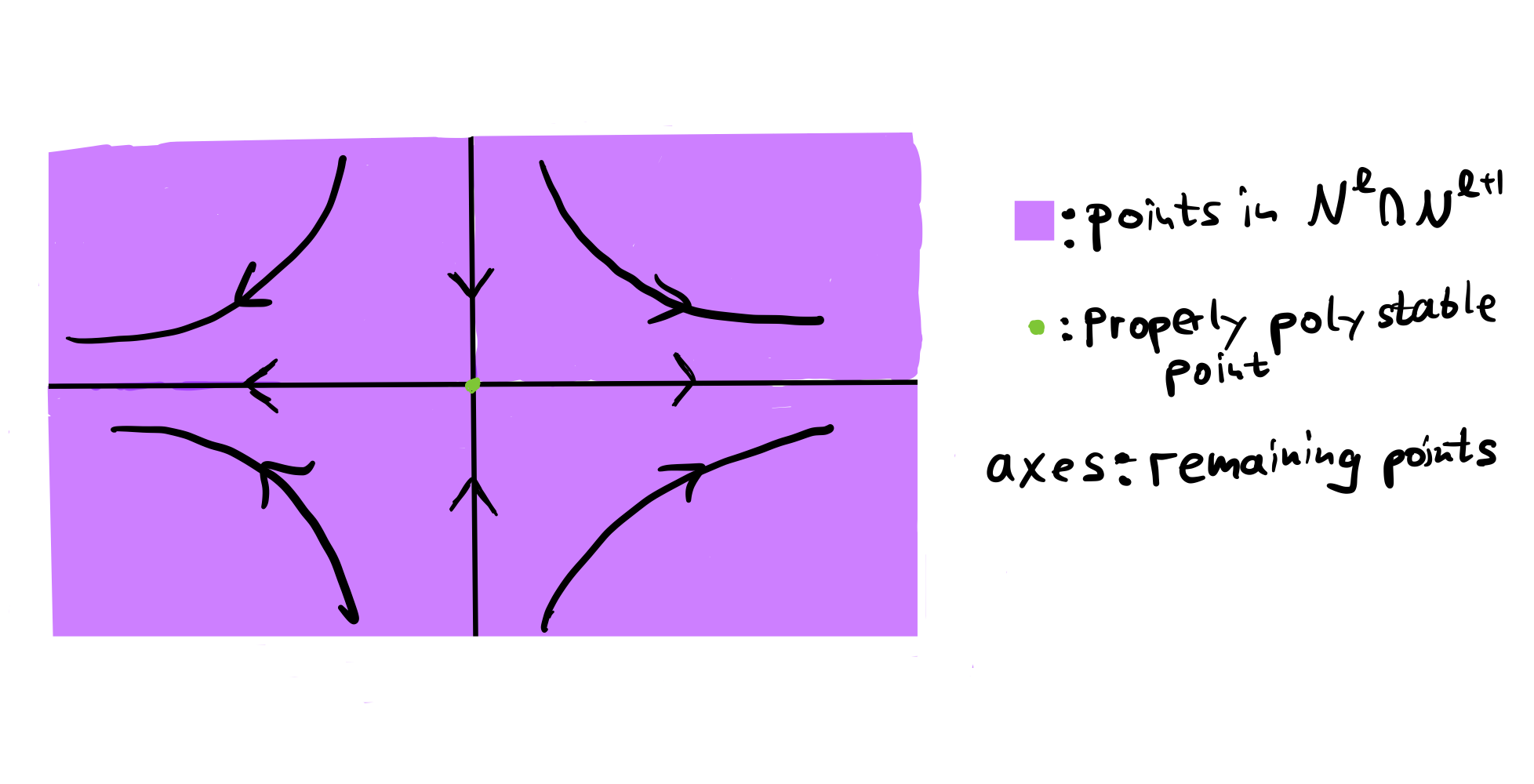}
			\centering 
			\caption{Deformation space of a polystable point.}
			\label{fig:2}
		\end{figure}
		
		The arrows in Figure \ref{fig:2} indicate the natural action of the stabilizer group (after some choice of orientation). By taking the quotient by the $\C^*$-action, one obtains a model for an \'etale neighborhood of $E$ in $\mathcal{N}^{\ell,\ell+1}$. The idea is now to write $\mathcal{N}^{\ell,\ell+1}$ as a global $\C^*$-quotient of some Deligne--Mumford stack $\mathcal{W}$, such that around every polystable point, we recover the situation illustrated in Figure \ref{fig:2}. To do this, we determine a suitable line bundle $\mathcal{L}$ on $\mathcal{N}^{\ell,\ell+1}$ such that for any point properly polystable point the automorphism group acts non-trivially (and with a certain orientation) on $\mathcal{L}$. Then we take $\mathcal{W}$ to be the total space of $\mathcal{L}^{\vee}$ with its natural $\C^*$-action (taking the dual is just a matter of sign choice). To find the line bundle $\mathcal{L}$, one uses determinant line bundles associated to pushforwards of the universal sheaf over $\mathcal{N}^{\ell,\ell+1}$. The resulting stack $\mathcal{W}$ is illustrated in Figure \ref{fig:3}.
		
		\begin{figure}[h]
			\includegraphics[width=.9\textwidth]{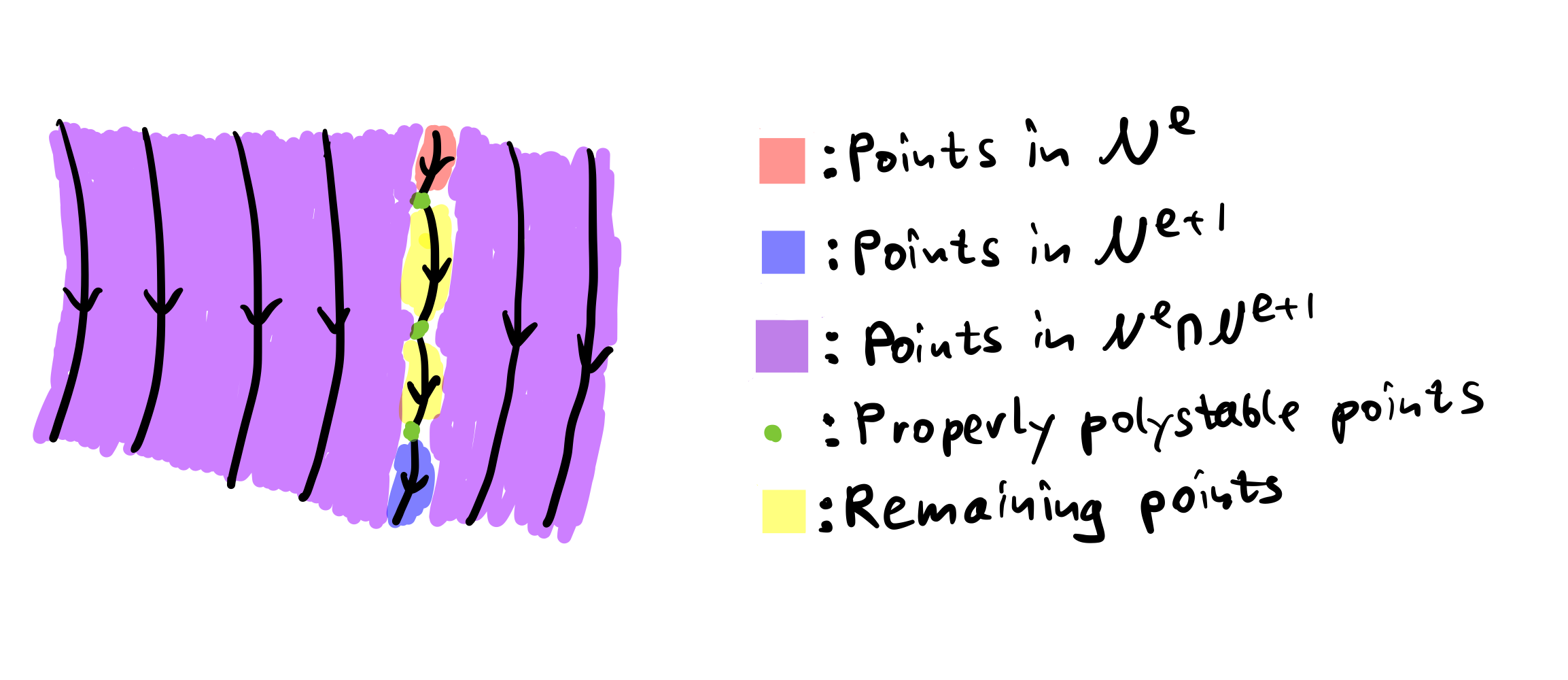}
			\centering 
			\caption{Illustration of $\mathcal{W}$.}
			\label{fig:3}
		\end{figure}
		
		To obtain the master space $\mathcal{Z}$ one ``completes'' $\mathcal{W}$ by adding in a copy of $\mathcal{N}^{\ell+1}$ at the ``top'' and of $\mathcal{N}^{\ell}$ at the ``bottom''. Concretely, one considers the open subset 
		\[\left(\mathcal{W}\times \mathbb{P}^1\right)^s\subset \mathcal{W}\times \mathbb{P}^1\]
		obtained by removing the preimage of $\mathcal{N}^{\ell,\ell+1}\setminus \mathcal{N}^{\ell+1}$ in the fiber over $0\in \pone$ and the preimage of $\mathcal{N}^{\ell,\ell+1}\setminus \mathcal{N}^{\ell}$ in the fiber over $\infty\in \pone$. Then $\mathcal{Z}$ is the quotient by the diagonal $\C^*$ action on $\left(\mathcal{W}\times\pone\right)^s$ and has a natural $\C^*$-action extending the action on $\mathcal{W}$.

		\paragraph{Structure of this paper.}
		
		We begin by giving a short review of the analytic Donaldson invariant and their algebraic counterparts in Section \ref{sec:background}
		
		In Section \ref{sec:perverse}, we review the basic properties of perverse coherent sheaves on the blowup of a smooth projective surface and construct their moduli spaces. We begin by reviewing the definition and basic properties of stable perverse coherent (also called \emph{$0$-stable}) sheaves on the blowup in \S \ref{subsec:pervcoh}. For us, these are exactly the sheaves that arise as pullbacks from Gieseker stable sheaves from $X$ or by quotients thereof by subsheaves of the form $\mathcal{O}_C(-1)$. In \S \ref{subsecmodspaces} we therefore construct the moduli spaces of oriented $0$-stable sheaves as Quot-schemes over the moduli space of oriented Gieseker stable sheaves on $X$. By twisting, they give rise to moduli stacks of $m$-stable sheaves introduced in \S \ref{subsec:mstable}, where we also prove that for $m \gg 0$, we recover the stack of Gieseker stable sheaves on the blowup. We conclude this section by setting up the basic geometric picture needed for the pushdown formula Theorem \ref{th:compthm} in \S \ref{subsec:cohsys}. 
		
		In Section \ref{sec:master}, we prepare the proof of Theorem \ref{th:mainthwallcrossing} by implementing the strategy outlined in the above. 
		We first define the stack $\mathcal{M}^{m,m+1}$ of $m,m+1$-semistable sheaves and show its basic properties in \S \ref{subsec:semistable}. In particular, we show that every $\C$-valued object of this stack that is not $m$ (respectively $m+1$) stable has a good notion of maximal $m$ (respectively $m+1$)-destabilizing subobject. In \S \ref{subsecunderlyingstack} we consider the notion of \emph{sheaf with a flag structure} and introduce a corresponding relative stack of flag structures $\mathcal{N}\to \mathcal{M}^{m,m+1}$. We define the notion of stability with respect to a given piece of the flag, which gives rise to the stacks $\mathcal{N}^{\ell}$ of \emph{flag stable objects}. Similarly, we define the notion of semistability for two adjacent pieces of the flag giving rise to the stacks $\mathcal{N}^{\ell,\ell+1}$. We define the notion of properly polystable sheaves and show the basic properties of these stacks. We prove properness of $\mathcal{N}^{\ell}$ via the usual valuative criterion. The existence part is somewhat delicate and is done in \S \ref{subsecpropex} using elementary modifications. We construct the master space $\mathcal{Z}$ for the spaces $\mathcal{N}^{\ell}$ and $\mathcal{N}^{\ell+1}$ in \S \ref{subsecmasterspace} and show that it is a proper Deligne--Mumford stack. In \S \ref{subsec:fix}, we analyze the stack-theoretic fixed locus $\mathcal{Z}^{\C^*}$ and give an explicit moduli-theoretic description of its components. 
		
		Section \ref{sec:perfob} is about constructing compatible perfect obstruction theories on all spaces appearing throughout Sections \ref{sec:perverse} and \ref{sec:master}. We review some background in \S \ref{subsec:obthy}, in particular the notions of compatibility and \emph{virtual pullbacks}. We construct the obstruction theory for spaces of oriented sheaves on a surface in \S \ref{subsecobstrbasic}, and the relative obstruction theories for the maps $\mathcal{N}\to \mathcal{M}^{m,m+1}$ and $\mathcal{W}\to \mathcal{N}^{\ell,\ell+1}$ in \S \ref{subsec:relobs} where we also establish their compatibilities. \S \ref{subsec:obstrfix} describes the obstruction theory on the master space. The most subtle point of this section is the obstruction theoretic analysis of the components of $\mathcal{Z}^{\C^*}$ done in \S \ref{subsec:fix}: There we determine the virtual normal bundle and show that the perfect obstruction theory induced from the $\C^*$-action on $\mathcal{Z}$ induces the same fundamental class as the one coming from their intrinsic moduli-theoretic characterization. Finally, in \S \ref{subsec:obstrquot} we construct the obstruction theory on the space $\mathcal{Q}(\chat,j)$ appearing in Theorem \ref{th:compthm} and show its compatibilities. 
		
		In Section \ref{sec:technicalresults}, we prove our technical main results. The proof of Theorem \ref{th:mainthwallcrossing} via virtual localization on the master space $\mathcal{Z}$ is given in \S \ref{subsec:wallcr}. There, we use the results of \S \ref{subsec:fix} and \S \ref{subsec:obstrfix} to compute the contributions of $\C^*$-fixed loci to get a wall-crossing result for $\mathcal{N}^{\ell}$ and $\mathcal{N}^{\ell+1}$. Then careful summation over all $\ell$ gives the wall-crossing result for $\mathcal{M}^m$ and $\mathcal{M}^{m+1}$. The proof of Theorem \ref{th:compthm} is given in \S \ref{subsec:compthm} and amounts to a closer analysis of the compatibilities proven in \S \ref{subsec:obstrquot}.
		
		Finally, these results are applied to the study of blowup formulae in Section \ref{sec:MainRes}. In \S \ref{subsec:donmoch}, we give the proof of Theorem \ref{th:blowupdon}. This amounts to an explicit calculation using the wallcrossing and pushdown formulas. A subtle point which is not addressed in detail here is that one may want to define the Donaldson invariant using insertions coming from non-algebraic classes on $X$. The last \S \ref{subsec:structhm} contains the proof of Theorem \ref{thm:structhm}.

		\vspace{0.3cm}
		\noindent
		{\it Acknowledgements.} 
		The first-named author was partially supported by NSF grant DMS-160121. 
The second-named author was partially supported by the Simons Collaboration on Special Holonomy in Geometry, Analysis and Physics and JSPS Grant-in-Aid for Scientific Research numbers JP16K05125, 16H06337, 20K03582, 20K03609. 
The authors would like to thank Tom Graber for helpful discussions about perfect obstruction theories and  Rachel Webb for explanations about the cotangent complex on algebraic stacks. We also thank	Hiraku Nakajima and Kota Yoshioka for explanations of their work. Finally, we would like to thank Martijn Kool and Rafe Mazzeo for helpful conversations and ongoing support.  
	
		\vspace{0.3cm}
	\noindent
		A version of this article will appear as part of the first-named authors PhD thesis. 

		\subsection*{Conventions and index of notations} 
		We often suppress pullbacks along the projections from a product: For example, we denote by $\mathcal{O}_C(-1)$ the pullback of that sheaf along any projection $\Xhat\times U\to \Xhat$ for $U$ an algebraic stack. Similarly, use the notation $\mathcal{E}_U$ to denote a naturally given sheaf on $\Xhat\times U$, for example by a universal property satisfied by $U$, or if $\mathcal{E}_U$ is the pullback of a sheaf $\mathcal{E}_{V}$ on $\Xhat\times V$ along a map $U\to V$.
		If $\mathcal{M}$ is a proper Deligne--Mumford stack with given perfect obstruction theory and $\alpha$ a Chow cohomology class, we write  
		\[\int_{\mathcal{M}} \alpha\]
		to denote the rational number obtained as the degree of $\alpha \cap [\mathcal{M}]^{\vir}$. 
		All the stacks we work with are quasi-separated and locally of finite type over $\Spec \C$.
		We make a list of the most important symbols used throughout the text, sorted by where they are introduced in the main text \S\S \ref{sec:perverse}--\ref{sec:MainRes}. 
		\subsubsection*{Introduction 
		}
		\begin{itemize}
			\item  $C_m:=\mathcal{O}_C(-m-1)$ 
			\item $\chexc_m:= \ch(C_m)$ with $e:=e_0$.
		\end{itemize}
		\subsubsection*{Subsection \ref{subsec:pervcoh}}
		\begin{itemize}	
			\item $r,c_1$: Fixed choice of rank $r\geq 1$ and first Chern class $c_1\in H^2(X,\Z)$.
			\item $c$ : An object in $\Z\oplus H^2(X,\Z)\oplus \frac{1}{2}\Z$ of the form $r+c_1+\ch_2$, interpreted as a \emph{Chern character} of sheaves.
			\item $\chat$ : An admissible Chern character on $\chat$
		\end{itemize}
	\paragraph{Subsection \ref{subsecmodspaces}}	
		\begin{itemize}
			\item $\mathcal{M}_{X,c}, \mathcal{M}_X(c)$: The moduli stack of oriented Gieseker-stable sheaves of Chern character $c$ (with respect to the fixed polarization $H$).
			\item $\mathcal{M}^p_{\chat},\mathcal{M}^0_{\chat}$: The moduli stack of stable perverse coherent (or $0$-stable) sheaves.
			\item  $\widetilde{\mathcal{M}}_{X,c},\widetilde{\mathcal{M}}^p_{\chat}$: The unoriented variants of the moduli spaces.
		\end{itemize}
	\subsubsection*{Subsection \ref{subsec:mstable}}
		\begin{itemize}	
			\item $\mathcal{M}^m,\mathcal{M}^m_{\chat}$: The moduli stack of oriented $m$-stable sheaves.
			
			\end{itemize}
  \subsubsection*{Subsection \ref{subsec:cohsys}}
\begin{itemize}	
			
			\item $\mathcal{Q}(\chat,j),q_1,q_2$: The relative Quot-scheme over $\mathcal{M}^1_{\chat}$, where $\chat=p^*c+j\chexc$, with structure map $q_1$ and forgetful map $q_2$ to $\mathcal{M}^1_{p^*c}$.
			
		\end{itemize}
			\subsubsection*{Subsection \ref{subsec:semistable}}
			\begin{itemize}	

			\item $\mathcal{M}^{0,1}, \mathcal{M}^{0,1}_{\chat}$: The moduli stack of oriented $0,1$-semistable sheaves with Chern character $\chat$.

			\end{itemize}
		\subsubsection*{Subsection \ref{subsecunderlyingstack}}
		\begin{itemize}	
			
			\item $L$: A fixed ample line bundle on $\Xhat$.
			\item $m$: A fixed large enough integer.
			\item $\globm{\mathcal{E}}$: The sheaf $\pi_{U*}(\mathcal{E}\otimes L^{m})$, for $\mathcal{E}$ a coherent sheaf on $\Xhat\times U$.
			\item $N:=\dim \globm{E}$ for $E$ a coherent sheaf on $\Xhat$ with $\ch(E)=\chat$, $d= \dim \globm{\mathcal{O}_C(-1)}$.
			\item $\mathcal{N}:=\mathcal{N}_{\chat}$: Full flag bundle for the vector bundle $\globm{\mathcal{E}_{\mathcal{M}^{0,1}}}$ over $\mathcal{M}^{0,1}$.
			\item $(\stabS_{\ell}),(\stabT_{\ell})$: Conditions of S-and T-stability for the $\ell$-th piece of the flag.
			\item $\mathcal{N}^{\ell}$: The open substack of $\mathcal{N}$ defined by $(\stabS_{\ell})$ and $(\stabT_{\ell})$.
			\item $\mathcal{N}^{\ell,\ell+1}$: The open substack of $\mathcal{N}$ defined by $(\stabS_{\ell})$ and $(\stabT_{\ell+1})$.
						\end{itemize}
		\subsubsection*{Subsection \ref{subsecmasterspace}}
	\begin{itemize}	
			\item $\detL(-)$: A functor associating to a sheaf $\mathcal{E}$ on $\Xhat\times U$ a line bundle on $U$.
			\item $\mathcal{W}$: The total space of $\detL(\mathcal{E}_{\mathcal{N}^{\ell,\ell+1}})$.
			\item $\mathcal{Z}$: The enhanced master space.
						\end{itemize}
		\subsubsection*{Subsection \ref{subsec:fix}}
	\begin{itemize}	
			\item $\mathcal{Z}_{k,\Lambda}, \rho_{k,\Lambda}$: The fixed loci of the enhanced master space, and the corresponding closed embedding respectively.
		\end{itemize}
	\subsubsection*{Subsection \ref{subsec:obstrfix}}
	\begin{itemize}
		\item $\mathfrak{N}(\mathcal{F},\mathcal{G}):= - [R\pi_{U*}(R\Sheafhom(\mathcal{F},\mathcal{G})\otimes \omega_X)^{\vee}]$ for $\mathcal{F}, \mathcal{G}$ coherent sheaves on $\Xhat\times U$.
	\end{itemize}
		\subsubsection*{ }

		\section{Background}\label{sec:background}
		\subsection{Donaldson invariants}\label{subsec:donaldson}
		We give some background about the analytic Donaldson invariants. Good general references for the Donaldson theory are the book by Donaldson--Kronheimer \cite{DoKr} and the collection by Friedman--Morgan \cite{FrMo2}.

		Let $X$ be a closed, oriented, smooth $4$-manifold. For simplicity and lack of a suitable reference, we will assume here that $X$ is simply-connected.

		The set of isomorphism classes of $SU(2)$-bundles on $X$ modulo isomorphism is in bijection with $H^4(X,\mathbb{Z})$ via the second Chern class map and thus with $\mathbb{Z}$, using the orientation. Given an integer $k\in \Z$, let $P=P_k$ denote the $SU(2)$-bundle with second Chern class $c_2(P)=k$. We consider the space of connections on this bundle. If $A$ is a connection on $P_k$, then its curvature $F_A$ is a  $2$-form on $X$ with values in the adjoint bundle $\mathfrak{su}_{2,P}$. Any choice of Riemannian metric $g$ on $X$ gives the Hodge star operator on forms and thereby a splitting 
		\[\Omega^2(X,\mathfrak{su}_{2,P})= \Omega^2_+(X,\mathfrak{su}_{2,P})\oplus \Omega^2_-(X,\mathfrak{su}_{2,P}) \] 
		into the positive and negative eigenspaces of the Hodge star operator. A connection $A$ is called \emph{anti-self-dual}, or ASD, if its curvature $F_A$ lies in $\Omega^2_-(X,\mathfrak{su}_{2,P})$. There is an action of the \emph{gauge group} (i.e. the group of global bundle automorphisms of $P$) on the space of connections, and each connection has automorphism group either $\{\pm 1\}$ or $S^1$. The former are called \emph{irreducible} connections and the latter \emph{reducible}. 
		There is a moduli space  $\mathcal{M}^*(P_k)$ of irreducible ASD connections modulo gauge transformations whose expected dimension is $8 k -3(1+b_2^+(X))$.

		Now suppose that $b_2^+=b_2^+(X)$ is odd. 
		Donaldson \cite{Dona} showed that if $b_2^+\geq1$, then for a generic choice of Riemannian metrics $g$ of $X$ (which can be chosen independently of $k$), the space $\mathcal{M}^*(P_k)$ is smooth of the expected dimension and there are no reducible ASD connections, in which case we write $\mathcal{M}^*(P_k)=\mathcal{M}(P_k)$. An orientation of $\mathcal{M}(P_k)$ is induced by a choice of orientation on $H_2^+(X,\mathbb{Z})$, which we fix from here on (the opposite choice gives the opposite sign for all invariants). There is a compactification $\mathcal{M}(P_k)\subset \overline{\mathcal{M}}(P_k)$ obtained by adding certain singular connections, which is called the \emph{Uhlenbeck compactification}. When $4k\geq 3b_2^+(X)+5$, the boundary of this moduli space is well-behaved, in particular, all boundary strata have lower dimension and $\overline{\mathcal{M}}(P_k)$ has a natural fundamental class. In this case we say that $k$ lies in the $\emph{stable range}$. 
		
		There is a canonical map 
		\[\overline{\mu}:H_{2}(X,\mathbb{Z})\to H^2(\overline{\mathcal{M}}(P_k)),\]
		whose restriction $\mu$ to the open subset $\mathcal{M}(P_k)$ can be described as follows: There is a universal $SO(3)$-bundle $\xi$ on $\mathcal{M}(P_k)\times X$. Then for $\alpha\in H_2(X,\mathbb{Z})$, we have 
		\[\mu(\alpha)= -\frac{1}{4}p_1(\xi)/\alpha,\]
		where $p_1$ denote the first Pontrjagin class. Here one uses the fact that the class $p_1(\xi)/\alpha$ is divisible by $4$ in $H^2(\mathcal{M}, \mathbb{Z})$. 
		When the bundle $\xi$ admits a lift to a principal $SU(2)$-bundle $\eta$ (which happens when $k$ is odd), we have that in fact $\mu(\alpha) = c_2(\xi)/\alpha$.
		
		Similarly, one has a map $\overline{\mu}:H_0(X,\mathbb{Z})\to H^4(\overline{\mathcal{M}}(P_k)\setminus T,\mathbb{Z})$, where $T$ is a certain boundary locus. Let $\alpha_1,\ldots,\alpha_n\in H_2(X,\mathbb{Z})$ and let $\nu\in H_0(X,\mathbb{Z})$ be the class of a point. One defines the \emph{stable Donaldson invariant} $D_{X,k}$ for $k$ in the stable range as 
		\[D_{X,k}(\alpha_1\cdot\cdots \cdot \alpha_n\cdot \nu^{\cdot m}) = \int_{\overline{\mathcal{M}}(P_k)}\mu(\alpha_1)\cup \cdots \mu(\alpha_n)\cup \mu(\nu)^{\cup m},\]
		where the integral denotes the pairing against the fundamental class. If $b_2^+>1$, then these invariants are in fact independent of the choice of a generic metric $g$, and therefore are invariants of the smooth structure of $X$. 
		
		\paragraph{Unstable range and a blowup formula.}
		Let $\Xhat$ be an oriented smooth connected sum of $X$ with $\overline{\mathbb{C}\mathbb{P}}^2$ and let $[C]$ be the homology class coming from a generator of $H_2(\overline{\mathbb{C}\mathbb{P}}^2,\Z)$ (for example, this is the case if $X$ is algebraic and $\Xhat$ the one-point blowup of $X$). Then we have $b_2^+(X)=b_2^+(\Xhat)$, and so from the condition 
		\[4k\geq b_2^+ +5,\]
		it is clear that if $k$ is in the stable range for $X$, then it is also in the stable range for $\Xhat$. For such $k$, one has the following \emph{blowup formula}: 
		\[D_{\Xhat,k+1}(\alpha_1\cdot\cdots \cdot \alpha_n\cdot[C]^{\cdot 4} \cdot \nu^{\cdot m}) = -2 D_{X,k}(\alpha_1\cdot\cdots \cdot \alpha_n\cdot \nu^{\cdot m}).\]
		 By stipulating this formula to hold, one can extend the definition of the Donaldson invariant to $k$ in the unstable range: Choose a positive integer $\ell$, such that $4(k+\ell)\geq b_2^+(X)+5$, and let $\widehat{X}^{(\ell)}$ be the smooth connected sum of $X$ with $\ell$ copies of  $\overline{\mathbb{C}\mathbb{P}}^2$ with $[C_i]$ the class coming from a generator of the second homology group of the $i$-th summand. Then one defines
		 \[ D_{X,k}(\alpha_1\cdot\cdots \cdot \alpha_n\cdot \nu^{\cdot m}):=\left(-\frac{1}{2}\right)^{\ell} D_{\Xhat^{(\ell)},k+\ell}(\alpha_1\cdot\cdots \cdot \alpha_n\cdot[C_1]^{\cdot 4}\cdots [C_{\ell}]^{\cdot 4} \cdot \nu^{\cdot m}).\]
		 This is the \emph{full Donaldson invariant}. It takes values in $(\frac{1}{2})^{\ell_0}\Z$, where $\ell_0$ is the minimal positive integer for which $4(k+\ell_0) \geq b_2^+(X)+5$. Due to the blowup formula for the stable invariant, the full Donaldson invariant is independent of choice of $\ell$ and agrees with the stable invariant for $k$ in the stable range.

		\paragraph{Definition using $SO(3)$-bundles.}
		Donaldson's definition goes through when one takes $P$ more generally to be a principal $SO(3)$ bundle. An $SO(3)$-bundle is classified by its second Stiefel--Whitney class $w_2(P)$ together with its first Pontrjagin class $p_1(P)$. These two classes satisfy the relation
		\[\mathcal{P}(w_2(P))=p_1(P) \mbox{ mod } 4.\] For a generic metric $g$ on $X$, we again have a well-behaved moduli space $\mathcal{M}(P)$ of ASD connections and its Uhlenbeck compactification $\overline{\mathcal{M}}(P)$. The expected dimension is 
		\[-2p_1(P) -3(b^+(X)+1)\]
		
		When $w_2(P)=0$, then $P$ lifts to an $SU(2)$-bundle $P'$ , and the two situations are related by 
		\[c_2(P')=-4p_1(P).\] 
		
		The definition of invariants goes through as before. One difference is that the class $p_1(P)$ may not be divisible by $4$, so the map $\overline{\mu}$ takes values in the rational cohomology group $H^*(\overline{\mathcal{M}},\mathbb{Q})$. On the upside, when $w_2(P)\neq 0$, there is no issue arising with the stable range, and one gets an invariant $D_{w_2,p_1}$ for each pair $(w_2,p_1)$ without additional work. 
		\paragraph{Relation to algebro-geometric Donaldson invariants.}
		Now assume additionally that $X$ is a complex projective surface with a given ample line bundle $H$. 
		For a fixed choice of $c_1\in H^2(X,\mathbb{Z})$ and $c_2\in H^4(X,\mathbb{Z})$, we consider the moduli space $\mathcal{M}^{lf}_{X}(2,c_1,c_2)$ of $H$-semistable locally free sheaves and its  closure $\overline{\mathcal{M}^{lf}}_{X}(2,c_1,c_2)$ in the Gieseker moduli space of $H$-semistable coherent sheaves. A rank $2$ holomorphic vector bundle is in particular a complex vector bundle, and can be viewed as a principal $U(2)$-bundle after a choice of a Hermitian metric. We have the exact sequence of groups 
		\[0\to U(1)\to U(2)\to PSU(2)\simeq SO(3)\to 0,  \]
		which induces an exact sequence
		\[0\to H^1(X,U(1))\to H^1(X,U(2))\to H^1(X,SO(3)).\]
		This means that each principal $U(2)$ bundle induces a principal $SO(3)$-bundle, and there is a bijection between isomorphism classes of $SO(3)$-bundles that arise in this way and the set of isomorphism classes of $U(2)$-bundles up to tensoring with a complex line bundle. 
		If $E$ is a $U(2)$-bundle and $P$ the induced $SO(3)$-bundle, their invariants are related by
		\begin{align}
			w_2(P)&=c_1(E) \mbox{ mod } 2,\label{eq:chtotop1}\\
		p_1(P)&=c_1(E)^2-4c_2(E).\label{eq:chtotop2}
		\end{align}
	 
		In particular, two $U(2)$ bundles with the same Chern classes induce the same topological $SO(3)$-bundle. 
		
		Suppose that $c_1$ is a $(1,1)$-class, i.e. that it lies in the intersection $H^2(X,\mathbb{Z})\cap H^1(X,\Omega_X)$ and fix the $SO(3)$-bundle with the corresponding invariants.  Then every ASD-connection on the bundle $P$ lifts to a unique Hermitian Yang-Mills connection on $E$ (the uniqueness uses crucially the simply-connectedness of $X$). 
		
		Therefore, there is a natural bijection $\mathcal{M}^{lf}_{X}(2,c_1,c_2)\to \mathcal{M}(P)$. A priori this map is only set-theoretic, but it can be extended to a continuous map between the moduli spaces and even to an algebraic map from $\overline{\mathcal{M}^{lf}}_{X}(2,c_1,c_2)$ to the Uhlenbeck compactification of $\mathcal{M}(P)$ \cite{Liju}.
		This has been used by Li \cite{Liju} as well as by O'Grady \cite{Ogra} and Morgan \cite{Morg} to give a definition of the Donaldson invariants purely in terms of algebraic geometry (at least for $SU(2)$-invariants), which recovers the usual definition. Their definition requires that the Gieseker moduli spaces are generically smooth of the expected dimension. For a fixed choice of $c_1$, this holds whenever $c_2$ is large enough. The definition can be extended to all values of $c_2$ by stipulating the blowup formula as before. 
		
		\paragraph{Mochizuki's Donaldson invariants.}
		Another definition of Donaldson invariants in the case of a smooth complex projective surface has been developed by Mochizuki \cite{Moch} using the \emph{virtual fundamental class} in algebraic geometry. This approach works for all values of $c_2$, and for $c_2$ large enough agrees with the algebro-geometric Donaldson invariants of Li and Morgan--O'Grady and therefore with the analytic invariants. Mochizuki's invariants have an advantage also in describing the {\it wall-crossing} of the invariants compared to the algebro-geometric Donaldson invariants mentioned above, in particular, they allow one to remove the assumption of {\it good walls} used in the formulation of wall-crossing of Donaldson invariants. 
This with the resolution of a Nekrasov conjecture led to the determination of the wall-crossing terms on a complex projective surface $X$ with $H^1 (X)=0$ by G\"{o}ttsche--Nakajima--Yoshioka \cite{GoNaYo1}. 

However, for small values of $c_2$ it is a priori unclear whether his invariants agree with the other ones. 
 The equivalence of them may have been obtained for simply-connected surfaces with positive geometric genus via the arguments of G\"{o}ttsche--Nakajima--Yoshioka in \cite{GoNaYo1}.  
We give a more direct proof for this by showing the Friedmann--Morgan style blowup formula for the virtual invariants in this paper.

G\"ottsche and Kool \cite{GoKo} express the virtual Euler characteristics of the moduli space of semistable sheaves on a surface as Mochizuki's invariants, and the $K$-theoretic counterpart, virtual $\chi_y$-genera of them, as well.   
They then study them by using Mochizuki's formula which expresses his virtual invariants in terms of Seiberg--Witten invariants and certain integrals over the product of Hilbert schemes of points on the surface.

Related to the work by G\"ottsche and Kool \cite{GoKo}, 
Thomas and the second-named author \cite{TaTh1}, \cite{TaTh2} define {\it Vafa--Witten invariants}, 
originally proposed in the work of Vafa and Witten \cite{VaWi}. Then the virtual Euler characteristics are realized as {\it instanton part} of the invariants.  
More generally, the virtual $\chi_{y}$-genera of the moduli spaces appear in the same manner in a $K$-theoretic refinement of the Vafa--Witten invariants studied by Thomas \cite{Thom}, Laarakker \cite{Laar} and others. 
Then, the {$S$-duality conjecture} described in \cite{VaWi} predicts the possible form of the virtual Euler characteristics and  the virtual $\chi_{y}$-genera of the moduli spaces from that of {\it monopole part} of these invariants and vice versa.  
The article \cite{GoKo2} is a good survey paper about recent developments of research in this direction.

		\section{Perverse coherent sheaves and their moduli spaces}
		\label{sec:perverse}

		\subsection{Perverse coherent sheaves on a blowup}
		\label{subsec:pervcoh}
		Let $X$ be a smooth projective surface with a fixed ample line bundle $H$. We denote by $p:\widehat{X}\to X$ the blowup of $X$ at a reduced point $q\in X$. We also denote by $C:=p^{-1}(q)$ the exceptional curve of the blowup and by $\mathcal{O}_C(m)$ the unique degree $m$ line bundle on $C$ for any $m\in \Z$.
		Our discussion is based on the notion of perverse coherent sheaf on the blowup $\Xhat$ as studied by Nakajima and Yoshioka in \cite{NaYo2}. 
		\begin{proposition}\label{propdefpervcoh}
			For a coherent sheaf $E$ on $\Xhat$ the following conditions are equivalent:
			\begin{enumerate}[label=(\arabic*)]
				\item The natural map $\phi:p^*p_*E\to E$ is surjective, and \label{propdefpervcoh1}
				\item $\Hom(E,\mathcal{O}_C(-1))=0$. \label{propdefpervcoh2}
			\end{enumerate}
			Moreover, if either of these hold, then we also have
			\begin{enumerate}[resume,label=(\arabic*)]
				\item $R^1p_*E=0$, so in particular $Rp_*E=p_*E$;  \label{propdefpervcoh3}
				\item $\Ker\phi$ is isomorphic to  $V^{\vee}\otimes \mathcal{O}_C(-1)$, where $V\simeq \Ext^1(E,\mathcal{O}_C(-1))$. This makes $p^*p_*E$ the universal extension of $E$ by $\mathcal{O}_C(-1)$; and \label{propdefpervcoh4}
				\item $R^1p_*E(C)=0$. \label{propdefpervcoh5}
			\end{enumerate}
		\end{proposition} 
		\begin{proof}
			The equivalence of \ref{propdefpervcoh1} and \ref{propdefpervcoh2} follows from Lemma 1.2 (3) and Proposition 1.9 (1) in \cite{NaYo2}. Then \ref{propdefpervcoh3} follows from Definition 1.1 in \cite{NaYo2}. Finally \ref{propdefpervcoh4} is Proposition 1.9 (2) in \cite{NaYo2}, and \ref{propdefpervcoh5} is Lemma 1.10 there. 
		\end{proof}

		\begin{definition}
			We call a coherent sheaf $E$ on $\widehat{X}$ \emph{perverse coherent} if it satisfies the equivalent conditions \ref{propdefpervcoh1} and \ref{propdefpervcoh2} of Proposition \ref{propdefpervcoh}.
		\end{definition}

		\begin{remark}
			This definition is a special case of one  studied by Bridgeland in \cite{Brid}. More generally, whenever $f:Y\to X$ is a birational morphism between projective varieties such that all fibers of $f$ have dimension $\leq 1$ and $Rf_*\mathcal{O}_Y=\mathcal{O}_X$, one can define an abelian subcategory  $\Per(Y/X)$ of the derived category of coherent sheaves on $Y$, whose objects are called \emph{perverse coherent sheaves}. They are in general certain complexes concentrated in $[-1,0]$. For our purposes, it is sufficient to consider only those objects in $\Per(\widehat{X}/X)$ which are in fact sheaves (i.e. complexes concentrated in degree $0$). In this case, Bridgeland's definition agrees with ours. 
		\end{remark}
		
		\begin{example}\label{expushpull}
			Let $F$ be a torsion-free coherent sheaf on $X$. The natural map  $p^*p_*p^*F\to p^*F$ has a right inverse $p^*F\to p^*p_*p^*F$ coming from adjunction and is therefore surjective. We find that $p^*F$ is perverse coherent (this argument works if $F$ has torsion, but then the right object to consider would really be $Lp^*F$). By the Auslander--Buchsbaum formula, the sheaf $F$ can be resolved by a sequence of two locally free sheaves $F^{-1}\hookrightarrow F^0$. Since $\widehat{X}$ is an integral scheme, it follows that $p^*F^{-1}\to p^*F^0$ remains an inclusion. Therefore $Lp^*F=p^*F$. It follows from this that $p_*p^*F\simeq F$. Indeed, we have $p_*p^*F=Rp_*Lp^*F=Rp_*(\mathcal{O}_{\Xhat}\otimes Lp^*F)$, so by the derived version of the projection formula $p_*p^*F\simeq (Rp_*\mathcal{O}_Y) \otimes F\simeq \mathcal{O}_X\otimes F\simeq F$. Since $p^*F=Lp^*F$, we also find that $c(p^*F)= p^*c(F)$.
		\end{example}

		We will be interested in stable sheaves on $X$ of rank $r\in \Z_{\geq 1}$ and with first Chern class $c_1$. We make the following assumption:

		\begin{assumption}\label{assumpcoprime}
			The integers $r$ and $\deg_{H} c_1:=c_1\cdot H$ are coprime. 
		\end{assumption}

		This assumption implies that Gieseker stability and Gieseker semistability coincide for sheaves on $X$ of rank $r$ with first Chern class $c_1$. If we also assume the sheaves to be  torsion-free, then both conditions agree also with the notions of slope (semi-)stability.  
		
		\begin{definition}\label{defadmiss}
			A cohomology class $\widehat{c}\in H^*(\Xhat,\Q)$ is called \emph{admissible} if its zero'th degree component is equal to $r$, and its first degree component is equal to $p^*c_1+k[C]$ for some $k\in \Z$. Equivalently, if $p_* \chat= r+c_1 +a $ for some $a\in H^4(X,\Z)$.  
		\end{definition}  
		\begin{remark} \label{rempstarchernclass}
			Let $E$ is a perverse coherent sheaf on $\widehat{X}$ with admissible total Chern class. Then $p_*E$ has rank $r$ and first Chern class $c_1$. In particular, $p_*E$ is Gieseker stable if and only if it is Gieseker semistable. 
			We can also compute the second Chern class of $p_*E$ using the Grothendieck--Riemann--Roch theorem: Say $c_1(E)=p^*c_1+k[C]$. Then 
			\[\operatorname{ch}_2(p_*E) = \operatorname{ch}_2(E)-\frac{1}{2} c_1(E)\cdot [C] = \operatorname{ch}_2(E)+\frac{k}{2}. \]
		\end{remark}

		\begin{definition}
			Let $E$ be a perverse coherent sheaf on $\widehat{X}$ with admissible Chern character. We say that $E$ is \emph{stable} as a perverse coherent sheaf, or simply \emph{stable}, if $p_*E$ is stable. 
		\end{definition}
		\begin{remark}
			This is again a special case of a more general definition of stability for perverse coherent sheaves, see Definition 2.2 in \cite{NaYo2}. It agrees with our definition thanks to Assumption \ref{assumpcoprime}, as shown in Lemma 2.6 in \cite{NaYo2}.
		\end{remark}
	
		We will be mainly interested in the moduli spaces of stable perverse coherent sheaves. A stable perverse coherent sheaf may have some torsion along the exceptional curve. The torsion part is, however, bounded by the following:
		
		\begin{lemma}\label{lemhomvanishing}
		If $E$ is a stable perverse coherent sheaf, then $\Hom(\mathcal{O}_C, E)=0$. 
		\end{lemma}
		\begin{proof}
			We have $\mathcal{O}_C=p^*\mathcal{O}_q$. Therefore by adjunction, we have $\Hom(\mathcal{O}_C, E)=\Hom(\mathcal{O}_q, p_*E)$. Since $p_*E$ is Gieseker stable and therefore torsion-free, this group vanishes. 
		\end{proof}

		Just like Gieseker stable sheaves, stable perverse coherent sheaves are simple: 
		
		\begin{proposition}\label{corstabsimple}
			Let $E_1$ and $E_2$ be stable perverse coherent sheaves on $\Xhat$ with Chern character $\chat$ and let $f:E_1\to E_2$ be a morphism between them. Then $f$ is either zero or an isomorphism.  
			In particular, we have $\Hom_{\Xhat}(E,E)=\C$ for any stable perverse coherent sheaf $E$ on $\Xhat$. 
		\end{proposition}
		\begin{proof}
			The map $p_*f:p_*E_1\to p_*E_2$ is either zero or an isomorphism by Gieseker stability (cf. Proposition 1.2.7 in \cite{HuLe}). We have the commutative diagram
			\begin{equation*}
				\begin{tikzcd}
					p^*p_*E_1\ar[r,"p^*p_*f"]\ar[d]& p^*p_*E_2\ar[d] \\
					E_1\ar[r, "f"] & E_2
				\end{tikzcd}
			\end{equation*}
			with the vertical maps being the ones from adjunction. But since $p^*p_*f$ is multiplication by a scalar, and since the map $p^*p_*E_2\to E_2$ is surjective, it follows that $f$ is surjective. Since $E_1$ and $E_2$ have the same Chern character, $f$ is an isomorphism. The last statement follows from this: Indeed, we have shown that $\Hom_{\Xhat}(E,E)$ is a division algebra over $\C$. 
		\end{proof}

		We now prepare some results which are important for the construction of moduli spaces in Subsection \ref{subsecmodspaces}. First, we have openness for families of stable perverse coherent sheaves. 
		\begin{proposition}\label{propstabisopen}
			Being perverse coherent is an open property for flat families of coherent sheaves on $\widehat{X}$. For a flat family of perverse coherent sheaves with admissible total Chern class, stability is an open property. 
		\end{proposition}
		\begin{proof}
			For any $\C$-scheme $T$, consider the base change $p_T:\widehat{X}_T \to X_T$ of $p$. Then both statements essentially follow from the fact that formation of $R(p_T)_*$ commutes with base change for $T$-flat coherent sheaves on $\Xhat_T$.  By Theorem 8.3.2 in \cite{FGA05}, for any $T$-flat sheaf $\mathcal{E}$ and any $g:T'\to T$, we have a natural isomorphism
			\begin{equation}\label{eqpbasechange}
				Lg_X^*R(p_T)_* \mathcal{E}\simeq R(p_{T'})_*g_{\Xhat}^*\mathcal{E} 
			\end{equation}
			in the bounded derived category of $X_{T'}$.   Since pullback is right exact, it is not hard to see that derived pullback behaves well with respect to upper truncation functors in the sense that $\tau_{\geq i}\circ Lg_{X}^* =\tau_{\geq i} \circ Lg_{X}^* \circ \tau_{\geq i}$.   Since $p$ has fibers of dimension at most one, $Rp_{T*}$ and $Rp_{T'*}$ both have cohomological dimension $\leq 1$. Therefore, by applying $\tau_{\geq 1}$ to \eqref{eqpbasechange}, we obtain
			
			\[g_{X}^*R^1(p_T)_*\mathcal{E}\simeq R^1(p_{T'})_*g_{\widehat{X}}^*\mathcal{E}.\]
			
			This exactly means that formation of $R^1(p_T)_*$ commutes with arbitrary base change. In particular, vanishing of $R^1p_*$ is an open property in flat families.  
			Since vanishing of $R^1p_*$ is a necessary condition for being perverse coherent, to prove the proposition we may therefore replace $T$ by the open subset where $R^1(p_T)_*\mathcal{E}$ vanishes. Then  \eqref{eqpbasechange} becomes 
			\[Lg_X^*(p_T)_*\mathcal{E}\simeq (p_{T'})_*g_{\Xhat}^*\mathcal{E}.\] 
			
			This shows that formation of $(p_T)_*\mathcal{E}$ commutes with base change and that $Lg_X^*R(p_T)_*\mathcal{E}=g_X^*R(p_T)_*\mathcal{E}$ for any $g:T'\to T$, which implies that $(p_T)_*\mathcal{E}$ is $T$-flat. Once we know that $(p_T)_*$ commutes with base change for $\mathcal{E}$, it follows that the same is true for the map $\Phi: p_T^* (p_T)_*\mathcal{E}\to \mathcal{E}$. In particular, the locus of perverse coherent sheaves is the open set in $T$ over which $\Phi$ is surjective. Finally, openness of stability for perverse coherent sheaves now follows from openness of Gieseker stability together with the statements that $(p_T)_*\mathcal{E}$ is flat and commutes with base change if $R^1(p_T)_*\mathcal{E}=0$, which we have already proven.
		\end{proof}
		For later reference we note two consequences of the proof:
		\begin{corollary}\label{lemronevanishing}
			Let $T$ be a scheme and $\mathcal{E}$ be a $T$-flat sheaf on $T\times \Xhat$. Then formation of $R^1(p_T)_*$ commutes with base change.  
		\end{corollary}
		
		\begin{corollary}\label{corpstarbasechange}
			Suppose $\mathcal{E}$ is a $T$-flat family of sheaves on $\widehat{X}$, flat over $T$. Suppose that $R^1(p_T)_*$ is identically zero (e.g. if $\mathcal{E}$ is a family of perverse coherent sheaves). Then $(p_T)_*\mathcal{E}$ is a $T$-flat family of sheaves on $X$ and its formation commutes with arbitrary base change on $T$. 
		\end{corollary}
	
		From the basic properties of perverse coherent sheaves in Proposition \ref{propdefpervcoh}, we have that every perverse coherent sheaf $E$ on $\Xhat$ sits in a sequence 
		\[0\to K\to F\to E\to 0,\]
		where $K\simeq \mathcal{O}_C(-1)^{\oplus j}$ for some $j$, and $F$ is pulled back from the base. In fact, modifying perverse coherent sheaves by sheaves of the form $\mathcal{O}_C(-1)^{\oplus k}$ is an important operation. In what follows we give some results to this extent.

		\begin{lemma}\label{lempushvanishing}
			Suppose $K$ is a coherent sheaf on $\widehat{X}$. Then $Rp_*K=0$ if and only if $K$ is a direct sum of copies of $\mathcal{O}_C(-1)$.
		\end{lemma}
		\begin{proof}
			See Lemma 1.7, (2) in \cite{NaYo2}.
		\end{proof}
		\begin{lemma}\label{lemvanishing}
			Let $m\in \Z$. Then 
			\begin{enumerate}[label=(\roman*)]
				\item 	$\Ext^1_{\Xhat}(\mathcal{O}_C(m),\mathcal{O}_C(m))=0$, and \label{lemvani}
				\item $\Hom_{\Xhat}(\mathcal{O}_C(m), \mathcal{O}_C(m-1))=\Ext^1_{\Xhat}(\mathcal{O}_C(m), \mathcal{O}_C(m-1))=0$. \label{lemvanii}
			\end{enumerate}
			
		\end{lemma}
		\begin{proof}
			It is clearly enough to show this when $m=0$. The statement $\Hom_{\Xhat}(\mathcal{O}_C,\mathcal{O}_C(-1))=0$ follows from restricting to $C$.  We consider the resolution 
			\[0\to \mathcal{O}_{\Xhat}(-C)\to \mathcal{O}_{\Xhat}\to \mathcal{O}_C\to 0.\]
			The associated long exact sequence for $\Hom_{\Xhat}(-,\mathcal{O}_C)$ gives an exact sequence
			\[\Hom_{\Xhat}(\mathcal{O}_{\Xhat}(-C),\mathcal{O}_C)\to \Ext^1_{\Xhat}(\mathcal{O}_C,\mathcal{O}_C)\to \Ext^1_{\Xhat}(\mathcal{O}_{\Xhat},\mathcal{O}_C).\]
			But $\Hom(\mathcal{O}_{\Xhat}(-C),\mathcal{O}_C)=H^0(\mathcal{O}_C(-1))=0$ and $\Ext^1_{\Xhat}(\mathcal{O}_{\Xhat},\mathcal{O}_C)=H^1(\mathcal{O}_C)=0$, by Serre duality. This shows \ref{lemvani}.
			Considering instead the long exact sequence for $\Hom_{\Xhat}(-,\mathcal{O}_C(-1))$, we obtain the exact sequence
			\[\Hom_{\Xhat}(\mathcal{O}_{\Xhat}(-C),\mathcal{O}_C(-1))\to \Ext^1_{\Xhat}(\mathcal{O}_C,\mathcal{O}_C(-1))\to \Ext^1_{\Xhat}(\mathcal{O}_{\Xhat},\mathcal{O}_C(-1)).\]
			Here the term on the left is isomorphic to $H^0(\mathcal{O}_C(-2))=0$, and the term on the right is isomorphic to $H^1(\mathcal{O}_C(-1))=0$. Therefore the term in the middle vanishes, which proves the remaining statement from \ref{lemvanii}.
		\end{proof}

		\begin{lemma}\label{lemexcchercarisexc}
			Let $E$ is a coherent sheaf on $\widehat{X}$, and let $K$ be coherent sheaf on $\Xhat$ with $\operatorname{ch}(K)=k\operatorname{ch}(\mathcal{O}_C(-1))$ for some $k \in \Z_{\geq 0}$. Then $K\simeq \mathcal{O}_C(-1)^{\oplus k}$ holds in either of the following two situations: 
			\begin{enumerate}[label=(\roman*)]
				\item $K$ is a subsheaf of $E$ and $p_*E$ is stable, or \label{lemexcchercarisexc1}
				\item $K$ is a quotient of $E$ and $E$ satisfies $R^1p_*E=0$. \label{lemexcchercarisexc2}
			\end{enumerate} 
		\end{lemma}
		
		\begin{proof}
			
			By Lemma \ref{lempushvanishing}, in either case it is enough to show that $Rp_*K=0$.
			From knowledge of the Chern character, we can immediately conclude that $K$ has rank $0$, and that the only one-dimensional component of its reduced support is the exceptional curve $C$. It follows that $p_*K$ is a torsion subsheaf with zero-dimensional support.
			Suppose we are in the case \ref{lemexcchercarisexc1}.
			Then $p_*K\subset p_*E$, but since the latter is stable and hence torsion-free, this implies $p_*K=0$. Therefore $Rp_*K=R^1p_*K$. Note that $R^1p_*K$ is supported at the center of the blow-up. Using the properties of the holomorphic euler characteristic, we find that 
			\[-\chi(R^1 p_*K)=\chi(Rp_*K)=\chi(K)=k\chi(\mathcal{O}_C(-1))=0 .\]
			Since the Euler characteristic of a sheaf with zero-dimensional support is the same as its dimension as a vector space over $\C$, we find $R^1p_*K =0$. 
			
			Now assume instead that \ref{lemexcchercarisexc2} holds. We find that $R^1p_*K=0$, since $p_*$ has cohomological dimension $\leq 1$ and $R^1p_*E=0$. Therefore $Rp_*K=R^0p_*K$. The same consideration of Chern classes as in the first part shows that $ p_*K$ must in fact be zero.
		\end{proof}
		\begin{corollary}\label{corquotispervstable}
			Suppose $E$ is a stable perverse coherent sheaf on $\Xhat$ and that  we have an exact sequence 
			\[0\to K\to E\to E'\to 0\]
			with $\operatorname{ch}(K)=\operatorname{ch}(\mathcal{O}_C(-1))$ for some $k\in \Z_{\geq 0}$. Then $K\simeq \mathcal{O}_C(-1)^{\oplus k}$ and $E'$ is stable perverse coherent. 
		\end{corollary}
		\begin{proof}
			By Lemma \ref{lemexcchercarisexc} and stability of $E$, we know that in fact $K\simeq \mathcal{O}_C(-1)^{\oplus k}$, so in particular $Rp_*K=0$. By taking long exact sequences, we conclude that $\Hom(E',\mathcal{O}_C(-1))=0$ and that $p_*E\simeq p_*E'$. 
		\end{proof}

		\subsection{Moduli spaces of stable perverse coherent sheaves}\label{subsecmodspaces}

		We now address the construction of moduli stacks of stable perverse coherent sheaves. Moduli stacks of coherent sheaves are never proper (except in trivial cases), due to the presence of scalar automorphisms. There are two common approaches to remedy this issue. One may consider the rigidifcation of the moduli stack where one, roughly speaking, ``divides out'' the scalar automorphism group. Alternatively, one may add some rigidifying data to the moduli problem, such as an orientation. We take the latter one here, as it seems to allow for a more natural treatment of obstruction theories later on. We first address the case without orientations. 
		
		Fix an admissible $\widehat{c}=r+\widehat{c}_1+\widehat{ch}_2\in H^*(\Xhat,\mathbb{Z})$. Then $\widehat{c}=p^*c-j\chexc$, where $j:=\widehat{c}_1\cdot [C]$ for some Chern character $c=r+c_1+\ch_2$ on $X$ (recall that $e=\ch(\mathcal{O}_C(-1)) = [C]-\frac{1}{2}[pt]$).  We throughout use the letter $T$ to denote an arbitrary test scheme over $\C$. 
		We let $\widetilde{\mathcal{M}}^p_{\chat}$ denote the moduli stack of stable perverse coherent sheaves on $\widehat{X}$ with Chern character $\widehat{c}$. By Proposition \ref{propstabisopen}, this makes sense as an open substack of the stack of all coherent sheaves on $\widehat{X}$.

		Suppose $E\in\widetilde{\mathcal{M}}^p_{\chat}(\C) $. By definition, $p_*E$ is a Gieseker stable sheaf on $X$ and we have a surjection $p^*p_*E\to E$. By Remark \ref{rempstarchernclass}, we have $\operatorname{ch}(p_*E)=c$.  Moreover, these associations are functorial in $E$, and work well in families by Corollary \ref{corpstarbasechange}.  Let $\widetilde{\mathcal{M}}_{X,c}$ denote the moduli stack of Gieseker stable coherent sheaves on $X$ of Chern character $c$ and $\mathcal{F}_{X,c}$ the universal sheaf over it. Let moreover $\widetilde{\mathfrak{Q}}\to \widetilde{\mathcal{M}}_{X,c}$ denote the relative Quot scheme $\Quot_{\Xhat\times\widetilde{\mathcal{M}}_{X,c}/\widetilde{\mathcal{M}}_{X,c}}(p^*\mathcal{F}_{X,c},p^*c-k\chexc)$. By what we just explained, we obtain a morphism of stacks $[p_*]:\widetilde{\mathcal{M}}^p_{\chat}\to\widetilde{\mathcal{M}}_{X,c}$, which sends a $T$-valued object $\mathcal{E}$ to $(p_T)_*\mathcal{E}$ and a morphism $\Phi:\widetilde{\mathcal{M}}^p_{\widehat{c}}\to\widetilde{\mathfrak{Q}}$ lifting $[p_*]$, which sends $\mathcal{E}$ to the quotient $p^*_T(p_T)_*\mathcal{E}\twoheadrightarrow \mathcal{E}$.
		\begin{proposition} \label{propmodisquot}
			$\Phi$ is an equivalence. A natural quasi-inverse is given by the functor $\Psi$ sending a quotient $p_T^*\mathcal{F}\to \mathcal{E}$ to $\mathcal{E}$, where $\mathcal{F}$ is an object of $\widetilde{\mathcal{M}}_{X,c}(T)$. 
		\end{proposition}
		
		\begin{proof}
			$\Psi$ is well-defined  on the open substack of $\widetilde{\mathfrak{Q}}$ corresponding to quotients which are stable perverse coherent sheaves. This substack is in fact all of $\widetilde{\mathfrak{Q}}$, due to Corollary \ref{corquotispervstable}. It is immediate that $\Psi\circ \Phi$ is the identity on $\widetilde{\mathcal{M}}^p_{\widehat{c}}$.   On the other hand, $\Phi\circ \Psi$ sends an quotient $p_T^*\mathcal{F}\to \mathcal{E}$ to the composition $p_T^*(p_T)_*p_T^\mathcal{F}\to p_T^*\mathcal{F}\to \mathcal{E}$, where the first map is the pullback of the natural map $\varepsilon:(p_T)_*p_T^*\mathcal{F}\to \mathcal{F}$. However, by Example \ref{expushpull}, the map $\varepsilon $ is an isomorphism. Thus we have exhibited a natural isomorphism  $\Phi\circ\Psi\xrightarrow{\sim} id_{\widetilde{\mathfrak{Q}}}$.  
		\end{proof}
		
		\begin{corollary}
			We have $p^*c=\widehat{c}+j\operatorname{ch}(\mathcal{O}_C(-1))$. If $j=0$, we have an isomorphism 
			\[\widetilde{\mathcal{M}}^p_{p^*c}\simeq \widetilde{\mathcal{M}}_{X,c}.\]
		\end{corollary}
		
		\begin{corollary}
			The stack $\widetilde{\mathcal{M}}_{\chat}^p$ is proper over $\widetilde{\mathcal{M}}_{X,c}$. 
		\end{corollary}
		
		\begin{lemma}\label{corkbound}
			For large enough $k$ (depending on $\chat$), the moduli stack $\widetilde{\mathcal{M}}^p_{\chat-k\chexc}$ is empty. 
		\end{lemma}
		\begin{proof}
			Say $\chat=p^*c-j\chexc$. We can assume that $j=0$ by changing $k$ appropriately. By Proposition \ref{propdefpervcoh} and the definition of perverse coherent sheaf, every object in $\widetilde{\mathcal{M}}^p_{\chat-k\chexc}(\C)$ is isomorphic to a quotient of $p^*E$ by  a subsheaf isomorphic to $\mathcal{O}_C(-1)^{\oplus k}$, for some $E\in \widetilde{\mathcal{M}}_{X,c}(\C)$. Since the set of sheaves parametrized by $\widetilde{\mathcal{M}}_{X,c}$ is bounded, the upper semicontinuous function $E\mapsto h^0(p^*E(-C))$ is bounded  above on $\widetilde{\mathcal{M}}_{X,c}(\C)$, say by $M>0$. On the other hand $h^0(\mathcal{O}_C(-1)(-C))=h^0(\mathcal{O}_C)=1$, so that for any subsheaf of $p^*E$ isomorphic to $\mathcal{O}_C(-1)^{\oplus k}$ and any $E\in \widetilde{\mathcal{M}}_{X,c}(\C)$, we have $k\leq M$. It follows that $\widetilde{\mathcal{M}}^p_{\chat-k\chexc}$ is empty for $k>M$. 
		\end{proof}
		
		\paragraph{Orientations.}

		Now we prove that these isomorphisms lift when we include orientations. First, we briefly recall the relevant terminology.  
		Let $Y$ be a smooth quasi-projective variety. We let $\det$ denote the determinant functor on the collection of coherent sheaves on $Y$. It associates to any coherent sheaf $E$ on $Y$ a line bundle $\det E$, and to an isomorphism $f:E\xrightarrow{\sim}E'$ an isomorphism $\det f: \det E\xrightarrow{\sim} \det E'$ in a functorial way, and satisfies additional naturality properties. More generally, it is defined for any scheme if one restricts to the class of coherent sheaves that have finite resolutions by locally free coherent sheaves. In particular, it is defined for any $T$-flat coherent sheaf $\mathcal{E}$ on $Y\times T$, such that the fibers $\mathcal{E}_{t}$ vary inside some bounded family of sheaves on $Y$ (e.g. Gieseker stable of some fixed Chern class).  
		Let $y\in Y $ be a chosen base point. Let $\Poin_Y$ on $Y\times \Pic Y$ be the Poincaré line bundle representing the relative Picard functor \cite[Part 5]{FGA05}, with $\Poin_Y\mid_{\{y\}\times \Pic Y} \simeq \mathcal{O}_{\{y\}\times \Pic Y}$. We recall its universal property: Given a line bundle $L$ on $Y\times T$ for some $T$, there is a classifying map $J_{L}:T\to \Pic (Y)$, such that $J_L^*\Poin_Y\simeq L\otimes \pi_Y^*M$ for some line bundle $M$ on $T$. 
		For a $T$-flat family of coherent sheaves $\mathcal{E}$ on $Y\times T$ for some $T$, we let $\det_{\mathcal{E}}=J_{\det \mathcal{E}}:T\to \Pic Y$ denote the classifying map associated to $\det E$. 
		\begin{definition}
			Let $Y$ be a smooth projective $\C$-scheme with basepoint, so that we have a Poincaré line bundle on $\Poin_Y$, and let $\mathcal{E}$ be a $T$-flat family of coherent sheaves on $Y\times T$ for some $T$. An \emph{orientation} of $\mathcal{E}$ is an isomorphism $\operatorname{or}:\det E\to \det_E^*\Poin_Y$. We also call the pair $(\mathcal{E},\operatorname{or})$ a family of \emph{oriented sheaves}.  An isomorphism of families of oriented sheaves $(\mathcal{E},\operatorname{or})\to (\mathcal{E}',\operatorname{or}')$ is an isomorphism $f:\mathcal{E}\to \mathcal{E}'$, such that $\operatorname{or}'\circ \det f =\operatorname{or}$. We will usually suppress reference to $\operatorname{or}$, when this does not cause confusion.
		\end{definition}
		
		Given a scheme, or more generally algebraic stack, $\widetilde{U}$ and a family $\mathcal{E}$ of sheaves on the smooth projective variety with basepoint $Y$ over $\widetilde{U}$, we may form the relative scheme $U$ over $\widetilde{U}$, parametrizing maps $F:T\to \widetilde{U}$ together with an orientation on $F^*\mathcal{E}$. In fact, $U$ is a $\C^*$-bundle over $\widetilde{U}$, where the $\C^*$ action consists of scaling a given orientation. 
		
		\begin{remark}\label{rem:Picarstack}
			Let $\mathcal{P}ic_Y$ denote the Picard stack of line bundles on $Y$ and let $\widetilde{U}$ be an algebraic stack together with a family of sheaves on $Y$ over $\widetilde{U}$.  Then the relative scheme of orientations $\pi_{\operatorname{or}}:U\to \widetilde{U}$ naturally fits into a cartesian square
			\begin{equation*}
				\begin{tikzcd}
					U\ar[r,]\ar[d,"\pi_{\operatorname{or}}"]& \Pic_X\ar[d,"\pi_{\operatorname{or}}"]\\
					\widetilde{U}\ar[r] & \mathcal{P}ic_X,
				\end{tikzcd}
			\end{equation*}
			where the horizontal maps are the respective classifying maps induced by $\det \mathcal{E}$, and the right vertical map is the classifying map of the Poincar\'e line bundle. 
		\end{remark}

		We now choose once and for all a basepoint $x\in X\setminus\{q\}$, which we identify with its preimage in $\widehat{X}$.
		We write $\mathcal{M}^p_{\chat}$ and $\mathcal{M}_{X,c}$ for the obvious moduli spaces of sheaves with orientations and $\mathfrak{Q}$ for the fiber product of the diagram $\widetilde{\mathfrak{Q}}\to\widetilde{\mathcal{M}}_{X,c}\xleftarrow{}\mathcal{M}_{X,c}$. We want to lift the morphisms $\Phi$ and $\Psi$ of Proposition \ref{propmodisquot} to include orientations .
		For each $k\in \Z$, there is an isomorphism $J_k:\Pic^{c_1} X\xrightarrow{\sim}\Pic^{c_1+k[C]}$ given by $L\mapsto p^*L\otimes \mathcal{O}(k C)$.  Due to our compatible choice of basepoint, there exists an isomorphism $J_k^*\Poin_{\Xhat}^{p^*c_1+k[C]}\simeq p^*\Poin_{X}^{c_1}\otimes \mathcal{O}_{\Xhat}(kC)$  on $\Xhat\times \Pic^{c_1} X$. For each $k$, make a choice of such an isomorphism once and for all. Once this choice is made, then for any family of line bundles $L$  on $X$  with fiberwise first Chern class $c_1$, we have natural isomorphisms $J_{p^*L\otimes \mathcal{O}_X(kC)}^*\Poin_{\widehat{X}}\simeq p^*J_{L}^*\Poin_X \otimes \pi_{\Xhat}^*\mathcal{O}_{\Xhat}(kC)$ for all $k$ .

		\begin{proposition}\label{propisrelquot} 
		
			The morphisms $\Phi$ and $\Psi$ lift to $\Phi:\mathcal{M}^p_{\chat} \to \mathfrak{Q}$ and $\Psi: \mathfrak{Q}\to \mathcal{M}^p_{\chat}$ (denoted by the same letters), which are still mutually inverse equivalences. 
		\end{proposition}
		\begin{proof}
			Since we have already defined $\Phi$ and $\Psi$ without orientations, we are reduced to the following problem: Suppose we have a $T$-valued point of $\widetilde{\mathfrak{Q}}$, i.e. an object $\mathcal{F}\in \widetilde{\mathcal{M}}_{X,c}(T)$ and an exact sequence 
			\[0\to \mathcal{K}\to p_T^*\mathcal{F}\to \mathcal{E}\to 0,\]
			where $\mathcal{E}\in \widetilde{\mathcal{M}}_{\chat}^p(T)$. We need to give natural ways to construct an orientation on $\mathcal{E}$ from a given one on $\mathcal{F}$ and vice-versa, and prove that the two constructions are inverse to each other. Note that $\mathcal{K}$ is fiberwise isomorphic to $\mathcal{O}_{\Xhat}(k_0 C)$ and so $\det \mathcal{K}\simeq \pi_{\Xhat}^*\mathcal{O}_{\Xhat}(k_0 C)\otimes \pi_T^*M$, for some line bundle $M$ on $T$. Since $\mathcal{K}$ is zero outside $C\times T$, we must in fact have that $M$ is canonically isomorphic to $\mathcal{O}_T$.  From the short exact sequence we have a natural isomorphism $\det \mathcal{E}\simeq \det p_T^*\mathcal{F}\otimes \left(\det\mathcal{K}\right)^{-1}\simeq   \det p_T^*\mathcal{F}\otimes\pi_{\Xhat}^*\mathcal{O}_{\Xhat}(-k_0 C)$. It follows by Remark \ref{rem:Picarstack} that we have natural isomorphisms
			\begin{equation}\label{eqorpullback}
				\operatorname{det}_{\mathcal{E}}^*\Poin_{\Xhat}\simeq J^*_{\det\mathcal{F}\otimes \pi_{\Xhat}^*\mathcal{O}_{\Xhat}(-k_0C)}\Poin_{\Xhat}\simeq p_T^*\operatorname{det}_{\mathcal{F}}^*\Poin_X\otimes \pi_{\Xhat}^*\mathcal{O}_{\Xhat}(-k_0C).
			\end{equation}
			Now given an orientation $\rho$ on $\mathcal{F}$ we can obtain an orientation on $\mathcal{E}$ via pullback and \eqref{eqorpullback}: 
			\[\det \mathcal{E}\xrightarrow{p_T^*\rho\otimes \left(\det \mathcal{K}\right)^{-1}}p_T^*\operatorname{det}_{\mathcal{F}}^*\Poin_X\otimes \pi_{\Xhat}^*\mathcal{O}_{\Xhat}(-k_0C)\simeq \operatorname{det}_{\mathcal{E}}^*\Poin_{\Xhat}.\]
			
			Conversely, given an orientation on $\mathcal{E}$, we may use \eqref{eqorpullback} to at least obtain an isomorphism $p_T^*\det \mathcal{F} \to p_T^*\det_{\mathcal{F}}^*\Poin_X$. By pushing forward, one obtains an orientation on $\mathcal{F}$. A straightforward check shows that these constructions are inverse to each other. 
		\end{proof}
		We have shown that $\mathcal{M}^p_{\chat}$ is a relative Quot-scheme over $\mathcal{M}_{X,c}$. 
		Since  $\mathcal{M}_{X,c}$ is a proper Deligne--Mumford stack, we have the following
		\begin{corollary}
			The moduli stack $\mathcal{M}^p_{\chat}$ is a proper Deligne--Mumford stack. 
		\end{corollary}
		\begin{corollary}\label{corpullbackiso}
			We have an isomorphism
			\[\mathcal{M}_{p^*c}^p \simeq \mathcal{M}_{X,c}.\]
		\end{corollary}

		\subsection{Moduli spaces of $m$-stable sheaves}
		\label{subsec:mstable}
				
		Let $E$ be a sheaf on $\widehat{X}$ with an admissible Chern character and let $m\in \mathbb{Z}$. Note that $E(-mC)$ has Chern character 
		$$\ch (E)\cdot \ch(\mathcal{O}_{\Xhat}(-mC))= r+ \operatorname{ch}_1(E)-mr[C] +  \ch_2 (E) - m [C] \cdot \ch_1 (E) -  \frac{r m^2}{2} \cdot [pt] , 
		$$
		which is still admissible. Following Nakajima and Yoshioka \cite{NaYo2}, we define: 
		\begin{definition}[cf. \S 3 in \cite{NaYo2}]
			Let $m\in \mathbb{Z}$. We say that a sheaf $E$ on $\Xhat$ with admissible Chern character is \emph{$m$-stable}, if $E(-mC)$ is a stable perverse coherent sheaf. 
			We let $\mathcal{M}^m_{\chat}$ denote the moduli stack of oriented $m$-stable sheaves of Chern character $\chat$ on $\Xhat$ and also write $\mathcal{M}^m$ if $\chat$ is understood. 
		\end{definition}
		\begin{lemma}\label{lemmstableprops}
			Suppose $E$ is an $m$-stable coherent sheaf on $\Xhat$. We have
			\begin{enumerate}[label=(\roman*)]
				\item $\Hom(\mathcal{O}_C(-m),E)=0$,
				\item $\Hom(E,\mathcal{O}_C(-m-1))=0$,
				\item if $m\geq 0$, then $p_*E$ is Gieseker stable, \label{lem:mstableprops3}
			\end{enumerate}
		\end{lemma}
		\begin{proof}
			Note that $\mathcal{O}_{\Xhat}(C)\simeq \mathcal{O}_C(-1)$. Then the first two claims follow by twisting from Lemma \ref{lemhomvanishing} and the definition of perverse coherent sheaf respectively. Suppose that $m\geq 0$. Then $\Hom(\mathcal{O}_C(-m),E)=0$ implies $\Hom(\mathcal{O}_C,E)=0$, i.e. that $p_*E$ is torsion-free. Since $p_*E(-mC)$ and $p_*E$ agree outside the codimension two locus $p_*C$, any slope-destabilizing subsheaf $F\subset p_*E$ induces a destabilizing subsheaf of $p_*E(-mC)$ by intersecting. This implies slope-stability of $p_*E$ and therefore Gieseker stability, as i is torsion-free. This demonstrates \ref{lem:mstableprops3}. 
		\end{proof}
		\begin{corollary}
			If $E$ is both $m_1$- and $m_2$-stable for $m_1\leq m_2$, then it is $m$-stable for any $m_1\leq m\leq m_2$.
		\end{corollary}
		\begin{proof}
			Since $\Hom(E,\mathcal{O}_C(-m_1-1))=0$, we have $\Hom(E,\mathcal{O}_C(-m-1))=0$, so $E(-mC)$ is perverse coherent. On the other hand, by Lemma \ref{lemmstableprops}, $m_2$-stability of $E$ implies that $p_*E(-mC)$ is Gieseker stable. Thus, $E(-mC)$ is stable perverse coherent. 
		\end{proof}
		\begin{remark}
			By definition, $m$-stability and being a stable perverse coherent sheaf are the same up to a twist. Write $\chat(-mC):=\widehat{c}\cdot \operatorname{ch}(\mathcal{O}_{\Xhat}(-mC))$. Then tensoring by $\mathcal{O}_{\Xhat}(-mC)$ gives isomorphisms
			\[\mathcal{M}^m_{\chat}\simeq \mathcal{M}^p_{\widehat{c}(-mC)} .\]
			In particular, $\mathcal{M}^0_{\chat}$  is just $\mathcal{M}^p_{\widehat{c}}$. 
		\end{remark}
		
		We may trivially restate Corollary \ref{corpullbackiso} that for $\chat= p^*c$, the moduli stack of oriented $0$-stable sheaves on $\Xhat$ agrees with the moduli space of Gieseker stable sheaves of Chern character $c$ on $X$.
		On the other hand, for any $\chat$, we have that for large $m$ the space $\mathcal{M}^m_{\chat}$ agrees with the space of Gieseker stable sheaves of Chern character $\chat$ on $\Xhat$ for a suitably chosen polarization.
		To make this precise, we use the following standard result (see for example \cite[Proposition 4.3]{Kose}). Recall that $H$ denotes our fixed polarization on $X$.  
		\begin{proposition}

			Let $\chat$ be a cohomology class on $\Xhat$.
			For any $\varepsilon \in \mathbb{Q}_{>0}$, let 	$L_{\epsilon}:=p^*H\otimes \mathcal{O}_{\Xhat}(-\varepsilon C)$, which is a $\mathbb{Q}$-line bundle on $\Xhat$.
			There exists a constant $\varepsilon_0>0$, such that the following hold:
			\begin{enumerate}
				\item For any $0<\epsilon<\epsilon_0$, the $\mathbb{Q}$-line bundle $H_{\varepsilon}$ is ample. 
				\item The conditions of Gieseker (semi-)stability defined by $H_{\varepsilon}$ on the collection of sheaves of Chern character $\chat$ is independent of $0<\varepsilon<\varepsilon_0$.  
			\end{enumerate}
			Moreover, if Gieseker stability and semistability coincide for sheaves of Chern character $p_*\chat$ on $X$ with respect to $H$, then the same is true for sheaves of Chern character $\chat$ on $\Xhat$ with respect to any of the $H_{\varepsilon}$ for $0<\varepsilon<\varepsilon_0 $.
		\end{proposition}   
		When referring to Gieseker stable sheaves on $\Xhat$, it is always implicitly with respect to one of the line bundles $H_{\varepsilon}$ for sufficiently small $\varepsilon$.
		\begin{remark}\label{rempullbackstable}
			If $E$ is a sheaf on $\Xhat$, then $\lim_{\varepsilon\to 0} \deg_{L_{\varepsilon}}E =\deg_X p_*E$. Similarly, if $F$ is a sheaf on $X$, then $\lim_{\varepsilon\to 0} \deg_{L_{\epsilon}} p^*F = \deg_X F$. By our assumption on $r$ and $c_1$, it follows that a torsion-free sheaf on $\Xhat$ is Gieseker stable if and only if $p_*E$ is Gieseker stable, since a destabilizing subsheaf of one induces a destabilizing subsheaf of the other. 
		\end{remark}
		Nakajima and Yoshioka make the following observation:
		
		\begin{proposition}[cf. Proposition 3.37 in \cite{NaYo2}]
			For sufficiently large $m$ (depending on $\chat$), the conditions of $m$-stability and Gieseker stability for sheaves on $\Xhat$ of Chern character $\chat$ coincide, i.e. 
			\[\mathcal{M}^m_{\chat}= \mathcal{M}_{\Xhat, \chat}.\]
			In particular, $\mathcal{M}^m_{\chat}$ is independent of $m$ for $m$ large enough. 
		\end{proposition}
		
		\begin{proof}
			For convenience, we outline the proof in \cite{NaYo2}. 
			One can show directly that for sufficiently large $m$ one has $\mathcal{M}_{\Xhat, \chat}\subset \mathcal{M}^m_{\chat}$. In fact, since the class of Gieseker stable sheaves of Chern character $\chat$ is bounded and $\mathcal{O}_X(-C)$ is relatively ample for $p$, we find that for $m\gg 0$ and any $E\in \mathcal{M}_{\Xhat, \chat}$ we have $R^1p_*E(-mC)=0$ and that $p^*p_*E(-mC)\to E(-mC)$ is surjective. In particular, $E(-mC)$ is perverse coherent. By Remark \ref{rempullbackstable}, we also obtain that $p_*E$ is stable.
			
			For the converse direction, we can again use Remark \ref{rempullbackstable} to see that any $m$-stable sheaf on $\Xhat$ is Gieseker stable \emph{provided it is torsion-free}. The main point here is to show that all $m$-stable sheaves are torsion-free once $m$ is sufficiently large. Nakajima and Yoshioka first establish that there is some $m_0>0$ such that $\mathcal{M}_{\chat}^m=\mathcal{M}_{\chat}^{m_0}$ for all $m\geq m_0$ (cf. Proposition 3.16 in \cite{NaYo2}). Then all $m_0$ -stable sheaves are in fact torsion-free: 
			To see this, suppose $E$ is $m_0$-stable, and let $T\subset E$ be the torsion subsheaf. Then for any $m\geq m_0$ we have $\Hom(\mathcal{O}_C, T(-mC))\subseteq\Hom(\mathcal{O}_C,E(-mC))=0$, as $E$ is $m$-stable. By adjunction, it follows that $\Hom(\mathcal{O}_q,p_*T(-mC))=0$ which implies $p_*T(-mC)=0$. On the other hand, $\mathcal{O}_{\Xhat}(-C)$ is ample relative to $p$, so $p^*p_*T(-mC)\to T(-mC)$ is surjective for $m\gg 0$. This shows that $T=0$.
		\end{proof}
		\subsection{Moduli spaces of coherent systems}
		\label{subsec:cohsys}
		Let $\chat = p^*c+j\chexc$ be an admissible Chern character on $\Xhat$ for some $j\geq 0$. We will see in this section, that the moduli space $\mathcal{M}^1_{\chat}$ can be related to the space $\mathcal{M}^1_{p^*c}$. This will be an essential tool in computing integrals over moduli spaces of sheaves on $\Xhat$ with first Chern class not pulled back from $X$.

		Let $\mathcal{Q}(\chat, j)$ be the relative Quot-scheme over $\mathcal{M}^1_{\chat}$ parametrizing quotients of $\mathcal{E}$ with Chern character $\ch(\mathcal{O}_C(-1)^{\oplus j})$. We will show that the kernel of the universal quotient on $\mathcal{Q}(\chat,j)$ is in fact $1$-stable, which gives a map $\mathcal{Q}(\chat,j)\to \mathcal{M}^1_{p^*c}$. 
		
		The following lemma is analogous to Corollary \ref{corquotispervstable}:
		\begin{lemma}\label{lem:exseqquot}
			Let $E$ be a $1$-stable sheaf on $\Xhat$ with an exact sequence 
			\[0\to F\to E\to K \to 0.\]
			Suppose that $\ch(K)=j\ch(\mathcal{O})$. Then $K\simeq \mathcal{O}_C(-1)^{\oplus j}$ and $F$ is $1$-stable. 
		\end{lemma}
		\begin{proof}
			By Proposition \ref{propdefpervcoh} \ref{propdefpervcoh5}, any $1$-stable sheaf $E$ satisfies $R^1p_*E=0$, and thus by Lemma \ref{lemexcchercarisexc}, the sheaf $K$ must be isomorphic to $\mathcal{O}_C(-1)^{\oplus j}$. We examine the sheaf $p_*F(-C)$. The sheaf $E(-C)$ is $0$-stable by assumption. Thus, $p_*E(-C)$ is slope-stable and torsion-free. Since $p_*F(-C)\subset p^*E(-C)$ is an isomorphism outside codimension two, the sheaf $p_*F(-C)$ inherits both of these properties. Therefore, $p_*F(-C)$ is even Gieseker stable. We have the following exact sequence:
			\[\Hom_{\Xhat}(E(-C),\mathcal{O}_C(-1))\to \Hom_{\Xhat}(F(-C),\mathcal{O}_C(-1)) \to \Ext^1_{\Xhat}(\mathcal{O}_C,\mathcal{O}_C(-1)).\]
			Here the first term is zero by $1$-stability of $E$, and the last term is zero by Lemma \ref{lemvanishing}, therefore the middle term is zero. This shows that $F(-C)$ is stable perverse coherent, therefore $F$ is $1$-stable.  
		\end{proof}

		 We have the following result, which is a version of Lemma 3.29 in \cite{NaYo2} and the discussion preceding it.

		\begin{proposition}\label{prop:cohsystem}
			\begin{enumerate}[label=\arabic*)]
				\item The stack $\mathcal{Q}(\chat, j)$ is naturally a moduli space for tuples $(\underline{\mathcal{E}},or_{\mathcal{F}},or_{\mathcal{E}})$, where $\underline{\mathcal{E}}$ is a short exact sequence $0\to \mathcal{F}\to \mathcal{E}\to \mathcal{O}_C(-1)\otimes \mathcal{V}\to 0$, such that $\mathcal{F}$ is an object of $\mathcal{M}^1_{p^*c}$, $\mathcal{E}$ is an object of $\mathcal{M}^1_{\chat}$ and such that $\mathcal{V}$ is a vector bundle on $\mathcal{Q}(\chat,j)$; and where $or_{\mathcal{F}}$ and  $or_{\mathcal{E}}$ are compatible orientations on $\mathcal{F}$ and $\mathcal{E}$ respectively. In particular, we have a diagram of forgetful maps 
				\begin{equation*}
					\begin{tikzcd}
						& \mathcal{Q}(\chat,j)\ar[dr,"q_2"]\ar[dl,"q_1"']& \\
						\mathcal{M}^1_{p^*c+j\chexc}& &\mathcal{M}^1_{p^*c} 
					\end{tikzcd}
				\end{equation*}
				\label{prop:cohsystem1}
				\item The map $q_1$ is canonically identified with the structure map of the relative Quot-scheme $\mathcal{Q}(\chat,j)\to \mathcal{M}_{\chat}$ under the given identification.\label{prop:cohsystem2}
				\item  The forgetful map $\mathcal{Q}(\chat,j)\to\mathcal{M}^1_{p^*c}$ identifies $\mathcal{Q}(\chat,j)$ with the relative Grassmann bundle $Gr(j, \Sheafext^1_{\Xhat}(\mathcal{O}_C(-1),\mathcal{F}))$. This identification also identifies $\mathcal{V}$ with the universal rank $j$ subbundle of the rank $r$ bundle \linebreak $q_2^*\Sheafext^1_{\Xhat}(\mathcal{O}_C(-1),\mathcal{F})$. \label{prop:cohsystem3}
			\end{enumerate}
		\end{proposition} 
		\begin{proof}
			It follows from Lemma \ref{lem:exseqquot} and the rigidity of $\mathcal{O}_C(-1)$ that the data of a quotient $\mathcal{E}\to \mathcal{K}$, where $\mathcal{E}$ is an object of $\mathcal{M}^1_{\chat}$ and $\mathcal{K}$ is a flat family of sheaves on $\Xhat$ with Chern character $j\ch(\mathcal{O}_C(-1))$ is equivalent to the data of an exact sequence of the type given in \ref{prop:cohsystem1}. We also note that an orientation on $\mathcal{E}$ induces a unique compatible orientation on $\mathcal{F}$. This implies \ref{prop:cohsystem1} and \ref{prop:cohsystem2} is clear.  
	
			We demonstrate \ref{prop:cohsystem3}. Let $\mathcal{F}$ be an object of $\mathcal{M}^1_{p^*c}(T)$ for a base scheme $T$. By definition of $1$-stability, we have the fiberwise $\Hom(\mathcal{O}_C(-1), \mathcal{F}_t)=0$ and $\Ext^2(\mathcal{O}_C(-1), \mathcal{F}_t)= \Hom(\mathcal{F}_t,\mathcal{O}_C(-2))^{\vee}=0$. It follows that $\Ext^1_{\Xhat}(\mathcal{O}_C(-1), \mathcal{F})$ is a vector bundle -- and using Grothendieck--Riemann--Roch one can show its rank is $r=\operatorname{rk} \mathcal{F}$. Now let $\mathcal{V}_T\subset\Sheafext^1_{\Xhat}(\mathcal{O}_C(-1),\mathcal{F})$ be a rank $j$ subbundle, which induces a section $\mathcal{O}_T\to \Sheafext^1_{\Xhat}(\mathcal{O}_C(-1)\otimes \mathcal{V}_T,\mathcal{F})$. The Grothendieck spectral sequence for the composition of functors $\Hom_{\Xhat\times T}(-,-)=\Gamma(T,-)\circ \Sheafhom_{\Xhat}(-,-)$ gives the exact sequence of low degrees
			\begin{gather*}
				0\to H^1(T,\Sheafhom_{\Xhat}(-,-))\to \Ext^1_{\Xhat\times T}(-,-) \qquad\qquad \qquad\\ \qquad\qquad \qquad\to H^0(T,\Sheafext^1_{\Xhat}(-,-))\to H^2(T,\Sheafhom_{\Xhat}(-,-)).
			\end{gather*}
			Since $\Sheafhom_{\Xhat}(\mathcal{O}_C(-1),\mathcal{F})$ is universally zero, we obtain an isomorphism
			\[\Ext^1_{\Xhat\times T}(\mathcal{V}_T\otimes \mathcal{O}_C(-1),\mathcal{F}) \to H^0(T,\Sheafext^1_{\Xhat}(\mathcal{V}_T\otimes \mathcal{O}_C(-1),\mathcal{F})).\]
			In particular, the given section of $\Sheafext^1_{\Xhat}(\mathcal{V}_T\otimes \mathcal{O}_C(-1),\mathcal{F})$ corresponds to a canonical extension 
			\[0\to\mathcal{F}\to \mathcal{E} \to \pi_{T}^*\mathcal{V}_T\otimes \mathcal{O}_C(-1)\to  0.\]
			In particular, the sheaf $\mathcal{E}$ constructed in this way is still $T$-flat. It is fiberwise $1$-stable, which follows from  \cite[Lemmas 3.9 (4) and 3.13 (4)]{NaYo2}. Moreover, the orientation on $\mathcal{F}$ defines uniquely a compatible orientation on $\mathcal{E}$. This shows that we get a natural map $Gr(j,\Sheafext(\mathcal{O}_C(-1), \mathcal{F}))\to \mathcal{Q}(\chat,j)$. The construction we made can easily be checked to be invertible, which shows that this map is an equivalence. 
		\end{proof}

		One can give another characterization of $\mathcal{Q}(\chat,j)$ as a moduli space of \emph{coherent systems}, i.e. of families of pairs $(E,V)$, where $E$ is a $1$-stable perverse coherent sheaf and $V\subset \Hom(E,\mathcal{O}_C(-1))$ is a subspace, see e.g. Definition 3.26 in \cite{NaYo2}. We now make this precise.
		\begin{construction}
			Choose a  resolution of the universal sheaf $\mathcal{E}$ on $\Xhat\times \mathcal{M}^1_{\chat}$ by a two-term complex of vector bundles $V_1\to V_0$, where $V_0$ is chosen sufficiently negative in the sense that $\Sheafext^i_{\Xhat}(V_0\mid_{x},\mathcal{O}_C(-1))=0$ for all $\C$-valued points $x$ of $\mathcal{M}^1_{\chat}$ and $i\geq 0$. Then $R\Sheafhom_{\Xhat}(\mathcal{E}, \mathcal{O}_C(-1))$ is represented by the two-term complex $W_0\to W_1$, where $W_i=\Sheafhom_{\Xhat}(V_i,\mathcal{O}_C(-1))$ (this uses that $\Sheafext^2_{\Xhat}(\mathcal{E},\mathcal{O}_C(-1))$ vanishes universally). In particular, for any $T\to \mathcal{M}^1_{\chat}$, we have an exact sequence 
			\[0\to \Sheafhom_{\Xhat}(\mathcal{E}_T,\mathcal{O}_C(-1)) \to W_{0,T} \to W_{1,T} . \]

			Consider the relative Grassmannian $\pi_G:Gr(j, W_0)\to \mathcal{M}^1_{\chat}$ and let $\mathcal{S}\subset \pi_G^*W_0$ denote the universal subbundle. We define $\operatorname{Coh}( \mathcal{E},\mathcal{O}_C(-1),j)$ as the vanishing locus of the morphism $\mathcal{S}\to \pi_G^*W_0\to \pi_G^*W_1$. By construction, the restriction of $\mathcal{S}$ to $\operatorname{Coh}( \mathcal{E},\mathcal{O}_C(-1),j)$ factors through $\Sheafhom_{\Xhat}(\mathcal{E}_{\operatorname{Coh}( \mathcal{E},\mathcal{O}_C(-1),j)},\mathcal{O}_C(-1))$. \end{construction}
		
		\begin{remark}\label{rem:cohsysbch}
			One can check that the stack $\operatorname{Coh}(\mathcal{E},\mathcal{O}_C(-1),j)$ parametrizes pairs $(\mathcal{E}_T, \mathcal{S}_T)$, where $\mathcal{E}_T$ is an object of $\mathcal{M}^1_{\chat}(T)$ and $\mathcal{S}_T\hookrightarrow \Sheafhom_{\Xhat}(\mathcal{E}_T,\mathcal{O}_C(-1))$ is a rank $j$ locally free subsheaf with the property that it remains a subsheaf after any base change, i.e. for any $T'\to T$, the composition 
			\[\mathcal{S}_{T'}\to \Sheafhom_{\Xhat}(\mathcal{E}_T,\mathcal{O}_C(-1))\mid_{T'}\to  \Sheafhom_{\Xhat}(\mathcal{E}_{T'},\mathcal{O}_C(-1))\]
			remains injective. In particular, this stack is independent of the choice of resolution $V_1\to V_0$. 
		\end{remark}
		
		\begin{proposition}\label{prop:qiscohsys}
			There is a canonical isomorphism $\mathcal{Q}(\chat,j)\simeq \operatorname{Coh}(j, \mathcal{E},\mathcal{O}_C(-1))$ over $\mathcal{M}_{\chat}^1$. This isomorphism gives an identification $\mathcal{S}\simeq \mathcal{V}^{\vee}$, where $\mathcal{S}$ is the universal rank $j$ subsheaf of $\Sheafhom_{\Xhat}(\mathcal{E},\mathcal{O}_C(-1))$ and $\mathcal{V}$ is as in Proposition \ref{prop:cohsystem}. 
		\end{proposition}
		\begin{proof}
			For a test scheme $T$, let $\mathcal{E}_T$ be an object of $\mathcal{M}_{\chat}^1(T)$. For a rank $j$ locally free sheaf $\mathcal{V}_T$ on $T$, giving a map $f:\mathcal{E}_T\to \pi_{T}^*\mathcal{V}_T\otimes \mathcal{O}_C(-1)$ is the same as giving a map $F:\mathcal{V}_T^\vee\to \Sheafhom_{\Xhat}(\mathcal{E}_T,\mathcal{O}_C(-1))$. We need to show that $f$ is a surjection if and only if $F$ gives an inclusion into the $\Hom$-sheaf after any base change. Assume that $f$ is surjective. Then $F$ may be recovered by applying $\Sheafhom_{\Xhat}(-,\mathcal{O}_C(-1))$ to $f$ and is therefore injective by left-exactness. Since any base change of $f$ is still surjective, the same argument works to show that $F$ remains injective after base change. For the converse, suppose we are given $F$ such that for any $T'\to T$ the induced map $\mathcal{V}_T'\to \Sheafhom_{\Xhat}(\mathcal{E}_T',\mathcal{O}_C(-1))$ is injective. We claim that $f$ is surjective. Since surjectivity of a map of coherent sheaves on a Noetherian scheme is an open statement, this may be checked on $\C$-points. So by the asserted base change property, we may assume that $T=\Spec \C$. Then the surjectivity of $f$ follows from Lemma 3.9 in \cite{NaYo2}.
		\end{proof}

		\section{Master space for wall-crossing}
		\label{sec:master}

	Nakajima and Yoshioka \cite{NaYo3} prove Theorem \ref{th:mainthwallcrossing} in the context of moduli spaces of framed sheaves on $\mathbb{P}^2$ by relating $\zerostable$ and $\onestable$, using an enhanced master space construction. They rely in an essential way on a GIT description of the moduli problem, which comes from an interpretation of moduli spaces as certain quiver varieties. In this section, we lay the geometric groundwork to generalize their formula to moduli spaces of perverse coherent sheaves on an arbitrary smooth projective surface, where such a description is not available. Since we have the identification $\mathcal{M}^m_{\chat}\simeq\mathcal{M}^0(\chat(-mC))$, we only need to study the wall-crossing formula for one value of $m$. 
		
		The notion of the master space goes back to Thaddeus \cite{Thad} in the context of studying how GIT quotients behave under change of linearization and has since then been used more generally to understand the behavior of moduli spaces under change of stability conditions. The general idea is as follows: Given two moduli spaces $\mathcal{M}^-$ and $\mathcal{M}^+$ corresponding to different stability conditions, the master space is a space together with an action of $\C^*$  such that $\mathcal{M}^{\pm}$ appear as fixed components. The remaining fixed points correspond to semistable objects with positive-dimensional automorphism group, and can usually be described in terms of moduli spaces which are in some sense simpler than the ones that one has started out with. Using an equivariant localization formula, one can then hope to relate  invariants on the moduli spaces $\mathcal{M}^+$ and $\mathcal{M}^-$ in terms of invariants on these simpler moduli spaces. 
		In our situation, we encounter two difficulties in constructing a master space:
		\begin{enumerate}[label=\arabic*.]
			\item In order for the master space to be well-behaved (for us, this means it should be a proper Deligne--Mumford stack), one needs that the connected components of automorphism groups of semistable objects are at most $\C^*$. However, for the natural notion of semistability in our situation, automorphism groups of any dimension can appear. In order to deal with this, one can use Mochizuki's notion of an {\it enhanced master space}. The idea is to construct moduli spaces $\flagzero$ and $\flagN$ of sheaves with additional data over $\zerostable$ and $\onestable$ respectively, and introduce a sequence of spaces $\flagell$ ($\ell = 0,1, \dots , N$) interpolating between these ``enhanced'' spaces. By defining a good notion of semistability for $\flagell$ and $\flagellplusone$, we can achieve that the stabilizer groups are sufficiently simple, and thus construct a well-behaved master space. Then we can recover the wall-crossing term between $\zerostable$ and $\onestable$ from the wall-crossing terms of the intermediate spaces. This part is similar to how Nakajima and Yoshioka proceed in \cite{NaYo3}. 
			\item  There is no apparent underlying GIT construction of $\zerostable$ and $\onestable$ as GIT quotients of a common variety. Therefore, we make use of a more general construction of master spaces, which we have learnt from a paper of Kiem and Li \cite{KiLi}. Their construction is purely based on the geometry of the underlying moduli problem and gets by without referring to GIT. On the other hand, this requires some care and repeated arguments using valuative methods to show that our moduli problems have the desired properties such as properness. 
		\end{enumerate}

		\subsection{Stack of semistable objects}\label{subsec:semistable}
		Fix throughout an admissible Chern character $\chat$. We will consider moduli stacks of sheaves with this Chern character.
		We now define the stack connecting the stacks $\mathcal{M}^0$ and $\mathcal{M}^1$. This is done in an ad-hoc manner by writing down a common weakening of their defining properties, and it turns out that this has the desirable geometric properties. A variant of this definition already appears in \cite[Definition 4.6]{NaYo3}.
		\begin{definition}
			Let $E$ be a coherent sheaf on $\widehat{X}$ with admissible Chern character (see Definition \ref{defadmiss}). We say that $E$ is \emph{$0,1$-semistable}, if and only if it satisfies the following conditions:
			\begin{enumerate}[label=(\arabic*)]
				\item $R^1p_*E=0$; \label{def01stable2} 
				\item $p_*E$ is stable; \text{and}  \label{def01stable3}
				\item $\Hom_{\widehat{X}}(E, \mathcal{O}_C(-2))=0$. \label{def01stable4}
				
			\end{enumerate}
		\end{definition}
		\begin{remark}\label{rem01stable}\,
			\begin{enumerate}[label=(\roman*)]
				\item If $E$ is $0,1$-semistable, then by \ref{def01stable3}, the sheaf $p_*E$ is torsion-free. Equivalently, $E$ is torsion-free outside of $C$, and we have $\Hom_{\widehat{X}}(\mathcal{O}_{C}, E)=0$. \label{rem01stable1}
				\item If we have an exact sequence of sheaves on $\widehat{X}$
				\[0\to \mathcal{O}_C(-1)^{\oplus k}\to E\to E'\to 0,\]
				then $E$ is $0,1$-semistable if and only if $E'$ is. The same statement holds with the arrows going in the other direction.  This is because (in either case) we have $p_*E=p_*E'$ and  $R^1p_*E=R^1p_*E'$ by the long exact sequence associated to $p_*$, and we also have \[\Hom_{\widehat{X}}(E, \mathcal{O}_C(-2))=\Hom_{\widehat{X}}(E', \mathcal{O}_C(-2))\] due to the long exact sequence associated to $\Hom_{\Xhat}(-,\mathcal{O}_C(-2))$ and Lemma \ref{lemvanishing} \ref{lemvanii}. \label{rem01stable2}
			\end{enumerate}
		\end{remark}
		
		\begin{lemma}
			The condition to be $0,1$-semistable is open in flat families.
		\end{lemma}
		\begin{proof}
			Since Chern classes are locally constant in flat families, the condition to have an admissible Chern character is open. From Corollary \ref{lemronevanishing} it follows that \ref{def01stable2} is open. On the locus where \ref{def01stable2} holds, it follows from Corollary \ref{corpstarbasechange} that $p_*$ commutes with base change and preserves flatness over a base. Since stability is an open condition, it then follows that the conditions \ref{def01stable2} and \ref{def01stable3} together form an open condition. Finally, the condition \ref{def01stable4} is equivalent to $E(-C)$ being perverse coherent, which is an open condition by Proposition \ref{propstabisopen}. 
		\end{proof}
		
		\begin{definition}
			For an admissible Chern character $\widehat{c}$, we let $\zeronestablechat$ denote the moduli stack of oriented $0,1$-semistable sheaves on $\widehat{X}$ with Chern character $\widehat{c}$. We also write $\zeronestable$, if the Chern character $\widehat{c}$ is clear from the context. 
		\end{definition}

		\begin{lemma}\label{lemsstabletostable}
			Suppose $E$ is a $0,1$-semistable coherent sheaf on $\widehat{X}$. Then
			\begin{enumerate}[label=(\roman*)]
				\item $E$ is $0$-stable if and only if  $\Hom_{\widehat{X}}(E, \mathcal{O}_C(-1))=0$; and \label{lemsstabletostablei}  
				\item $E$ is $1$-stable, if and only if $\Hom_{\widehat{X}}(\mathcal{O}_C(-1),E)=0$. \label{lemsstabletostableii}
			\end{enumerate}
		\end{lemma}
		\begin{proof}
			Since we are assuming that $p_*E$ is Gieseker stable, $E$ is $0$-stable if and only if it is perverse coherent. By definition, $E$ is perverse coherent if and only if $\Hom_{\Xhat}(E,\mathcal{O}_C(-1)) = 0$.  This proves \ref{lemsstabletostablei}. 
			We now show \ref{lemsstabletostableii}. Suppose first that $E$ is $1$-stable. Then $p_*E(-C)$ is stable, and therefore torsion-free, and in particular, torsion-free at the center of the blowup. Therefore, $\Hom_{\widehat{X}}(\mathcal{O}_C(-1), E)=\Hom_{\widehat{X}}(\mathcal{O}_C, E(-C))=0$. For the converse we use the argument in the proof of Lemma 3.4 in \cite{NaYo2}: From the assumption that $E$ is $0,1$-semistable, we have  $\Hom_{\widehat{X}}(E(-C),\mathcal{O}_C(-1))=\Hom(E(-C),\mathcal{O}_C(-2))=0$, from which it follows that $E(-C)$ is perverse coherent. The map $p_*E(-C)\to p_*E$ is an isomorphism outside the blowup center, which has codimension two. It follows, that there is a one-to-one correspondence between saturated subsheaves of $p_*E(-C)$ and of $p_*E$, which moreover preserves degree and rank. This implies that $p_*E(-C)$ is slope stable. Now, we assume that $\Hom_{\widehat{X}}(\mathcal{O}_C(-1), E)=0$, i.e. that $p_*E(-C)$ is torsion-free. Note that $p_*E(-C)$ has the same rank and first Chern class as $E$. Gieseker stability therefore follows from slope stability and torsion-freeness. We have shown that $E(-C)$ is $0$-stable, i.e. that $E$ is $1$-stable.    
		\end{proof}
		
		The following result is the first indication that the notion of $0,1$-semistability is suitable. It says that if a $0,1$-semistable sheaf is not stable (for either $0$- or $1$-stability), then this is due to the existence of a certain destabilizing subsheaf. 
		\begin{proposition}\label{propfiltrations}
			Let $E$ be a $0,1$-semistable sheaf.
			\begin{enumerate}[label=(\arabic*)]
				\item There exists a unique subsheaf $S\subset E$, such that $S\simeq \mathcal{O}_C(-1)^{\oplus k}$ for some $k\geq 0$ and such that $E/S$ is $1$-stable.  \label{propfilt1}
				\item There exists a unique subsheaf $T\subset E$, such that $E/T\simeq \mathcal{O}_C(-1)^{\oplus j}$ for some $j\geq 0$ and such that $T$ is $0$-stable. \label{propfilt2}
				\item The image of $S$ in $E/T$ lifts to a direct summand $K$ of $E$. Moreover, $K$ is isomorphic to a direct sum of copies of $\mathcal{O}_C(-1)$, and $E/K$ does not have a direct summand isomorphic to $\mathcal{O}_C(-1)$. \label{propfilt3} 
			\end{enumerate}
		\end{proposition}

		\begin{remark}\label{remfiltrationsanyfield}
			The conclusions hold as well over any algebraically closed field of characteristic zero, and in fact over any field of characteristic zero, as the uniqueness statement guarantees that the subsheaves are Galois-invariant. 
		\end{remark}

		\begin{proof}[Proof of Proposition \ref{propfiltrations}.]
			We argue along similar lines as in the proof of Lemma 3.9 in \cite{NaYo2}. 
			\begin{enumerate}[label=(\arabic*)]
				\item \label{propfiltproof1} Let $U:=\Hom_{\widehat{X}}(\mathcal{O}_C(-1),E)$ and consider the evaluation map $\varphi:U\otimes\mathcal{O}_C(-1)\to E$. The map $\varphi$ has the defining property that any $f:\mathcal{O}_C(-1)\to E$ factors as $f=\varphi\circ \iota_f$, where $\iota_f:\mathcal{O}_C(-1)\to U\otimes \mathcal{O}_C(-1)$ is defined by $s\mapsto f\otimes s$. We claim that this map is injective. Indeed, since $\operatorname{Im}\varphi$ is scheme-theoretically supported on $C\simeq \mathbb{P}^1$, we can write  as a direct sum of sheaves $\mathcal{O}_C(m_i)$ for different $m_i$. There are no maps $\mathcal{O}_C(-1)\to \mathcal{O}_C(m)$ for $m<-1$, so all $m_i\geq -1$. On the other hand, we have that all $m_i<0$, as $\Hom(\mathcal{O}_C,E)=0$. Therefore $\operatorname{Im}\varphi$ is a direct sum of copies of $\mathcal{O}_C(-1)$. It follows that we have a decomposition $U\otimes \mathcal{O}_C(-1)=\Ker\varphi\oplus \operatorname{Im}(\varphi)$. In particular, $\Ker\varphi$ is generated by subsheaves of the form $f\otimes \mathcal{O}_C(-1)$ for $f\in U$. However, if $f\otimes \mathcal{O}_C(-1)\subset \Ker\varphi$, then $f=\varphi\circ \iota_f=0$ by the universal property of $U$. Therefore $\Ker\varphi=0$. 
				Setting $S:=\operatorname{Im}\varphi$, we therefore have an exact sequence
				\[0\to U\otimes \mathcal{O}_C(-1)\to E\to E/S\to 0.\]
				By Remark \ref{rem01stable}, \ref{rem01stable2} we find that $E/S$ is $0,1$-semistable. Now consider part of the long exact sequence associated to $\Hom_{\widehat{X}}(\mathcal{O}_C(-1), {-})$: 
				
				\begin{multline*}
					\Hom_{\widehat{X}}(\mathcal{O}_C(-1), S)\to \Hom_{\widehat{X}}(\mathcal{O}_C(-1), E)\\ \to \Hom_{\widehat{X}}(\mathcal{O}_C(-1), E/S)\to \Ext^1_{\widehat{X}}(\mathcal{O}_C(-1), S).
				\end{multline*}
				The first map is an isomorphism by construction of $S$, and the last group is zero due to Lemma \ref{lemvanishing} \ref{lemvani}. Therefore, $\Hom_{\widehat{X}}(\mathcal{O}_C(-1), E/S)=0$, and it follows from Lemma \ref{lemsstabletostable} that $E/S$ is $1$-stable.

				For uniqueness of $S$, suppose $S'$ is another subsheaf of $E$ satisfying the conditions. Then $S'$ is a direct sum of copies of $\mathcal{O}_C(-1)$, and by the definition of $S$ as the image of $\varphi$, it follows that $S'$ is contained in - and therefore a direct summand of $S$. But then $\Hom_{\widehat{X}}(\mathcal{O}_C(-1),E/S')=0$ implies that $S'=S$. 
				\item Let $V:=\Hom(E,\mathcal{O}_C(-1))$. There is a universal map $\psi:E\to V^\vee\otimes \mathcal{O}_C(-1)$. By an argument analogous as in \ref{propfiltproof1}, this time using that $\Hom_{\widehat{X}}(E,\mathcal{O}_C(-2))$, we find that $\psi$ is surjective. Then, setting $T:=\Ker \psi$, we use the short exact sequence 
				\[0\to T\to E\to E/T\to 0\]
				to conclude that $E/T$ is in fact $0$-stable. Then we conclude that $T$ is the unique subobject with the desired properties. The argument is again exactly analogous to the one in \ref{propfiltproof1}.   
				\item The image $K'$ of $S$ in $E/T$ is again a direct sum of copies of $\mathcal{O}_C(-1)$. It therefore splits both the surjection $S\to K'$ and the injection $K'\to E/T$ and it follows that we have a map $E\to K'$ with a section $K'\to E$. Therefore $K'$ induces a direct summand $K$ of $E$. Finally, a nontrivial direct summand of $E/K$ isomorphic to a sum of copies of $\mathcal{O}_C(-1)$ would lift to such a summand $K''$ of $E$ strictly containing $K$. One finds that $K''\subset S$ and that $T\cap K''=0$, so that the induced map $K''\to S\to \operatorname{Im}\varphi\simeq K$ is an inclusion, which contradicts our assumption. 
			\end{enumerate}
		\end{proof}

		\begin{corollary}\label{corintsumexc}
			\begin{enumerate} [label=(\roman*)]
				\item Any map $\mathcal{O}_C(-1)^{\oplus j}\to E$ factors through $S$ and has image isomorphic to $\mathcal{O}_C(-1)^{\oplus j'}$ for some $j'$.  
				\item The collection of subobjects $S'\subset E$ that are isomorphic to a direct sum of copies of $\mathcal{O}_C(-1)$ is closed under sum and intersection.
				\item Any map $E\to \mathcal{O}_C(-1)^{\oplus j}$ factors through $E/T$ and has image isomorphic to $\mathcal{O}_C(-1)^{\oplus j'}$ for some $j'$.
				\item The collection of subobjects $T'\subset E$ for which $E'/T$ is isomorphic to a direct sum of copies of $\mathcal{O}_C(-1)$ is closed under sum and intersection. 
			\end{enumerate}
		\end{corollary}
		\begin{proposition}\label{propzeronebounded}
			The stack $\zeronestablechat$ is quasi-compact. 
		\end{proposition}
		\begin{proof}
			By Proposition \ref{propfiltrations}, every object $E\in \zeronestablechat(\C)$ sits in an exact sequence 
			
			\begin{equation}
				0\to E'\to E\to \mathcal{O}_C(-1)^{\oplus k}\to 0 \label{corproofexteqn}
			\end{equation}
			for some $k\geq 0$, where $E'$ is $0$-stable of Chern character $\widehat{c}-k\chexc$. By Corollary \ref{corkbound}, the moduli spaces $\zerostable[\widehat{c}-k\chexc] $ are empty for large enough $k$. Therefore, it is enough to show that for fixed $k$, the class of objects $E$ arising as extensions of the form \eqref{corproofexteqn} is bounded. By  \cite[Lemma 1.7.2]{HuLe}, a family of sheaves with fixed Chern character is bounded if and only if there is a global bound for the Castelnuovo--Mumford regularity of its sheaves. In our situation, this is clear: Simply take as a bound the maximum of the regularity of $\mathcal{O}_C(-1)$ and of an upper bound for the collection of sheaves appearing as objects of $\mathcal{M}^0_{\chat-ke}$ (which we already know to be bounded). 
		\end{proof}
		\begin{corollary}\label{corkboundzerone}
			For large enough $k$ (depending on $\chat$), the stack $\mathcal{M}_{\chat-k\chexc}^{0,1}$ is empty. 
		\end{corollary}
		\begin{proof}
			If $E\in \mathcal{M}_{\chat-k\chexc}^{0,1}$ for some $k>0$, then the associated $0$-stable sheaf $E'$ in \eqref{corproofexteqn} lies in $\mathcal{M}^0_{\chat-k'\chexc}$ for some $k'\geq k$. By Corollary \ref{corkbound}, there exists some $M\geq 0$, such that  $\mathcal{M}^0_{\chat-k'\chexc}$ is empty for $k'\geq M$. Therefore $\mathcal{M}_{\chat-k\chexc}^{0,1}$ is empty for $k\geq M$.  
		\end{proof}

		\begin{lemma}\label{lemmapsofsstable}
			Let $E_1$ and $E_2$ be $0,1$-semistable with admissible Chern character $\widehat{c}$, and let $f:E_1\to E_2$ be an homomorphism. 
			Then exactly one of the following two cases holds: 
			\begin{enumerate}[label=(\alph*)]
				\item $\operatorname{Im}f\simeq \mathcal{O}_C(-1)^{\oplus k}$ for some $k\geq 0$, and both $\Ker f$ and $\operatorname{Coker} f$ are $0,1$-semistable with Chern character $\chat-k\chexc$, or \label{lemmapsofsstablea}
				\item $\Ker f\simeq \operatorname{Coker} f\simeq \mathcal{O}_C(-1)^{\oplus k}$ for some $k\geq 0$, and $\operatorname{Im} f$ is $0,1$-semistable with Chern character $\chat-k\chexc$. \label{lemmapsofsstableb}
			\end{enumerate}
			If $f:E\to E$ is an endomorphism of a $0,1$-semistable sheaf, then there exists a unique $\lambda\in \C$, such that \ref{lemmapsofsstablea} holds for $f-\lambda \operatorname{Id}_E$. 
		\end{lemma}

		\begin{proof}
			For $i=1,2$, let $T_i\subset E_i$ be the $0$-stable subobject, such that $E_i/T_i\simeq \mathcal{O}_C(-1)^{\oplus k_i}$ as in Proposition \ref{propfiltrations} \ref{propfilt2}. Since $T_1$ is $0$-stable, the composition $T_1\to E_1\xrightarrow{f} E_2\to E_2/T_2$ is zero. Therefore $f$ restricts to a map $g:T_1\to T_2$. By Proposition \ref{corstabsimple}, the map $g$ is either an isomorphism or zero. If $g$ is zero, then $f$ factors through a map $T_1/E_1\to E_2$ and we are in case \ref{lemmapsofsstablea} by Corollary \ref{corintsumexc} and Remark \ref{rem01stable} \ref{rem01stable2}. If $g$ is an isomorphism, we are instead in the case \ref{lemmapsofsstableb}.
			The remaining statement is clear.  
		\end{proof}

		\subsection{Stacks of sheaves with flag structures} \label{subsecunderlyingstack}

		For the rest of this section we fix an admissible Chern character $\widehat{c}$ on $\widehat{X}$ and an ample line bundle $L$ on $\widehat{X}$. Furthermore, we let $m$ be a sufficiently large integer, by which we mean that the following conditions should be satisfied:
		\begin{assumption}\label{assumpmlarge}
			\begin{enumerate}[label = (\roman*)]
				\item For any $E'\in \coprod_{j=0}^\infty  \zeronestable[\widehat{c}-j\chexc](\C)$ and any $m'\geq m$, the sheaf $E'\otimes L^{\otimes m'}$ is globally generated with vanishing higher cohomology groups.\label{enumlargeenough1}
				\item For any $m'\geq m$, the sheaf $\mathcal{O}_C(-1)\otimes L^{\otimes m'}$ is globally generated with vanishing higher cohomology groups. \label{enumlargeenough2}
			\end{enumerate} 
		\end{assumption}
		Such an $m$ exists, since the spaces $\mathcal{M}^{0,1}_{\chat-j\chexc}$ are quasicompact by Proposition \ref{propzeronebounded}, and nonempty for only finitely many values of $j$ by Corollary \ref{corkboundzerone}.
		For a coherent sheaf $E$ on $\widehat{X}$, or more generally a family $\mathcal{E}$ of coherent sheaves on $\widehat{X}\times U$, where $U$ is a quasicompact algebraic stack, we abbreviate
		\[\globm{E}:= \Gamma(\widehat{X}, E\otimes L^{\otimes m} ),\]
		\[\globm{\mathcal{E}}:= \pi_{U *} (\mathcal{E}\otimes L^{\otimes m}).\]
		We use $\globm{f}$ for the obvious induced map if $f$ is a morphism of (families of) coherent sheaves over $\widehat{X}$. 
		We also define
		\[N:=\dim \globm{E}; \mbox{ for } E\in \zeronestable=\zeronestable[\chat].\]
		We now define the ambient stack which we will use to define stacks of stable objects:
		\begin{equation}\label{eqNdef}
			\flag{} :=\left\{(E,V^{\bullet}) \, \left\lvert \, \begin{array}{ll}
				E \in \zeronestable;\\
				V^{\bullet }=(V^i)_{i=0}^N \mbox{ is a complete flag in } \globm{E}
			\end{array}\right. \right\}.
		\end{equation}

		\begin{remark}
			Due to Assumption \ref{assumpmlarge}, if $\mathcal{E}$ is a family of $0,1$-semistable sheaves of Chern character $\widehat{c}$ on $\widehat{X}$ over some base $T$, then $\globm{\mathcal{E}}$ is a vector bundle over $T$. Therefore objects of $\flag{}$ make sense over any finite-type $\C$-scheme, and $\flag{}$ is naturally a bundle of flag varieties over $\zeronestable$, in particular it is an algebraic stack of finite type over $\C$. 
		\end{remark}
		
		\begin{definition}
			If $\mathcal{E}$ is a family of $0,1$-semistable sheaves on $\widehat{X}$ over some base $T$, then a \emph{(full) flag structure} on $\mathcal{E}$ is the data of a length $N$-filtration
			\[0=\mathcal{V}^0\subsetneq \mathcal{V}^1\subsetneq \cdots\subsetneq \mathcal{V}^N=\globm{\mathcal{E}}\]
			by vector bundles. 
		\end{definition}
		
		Let $E$ be a coherent sheaf on $\widehat{X}$ with Chern character $\widehat{c}$, and let $V\subset \globm{E}$ be a subspace. 
		Following \cite[Definition 5.1]{NaYo3}, we introduce a notion of stability, which lies in between $0$-stability and $1$-stability and depends on the choice of subspace $V$.
		
		\begin{definition}\label{defvstability}
			\begin{enumerate}[label =(\arabic*)]
				\item 	 We say that $E$ is \emph{stable with respect to $V$} or \emph{$V$-stable}, if it is $0,1$-semistable and the following two properties hold:
				\begin{itemize}
					\item[$(\stabS (V))$:] For all subobjects $S\subset E$ such that $S\simeq \mathcal{O}_C(-1)^{\oplus k}$ for some $k\geq 0$, we have $V\cap \globm{S}=\{0\}\subset \globm{E}$.
					\item[$(\stabT (V))$:] For all subobjects $T\subset E$ such that $E/T\simeq \mathcal{O}_C(-1)^{\oplus k}$ for some $k\geq 1$, we have $V\not\subset \globm{T}$.  
				\end{itemize}
				\item We say a subobject $S\subset E$ \emph{violates $\stabS (V)$} if it is isomorphic to a direct sum of copies of $\mathcal{O}_C(-1)$ and $V\cap \globm{S}\neq \{0\}$. Similarly, a proper subobject $T\subsetneq E$ \emph{violates $\stabT(V)$}, if $E/T$ is isomorphic to a direct sum of copies of $\mathcal{O}_C(-1)$, and $V\subset \globm{T}$.  We also call such objects $\stabS(V)$-destabilizing and $\stabT(V)$-destabilizing respectively. 
				\item For $(E,V^{\bullet})$ a semistable sheaf with flag structure, we also write $(\stabS_\ell)$ for $(\stabS(V^\ell))$ and $(\stabT_\ell)$ for $(\stabT(V^\ell))$. \label{defvstab3}
			\end{enumerate}
		\end{definition} 
		\begin{remark}
			\begin{enumerate}[label=(\roman*)]
				\item 
				It is straightforward to check that if $W\subset V\subset \globm{E}$, then we have the implications $(\stabS(V))\Rightarrow( \stabS(W))$ and $(\stabT(W))\Rightarrow (\stabT(V))$. In particular, for $(E,V^\bullet)\in \flag{}(\C)$ we have $(\stabS_{\ell+1})\Rightarrow (\stabS_{\ell})$ and $(\stabT_{\ell})\Rightarrow (\stabT_{\ell+1})$ for $0\leq \ell\leq N-1$. 
				\item Being stable with respect to $V=\{0\}$ is the same as being $0$-stable,  being stable with respect to $V=\globm{E}$ is the same as being $1$-stable. 
			\end{enumerate}
		\end{remark}
		
		We now show that $V$-stability is open in flat families:
		\begin{proposition}
			Suppose $\mathcal{E}$ is a flat family of sheaves in $\zeronestable$ over some base $T$, and that $\mathcal{V}\subset \globm{\mathcal{E}}$ is a locally free subsheaf of constant rank. Then
			\begin{enumerate}[label=\roman*)]
				\item the set of points $t\in T$ such that $(\stabS(\mathcal{V}_t))$ holds for $(\mathcal{E}_t,\mathcal{V}_t)$ is open,
				\item the set of points $t\in T$ such that $(\stabT(\mathcal{V}_t))$ holds for $(\mathcal{E}_t,\mathcal{V}_t)$ is open and 
				\item the set of $t\in T$ such that $\mathcal{E}_t$ is stable with respect to $\mathcal{V}_t$ is open.
			\end{enumerate}
		\end{proposition}

		\begin{proof}
			To prove that $(\stabS(\mathcal{V}_t))$ is open, we show that there is a proper map from a scheme $Z$ to $T$ whose image is the complement of the locus where $(\stabS(\mathcal{V}_t))$ holds. 
			
			Using the universal sheaf $\mathcal{E}$ on $\widehat{X}\times T$, we construct the scheme 
			\[\mathcal{Q}:= \coprod_{k=1}^{\infty} \operatorname{Quot}_{\widehat{X}\times\flag{}/\flag{}, \widehat{c}-j\chexc}(\mathcal{E}). \]
			It follows from Lemma \ref{lemexcchercarisexc} that $\mathcal{Q}$ parametrizes families of objects $(E, V^\bullet)$ pulled back from $T$, together with a choice of subsheaf $S\subset E$, such that $S\simeq \mathcal{O}_C(-1)^{\oplus k}$ for some $k$. Note also that $\mathcal{Q}$ is in fact quasi-compact over $T$, since the numbers $k$ for which the Quot-schemes appearing in the union are nonempty are bounded by the condition $k\dim\globm{\mathcal{O}_C(-1)}\leq N$. Consequently the natural map $f:\mathcal{Q}\to T$ is proper.  Let $\mathcal{S}\subset f^*\mathcal{E}\twoheadrightarrow \mathcal{E}'$ be the respective universal sub-and quotient sheaves on $\mathcal{Q}$. By our Assumption \ref{assumpmlarge}, the sheaf $\mathcal{O}_C(-1)$ and all sheaves arising as fibers of $\mathcal{E}'$ have vanishing higher cohomology after tensoring with $L^{\otimes m}$. It follows that on $\mathcal{Q}$ we have a short exact sequence of locally free sheaves (whose ranks are only locally constant)
			\[0\to \globm{\mathcal{S}}\to  f^*\globm{\mathcal{E}}\to \globm{\mathcal{E}'}\to 0.\]
			Now let $Z$ be the closed determinantal subscheme of $\mathcal{Q}$ parametrizing the locus where the composition $f^*\mathcal{V}\to  f^*\globm{\mathcal{E}} \to \globm{\mathcal{E}'}$ is not injective. Then the set of points where condition $(\stabS(\mathcal{V}_t))$ does not hold is precisely the image of $Z$, which is closed by properness of $f$. 
			
			The proof of the second statement is analogous and left as an exercise. The last one follows immediately from the first two. 
		\end{proof}
		
		We can now define the stacks that we will use to construct the master space. 
		\begin{definition}
			\begin{enumerate}[label=(\arabic*)]
				\item For any $0\leq \ell\leq N$, we let $\flag{\ell}$ denote the open substack of $\flag{}$ whose $\C$-valued objects are those pairs $(E,V^{\bullet})$ for which $E$ is stable with respect to $V^{\ell}$. By abuse of notation, objects in $\flagell$ are also called \emph{$V^{\ell}$-stable}. We write $\mathcal{N}^{\ell}_{\chat}$ when we want to emphasize the dependence on the Chern character. 
				
				\item For any $0\leq \ell\leq N-1$, we let $\flagellss$ denote the open substack of $\flag{}$ whose $\C$-valued objects are the pairs $(E,V^{\bullet})$ that satisfy $(\stabS_\ell)$ and $(\stabT_{\ell+1})$. Objects in $\flagellss$ are also called \emph{$V^{\ell},V^{\ell+1}$-semistable}. We write $\mathcal{N}^{\ell,\ell+1}_{\chat}$ when we want to emphasize the dependence on the Chern character. 
			\end{enumerate}
		\end{definition}
		\begin{remark}
			The stack $\flagellss$ contains both $\flagell$ and $\flagellplusone$ as open substacks. It should be regarded as a stack of semistable objects  for a stability condition ``on the wall'' between stability with respect to $V^{\ell}$ and stability with respect to $V^{\ell+1}$ respectively. 
		\end{remark}
		
		A $V$-stable sheaf only has scalar automorphisms, due to the following: 
		\begin{lemma}\label{lemautscalarflag}
			Suppose $E_1$ and $E_2$ are $0,1$-semistable sheaves on $\widehat{X}$ with Chern character $\widehat{c}$, and that $E_1$ is stable with respect to $V_1\subset\globm{E_1}$ and $E_2$ is stable with respect to some $V_2\subset \globm{E_2}$. Suppose further that $V_1$ and $V_2$ have the same dimension. Then any homomorphism $\varphi:E_1\to E_2$ that carries $V_1$ into $V_2$ is either zero or an isomorphism. It follows that the only endomorphisms of $E_1$ that preserve $V_1$ are multiplications by scalars.  
		\end{lemma}
		\begin{proof}
			Suppose $f:E_1\to E_2$ is such a homomorphism. By Lemma \ref{lemmapsofsstable}, we are in one of two possible cases: Either $\operatorname{Im} f \simeq \mathcal{O}_C(-1)^{\oplus k}$ or $\Ker f\simeq \operatorname{Coker} f\simeq \mathcal{O}_C(-1)^{\oplus k}$ for some $k$. Suppose the latter holds. Then by condition $(\stabS(V_1))$, it follows that $V_1\cap \globm{\Ker f}=\{0\}$ and therefore, $f$ induces an isomorphism $\globm{f}:V_1\xrightarrow{\sim} V_2$. But then $V_2\subset \globm{\operatorname{Im}(f)}$, and $E_2/\operatorname{Im} f\simeq \mathcal{O}_C(-1)^{\oplus k}$. By condition $(\stabT(V_2))$, we must have $k=0$, such that $f$ is an isomorphism. Now suppose instead that we have $\operatorname{Im}(f)\simeq \mathcal{O}_C(-1)^{\oplus k}$. Then condition $(\stabS(V_2))$ applied to $\operatorname{Im} f$ shows that $V_2\cap \globm{\operatorname{Im} f}=\{0\}$. Therefore $\globm{f}(V_1)=0$, and it follows that $V_1\subset\globm{\Ker f}$. Then $(\stabT(V_1))$ applied to $\Ker f$ shows that $k=0$, hence $f=0$. For the last statement of the lemma, suppose that $(E_1,V_1)=(E_2,V_2)$. Then one can find $\lambda\in \C$, such that $f-\lambda \operatorname{Id}_{E_1}$ is not injective and therefore zero by what we have already shown. It follows that $f=\lambda\operatorname{Id}_{E_1}$. 
		\end{proof}
		
		The following lemma shows that a sheaf with flag structure that is $V^{\ell},V^{\ell+1}$-semistable but not $V^{\ell}$-stable has a distinguished ``destabilizing'' subobject, and analogously if it is semistable, but not $V^{\ell+1}$-stable. This will allows us to define a notion of polystable sheaf and will be important in the construction of the master space in \S \ref{subsecmasterspace}.  
		\begin{lemma}\label{lemmindest}
			Let $(E,V^{\bullet})$ be a $0,1$-semistable sheaf on $\widehat{X}$ with flag structure.
			\begin{enumerate}[label=(\roman*)]
				\item  	Suppose $(E,V^{\bullet})\in \flag{}$ satisfies $(\stabS_j)$ but not $(\stabS_{j+1})$ for some $0\leq j\leq N-1$. Then there exists a unique minimal $(\stabS_{j+1})$-destabilizing subobject $S\subset E$. Moreover, $\globm{S}\cap V^{j+1}$ is one-dimensional, and equals $\globm{S'}\cap V^{j+1}$ for any other $(\stabS_{j+1})$-destabilizing subobject of $E$. \label{lemmindesti}
				\item  Suppose $(E,V^{\bullet})\in \flag{}$ satisfies $(\stabT_{j+1})$ but not $(\stabT_j)$ for some $0\leq j\leq N-1$. Then there exists a unique minimal $(\stabT_j)$-destabilizing subobject $T\subset E$. 		\label{lemmindestii}
				\item Suppose $f:(E,V^\bullet)\to (E',V'^\bullet)$ is a morphism of $V^{\ell},V^{\ell+1}$-semistable sheaves of the same Chern character which is not an isomorphism. 
				Then we are in exactly one of the following two cases: \label{lemmindestiii}
				\begin{enumerate}[label=(\alph*)]
					\item $\operatorname{Im} f$ is a direct sum of copies of $\mathcal{O}_C(-1)$. Then $\Ker f$ is $(\stabT_{\ell})$-destabilizing for $E$, and $\operatorname{Im} f$ is the minimal $(\stabS_{\ell+1})$-destabilizing subobject of $E'$.  \label{lemmindestiiia}
					\item $\Ker f$ is a direct sum of copies of $\mathcal{O}_C(-1)$. Then $\Ker f$ is a $(\stabS_{\ell+1})$-destabilizing subobject of $E$ and $\operatorname{Im}f$ is the minimal $(\stabT_{\ell})$-destabilizing subobject of $E'$. \label{lemmindestiiib}
				\end{enumerate}
			\end{enumerate}
		\end{lemma}
		
		\begin{proof}
			\begin{enumerate}[label=(\roman*)]
				\item By Corollary \ref{corintsumexc}, the set of subobjects of $E$ isomorphic to a sum of copies of $\mathcal{O}_C(-1)$ is closed under sum and intersection. We will show that in fact the set of $(\stabS_{j+1})$-destabilizing subobjects is closed under intersections, from which it follows  that there must be a minimal one.  Let $S_1$ and $S_2$ be $(\stabS_{j+1})$-destabilizing subobjects of $E$. Then $W_1:=V^{j+1}\cap \globm{S_1}$ is nonzero. Since $(\stabS_j)$ holds for $E$, we have $V^j\cap \globm{S_1}=\{0\}$. Since $V^{j}\subset V^{j+1}$ has codimension one, it follows that $W_1$ is one-dimensional. The same argument shows that $W_2:=V^{j+1}\cap \globm{S_2}$ is one-dimensional. Note that $S_1+S_2$ also violates $(\stabS_{j+1})$, and therefore the subspace $\globm{S_1+S_2}\cap V^{j+1}$, which contains $W_1+W_2$, is also one-dimensional. This shows that $W_1=W_2$ and that $S_1\cap S_2$ is again $(\stabS_{j+1})$-destabilizing.  
				\item This case is slightly simpler: By Corollary \ref{corintsumexc}, and since the condition $V_{j}\subset \globm{T}$ is preserved by intersections in $T$, it follows that the intersection of two $(\stabT_{j})$-destabilizing objects is again $(\stabT_{j})$-destabilizing. Therefore the unique minimal $(\stabT_{j})$-destabilizing object $T$ is characterized by the condition that $E/T\simeq \mathcal{O}_C(-1)^{\oplus k}$ for $k$ maximal among all $(\stabT_j)$-destabilizing subobjects of $E$.   
				\item The two cases correspond to the cases of Lemma \ref{lemmapsofsstable}. We only discuss the case that $\Ker f$ is a direct sum of copies of $\mathcal{O}_C(-1)$, since the other one goes similarly. Let $S:=\Ker f\subset E$ and $T:=\operatorname{Im} f\subset E'$. Since $(\stabS_\ell)$ holds for $E$, we have that $V^\ell\cap \globm{S}=\{0\}$, so that $\globm{f}$ induces an isomorphism $V^\ell\xrightarrow{\sim}V'^\ell$. Since $f$ factors through $T$, we find that $T$ is $(\stabT_\ell)$-destabilizing for $E'$. On the other hand, $(\stabT_{\ell+1})$ holds for $E'$, so $\globm{f}:V^{\ell+1}\to V'^{\ell+1}$ is not an isomorphism. We must therefore have $V^{\ell+1}\cap \globm{S}\neq \{0\}$, so $S$ violates $(\stabS_{\ell+1})$. Note that this also shows that $\globm{f}(V^{\ell+1})\subset V'^{\ell}$.  Finally, let $T'\subset T$ be the minimal $(\stabT_{\ell})$-destabilizing subobject of $E'$. Then $\globm{f^{-1}(T)}$ contains $V^{\ell+1}$. Suppose the inclusion $T'\subset T$ was proper. Then $f^{-1}(T')$ would moreover be a proper subobject of $E$ with $E/f^{-1}(T)$ a direct sum of copies of $\mathcal{O}_C(-1)$. So we would have found a $(\stabT_{\ell+1})$-destabilizing subobject of $E$. This shows that, in fact, we must have $T'=T$. 
			\end{enumerate}
		\end{proof}
		
		\begin{remark}\label{remdestissimple}
			\begin{enumerate}[label=(\roman*)]
				\item Let $S$ be a minimal $(\stabS_{j+1})$-destabilizing subobject of $E$ as in Lemma \ref{lemmindest}, and write $V_S:=V^{j+1}\cap\globm{S}$. Then the pair $(S,V_S)$ is simple in the sense that the only endomorphisms of $S$ that respect $V_S$ are multiplications by scalars. Equivalently, there exists no nontrivial direct summand $S'\subset S$ such that $V_S\subset \globm{S'}$. To see this, we may suppose that $\varphi$ is a nonzero endomorphism of $S$ with a nontrivial kernel. Then either $\globm{\varphi}$ sends $V_S$ to zero, in which case $\Ker\varphi$ is a proper subobject of $S$ that still violates $(\stabS_{j+1})$, or it does not send $V_S$ to zero, in which case $\varphi(S)$ is a proper subobject that violates $(\stabS_{j+1})$. In either case, we have a contradiction to minimality of $S$.
				\item 	In a similar way one sees that a minimal $(\stabT_j)$ destabilizing subobject $T$ is stable with respect to the subspace $V^j\subset \globm{T}$. 
			\end{enumerate}
		\end{remark}
		
		We now define a notion of polystable sheaf, and show some of its basic properties. 
		
		\begin{definition}\label{defpolystable}
			A $V^{\ell},V^{\ell+1}$-semistable pair $(E,V^{\bullet})$  of a sheaf with flag structure is called \emph{ $V^{\ell},V^{\ell+1}$-polystable} if either
			\begin{enumerate}[label=(\roman*)]
				\item  it is both $V^{\ell}$ and $V^{\ell+1}$-stable, or if
				\item 	it is a direct sum $S\oplus T$, where $S$ and $T$ are its minimal $(\stabS_{\ell+1})$-destabilizing and $(\stabT_\ell)$-destabilizing subobjects respectively such that also $V^{j}=V_S^j\oplus V_T^{j}$ for all $j$, where $V_S^j:=V^j\cap \globm{S}$ and $V_T^j:=V^j\cap\globm{T}$.
			\end{enumerate}
			In the latter case, we also call $(E,V^{\bullet})$ \emph{properly polystable}.   We drop the reference to the flag structure $V^{\bullet}$ if it is implicit. 
		\end{definition}
		
		\begin{lemma}\label{lempolystabdefs}
			Let $(E,V^{\bullet})$ be a $V^{\ell},V^{\ell+1}$-semistable sheaf with flag structure. Then the following are equivalent:
			\begin{enumerate}[label=(\roman*)]
				\item The sheaf $E$ is properly polystable. \label{lempolystabdef1}
				\item The sheaf $E$ is not $V^{\ell+1}$-stable, and $(E,V^{\bullet})\simeq (S\oplus E/S, V_S^{\bullet}\oplus V^\bullet/V_S^\bullet)$, where $S$ is the  minimal $(\stabS_{\ell+1})$-destabilizing subobject and $V_S^j=V^j\cap\globm{S}$. \label{lempolystabdef2}
				\item The sheaf $E$ is not $V^{\ell}$-stable, and $(E,V^{\bullet})\simeq (T\oplus E/T, V_T^{\bullet}\oplus V^\bullet/V_T^\bullet)$, where $T$ is the  minimal $(\stabT_{\ell})$-destabilizing subobject and $V_T^j=V^j\cap\globm{S}$. \label{lempolystabdef3}
				\item The pair $(E,V^{\bullet})$ has an endomorphism $\varphi$ which is not multiplication by a scalar. \label{lempolystabdef4}
				\item The pair $(E,V^{\bullet})$ has automorphism group $\C^*\times \C^*$.  \label{lempolystabdef5}
			\end{enumerate}
		\end{lemma}
		\begin{proof}
			The implications \ref{lempolystabdef1}$\Rightarrow$\ref{lempolystabdef2} and \ref{lempolystabdef1}$\Rightarrow$\ref{lempolystabdef3} are immediate by definition of polystability. Also we clearly have that \ref{lempolystabdef5} $\Rightarrow$ \ref{lempolystabdef4}, and moreover that either of \ref{lempolystabdef2} or \ref{lempolystabdef3} imply \ref{lempolystabdef4}, e.g. by considering the projection to followed by inclusion of a summand.  To conclude, it is enough to show that \ref{lempolystabdef4} implies \ref{lempolystabdef1} and \ref{lempolystabdef5}, which is due to the the following
			\begin{claim}
				Suppose that $\varphi:E\to E$ is a non-scalar endomorphism that respects the flag structure. Then $(E,V^\bullet)\simeq (S\oplus T, V_S^{\bullet}\oplus V_T^{\bullet})$ is polystable and $\varphi$ is scalar on each factor.  
			\end{claim} 
			\begin{proof}
				According to Lemma \ref{lemmapsofsstable}, after adding a scalar endomorphism to $\varphi$, we can assume that $\operatorname{Im} \varphi$ is a direct sum of $k\geq 1$ copies of $\mathcal{O}_C(-1)$.  We are therefore in case \ref{lemmindestiiia} of Lemma \ref{lemmindest} \ref{lemmindestiii}, and conclude that $T:=\Ker\varphi$ is a $(\stabT_{\ell})$-destabilizing subobject and $S:=\operatorname{Im}\varphi$ is the minimal $(\stabS_{\ell+1})$-destabilizing subobject of $E$.  By Lemma \ref{lemmindest} \ref{lemmindesti}, the subspace $V_S:=\globm{S}\cap V^{\ell+1}$ is a one-dimensional complement to $V^\ell$ in $V^{\ell+1}$. Since $\globm{\varphi}(V^{\ell+1})\neq \{0\}$, but $\globm{\varphi}(V^\ell)=\{0$\}, it follows that $0\neq\globm{\varphi}(V_S)\subset \globm{\varphi}(S)\cap V^{\ell+1}$. So $\varphi(S)$ is $(\stabS_{\ell+1})$-destabilizing and therefore contains $S$. By comparing Chern characters, it follows that $\varphi(S)=S$. This shows that $E$ splits as $S\oplus T$, and that $\varphi=\varphi_S\oplus 0$, where $\varphi_S$ is an automorphism of $S$ that respects $V_S$ to itself. By Remark \ref{remdestissimple}, $\varphi_S$ must be multiplication by a scalar. Suppose $T'\subset T$ is the minimal $(\stabT_{\ell})$-destabilizing subobject. If the inclusion is proper, then $S+T'$ would be a subobject violating $(\stabT_{\ell+1})$. Therefore $T=T'$. Finally, the claim that the flags split as required follows by observing that $V_T^j=\Ker\globm{\varphi}\mid_ {V^j}$ and $V_S^j=\operatorname{Im}\globm{\varphi}\mid_{V^j}$ for any $j$. 
			\end{proof}
		\end{proof}
		
		The following lemma shows that we can build a polystable sheaf out of its direct summands. 
		\begin{lemma}\label{lemdecpolystables}
			Suppose $(E, V^\bullet)$ is a $0,1$-semistable sheaf with flag structure and that we have a decomposition $E\simeq S\oplus T$, where $S\simeq \mathcal{O}_C(-1)^{\oplus k}$ for some $k>0$, and which induces a decomposition $V^{\bullet}=V_S^\bullet\oplus V_T^\bullet$. Then $E$ is $V^{\ell},V^{\ell+1}$-polystable with associated decomposition $E\simeq S\oplus T$ if and only if all of the following conditions are satisfied:
			\begin{enumerate}[label= (\roman*)]
				\item For all $0\leq j\leq \ell$, we have $V_S^j=\{0\}$,
				\item The pair $(S,V_S^{\ell+1})$ is simple in the sense of Remark \ref{remdestissimple} (In particular, $V_S^{\ell+1}$ is one-dimensional),
				\item The sheaf $T$ is stable with respect to the $\ell$-dimensional subspace $V_T^\ell\subset \globm{T}$. 
			\end{enumerate}
		\end{lemma}
		
		\begin{proof}
			The only if direction follows from Remark \ref{remdestissimple} and the definition of polystable sheaf. For the if direction, we only need to show that the given conditions imply $V^{\ell},V^{\ell+1}$-semistability. Polystability then follows due to the presence of extra automorphisms from Lemma \ref{lempolystabdefs}. Any $(\stabS_\ell)$-destabilizing  $S'\subset E$ induces a subobject $S'\cap T\subset T$ which violates $(\stabS_{V_T^\ell})$. Therefore $(\stabS_\ell)$ holds for $E$. On the other hand, if $T'\subset E$ violates $(\stabT_{\ell+1})$, then  $T\subset T'$, since otherwise $T'\cap T\subset T$ violates $(\stabT_{V_T^\ell})$. Moreover, for such $T'$, we have $V_S^{\ell+1}\subset V^{\ell+1}\subset \globm{T'}$, and therefore $S\subset T'$ due to simpleness of $(S,V_S^{\ell+1})$. It follows that $T'=E$, contradicting the assumption that it is $(\stabT_{\ell+1})$-destabilizing. This shows that in fact $E$ is semistable.  
		\end{proof}

		To an object which is $V^\ell,V^{\ell+1}$-semistable but not stable with respect to one of $V^{\ell}$ or $V^{\ell+1}$, we have an associated polystable object as the following lemma shows. 
		\begin{lemma}\label{lemasspolystable}
			Let $(E,V^{\bullet})$ be a $V^{\ell},V^{\ell+1}$-semistable sheaf with flag structure. 
			\begin{enumerate}[label=(\roman*)]
				\item 	If $E$ is not $V^{\ell+1}$-stable, let $S\subset E$ be the minimal subobject violating $(\stabS_{\ell+1})$. Then $(S\oplus E/S, V_S^{\bullet}\oplus V^{\bullet}/V_S^{\bullet})$ is polystable, where $V_S^j=V^j\cap \globm{S}$.
				\item If $E$ is not $V^{\ell}$-stable, let $T\subset E$ be the mininmal subobject violating $(\stabT_\ell)$. Then $(T\oplus E/T, V_T^\bullet\oplus V^{\bullet}/V_T^{\bullet})$ is polystable, where $V_T^j=V^j\cap \globm{T}$. \label{lemasspolystableii}
			\end{enumerate}  
		\end{lemma}

		\begin{proof}
			Suppose that $S\subset E$ is a minimal $(\stabS_{\ell+1})$-destabilizing subobject. Then $(S,V_S^{\ell+1})$ is simple due to Remark \ref{remdestissimple}. Moreover, $E/S$ is stable with respect to $V^{\ell}/V_S^\ell=V^\ell$: By taking the preimage of a subobject violating $(\stabS({V^\ell/V_S^\ell}))$ or $(\stabT(V^\ell/V_S^\ell))$ in $E/S$, we would get a subobject of $E$ violating $(\stabS_\ell)$ or $(\stabT_{\ell+1})$ respectively. Now the first claim follows from Lemma \ref{lemdecpolystables}. 
			Next, suppose instead that $T\subset E$ is a minimal $(\stabT_\ell)$-destabilizing subobject. By Remark \ref{remdestissimple}, $T$ is stable with respect to $V_T^\ell=V^\ell$. We claim that $(E/T, V^{\ell+1}/V_T^{\ell+1})$ is simple in the sense of Remark \ref{remdestissimple}. In fact, any proper subobject $S'\subset E/T$ which is a sum of copies of $\mathcal{O}_C(-1)$ and such that $V^{\ell+1}/V_T^{\ell+1}\subset\globm{S'}$ pulls back to a subobject of $E$ violating $(\stabT_{\ell+1})$. Again, the claimed polystability follows from Lemma \ref{lemdecpolystables} .
		\end{proof}

		In Subsection \ref{subsec:fix}, we will use the following partial generalization of Lemma \ref{lempolystabdefs} to families
		\begin{lemma}\label{lempolystabfamilies} 
			Let $\mathcal{E}$ be a family of $V^{\ell},V^{\ell+1}$-semistable sheaves over an finite type $\C$-scheme $Y$ and let $f$ be an endomorphism of $\mathcal{E}$ respecting the flag structure. Suppose $y\in Y$ is a closed point, such that the restriction $f_y:\mathcal{E}_y\to\mathcal{E}_y$ is not equal to multiplication by a scalar. Then 
			\begin{enumerate}[label=(\roman*)]
				\item there exists a neighborhood $U$ of $y$, such that the subsheaf $\Sheafhom_{flag}(\mathcal{E},\mathcal{E})$ of endomorphisms that respect the flag structure is free over $U$ with generators $\operatorname{id}_{\mathcal{E}_U}$ and $f\mid_U$, \label{lempolystabfam1}
				\item over each connected component of $U$, we have a splitting $\mathcal{E}\simeq \mathcal{S}\oplus \mathcal{T}$ which induces a splitting of flags $\mathcal{V}^{\bullet}\simeq \mathcal{V}_S^{\bullet}\oplus \mathcal{V}_T^j$,  such that $\mathcal{T}$ is a family of $\mathcal{V}_T^{\ell}$-stable sheaves of Chern character $\chat-k\chexc$ and such that  $\mathcal{S}$ is Zariski-locally isomorphic to  a constant family with fiber $\mathcal{O}_C(-1)^{\oplus k}$ for some $k>0$ and $(\mathcal{S}, \mathcal{V}_S^{\ell+1})$ defines a family of simple sheaves in the sense of Remark \ref{remdestissimple}.\label{lempolystabfam2}
			\end{enumerate}
		\end{lemma}
		\begin{proof}
			In order to prove \ref{lempolystabfam1}, we will prove a version of cohomology and base change for $\Sheafhom_{flag}(\mathcal{E},\mathcal{E})$. We will use the following version of the projection formula: 
			\begin{lemma}\label{lemprojection}
				For any coherent sheaf $M$ on $Y$, the natural morphism $\globm{\mathcal{E}}\otimes M \to \globm{\mathcal{E}\otimes \pi_T^*M}$ of coherent sheaves on $Y$ is an isomorphism. 
			\end{lemma}
			\begin{proof}
				This is a consequence of the usual results on cohomology and base change using that $\mathcal{E}\otimes L^{m}$ is flat and has vanishing higher cohomology.  
			\end{proof}

			We define for any coherent sheaf $M$ on $Y$ a coherent sheaf:
			\[\Homflag{M}:= \Ker\left(\Sheafhom_{\pi_Y}(\mathcal{E},\mathcal{E}\otimes \pi_Y^*M)\to \prod_{j=0}^{N} \Sheafhom(\mathcal{V}^j, \globm{\mathcal{E}}/\mathcal{V}^j \otimes M )\right).\]
			The map between the $\Hom$-sheaves is given as follows: A morphism $f:\mathcal{E}\to \mathcal{E}\otimes \pi_Y^*M$ is pushed down to $Y$, and  then precomposed with the inclusion of $\mathcal{V}^j$ and postcomposed with the quotient map $\globm{\mathcal{E}\otimes \pi_Y^*M}\simeq \Gamma_m(\mathcal{E})\otimes M\to \globm{E}/\mathcal{V}^j\otimes M$ (here we implicitly used Lemma \ref{lemprojection}). Notice that for $M=\mathcal{O}_Z$ the structure sheaf of a closed subscheme of $Y$, we have canonically $\Homflag{ \mathcal{O}_Z}\simeq  \Hom_{flag}(\mathcal{E}_Z,\mathcal{E}_Z)$ is the sheaf of endomorphisms of the restriction $\mathcal{E}_{Z}$ preserving the flag structure.
			
			The association  $M\mapsto \Homflag{M}$ defines a coherent functor in the sense of \cite{Hart}, \S 1. To see this, use Theorem 1.1 a) and Example 2.7 in \cite{Hart}. Moreover, as a kernel of left-exact functors it is a left-exact functor.  
			By our assumption that $f$ is $\C$-linearly independent of $\operatorname{id}_{\mathcal{E}}$ at $y$, the restriction map $\Homflag{ \mathcal{O}_Y}\otimes_{\mathcal{O}_Y} \mathcal{O}_{y}\to \Homflag{\mathcal{O}_y}\simeq \C\times \C$ is surjective. It follows from Proposition 3.9 in \cite{Hart}, that the functor $\Homflag{-}$  is right-exact on some neighborhood $U$ of $y$ in $Y$, and that over $U$ it is given by $\Homflag{M}=\Sheafhom_{flag}(\mathcal{E},\mathcal{E})\otimes M$ for every $M$. By left-exactness, we find that $\Sheafhom_{flag}(\mathcal{E},\mathcal{E})$ must in fact be flat, hence locally free, over $U$. After restricting $U$ further, we can assume that a basis is given by $\operatorname{id}_\mathcal{E}$ and $f$. This proves \ref{lempolystabfam1}. For \ref{lempolystabfam2},  we note that  $\mathcal{A}:=\Sheafhom_{flag}(\mathcal{E},\mathcal{E})$ is in fact a subalgebra of the endomorphism algebra. 
			We have shown it is flat, and since it is fiberwise \'etale it is therefore \'etale over $U$. That is, $U':=\Spec \mathcal{A}\to U$ is a finite \'etale double cover. Pulling back to $U'$, we find that $\Sheafhom_{flag}(\mathcal{E}_{U'},\mathcal{E}_{U'})$ splits as a product $\mathcal{O}_{U'}\times \mathcal{O}_{U'}$. In particular, we have a nontrivial idempotent endomorphism $g$ of $\mathcal{E}_{U'}$. We therefore obtain a decomposition $\mathcal{E}_{U'}\simeq \Ker g\oplus \operatorname{Im} g$ together with a decomposition of flags. Both summands have to be flat, and by restricting to any closed point in $U'$, we see that exactly one of them parametrizes sheaves with Chern character $k\chexc$ for some $k>0$, and the other one parametrizes $\mathcal{V}^{\ell}$-stable sheaves with Chern character $\chat-k\chexc$. Let $\mathcal{S}_{U'}$ denote the former and $\mathcal{T}_{U'}$ denote the latter. We claim that the decomposition $\mathcal{E}_{U'}\simeq \mathcal{S}_{U'}\oplus \mathcal{T}_{U'}$ descends to a decomposition $\mathcal{E}\simeq \mathcal{S}\oplus \mathcal{T}$. It is enough to show that the isomorphism defining the descent datum for $\mathcal{E}_{U'}$ respects $\mathcal{S}_{U'}$ and $\mathcal{T}_{U'}$. Since $\Hom_{flag}(\mathcal{E}_{U'},\mathcal{E}_{U'})$ is generated by the projections to $\mathcal{S}_{U'}$, and $\mathcal{T}_{U'}$, it follows that every flag-preserving endomorphism of $\mathcal{E}_{U'}$ preserves the decomposition $\mathcal{E}_{U'}=\mathcal{S}_{U'}\oplus \mathcal{T}_{U}'$, and this is also true after pullback to $U'\times_{U}U'$.    It remains to show that $\mathcal{S}$ is \'etale locally isomorphic to $\mathcal{O}_C(-1)^{\oplus k}$.  This is at least true fiberwise, from which the result follows, as by Lemma \ref{lemvanishing}, the sheaf $\mathcal{O}_C(-1)$ has no infinitesimal deformations. 
		\end{proof}
		
		The following is the main result about the spaces $\mathcal{N}^{\ell}$. The proof of the existence part of the valuative criterion for properness will take up \S \ref{subsecpropex}.

		\begin{proposition}\label{propproper}
			Each of the stacks $\flagell$ is a proper Deligne--Mumford stack over $\C$. 
		\end{proposition}

		\begin{proof}
			By construction, $\flagell$ is of finite presentation over $\C$. In particular, the diagonal $\flagell\to \flagell\times\flagell$ is of finite presentation. By Remark 8.3.4 in \cite{Olso2}, an algebraic stack whose diagonal morphism is of finite presentation is a  Deligne--Mumford stack if and only if the stabilizer group of every object of $\flagell$ over an algebraically closed field $K$ is unramified over $K$, which in characteristic $0$ is the same as being finite. By Lemma \ref{lemautscalarflag}, (which holds as well with $\C$ replaced by any algebraically closed extension), if $(E,V^{\bullet})\in \flagell(K)$, then the only automorphisms of the sheaf that preserve $V^\ell$ are the scalar ones, and those in fact preserve the whole flag $V^{\bullet}$. Multiplication by a scalar $\lambda$ preserves a given orientation of $E$ exactly if $\lambda$ is an $r$th root of unity, where $r=\operatorname{rk} E$. Therefore $(E,V^{\bullet})$ has automorphism group $\mu_r$ as an object of $\flagell$. This shows that $\mathcal{N}^{\ell}$ is a  Deligne--Mumford stack. We are left with showing properness. For a quasi-separated stack of finite type (in particular for a stack of finite presentation), this can be checked using the valuative criterion. We do this in Propositions \ref{propsepval} and \ref{proppropval}.  
		\end{proof}
		
		\begin{proposition}\label{propsepval}
			The uniqueness part of the valuative criterion for properness holds for $\flagell$. 
		\end{proposition}
		See \cite[\href{https://stacks.math.columbia.edu/tag/0CL9}{Tag 0CL9}]{St20} for a precise statement of the valuative criterion for algebraic stacks.

		\begin{proof}[\it Proof of Proposition \ref{propsepval}.]
			Let $R$ be a DVR over $\C$. Let $\eta$ denote the generic point and $\xi$ the closed point of $\Spec R$ respectively. Let further $(\mathcal{E},\mathcal{V}^\bullet)$ and $(\mathcal{E}',\mathcal{V}'^{\bullet}_2)$ be $\Spec R$-valued points of $\flagell$, together with an isomorphism $f_{\eta}:(\mathcal{E}'_{\eta},\mathcal{V}_{\eta}'^{\bullet})\to (\mathcal{E}_{\eta},\mathcal{V}_{\eta}^{\bullet})$. That means $f_{\eta}$ is an isomorphism $\mathcal{E}'_{\eta}\to \mathcal{E}_{\eta}$ that induces an isomorphism of flags and is compatible with orientations. The last point means that the composition 
			\[\mathcal{O}_{\widehat{X}_\eta}\simeq \det \mathcal{E}'_{\eta}\xrightarrow{\det f}\det \mathcal{E}\simeq \mathcal{O}_{\widehat{X}_\eta} \]
			is the identity map, where the isomorphisms are the natural ones induced by the orientations. 
			We need to show that there is a unique isomorphism $f$ extending $f_{\eta}$ to $\Spec R$.  Due to $R$-flatness of $\mathcal{E}$ and $\mathcal{E}'$, we can regard $\mathcal{E}'$ as a subsheaf of $\mathcal{E}_{\eta}= \mathcal{E}\otimes_R K$ via $f_{\eta}$. Since $\mathcal{E}_{\eta}= \cup_{j\leq 0} \pi^j\mathcal{E}$, and since $\mathcal{E}'$ is coherent on a quasi-compact scheme, there exists a unique $k\in \mathbb{Z}$, such that $\pi^k\mathcal{E}'\subset \mathcal{E}$ and $\pi^k\mathcal{E}'\not\subset \pi \mathcal{E}$.  This proves that $\pi^k f_\eta$ extends uniquely to a morphism $g:\mathcal{E}'\to \mathcal{E}$. Moreover, the restriction of $g$ to the special fiber is nontrivial. Since $g$ respects the flags away from the special fiber it does so over all of $\Spec R$, as this is a closed condition. It follows from Lemma \ref{lemautscalarflag} that $g$ is an isomorphism over the special fiber, hence an isomorphism $(\mathcal{E}',\mathcal{V}'^{\bullet})\to (\mathcal{E}',\mathcal{V})$ of $0,1$-semistable sheaves with flag structure. From the given orientations, it follows that $\det g$ induces an isomorphism $\mathcal{O}_{\widehat{X}_R}\to \mathcal{O}_{\widehat{X}_R}$, whose restriction to the generic fiber equals multiplication by $\pi^{r k}$. It follows that we must have $k=0$. Hence $g$ is in fact, an extension of $f_{\eta}$ to all of $\Spec R$ as an isomorphism of sheaves with flag structures. One easily sees that $g$ is also compatible with orientations. Since the sheaves have no $R$-torsion, it follows that any extension of a morphism defined over $\eta$ must be unique. This finishes the uniqueness part of the valuative criterion.
		\end{proof}

		\subsection{Properness of $\mathcal{N}^{\ell}$ (existence part)}
		\label{subsecpropex}
		In this subsection, we finish the proof of Proposition \ref{propproper} by establishing the following: 
		
		\begin{proposition}\label{proppropval}
			The existence part of the valuative criterion for properness holds for $\flagell$. 
		\end{proposition}
		The proof relies heavily on the use of elementary modifications, which will again appear in the construction of the master space in \S \ref{subsecmasterspace}. Therefore, we begin by including a review of elementary modifications suited to our setting.

		\subsubsection{Elementary modifications}

		Throughout, $R$ denotes a DVR over $\C$, with $\eta$ and $\xi$ denoting the generic and closed point of $\operatorname{Spec} R$ respectively. Moreover, we use $\pi$ to denote some uniformizer of $R$.  All of the constructions and properties work as well over a smooth curve with distinguished closed point instead of $\Spec R$. 
		
		Elementary modifications were introduced in the work of Langton \cite{Lang} as a tool to check valuative criteria for moduli spaces of stable sheaves. The idea is that given a family of sheaves over $\operatorname{Spec} R$ such that the generic fiber is stable,  we can modify the family in a controlled way that preserves the generic fiber, and improves the stability properties of the central fiber. Then after finitely many such modifications one can hope to obtain a central fiber which is also stable, or at least semistable. This works when the notion of stability is Gieseker stability and we will apply the same ideas to our notion of stability.
		
		\begin{definition}
			Let $X$ be any scheme, and let $R$ be a DVR. Suppose $\mathcal{E}$ is an $R$-flat coherent sheaf on $X\times \Spec R$. Let $\mathcal{E}_\xi\twoheadrightarrow Q$ be a quotient sheaf of the generic fiber.  Then the \emph{elementary modification of $\mathcal{E}$ along $Q$} is defined to be the kernel of the composition  $\mathcal{E}\to\mathcal{E}_{\xi}\to Q$. 
		\end{definition}
		
		\begin{lemma} \label{lemexacts}
			Let $\mathcal{E}'$ be the elementary modification of $\mathcal{E}$ along some quotient $\mathcal{E}_\xi\twoheadrightarrow Q$. Then $\mathcal{E}'$ is again flat over $\Spec R$. 
			If we have an exact sequence 
			\[0\to F\to \mathcal{E}_\xi\to Q\to 0\]
			of sheaves on the special fiber, then we have an exact sequence 
			\begin{equation}
				0\to Q\to \mathcal{E}_{\xi}'\to F\to 0,  \label{eqmodif}
			\end{equation}	
			where the map $\mathcal{E}_{\xi}'\to F$ comes from the isomorphism $F\simeq \mathcal{E}'/\pi \mathcal{E}$, and the map $Q\to \mathcal{E}_{\xi}'$ is obtained from $\mathcal{E}\xrightarrow{\cdot \pi}\mathcal{E}'$ by taking the quotient mod $\pi \mathcal{E}'$.   
		\end{lemma}
		This illustrates why we are interested in elementary modifications. They preserve flat families, and if $F$ is a destabilizing subobject of the central fiber of $\mathcal{E}$, it will become a quotient object after the modification. One can then reasonably hope that if the generic fiber is stable, then by a sequence of suitable elementary modifications, one can eliminate all the destablizing subobjects of the special fiber.

		\begin{proof}[\it Proof of Lemma \ref{lemexacts}.]
			Flatness follows from the fact that a module over a DVR is flat if and only if it is torsion-free together with the simple observation that a  submodule of a torsion-free module is torsion-free.  
			To show the statement about the exact sequences, we may assume that $Q=\mathcal{E}/\mathcal{E}'$ and that $F=\mathcal{E}'/\pi \mathcal{E}$. The claim now follows by considering the exact sequence associated to the  chain of inclusions to $\pi\mathcal{E}'\subset \pi \mathcal{E}'\subset \mathcal{E}'$:
			\begin{equation*}
				\begin{tikzcd}
					0\ar[r] &\pi\mathcal{E}/\pi\mathcal{E}'\ar[r]\ar[d,"\simeq"] &\mathcal{E}'/\pi\mathcal{E}'\ar[r]\ar[d,equal] &\mathcal{E}'/\pi\mathcal{E}\ar[r]\ar[d, equal]& 0 \\
					0\ar[r]& Q\ar[r] &\mathcal{E}'_{\xi}\ar[r]& F  \ar[r]& 0.
				\end{tikzcd}
			\end{equation*}
			Here the isomorphism $Q=\mathcal{E}/\mathcal{E}'\to \mathcal{E}\to \pi\mathcal{E}/\pi\mathcal{E}'$ is just multiplication by $\pi$. 
		\end{proof}

		\begin{remark} \label{remelmodinv}
			Suppose we have the exact sequence
			\[0\to F\to \mathcal{E}_\xi\to Q\to 0,\]
			and let $\mathcal{E}'$ be the elementary modification of $\mathcal{E}$ along $Q$. If we let $\mathcal{E}''$ be the elementary modification of $\mathcal{E}'$ along $F$ in the resulting exact sequence 
			\[0\to Q\to \mathcal{E}'_\xi\to F\to 0,\]
			then we have $\mathcal{E}''=\pi \mathcal{E}$. 
		\end{remark}
		Next, we consider the case, where after an elementary modification along $Q$, the sequence \eqref{eqmodif} is split. Then one may choose a splitting, and again do a modification of $\mathcal{E}'$ along $Q$. If one again ends up with a split exact sequence, one can keep going, potentially ad infinitum. It turns out that this process can be geometrically interpreted as extending the section of the Quot-scheme of $\mathcal{E}$ defined by $Q$ to increasingly large infinitesimal neighborhoods of the special fiber as we now make precise. 
		
		\begin{lemma}[cf. Chapter 2.B in \cite{HuLe}]\label{lemseqofmodifications}
			Fix a nonnegative integer $n$ and let $\mathcal{E}$ be an $R$-flat coherent sheaf on $X_R$.  Then we have natural bijections between the following two pieces of data:
			\begin{enumerate}[label=(\arabic*)]
				\item 	Sequences of sheaves $\mathcal{E}_n\subset  \cdots \subset \mathcal{E}_{i} \subset \mathcal{E}_{i-1}\subset \cdots \subset \mathcal{E}_0=\mathcal{E}$ such that $\mathcal{E}_{i+1}$ is obtained from $\mathcal{E}_i$ by an elementary modification along some quotient $\mathcal{E}_{i,\xi}\twoheadrightarrow Q_i$ for $0\leq i\leq n-1$ and such that moreover  $Q_{i}\to \mathcal{E}_{i+1,\xi}\to Q_{i+1}$ is an isomorphism for all $0\leq i\leq n-1$, where the first map is the one induced by multiplication by $\pi$. \label{lemseqofmod1}
				\item Quotients $\mathcal{E}\otimes_R R/(\pi^n)\to \mathcal{Q}_n$ which are flat over $R/(\pi^n)$. \label{lemseqofmod2}
			\end{enumerate}	
			The bijections are given by associating to a a sequence $\mathcal{E}_n\subset\cdots \subset\mathcal{E}$ as in \ref{lemseqofmod1}, the quotient $\mathcal{E}/\pi^n\mathcal{E}\twoheadrightarrow \mathcal{E}/\mathcal{E}_{n}$, and conversely to an $R/(\pi^n)$-flat quotient $\mathcal{E}/\pi^n\mathcal{E}\to \mathcal{Q}_n$ as in \ref{lemseqofmod2} the sequence defined by letting $\mathcal{E}_i$ be the preimage of $\pi^i\mathcal{Q}_n$ under the composition $\mathcal{E}\to \mathcal{E}/\pi^n\mathcal{E}\to \mathcal{Q}_n$. 
		\end{lemma}

		\begin{proof}
			This is essentially a consequence of the local flatness criterion \cite[Theorem 3.2]{AlKl}, which in our situation can be stated as follows: Let $\mathcal{F}$ be a coherent sheaf on $X\times \Spec R/(\pi^n)$, and consider the filtration of $\mathcal{F}$ given by $\mathcal{F}_i:=\pi^i\mathcal{F}$.  Then $\mathcal{F}$ is flat over $R/\pi^n$ if and only if the map $\mathcal{F}_i/\mathcal{F}_{i+1}\xrightarrow{\cdot \pi}\mathcal{F}_{i+1}/\mathcal{F}_{i+2}$ is an isomorphism for each $i=0,1,\ldots,n-2$.  Using this, the proof of the lemma is straightforward and left to the reader.   
		\end{proof}

		\subsubsection{Proof of Proposition \ref{proppropval}}
		The strategy to establish the existence part of the valuative criterion is summarized as follows: Suppose $(\mathcal{E}_{\eta},\mathcal{V}_{\eta}^{\bullet})$ is a family of $V^{\ell}$-stable sheaves on $\widehat{X}$ over the generic point of $\Spec R$. We first complete this to an arbitrary family $(\mathcal{E},\mathcal{V}^{\bullet})$ on $\Spec R$ for which $\mathcal{E}_{\xi}$ is $0,1$-semistable. The central fiber then automatically satisfies the trivial conditions $(\stabS_0)$ and $(\stabT_N)$.  By successively performing a finite number of suitable elementary modifications, one can construct a family whose restriction to the central fiber satisfies $(\stabS_1)$, then by performing more modifications one whose restriction satifies $(\stabS_2)$, and so on, until we arrive at a family whose central fiber satisfies $(\stabS_\ell)$.  By a similar process applied to this new family,  we subsequently obtain families such that $(\stabT_{N-1})$ holds, then $(\stabT_{N-2})$, etc. until we achieve a family satisfying $(\stabT_\ell)$. If we can do this in a way that preserves the property $(\stabS_\ell)$, we obtain an extension of $(\mathcal{E}_{\eta},\mathcal{V}^{\bullet})$ to a $V^{\ell}$-stable family over $\Spec R$. In what follows, we work out the steps in this strategy in detail. 
		
		First, we address the problem of extending a family of $0,1$-semistable sheaves. 
		\begin{lemma}\label{lem01stablext}
			Suppose $\mathcal{E}_{\eta}$ is a $0,1$-semistable sheaf on $\widehat{X}_{\eta}$. After a finite base change on $R$, there exists an extension to a family of $0,1$-semistable sheaves on $\widehat{X}_R$.  
		\end{lemma}

		\begin{proof}
			By abuse of notation, we write $\mathcal{O}_C(-1)$ for its pullbacks to $\widehat{X}_R$ or $\widehat{X}_\eta$.  By Proposition \ref{propfiltrations} and Remark \ref{remfiltrationsanyfield}, there exists a $0$-stable subsheaf $\mathcal{T}_{\eta}\subset \mathcal{E}_{\eta}$, such that $\mathcal{E}_{\eta}$ is an extension.
			\[0\to \mathcal{T}_{\eta}\to \mathcal{E}_{\eta} \to \mathcal{O}_C(-1)^{\oplus k} \to 0,\]
			for some $k$. This defines an extension class $\nu_{\eta}\in \Ext^1_{\widehat{X}_{\eta}}(\mathcal{O}_C(-1)^{\oplus k}, \mathcal{T}_{\eta} )$.
			After possibly performing a finite base change on $R$, we can extend $\mathcal{T}_{\eta}$  to a $0$-stable family $\mathcal{T}$ over $\Spec R$, for example by choosing any orientation and using properness of $\zerostable$. We claim that for some $n$, the class $\pi^n\nu_{\eta}$ is the restriction of a class $\nu\in \Ext^1_{\widehat{X}_{R}}(\mathcal{O}_C(-1)^{\oplus k}, \mathcal{T}_{R} )$. Then we can take $\mathcal{E}$ to be the extension defined by $\nu$. One way to see this, is to resolve $\mathcal{O}_C(-1)^{\oplus k}$ on $\widehat{X}_R$ as 
			\[ \mathcal{F}^1\to \mathcal{F}^0\to \mathcal{O}_C(-1)^{\oplus k}\to 0,\]
			where $\mathcal{F}^0$ is a direct sum of sufficiently negative powers of an ample line bundle, such that the higher derived pushfowwards of $\mathcal{T}_R\otimes \mathcal{F}^{0\vee}$ vanish.Then we get a commutative diagram with exact rows
			\begin{equation*}
				\begin{tikzcd}
					\Hom_{\widehat{X}_R}(\mathcal{F}^0, \mathcal{T}_R)\ar[r]\ar[d]&\Hom_{\widehat{X}_R}(\mathcal{F}^1, \mathcal{T}_R) \ar[r]\ar[d]& \Ext^1_{\widehat{X}_R}(\mathcal{O}_C(-1)^{\oplus k}, \mathcal{T}_R)\ar[r]\ar[d]& 0\\
					\Hom_{\widehat{X}_\eta}(\mathcal{F}_\eta^0, \mathcal{T}_\eta)\ar[r] & \Hom_{\widehat{X}_\eta}(\mathcal{F}_\eta^1, \mathcal{T}_\eta)\ar[r]& \Ext^1_{\widehat{X}_\eta}(\mathcal{O}_C(-1)^{\oplus k}, \mathcal{T}_\eta)\ar[r]& 0.
				\end{tikzcd}
			\end{equation*}
			The first two terms in the lower row are just the localizations at $\pi$ of the terms above them. Since localization is exact, this must also hold for the last term, which gives us what we want.
		\end{proof}

		\begin{definition}
			Let $(\mathcal{E},\mathcal{V}^{\bullet})$ be a family of $0,1$-semistable sheaves with flag structure on $\widehat{X}$ over $\Spec R$. 
			\begin{enumerate}[label=(\arabic*)]
				\item Suppose the central fiber does not satisfy $(\stabS_N)$. Let $0\leq j\leq N-1$ such that the central fiber satisfies $(\stabS_j)$ but not $(\stabS_{j+1})$. Let $S\subset \mathcal{E}_{\xi}$ be the minimal subobject violating $(\stabS_{j+1})$ as in Lemma \ref{lemmindest}. Then we call the elementary modification $\mathcal{E}'$ of $\mathcal{E}$ along $\mathcal{E}_{\xi}$, together with the induced flag structure and orientation an \emph{$S$-modification} of $(\mathcal{E}, \mathcal{V}^{\bullet})$. We also call it an $S$-modification of the sheaf $\mathcal{E}$, when the rest of the data are understood.  
				\item Suppose the central fiber does not satisfy $(\stabT_0)$ and let $0\leq j\leq N$ be such that the central fiber satisfies $(\stabT_{j+1})$ but not $(\stabT_j)$. Let $T\subset \mathcal{E}_{\xi}$ be the minimal subobject violating $(\stabT_j)$ as in Lemma \ref{lemmindest}. Then we call the modification of $\mathcal{E}$ along $\mathcal{E}_{\xi}/T$ together with its induced data a \emph{$T$-modification} of $(\mathcal{E},\mathcal{V}^{\bullet})$ or just of $\mathcal{E}$ if $\mathcal{V}^{\bullet}$ is understood.
			\end{enumerate} 
		\end{definition}
		\begin{remark}\label{remflagstr}
			\begin{enumerate}[label=(\roman*)]
				\item By Remark \ref{rem01stable}, \ref{rem01stable2} the central fiber of $\mathcal{E}'$ will still be $0,1$-semistable after an $S$- or $T$-modification. In each case, the flag structure is the unique one induced from the given one over the generic fiber and properness of Grassmannians.
				\item Suppose we perform an $S$-modification and $(\stabS_j)$ holds for $\mathcal{E}_{\xi}$. Then we have explicitly $\mathcal{V}'^{j}=\pi \mathcal{V}^j\subset \globm{\mathcal{E}'}$. This is since $\globm{S}\cap V_{\xi}^j=\{0\}$, and therefore each map in the composition $\mathcal{V}^j_{\xi}\to \globm{\mathcal{E}_{\xi}/S}\xrightarrow{\cdot \pi} \globm{\mathcal{E}'_\xi}$ is injective. \label{remflagstrii} 
				\item If $(\stabS_j)$ holds for $\mathcal{E}_{\xi}$ but not $(\stabS_{j+1})$, we can describe $\mathcal{V}'^{j+1}$ as follows: Let $V_S:=\globm{S}\cap \mathcal{V}_\xi^{j+1}$, which is one-dimensional by Lemma \ref{lemmindest}. Choose any extension of $V_S$ to a rank one subbundle $\mathcal{V}_S\subset \mathcal{V}^{j+1}$, so in particular $\mathcal{V}^{j+1}=\mathcal{V}^j+\mathcal{V}_S$. Then $\mathcal{V}_S\subset \globm{E}'$, and $\mathcal{V}'^{j+1}= \mathcal{V}'^j+\mathcal{V}_S =\pi \mathcal{V}^{j}+\mathcal{V}_S$. More generally, for any $i$ we can set $V^i_S:=\globm{S}\cap \mathcal{V}^i_\xi$ and choose an extension $\mathcal{V}_S^i$ to a subbundle of $\mathcal{V}^i$. Then $\mathcal{V}'^i=\pi \mathcal{V}^i+\mathcal{V}^i_S$. Moreover, it is true that $\mathcal{V}'^{i}$ is the elementary modification of $\mathcal{V}^{i}$ along its image in $\globm{ \mathcal{E}_{\xi}/S}$. \label{remflagstriii} 
				\item On the other hand, if $(\stabT_j)$ \emph{doesn't} hold for $\mathcal{E}_{\xi}$, and we perform a $T$-modification, then $\mathcal{V}'^j=\mathcal{V}^j\subset\globm{\mathcal{E}'}$. \label{remflagstriv} 
				\item If $(\stabT_j)$ doesn't hold, but $(\stabT_{j+1})$ does, we can describe $\mathcal{V}'^{j+1}$ as follows: Choose a splitting $\mathcal{V}^{j+1}=\mathcal{V}^j\oplus \mathcal{V}_T$ over $\Spec R$. Then $\mathcal{V}'^{j+1}=\mathcal{V}'^j+\pi \mathcal{V}_T =\mathcal{V}^j+\pi \mathcal{V}_T$. \label{remflagstrv} 
			\end{enumerate}
		\end{remark}
		We go on to show that an $S$- or $T$-modification will preserve the properties $(\stabS_{j})$ and $(\stabT_{j})$ appropriately. 
		
		\begin{lemma}\label{lemmodificationproperties}
			Let $(\mathcal{E},\mathcal{V}^{\bullet})$ be a family of $0,1$-semistable sheaves with flag structures over $\Spec R$.
			\begin{enumerate}[label=(\roman*)]
				\item Suppose $\mathcal{E}_\xi$ satisfies $(\stabS_j)$ but not $(\stabS_{j+1})$ and let $\mathcal{E}'$ be the $S$-modification of $\mathcal{E}$. Then the central fiber $\mathcal{E}'_{\xi}$ satisfies $(\stabS_j)$.  If $\mathcal{E}'_{\xi}$ does not satisfy $(\stabS_{j+1})$, let  $S'\subset \mathcal{E}'_{\xi}$ and $S\subset \mathcal{E}_{\xi}$ be the respective  $(\stabS_{j+1})$-destabilizing subobjects. Then the composition $S'\to \mathcal{E}'_{\xi}\to S$ is surjective.    
				\item Suppose $\mathcal{E}_{\xi}$ satisfies $(\stabT_{j+1})$ but not $(\stabT_j)$ and let $\mathcal{E}'$ be the $T$-modification of $\mathcal{E}$. Then the central fiber $\mathcal{E}'_{\xi}$ satisfies $(\stabT_{j+1})$.  If $\mathcal{E}_{\xi}'$ does not satisfy $(\stabT_j)$, let $T\subset \mathcal{E}_{\xi}$ and $T'\subset \mathcal{E}'_{\xi}$ be the minimal $(\stabT_{j})$-destabilizing subobjects. Then the composition $T'\to \mathcal{E}'_{\xi}\to T$ is surjective. 
				\item Suppose $\mathcal{E}_{\xi}$ satisfies $(\stabS_j)$, but not $(\stabT_j)$. Let $\mathcal{E}'$ be the $T$-modification of $\mathcal{E}$. Then $\mathcal{E}'_\xi$ satisfies $(\stabS_j)$. 
				\item Suppose $\mathcal{E}_{\xi}$ satisfies $(\stabT_j)$ but not $(\stabS_j)$. Let $\mathcal{E}'$ be the $S$-modification of $\mathcal{E}$. Then $\mathcal{E}'_\xi$ satisfies $(\stabT_j)$.  
			\end{enumerate}
		\end{lemma}

		\begin{proof}
			\begin{enumerate}[label =(\roman*)]
				\item First let $S'\subset \mathcal{E}_{\xi}'$ be any subobject isomorphic to a direct sum of copies of $\mathcal{O}_C(-1)$. Then the preimage $f^{-1}(S')$ of $S'$ under $f:\mathcal{E}_{\xi}\to \mathcal{E}_{\xi}/S\xrightarrow{\cdot \pi} \mathcal{E}'_{\xi}$ is again a direct sum of copies of $\mathcal{O}_C(-1)$. If $\globm{S'}\cap \mathcal{V}_{\xi}'^j\neq 0$, it follows from Remark \ref{remflagstr} \ref{remflagstrii} that $\globm{f^{-1}(S')}\cap \mathcal{V}^j_{\xi}\neq 0$, and therefore would violates $(\stabS_{j})$. This would contradict our assumption, so in fact we must have $\globm{S'}\cap \mathcal{V}'^{j}_{\xi} =\{0\}$. This shows that $\mathcal{E}'_\xi$ satisfies $(\stabS_j)$. Now suppose that $S'\subset \mathcal{E}'_{\xi}$ is the minimal subobject violating $(\stabS_{j+1})$, and consider the map $g:S'\to \mathcal{E}'_\xi\to S$. Then $\globm{g}$ is nonzero on $\mathcal{V}'^{j+1}\cap \globm{S}'$ for example by Remark \ref{remflagstr} \ref{remflagstriii}. Thus, $g(S')$ gives a subsheaf of $\mathcal{E}_{\xi}$ violating $(\stabS_{j+1})$. Since $S$ is the minimal such, it follows that $g$ is surjective. 
				\item This goes similarly.  Let $T'\subset \mathcal{E}'_\xi$ be a subobject such that $\mathcal{E}'_\xi/T'$ is isomorphic to a sum of copies of $\mathcal{O}_C(-1)$ and such that $\mathcal{V}^j_{\xi}\subset \globm{T'}$. Consider the composition $g:T'\to \mathcal{E}'_{\xi}\to T$. It follows from Remark \ref{remflagstr} that $\mathcal{V}_{\xi}^j\subset g(T')$. Moreover, $\mathcal{E}_{\xi}/g(T')$ is a direct sum of copies of $\mathcal{O}_C(-1)$. Since $T$ is the minimal $(\stabT_j)$-destabilizing subobject of $\mathcal{E}_\xi$, we find that $g(T')=T$. This already shows the last statement. Now suppose $T'$ was in fact $(\stabT_{j+1})$-destabilizing.  We view $S:=\mathcal{E}_{\xi}/T$ as a subobject of $\mathcal{E}'_\xi$ via the exact sequence 
				\[0\to S\to \mathcal{E}'_\xi\to T\to 0. \]
				Since $g$ is surjective, but $T'\neq \mathcal{E}'_{\xi}$, the inclusion $T'\cap S\subset S$ is proper. Moreover, we have the inclusion $V_T:=\mathcal{V}'^{j+1}_{\xi}\cap\globm{S} \subset\globm{T'\cap S}$. But by Remark \ref{remflagstr} \ref{remflagstrv}, $V_T$ equals the image of $\mathcal{V}^{j+1}_{\xi}$ in $\globm{Q}$. It follows that the preimage of $T'\cap S$ in $\mathcal{E}_{\xi}$ violates $(\stabT_{j+1})$, a contradiction. So, in fact, $\mathcal{E}'_\xi$ satisfies $(\stabT_{j+1})$. 
				\item  Let $S'\subset \mathcal{E}'_\xi$ be any subobject isomorphic to a direct sum of copies of $\mathcal{O}_C(-1)$. We consider the composition $f:\mathcal{E}'_\xi\to T\to \mathcal{E}_\xi$. Then $f(S')$ is again a direct sum of copies of $\mathcal{O}_C(-1)$. Since $\mathcal{E}_\xi$ satisfies $(\stabS_j)$, we have $\globm{f(S')}\cap \mathcal{V}^{j}_\xi=\{0\}$.  By Remark \ref{remflagstr} \ref{remflagstriv}, $\globm{f}$ induces an isomorphism on $\mathcal{V}'^j_{\xi}\to \mathcal{V}^j_{\xi}$. Therefore $\globm{S'}\cap \mathcal{V}'^j_{\xi}=0$. Hence, $\mathcal{E}'_\xi$ satisfies $(\stabS_j)$. 
				\item Suppose that $S$ is the minimal $(\stabS_{i+1})$-destabilizing subobject of $\mathcal{E}_\xi$ for some $i<j$, such that $\mathcal{E}_\xi$ satisfies $(\stabS_i)$ and let $Q:=\mathcal{E}_\xi/S$. Then $\mathcal{V}'^{i+1}=\pi \mathcal{V}^i+\mathcal{V}_S$ and $\mathcal{V}'^j=\pi \mathcal{V}^j+\mathcal{V}_S^j$ as in Remark \ref{remflagstr} \ref{remflagstriii}. Suppose $T\subset \mathcal{E}'_\xi$ violates $(\stabS_j)$. Then in particular $\mathcal{V}'^{i+1}\subset \globm{T}$. We consider the composition $g:\mathcal{E}_\xi'\to S\to \mathcal{E}_\xi$. It follows that $V_S\subset \globm{g(T)}$, and therefore that $g(T)$ violates $(\stabS_{i})$ in $\mathcal{E}_\xi$. So we must have $S\subset g(T)$ by minimality of $S$, i.e. $T$ maps onto $S$ under $\mathcal{E}_\xi\to S$. On the other hand, suppose $T\cap Q\subsetneq Q$ where the inclusion $Q\to \mathcal{E}_\xi'$ is induced by multiplication with $\pi$. Then as $\pi\mathcal{V}^j_\xi\subset T$, it is easily seen that the preimage of $T$ in $\mathcal{E}_\xi$ violates $(\stabT_j)$. So in fact we must have $T=\mathcal{E}'_\xi$ which contradicts the assumption that it is $(\stabT_{j})$-destabilizing for $\mathcal{E}'_\xi$.  
			\end{enumerate}
		\end{proof}

		\begin{lemma}\label{lemfinsmods}
			Suppose $(\mathcal{E}, \mathcal{V}^{\bullet})$ is a family of $0,1$-semistable sheaves with flag structure over $\Spec R$, such that the generic fiber satisfies $(\stabS_{j+1})$ and such that the special fiber satisfies $(\stabS_j)$. Then after finitely many $S$-modifications (possibly zero), we obtain a family that satisfies $(\stabS_{j+1})$.  
		\end{lemma}

		\begin{proof}
			We prove this by contradiction. Suppose that the statement of the lemma is false. Then we get an infinite descending sequence of sheaves $\mathcal{E}= \mathcal{E}_0 \supset \cdots\supset \mathcal{E}_n\supset\mathcal{E}_{n+1} \supset\cdots$ with flag structures $\mathcal{V}_n^{\bullet}$, such that $\mathcal{E}_{n+1}$ is obtained from $\mathcal{E}_n$ through an $S$-modification. By Lemma \ref{lemmodificationproperties}, each $\mathcal{E}_n$ satisfies $(\stabS_j)$, and by our assumption has a central fiber that does not satisfy $(\stabS_{j+1})$. Let $S_n\subset \mathcal{E}_{n,\xi}$ be the minimal $(\stabS_{j+1})$-destabilizing subobject and set $Q_n:=\mathcal{E}_{n,\xi}/S_n$.  Then Lemma \ref{lemmodificationproperties} further states that the maps $f_n: S_{n+1}\to \mathcal{E}_{n,\xi}\to S_{n}$ are all surjective.  For each $n$, define $k_n$ to be the rank of $S_n$ restricted to the exceptional divisor $C$, so that $S_n$ is a direct sum of $k_n$ copies of $\mathcal{O}_C(-1)$. By considering Chern characters, $k_{n+1}>k_n$ whenever $f_n$ is not an isomorphism. However, due to Assumption \ref{assumpmlarge}, all the $k_n$ are a priori bounded by $k_n\dim \globm{\mathcal{O}_C(-1)}\leq \dim\globm{\mathcal{E}_\xi}$. Thus, eventually all $f_n$ are isomorphisms. Replacing $\mathcal{E}$ by $\mathcal{E}_n$ for $n$ big enough, we can assume that all the $f_n$ are isomorphisms.  In other words, for each $n$ the inclusion  $S_{n+1}\subset \mathcal{E}_{n+1,\xi}$ splits the exact sequence
			\[0\to Q_n\to \mathcal{E}_{n+1}\to S_n\to 0.\]
			Equivalently, the quotient $Q_{n+1}$ splits the sequence, i.e. the compositions $Q_n\to \mathcal{E}_{n+1,\xi}\to Q_{n+1}$ are isomorphisms. Therefore, we are in the situation of Lemma \ref{lemseqofmodifications}.  It follows that we have a flat quotient $\mathcal{E}/(\pi^n)\to \mathcal{Q}_n$ over $\Spec R/(\pi^n)$ that restricts to $\mathcal{E}_{0,\xi}\to Q_0$ over $\xi$ for any $n\geq 0$.  
			Moreover, we also have the sequence 
			\[\cdots \subset \mathcal{V}^{j+1}_{n+1}\subset \mathcal{V}^{j+1}_n\subset \cdots \subset \mathcal{V}^{j+1}_0=\mathcal{V}^{j+1} \]
			where $\mathcal{V}^i_n$ is the induced $i$-th piece of the flag structure on $\mathcal{E}_n$. Due to Remark \ref{remflagstr} \ref{remflagstriii}, each entry in  ,this sequence is the elementary modification of the preceeding one along the image of $\mathcal{V}^{j+1}_n\to \globm{Q_n}$, or equivalently along the surjection $\mathcal{V}^{j+1}_{n,\xi}\to \mathcal{V}^j_{n,\xi}$ obtained from the splitting $\mathcal{V}^{j+1}_{n,\xi}=\mathcal{V}^j_{n,\xi}\oplus V_{n,S}$.   Since $\mathcal{V}^{j}_{n+1}=\pi\mathcal{V}^j_n$, in particular, the condition of Lemma \ref{lemseqofmodifications} is fulfilled. This means that for all $n$, the sheaf $\mathcal{W}_n:=\mathcal{V}^{j+1}_0/\mathcal{V}^{j+1}_n$ is an $R/(\pi^n)$-flat quotient of $\mathcal{V}^{j+1}\otimes_R R/(\pi^n)$ which restricts to $\mathcal{V}^{j+1}_\xi\to \mathcal{V}^{j+1}_\xi/V_S$ on the special fiber. We also have for any $n$ that $\mathcal{V}^{j+1}_{n+1}=\mathcal{V}^{j+1}_{n}\cap \globm{\mathcal{E}_{n+1}}$, and therefore by induction $\mathcal{V}^{j+1}_n=\mathcal{V}^{j+1}\cap \globm{\mathcal{E}_n}$. This shows that $\mathcal{W}_n$ is equal to the image of the map $\mathcal{V}^{j+1}\otimes_R R/(\pi^n)\to \globm{\mathcal{Q}_n}$. To summarize, we have shown that on each infinitesimal neighborhood $\iota_n:Z_n:=\Spec R/(\pi^n)\hookrightarrow \Spec R$ of the closed point we have some flat quotient $\iota_n^*\mathcal{E}\to \mathcal{Q}_n$ extending $\mathcal{E}_\xi\to Q$, such that the composition $\iota_n^*\mathcal{V}^{j+1}\to\globm{\iota^*\mathcal{E}} \to\globm{\mathcal{Q}_n}$ factors through some rank $j$ vector bundle $\mathcal{W}_n\subset \globm{\mathcal{Q}_n}$. We now show that this creates a contradiction to the assumption that the generic fiber satisfies $(\stabS_{j+1})$.
			Consider the relative Quot scheme $\sigma:\Quot_{X_R/R}(\mathcal{E},\chat(Q))\to \Spec R$ and let $\sigma^*\mathcal{E}\to \mathcal{Q}$ denote the universal quotient. We consider the map $f:\sigma^*\mathcal{V}^{j+1}\to \globm{\mathcal{Q}}$.  By Lemma \ref{lemexcchercarisexc}, all the points in the Quot-scheme parametrize a quotient with Kernel isomorphic to a direct sum of copies of $\mathcal{O}_C(-1)$. Thus, the condition $(\stabS_{j+1})$ guarantees that the map $f$ is injective over every point of the generic fiber. On the other hand, the restriction of $f$ to the point corresponding to  $E_{\xi}\to Q$ is not injective by assumption. We let $Z$ be the zero locus of the map $\wedge^{j+1} f:\sigma^*(\wedge^j\mathcal{V}^{j+1})\to \globm{\mathcal{Q}}^{j+1}$. It has the universal property, that a map $\alpha:T\to \Quot_{X_R/R}(\mathcal{E},H_Q)$ factors through $Z$ if and only if $\alpha^*(\wedge^{j+1}f)=0$. Then in particular, $Z$ contains $[E_{\xi}\to Q]$, but is set-theoretically supported only over the special fiber. Its scheme-theoretic image must therefore be contained in some closed subscheme of $\Spec R$ supported on $\{\xi\}$, i.e. in $Z_n$ for some $n$. Now consider the map $\alpha_{n+1}:Z_{n+1}\to \Quot_{X_R/R}(\mathcal{E},H_Q)$ induced by $\mathcal{Q}_{n+1}$. As we have shown, $\alpha_{n+1}^*f$ factors through the rank $j$ bundle $\mathcal{W}_{n+1}$. Hence, in particular, $\alpha_{n+1}^*(\wedge^{j+1} f)=0$, namely, $\alpha_{n+1}$ factors through $Z$. But this would imply that the closed subscheme $Z_{n+1}$ factors through $Z_n$, which is a contradiction. 
		\end{proof}
		
		\begin{lemma}\label{lemfintmods}
			Suppose $(\mathcal{E}, \mathcal{V}^{\bullet})$ is a family of $0,1$-semistable sheaves with flag structure over $\Spec R$, such that the generic fiber satisfies $(\stabT_{j})$ and such that the special fiber satisfies $(\stabT_{j+1})$. Then after finitely many $S$-modifications (possibly zero), we obtain a family that satisfies $(\stabT_{j+1})$. Moreover, if we assume that the special fiber initially satisfies $(\stabS_\ell)$ for some $\ell\leq j$, then this is still true for any of the modifications.   
		\end{lemma}

		\begin{proof}
			The last statement is immediate from Remark \ref{remflagstr}. The rest of the proof is along the same lines as the one in Lemma \ref{lemfinsmods}. In particular, we obtain a sequence of quotients  $\iota_n^*\mathcal{E}\to \mathcal{Q}_n$ over all infinitesimal neighborhoods $\iota_n:Z_n=\Spec R/(\pi^n)\hookrightarrow \Spec R$ of $\xi$ and for which the natural map $\iota^*\mathcal{V}^{j}\to \globm{\mathcal{Q}_n}$ vanishes. Using the relative Quot-scheme as in the proof of Lemma \ref{lemfinsmods}, we see that this contradicts the assumption that $(\stabT_j)$ holds for the generic fiber.
		\end{proof}

		We can now prove the existence part of the valuative criterion for properness of $\flagell$.
		\begin{proof}[Proof of Proposition \ref{proppropval} ]
			Suppose $(\mathcal{E}_\eta, \mathcal{V}_{\eta}^{\bullet})\in \flagell(\eta)$ is an oriented $V_\ell$-stable sheaf over the generic point of $\Spec R$. Thus, it satisfies properties $(\stabS_\ell)$ and $(\stabT_\ell)$ of Definition \ref{defvstability}. We need to show that, possibly after a base change, there exists a family of oriented $V_\ell$-stable sheaves over $\Spec R$, such that its restriction to the generic fiber is isomorphic to the given one. By Lemma \ref{lem01stablext}, we can extend the sheaf $\mathcal{E}_\eta$ to a family $\mathcal{E}'$ of $0,1$-semistable sheaves on $\Spec R$ (possibly after performing a base change). By properness of Grassmannians, we can automatically extend the flag $\mathcal{V}_{\eta}^{\bullet}$ to obtain a family $(\mathcal{E}',\mathcal{V}'^{\bullet})$ of $0,1$-semistable sheaves with flag structure. The special fiber of this family satisfies $(\stabS_0)$ and $(\stabT_N)$, while the generic fiber satisifes $(\stabS_j)$ for all $0\leq j\leq \ell$ and $(\stabT_i)$ for all $\ell\leq i\leq N$. By applying Lemma \ref{lemfinsmods} a number of $\ell$ times, after finitely many $S$-modifications we can assume that the special fiber satisfies $(\stabS_\ell)$. Then by applying Lemma \ref{lemfintmods} a number of $N-\ell$ times, we can assume that the special fiber satisfies both $(\stabS_\ell)$ and $(\stabT_\ell)$, i.e. that we have a $V_\ell$-stable family. Let $(\mathcal{E}', \mathcal{V}'^{\bullet})$ now denote this new family. Its generic fiber is isomorphic to $(\mathcal{E}_\eta,\mathcal{V}^{\bullet}_\eta)$ via some isomorphism $\varphi:\mathcal{E}'_\eta\to \mathcal{E}_\eta$. Pick any orientation of $\mathcal{E}'$ over $\Spec R$. Then $\varphi$ and the chosen orientations induce a map $\mathcal{O}_{X_\eta}\simeq \det\mathcal{E}'_\eta\xrightarrow{\det \varphi} \det \mathcal{E}_\eta  \simeq \mathcal{O}_{X_\eta}$, which by properness of $X$ must be the multiplication with some $a\in K^*$, where $K$ is the field of fractions of $R$. After passing to some valuation ring dominating $R$, we can assume that $K$ contains an $r$-th root of $a$. Then  $a^ {-1/r}\varphi$  is an isomorphism between $(\mathcal{E}'_\eta,\mathcal{V}'^\bullet_\eta)$ and $(\mathcal{E}_\eta,\mathcal{V}^\bullet_\eta)$ as oriented $V^\ell$-stable sheaves over $K$.  
		\end{proof}
		
		\begin{corollary}\label{corvalpropss}
			The existence part of the valuative criterion for properness holds for $\flagellss$. 
		\end{corollary}
		\begin{proof}
			The same proof works, except that one now applies Lemma \ref{lemfintmods} only $N-\ell-1$ times, to achieve a special fiber that satisfies $(\stabT_{\ell+1})$. 
		\end{proof}

		\subsection{Construction of the enhanced master space} \label{subsecmasterspace}

		We now construct the master space connecting the spaces $\flagell$ and $\flagellplusone$. We mostly follow the construction given by Kiem and Li in \cite{KiLi}, fill out some details, and exhibit how their method adapts to our situation. We suppress mentioning the flag structure on sheaves in this section unless needed.

		We fix a large enough integer $M$ which satisfies
		\begin{equation}
			\chi_1:=\int_X c\cdot \operatorname{ch}(H^{\otimes M}) \cdot \operatorname{td}(X)>0,  
		\end{equation} 
		where $\mathcal{O}_X(1) = H$ and $c$ are our fixed choices of an ample line bundle and a Chern character on $X$ respectively. This number is chosen such that for any $0,1$-semistable sheaf $E$ with Chern character $p^*c-j\chexc$ for some $j\geq 0$, we have $\chi(E\otimes p^*\mathcal{O}_X(M))=\chi_1$, as can be seen by the Grothendieck--Riemann--Roch formula. We also set 
		\begin{equation}
			\chi_2:= \int_X \widehat{c} \cdot \operatorname{ch}(p^*\mathcal{O}_X(M))\cdot \operatorname{ch}(\mathcal{O}_{\widehat{X}}(-C)) \cdot \operatorname{td}(\widehat{X}).
		\end{equation}
		For any coherent sheaf $E$ on $\widehat{X}$ of Chern character $\widehat{c}$, we have $\chi(E(-C)\otimes p^*\mathcal{O}_X(M))=\chi_2$. 
		The following is a variant of Definition 2.6 in \cite{KiLi}: 
		\begin{construction}\label{constrdetl}
			Let $T$ be a $\C$-scheme and let $\pi_{\widehat{X}}:\widehat{X}\times T\to T$ be the projection. To any $T$-flat coherent sheaf $\mathcal{E}$ on $\widehat{X}\times T$, we associate a line bundle $\detL(\mathcal{E})$ on $T$ via
			\[\detL(\mathcal{E}):= \left(\det R \pi_*(\mathcal{E}\otimes p^*\mathcal{O}_X(M))\right)^{\chi_2} \left( \det R\pi_*(\mathcal{E}(-C)\otimes p^*\mathcal{O}_X(M))\right)^{-\chi_1}.\]
			This is functorial for isomorphisms between coherent sheaves.
		\end{construction}
		\begin{remark}
			More generally, this makes sense - and the expected functorialities hold - for not neccessarily $T$-flat $\mathcal{E}$ whenever the determinants appearing in the definition of $\detL(\mathcal{E})$ are defined. For example, this is always the case when $T$ is the spectrum of a discrete valuation ring. We will  use the same notation also in this more general situation 
		\end{remark}

		The following properties follow from the properties of the determinant functor. By abuse of notation, let $\rho_t$ denote multiplication by $t\in \C^*$ on various sheaves.
		\begin{lemma}\label{lemlprops}
			
			Let $\mathcal{E}$ be a coherent sheaf on $\widehat{X} \times T$ for which $\detL(\mathcal{E})$ is defined.  
			\begin{enumerate}[label=(\roman*)]
				\item  If each fiber $\mathcal{E}_t$ has Chern character $\chat$, then we have $\detL(\rho_t)=\operatorname{id}_{\detL(E)}$.\label{lemlprops1}
				\item The functor $\detL$ is multiplicative in exact sequences, i.e. given an exact sequence of coherent sheaves on $\widehat{X} \times T$
				\[0\to \mathcal{F}_1\to \mathcal{E}\to \mathcal{F}_2\to 0,\]
				we have a natural isomorphism $\detL(\mathcal{E})\simeq \detL(\mathcal{F}_1)\otimes \detL(\mathcal{F}_2)$ whenever $\detL$ is defined for all terms.\label{lemlprops2}
				\item For the sheaf $\mathcal{O}_C(-1)^{\oplus k}$ on $C \times T$ we have that $\detL(\rho_t)$ acts as $t^{-k\chi_1}$ on $\detL(\mathcal{O}_C(-1))$. If $F$ is a sheaf with Chern character $\chat-k\chexc$ on $\widehat{X} \times T$, then $\detL(\rho_t)$ acts on $\detL(F)$ as multiplication by $t^{k\chi_1}$ . \label{lemlprops3}
				\item Let $\nu(\mathcal{E})$ denote the locally constant function on $T$ such that locally $\detL(\rho_t)$ acts as $t^{\nu(\mathcal{E})}$ on $\detL(\mathcal{E})$. Then for any line bundle $A$ on $T$, we have a natural isomorphism $\detL(\mathcal{E}\otimes \pi_{\Xhat}^*A)\simeq \detL(\mathcal{E})\otimes A^{\nu(E)}$. \label{lemlprops4}
				\item Suppose we are given an inclusion $\mathcal{E}'\subset \mathcal{E}$ of $R$-flat sheaves over $\Xhat_R$ for some DVR $R$, such that the quotient $Q:=\mathcal{E}/\mathcal{E}'$ is supported on the special fiber and suppose that on the special fiber we have $\detL(\rho_t) = t^{\nu(Q)} \operatorname{id}_{Q}$ on $\detL(Q)$. Then the rational map $r:\detL(\mathcal{E}')\dasharrow\detL(\mathcal{E})$ induced by functoriality over the generic fiber has a zero of order $\nu(Q)$.  \label{lemlprops5}
			\end{enumerate}
		\end{lemma}

		\begin{proof}
			We only show \ref{lemlprops5}, as the other parts being straightforward. Since the determinant behaves well with respect to pullback, and since the determinant of the zero sheaf comes with a natural trivialisation, we get a natural rational section $s:\mathcal{O}_{\Spec R}\dasharrow \detL(Q)$. We claim that this section has a zero of order $\nu(Q)$ (respectively a pole, if this is negative). In fact, suppose we had an exact sequence $0\to Q_R\xrightarrow{\cdot \pi} Q_R \to Q\to 0$, where $Q_R$ is $R$-flat. It induces an isomorphism $\detL(Q_R)\xrightarrow{\sim}\detL(Q_R)\otimes \detL(Q)$. Moreover, this isomorphism fits into the commutative diagram
			\begin{equation*}
				\begin{tikzcd}
					\detL(Q_R) \ar[r,dashed, "\detL(\pi)^{-1}"]\ar[dr, "\sim"]&\detL(Q_R) \ar[d,dashed,"\operatorname{id}\otimes s"]\\
					&  \detL(Q_R)   \otimes \detL(Q)
				\end{tikzcd}
			\end{equation*}
			as one sees by restricting to the generic point of $R$. 
			Since $\detL(\pi)=\pi^{\nu(Q)}$, this proves the claim about $s$. If $Q$ does not fit into such a sequence, one can still prove the result by applying this method to the derived pushforwards appearing in the definition of $\detL$. 
			Now we consider the sequence $0\to \mathcal{E}'\to \mathcal{E}\to Q\to 0$. We obtain an isomorphism $\detL(\mathcal{E})\xrightarrow{\sim} \detL(\mathcal{E}')\otimes \detL(Q)$, which fits into a commutative diagram 
			\begin{equation*}
				\begin{tikzcd}
					\detL(\mathcal{E}) \ar[r,dashed, "r^{-1}"]\ar[dr, "\sim"] &\detL(\mathcal{E}') \ar[d,dashed,"\operatorname{id}\otimes s"]\\
					&  \detL(\mathcal{E}')   \otimes \detL(Q). 
				\end{tikzcd}
			\end{equation*}
			It follows that $r$ has a zero of order $\nu(Q)$ as claimed.
		\end{proof}

		We now define a stack $\mathcal{W}$ as follows. For a scheme $T$, its $T$-valued points are
		\begin{equation}\label{eqnWdef}
			\mathcal{W}(T) :=\left\{(\mathcal{E},\varphi) \,\mid\, \begin{array}{ll}
				\mathcal{E} \in \flagellss(T);\\
				\varphi:\detL(\mathcal{E})\xrightarrow{\sim}\mathcal{O}_T
			\end{array}\right\}.
		\end{equation}
		A morphism from $(\mathcal{E}',\varphi')$ to $(\mathcal{E},\varphi)$ in the category $\mathcal{W}(T)$ is given by an isomorphism $f:\mathcal{E}\to \mathcal{E}'$ of sheaves with flag structure, such that $\varphi=\varphi'\circ\detL(f)$. 
		The stack $\mathcal{W}$ has a natural $\C^*$-action given by $t\cdot (\mathcal{E}, \varphi) = (\mathcal{E},t\varphi)$. There is a natural map $\mathcal{W}\to \flagellss$ given by forgetting $\varphi$, which makes $\mathcal{W}$ a principal $\C ^*$-bundle over $\flagellss$ with respect to the given action. 
		We go on to show some properties of $\mathcal{W}$ and of the $\C^*$-action on it. 
		\begin{proposition}\label{propstackwprops}
			\begin{enumerate}[label=(\arabic*)]
				\item The stack $\mathcal{W}$ is a separated Deligne--Mumford stack.\label{propstackw1}
				\item The set of $\C$-valued fixed points of the $\C^*$-action is exactly the set of pairs $(E,\varphi)$ such that $E\in \flagellss(\C)$ is properly polystable. \label{propstackw2}
				\item For $(E,\varphi)\in \mathcal{W}(\C)$, the limit $\lim_{t\mapsto 0} t\cdot(E,\varphi)$ exists if and only if $E$ does not satisfy condition $(\stabT_\ell)$. In that case the limit is isomorphic to $(T\oplus E/T,\varphi')$, where $T\oplus E/T$ is the polystable object associated to the minimal subobject violating $(\stabT_{\ell})$ as in Lemma \ref{lemasspolystable}, and any choice of $\varphi'$.  \label{propstackw3}
				\item For $(E,\varphi)\in \mathcal{W}(\C)$, the limit $\lim_{t\mapsto \infty} t\cdot(E,\varphi)$ exists if and only if $E$ does not satisfy condition $(\stabS_{\ell+1})$. In that case the limit is isomorphic to $(S\oplus E/S,\varphi')$, where $S\oplus E/S$ is the polystable object associated to the minimal subobject violating $(\stabS_{\ell+1})$ as in Lemma \ref{lemasspolystable}, and any choice of $\varphi'$.  \label{propstackw4}
			\end{enumerate}
		\end{proposition}
		\begin{proof}
			\begin{enumerate}[label = (\arabic*)]
				\item The Deligne--Mumford property here can be checked on closed points as in the proof of Proposition \ref{propproper}. By definition, the automorphism group of an object $(E,\varphi)\in \flagellss(\C)$ is a subgroup of the automorphism group of $E\in \flagellss(\C)$. If $E$ is not polystable, then due to Lemma \ref{lempolystabdefs}, the only automorphisms of $E$ as a sheaf with flag structure are the scalar automorphisms. Of those, the ones that preserve the given orientation are only the $r$-th roots of unity. Thus, the automorphism group of $E$ and a fortiori of $(E,\varphi)$ is finite in this case. Now suppose that $E$ is polystable, thus $E\simeq S\oplus T$, where $S\simeq \mathcal{O}_C(-1)^{\oplus k }$, $T$ is $V^{\ell}$-stable, and where $S$ and $T$ have the appropriate flag structures. By Lemma \ref{lempolystabdefs}, its automorphism group as a (unoriented) sheaf with flag structure is $\C^*\times \C^*$, acting via scalars on each of the factors. The group of automorphisms preserving the orientation is then $\C^*\times \mu_r$. For $(x,y)\in \C^*\times \mu_r$ let $\sigma_{x,y}$ be the corresponding automorphism of $E$. By Lemma \ref{lemlprops}, we have $\detL(\sigma_{x,y})=x^{-k\chi_1}y^{k\chi_1}=(y/x)^{k\chi_1}$. Thus, the only automorphisms of $(E,\varphi)$ are given by those pairs $(x,y)\in \C^*\times \mu_r$ for which $(y/x)^{k\chi_1}=1$. Since there are only finitely many possible values for $y$, there are only finitely many such pairs. Hence, the automorphism group is also finite in this case. This proves that in fact $\mathcal{W}$ is a Deligne--Mumford stack. 
				
				Now we check the valuative criterion for separatedness: Suppose that $(\mathcal{E},\Phi)$ and $(\mathcal{E}',\Phi')$ are two families in $\mathcal{W}(\Spec R)$ for a DVR $R$, and that we have an isomorphism $f_\eta:\mathcal{E}_\eta\xrightarrow{\sim} \mathcal{E}_\eta'$ respecting the flag structures, orientations and compatible with $\Phi_\eta$ and $\Phi_\eta'$. We need to show that $f_\eta$ extends to an isomorphism over all of $\Spec R$ (then it does so uniquely by $R$-flatness). Via $f_\eta$, we can view $\mathcal{E}'$ as a subsheaf of $\mathcal{E}_\eta$. Then there exists a unique $k\in \Z$, such that $\pi^k\mathcal{E}'\subset \mathcal{E}$, and $\pi^k\mathcal{E}'\not\subset\pi \mathcal{E}$. We let $\mathcal{E}'':=\pi^k\mathcal{E}'$ with flag structure induced from the one on $\mathcal{E}'$. 
				\begin{claim}\label{claimstmods}
					The family $\mathcal{E}''$ of $V^{\ell},V^{\ell+1}$-semistable sheaves is obtained from $\mathcal{E}$ either by a finite sequence of $S$-modifications or by a finite sequence of $T$-modifications.
				\end{claim}
				\begin{proof}
					By construction, the map $g:\mathcal{E}''_\xi\to\mathcal{E}_\xi$ induced by restricting the inclusion $\mathcal{E}''\subset \mathcal{E}$ to the special fiber is nonzero. We find that $g$ respects the flag structures, as this is a closed condition.  Thus, $g$ is a nonzero map of $V^\ell,V^{\ell+1}$-semistable sheaves.  We may assume that $g$ is not an isomorphism, since otherwise $\mathcal{E}''=\mathcal{E}$ and the claim is true. By Lemma \ref{lemmindest} \ref{lemmindestiii}, the image of $g$ is either a minimal $(\stabS_{\ell +1})$-destabilizing subobject or a minimal $(\stabT_\ell)$-destablizing subobject of $\mathcal{E}_{\xi}$. By respectively performing an $S$ or $T$-modification on $\mathcal{E}$, we therefore obtain a new family $\mathcal{E}_1$, which by Lemma \ref{lemmodificationproperties} is still $V^\ell,V^{\ell+1}$-semistable. We also have $\mathcal{E}''\subset \mathcal{E}_1\subset \mathcal{E}$ and $\mathcal{E}_1=\pi \mathcal{E}+\mathcal{E}''$. We can keep performing $S$- and $T$-modifications to obtain a sequence of sheaves $\mathcal{E}_i$ each being a modification of the preceeding one, until we obtain an equality $\mathcal{E}_n=\mathcal{E}''$ for some $n$. In fact, this process has to stop, since  on one hand the equality $\mathcal{E}^{i+1}=\pi\mathcal{E}^{i}+\mathcal{E}''$ implies $\mathcal{E}^i\subset \pi^i\mathcal{E}+\mathcal{E}''$, and on the other hand $\pi^i\mathcal{E} \subset \mathcal{E}''$ for $i$ large enough, so $\mathcal{E}_i\subset \mathcal{E}''$ for large enough $i$. We conclude that $\mathcal{E}''$ is obtained from $\mathcal{E}$ by a finite sequence of $S$ and $T$-modifications. We claim that necessarily only one type of modification appears throughout. Indeed, suppose we have $\mathcal{E}_{i+2}\subset \mathcal{E}_{i+1}\subset \mathcal{E}_i$, where, say, we performed an $S$-modification followed by a $T$-modification. Let $S\subset \mathcal{E}_{i,\xi}$ be the minimal $(\stabS_{\ell+1})$-destabilizing subobject, and let $T:=\mathcal{E}_{i,\xi}/S$, so that $\mathcal{E}_{i+1}$ is the modification of $\mathcal{E}_i$ along $T$. By the properties of elementary modifications, the map $\mathcal{E}_{i,\xi}\to \mathcal{E}_{i+1,\xi}$ induced by multiplication by $\pi$ has image isomorphic to $T$. It follows from Lemma \ref{lemmindest} \ref{lemmindestiii} that $T$ is the minimal $(\stabT_{\ell})$-destabilizing subobject of $\mathcal{E}_{i+1,\xi}$. This means that $\mathcal{E}_{i+2}$ is obtained from $\mathcal{E}_{i+1}$ by an elementary modification along the object $S$ arising in the exact sequence $0\to T\to \mathcal{E}_{i+1,\xi}\to S\to 0$,which we have due to the virtue of $\mathcal{E}_{i+1}$ being a modification of $\mathcal{E}_{i}$. But then by Remark \ref{remelmodinv}, we have that $\mathcal{E}_{i+2}=\pi \mathcal{E}_i$, and in particular $\mathcal{E}''\subset \pi\mathcal{E}$, which contradicts our assumptions. 
				\end{proof}

				By Claim \ref{claimstmods}, the sheaf $\mathcal{E}''$ is obtained by elementary modifications along successive quotients $\left(Q_i\right)_{i=1}^n$, where one of the following two situations holds: Either each $Q_i$ is a direct sum of $k_i$ copies of $\mathcal{O}_C(-1)$, or each $Q_i$ has Chern character $\chat-k_i\chexc$. In either case, the $k_i$ are positive integers. By a repeated application of Lemma \ref{lemlprops}, \ref{lemlprops5} one finds that the inclusion $\mathcal{E}''\to \mathcal{E}$ induces a rational map $r:\detL(\mathcal{E}'')\to \detL(\mathcal{E}) $, which has a zero of order $\sum_i \nu(Q_i)=\pm (\sum_i k_i) \chi_1$, where the sign depends on which case we are in. On the other hand, the isomorphism $\mathcal{E}''=\pi^k \mathcal{E}' \simeq \mathcal{E}'$ identifies the inclusion $\mathcal{E}''\subset \mathcal{E}$ with the map $\pi^kf_{\eta}$. Consider the rational maps $\detL(\pi^k):\detL(\mathcal{E})\dashrightarrow \detL(\mathcal{E})$ and $\det(f_{\eta}):\detL(\mathcal{E}')\dashrightarrow \detL(\mathcal{E})$.  By Lemma \ref{lemlprops} \ref{lemlprops5}, $\detL(\pi^k)=\operatorname{id}_{\mathcal{E}}$ vanishes to order zero at the special point. So does $\det(f_{\eta}) = \Phi^{-1}_{\eta} \circ \Phi'_{\eta}$. Therefore $\sum_i k_i \chi_1=0$, which implies that $n=0$, i.e. $\mathcal{E}''=\mathcal{E}$. We conclude that $\pi^k f_{\eta}:\mathcal{E}_\eta'\to \mathcal{E}_\eta$ extends to an isomorphism of sheaves with flag structures  over all of $\Spec R$, which we call $h$. Since $ \detL(\pi)=\operatorname{id}_{\detL(\mathcal{E}')}$, the isomorphism $\pi^kf_\eta$ is still compatible with $\Phi_\eta'$ and $\Phi_\eta$. Therefore $h$ is compatible with $\Phi$ and $\Phi'$. Finally, since $f_\eta$ was assumed to be compatible with orientations, we conclude as in the proof of  Proposition \ref{propsepval} that in fact $k=0$. Then $h$ is also compatible with orientations, and is the desired extension of $f_\eta$. 
				
				\item  Let $(E,\varphi)\in \mathcal{W}(\C)$ be a $\C$-valued point. It is fixed under the $\C^*$-action if and only if for each $t\in \C^*$, there exists an automorphism $f_t$ of $E$ that induces an isomorphism $f_t:(E,t\varphi)\xrightarrow{\sim}(E,\varphi)$. In particular, if $(E,\varphi)$ is a $\C^*$-fixed point, then $E$ must have infinitely many automorphisms that preserve the orientation, so at least one non-scalar one. By Lemma \ref{lempolystabdefs}, we obtain that $E\simeq S\oplus T$ is properly polystable, where the indicated splitting is as in Definition \ref{defpolystable}. Let $k$ be a positive integer such that $S\simeq \mathcal{O}_C(-1)^{\oplus k}$.  For $s\in \C^*$, let $\sigma_{s,1}$ be the automorphism of $E$ which is multiplication by $s$ on $S$ and the identity on $T$. Then by Lemma \ref{lemlprops}, we have $\detL(s)=s^{-k\chi_1}$. So given $t\in \C^*$, we can take $f_t=\sigma_{s,1}$, where $s$ is a $k\chi_1$-th root of $t$. Then the diagram 
				\begin{equation*}
					\begin{tikzcd}
						\detL(E) \ar[rr,"\detL(f_t)=t^{-1}"]\ar[dr,"\varphi"]&&\detL(E) \ar[dl,"t\varphi"']\\
						& \mathcal{O}_{\Spec \C}& 
					\end{tikzcd}
				\end{equation*}
				is commutative. So $f_t$ is an isomorphism $(E,t\varphi)\to (E,\varphi)$. This shows that a polystable object is indeed $\C^*$-fixed.  
				\item 
				Suppose first that $E$ does not satisfy $(\stabT_\ell)$. Let $x$ be the coordinate function on $\aone$ and let $T\subset E$ be the minimal $(\stabT_\ell)$-destabilizing subobject and $Q:=E/T$. Let $E_{\mathbb{A}^1}$ denote the constant family over $\mathbb{A}^1$ with fiber $E$ and similarly $T_{\aone}, Q_{\aone}$. We also let $\Phi:\detL(E_{\mathbb{A}^1})\to \mathcal{O}_{\aone}$ be the trivialisation which is constant over $\mathbb{A}^1$ with value $\varphi$. Now let $\mathcal{E}\subset E_{\aone}$ be the elementary modification along the quotient $E\to Q$ in the fiber over $0\in \mathbb{A}^1$. By Lemma \ref{lemlprops}  \ref{lemlprops3} and \ref{lemlprops5} the rational map $r:\detL(\mathcal{E})\dasharrow \detL(E_{\aone})$ induced from the isomorphism away from zero vanishes to order $\nu(Q)=-\nu(T)<0$, i.e. has a pole of order $\nu(T)$. Therefore, the composition $x^{\nu(T)}\Phi\circ r$ extends to a trivialization $\Psi:\detL(\mathcal{E})\xrightarrow{\sim} \mathcal{O}_{\aone}$.  Since $Q$ has rank $0$ on $X$, the orientation on $E_{\aone}$ induces an orientation on $\mathcal{E}$, which is compatible with the isomorphism induced by inclusion away from the fiber over $0$. It follows that $\mathcal{E}\mid_{\mathbb{A}^1\setminus\{0\}}$ is isomorphic to $(E_{\aone\setminus\{0\}},x^{\nu(T)}\varphi)$. We now examine the central fiber. From the elementary modification, we have the nontrivial morphism $E\to \mathcal{E}_0$ with image $Q$. It follows that $Q$ is the minimal $(S_{j+1})$-destabilizing subobject by Lemma \ref{lemmindest} \ref{lemmindestiii}. On the other hand, we have an inclusion $T\subset \mathcal{E}_0$, induced by $T_{\aone}\subset \mathcal{E}$. It is easy to check that this gives the minimal $(\stabT_\ell)$-destabilizing subobject. We have the splitting $\mathcal{E}_0\simeq Q\oplus T$, and this makes $\mathcal{E}_0$ into the polystable object associated to $E$ with destabilizing subsheaf $T$ as in Lemma \ref{lemasspolystable} \ref{lemasspolystableii}. We have shown that the limit exists, and that the limiting object (which is unique due to separatedness) is of the stated form. 
				Now suppose instead that the limit $\lim_{t\to 0}(E,t\varphi)$ exists. We need to show that $E$ does not satisfy $(\stabT_j)$.  By assumption we have a family $(\mathcal{E},\Phi)$ of objects in $\mathcal{W}$ over $\aone$ and positive integer $k> 0$, such that over any $\aone\setminus\{0\}$ it is isomorphic to $(E\mid_{\aone\setminus\{0\}},x^{k} \varphi)$. Using this isomorphism, we may regard $\mathcal{E}$ as a subsheaf of $E\otimes \C(x)$. And, disregarding the orientations, we may assume that $\mathcal{E}\subset E_{\aone}$ and $\mathcal{E}\not \subset xE_{\aone}$, after possibly multiplying $\mathcal{E}$ by a suitable power of $x$. By the same argument as in Claim \ref{claimstmods}, $\mathcal{E}$ is then obtained from $E_{\aone}$ either by a sequence of $S$-modifications or by a sequence of $T$-modifications. But as $(\mathcal{E},\Phi)\mid_{\aone\setminus \{0\}}$ is isomorphic to $(E_{\aone\setminus\{0\}}, x^{k} \varphi)$ for a \emph{positive} exponent $k$, it must have been a sequence of $T$-modifications, and moreover a nonempty sequence. In particular, $E$ must have a $(\stabT_j)$-destabilizing subobject.
				\item  The argument here is exactly analogous to the one for \ref{propstackw3} with a minimal $(\stabS_{j+1})$-destabilizing subobject in place of $T$. Only the argument regarding orientations is slightly different: Here the isomorphism $(\mathcal{E},\Psi )\mid_{\aone\setminus}\simeq (E_{\aone}, x^{\nu(S)}\Phi)$ is not induced by the identification $\mathcal{E}_{\aone\setminus\{0\}}=E_{\aone\setminus \{0\}}$ but instead by the composition $\mathcal{E}_{\aone\setminus\{0\}}=E_{\aone\setminus \{0\}}\xrightarrow{\cdot x^{-1}} E_{\aone\setminus \{0\}}$. 
			\end{enumerate}
		\end{proof}
		
		We can now finally construct the master space. Let $\mathcal{W}^-$  (respectively $\mathcal{W}^+$) be the preimage of $\flagell$ (respectively of $\flagellplusone$) under the map $\mathcal{W}\to \flagellss$. As preimages of open substacks of $\mathcal{N}^{\ell,\ell+1}$, they are open substacks of $\mathcal{W}$. We also let $\Sigma^+:=\mathcal{W}\setminus \mathcal{W}^-$, considered as a closed substack with its induced reduced structure, and similarly $\Sigma^-:=\mathcal{W}\setminus \mathcal{W}^+$. By Proposition \ref{propstackwprops}, the set $\Sigma^{\pm}$ contains exactly those $\C$-points $(E,\varphi)$ of $\mathcal{W}$ for which $\lim_{t\to 0}t^{\pm}(E,\varphi)$ exists.

		\begin{construction}[cf. \S 4 in \cite{KiLi}]
			Consider $\mathcal{W}\times \pone$. We write $0$ for $[0,1]\in \pone$ and $\infty$ for $[1,0]\in \pone$. We let the group $T:=\C^*$ act on $\mathcal{W}\times \pone$ via 
			\begin{equation}\label{eqnTaction}
				t\cdot (w, [s_0,s_1]):= (t\cdot w,[t s_0, s_1]).
			\end{equation}
			We also consider another $\C^*$-action via
			\begin{equation}
				(w,[s_0,s_1])^t:=(w,[s_0,ts_1]). \label{eqnsecondaction}
			\end{equation}
			Evidently, the two actions commute with each other. We define an open subset of "stable points" as
			\[\left( \mathcal{W}\times \pone\right)^s:=\left( \mathcal{W}\times \pone\right)\setminus \left(\Sigma^-\times \{0\}\cup \Sigma^+\times\{\infty\} \right),\]
			which is a $T$-invariant open subset. We define the master space as the stack quotient
			\[\mathcal{Z}:=\left[\left(\mathcal{W}\times\pone \right)^s/T \right] .\] 
		\end{construction}
		\begin{proposition}
			The master space $\mathcal{Z}$ is a proper Deligne--Mumford stack with $\C^*$ action induced by \eqref{eqnsecondaction}.
		\end{proposition}
		\begin{proof}
			It is easy to see that $T$ acts on $\mathcal{W}\times \pone$ without fixed points. Indeed, the action on the second coordinate only fixes the points $\{0\}$ and $\{\infty\}$, and over these points we have removed the fixed points of $\mathcal{W}$. It follows that $\mathcal{Z}$ is a Deligne--Mumford stack. Since the $\C^*$-action \eqref{eqnsecondaction} commutes with the one of $T$, it induces an action on the quotient.  
			In what follows $R$ will always denote a DVR with generic point $\eta$, closed point $\xi$ and field of fractions $K$. The $R$-valued points of $\pone$ are in bijection with the set $K\cup \{\infty\}$. To ease notation, we will use elements of $K\cup \{\infty\}$ to denote $R$-valued points of $\pone$. In particular, if $s\in K$ is regarded as a point of $\pone(R)$, then $s_{\xi}$ is the image in the residue field of $R$, when $s\in R$, and $\infty$ when $s\not \in R$.

			\paragraph{Separatedness of $\mathcal{Z}$.}
			Separatedness of $\mathcal{Z}$ is equivalent to properness of the action map 
			\[T\times \left(\mathcal{W}\times\pone \right)^s\xrightarrow{\mu\times \operatorname{proj}_2} \left(\mathcal{W}\times\pone \right)^s\times \left(\mathcal{W}\times\pone \right)^s, \]
			where we write $\mu$ for the $T$-action: $\mu(t, (w,[s_0,s_1])):=t\cdot (w,[s_0,s_1])$. We therefore want to check the valuative criterion for properness of the map $\mu\times \operatorname{proj}_2$. Separatedness follows from separatedness of $\mathcal{W}$, so we only need to check the existence part of the valuative criterion. Let $R$ be a DVR with field of fractions $K$. Suppose that we are given two objects $(w,s)$ and $(w',s')$ in $\left(\mathcal{W}\times\pone \right)^s(\Spec R)$, where $s,s'\in K\cup\{\infty\}$, together with an element $t\in \mathbb{G}_m(\eta) = K^*$ and an isomorphism $t\cdot (w_\eta,s_\eta)\xrightarrow{\sim}(w'_\eta,s'_\eta)$ over the generic point $\eta$. We need to show that one can extend $t$ and the given isomorphism uniquely to $\Spec R$.  In fact, once we can extend $t$, we can extend the isomorphism uniquely due to separatedness of $\mathcal{W}$ and $\pone$. But $t$ extends uniquely to $\mathbb{G}_m(\Spec R)$, if and only if it is in fact an element of $R^*$, which we will show. In any case, we can write $t=u\pi^k$ for some $k\in \Z$ and some $u\in R^*$. If we can show that $k=0$ we are done. Suppose it is not. Write $w=(\mathcal{E},\Phi)$ and $w'=(\mathcal{E}',\Phi')$. Using the given isomorphism $\mathcal{E}_\eta\simeq \mathcal{E}_\eta'$ and possibly multiplying it by a suitable power of $\pi$, we can realize $\mathcal{E}'$ as a subsheaf of $\mathcal{E}$ satisfying $\mathcal{E}'\not\subset \pi \mathcal{E}$ and such that the inclusion is an equality over the generic point and compatible with the flag structures. Due to Lemma \ref{lemlprops}, \ref{lemlprops1} and our assumptions, this inclusion will induce a rational map $r:\detL(\mathcal{E}')\dasharrow \detL(\mathcal{E})$, which commutes with $\Phi'_\eta $ and $t\Phi_\eta$ over the generic point and therefore vanishes to order $-k$ at the closed point of $\Spec R$. Moreover, by the same argument as in \ref{claimstmods}, it follows that $\mathcal{E}'$ is obtained from $\mathcal{E}$ through either a sequence of $S$-modifications or a sequence of $T$-modifications. Suppose $k>0$. Then by Lemma \ref{lemlprops}, \ref{lemlprops5}, $\mathcal{E}'$ must be obtained by a nontrivial sequence of $T$-modifications.  It follows from Lemma \ref{lemmindest} that $\mathcal{E}_\xi$ does not satisfy $(\stabT_{\ell})$ and that $\mathcal{E}'_{\xi}$ does not satisfy $(\stabS_{\ell+1})$. Thus $w_\xi $ lies in $\Sigma^+$, and $w'_\xi$ lies in $\Sigma^-$. It follows from definition of $\left(\mathcal{W}\times\pone \right)^s$ that $s_\xi \neq \infty$ and $s_\xi\neq 0$. But by assumption, we have that $s'= u\pi^k s$. This contradicts our assumption that $k>0$. With the exact analogue argument, one shows that $k<0$ is impossible. This concludes the proof of separatedness of $\mathcal{Z}$. 
			
			\paragraph{Properness of $\mathcal{Z}$.}
			To show that $\mathcal{Z}$ is proper, we want to check the existence part of the valuative criterion. Suppose we are given a map $b_\eta:\eta \to \mathcal{Z}$, where $\eta$ is the generic point of a DVR $R$. Since $\left(\mathcal{W}\times\pone \right)^s\to \mathcal{Z}$ is a map of finite type stacks,  after possibly passing to an extension of $R$, we can lift $b_\eta$ to a map $a_\eta:\eta\to \left(\mathcal{W}\times\pone \right)^s$. It is sufficient to show that there exists some choice of $a_\eta$ that extends to a map $\Spec R\to \left(\mathcal{W}\times\pone \right)^s$. Since any translate of a lift $a_\eta$ by an object in $\mathbb{G}_m(\eta)$ is still a lift of $b_\eta$, we are therefore reduced to the following problem:
			\begin{claim}
				Let  $(\mathcal{E}_\eta,\Phi_\eta)\in \mathcal{W}(\eta)$, and let $s\in \pone(\eta)$. Then there exists an element $t=u\pi^k\in K$, such that $\left((\mathcal{E}_\eta, u\pi^k\Phi_\eta),u\pi^k s_\eta\right)$ extends to a $\Spec R$-valued point of $\left(\mathcal{W}\times\pone \right)^s$.
			\end{claim}  
			\begin{proof}
				
				Suppose first that $s=0$. Then $(\mathcal{E}_\eta,\Phi_\eta)\in \mathcal{W}^+(\eta)$, so that $\mathcal{E}_\eta\in \flagellplusone(\eta)$. By properness of $\flagellplusone$, we can extend $\mathcal{E}_\eta$ to a family $\mathcal{E}_R$ over $\Spec R$. Then taking $k$ to be minus the vanishing order at $\xi$ of the rational map $\detL(\mathcal{E}_R)\dasharrow \mathcal{O}_{\Spec R}$ induced by $\Phi_\eta$ of $R$, we see that $\pi^{k}\Phi_\eta$ also extends over $\Spec R$, which concludes this case. In the case $s=\infty$, one uses the obvious modification of this argument. If we have $s\neq 0,\infty$, then by using the $T$-action, we may assume that $s=1$. By Corollary \ref{corvalpropss}, we can still extend $\mathcal{E}_\eta$ to a family $\mathcal{E}_R$ of oriented $V^{\ell},V^{\ell+1}$-semistable sheaves. Once this extension is chosen, we let again $k$ denote minus the vanishing order at $\xi$ of the rational map $\detL(\mathcal{E}_R)\dasharrow \mathcal{O}_{\Spec R}$ induced by $\Phi_\eta$, so that as before $\pi^{k}\Phi_{\eta}$ extends over $\Spec R$ as, say, $\Psi_R$.  We obtain an map $a_1:\Spec R\to \mathcal{W}\times\pone(\Spec R)$ defined by the family  $\left((\mathcal{E}_R, \Psi_R),\pi^k \right) $. If $a_1$ factors through$\left(\mathcal{W}\times\pone \right)^s(\Spec R)$, we are done; so we may suppose it does not. Then $a_1$ sends $\xi$ either to $ \Sigma^{-}\times \{0\}$ or to ${\Sigma^+}\times \{\infty\}$. Thus either $k>0$ and $\mathcal{E}_\xi$ is not $V^{\ell+1}$-stable, or $k<0$ and $\mathcal{E}_\xi$ is not $V^{\ell}$-stable. We only treat the first case with $k>0$, as the other one is analogous. Thus, by assumption $\mathcal{E}_\xi$ satisfies $(\stabS_\ell)$ and $(\stabT_{\ell+1})$, but is not $V^{\ell+1}$-stable, hence does not satisfy $(\stabS_{\ell+1})$. Let $S\subset\mathcal{E}_\xi$ be the minimal $(\stabS_{\ell+1})$-destabilizing subobject with the quotient $Q:=\mathcal{E}_\xi/S$. Let $\mathcal{E}_R'$ be the $S$-modification of $\mathcal{E}_R$. Then by Lemma \ref{lemlprops}, \ref{lemlprops5} we find that $\pi^{-\nu({Q})}\Psi_R$ induces a trivialisation of $\detL(\mathcal{E}'_R)$. Note that $\nu(Q)>0$  by Lemma \ref{lemlprops}, \ref{lemlprops3}. We consider the isomorphism $\mathcal{E}_{\eta}'\to \mathcal{E}_\eta$ induced by the inclusion composed with multiplication by $1/\pi$. The orientation on $\mathcal{E}_R$ pulls back along this morphism to an orientation on $\mathcal{E}_{\eta}'$, which extends over $\Spec R$. Suppose that $\nu(Q)<k$. Then we may replace $\left((\mathcal{E}_R, \Psi_R),\pi^k\right)$ by $\left((\mathcal{E}'_R, \pi^{-\nu(Q)}\Psi_R),\pi^{k-\nu(Q)} \right)$ and repeat the argument. Since $k$ becomes smaller every time we do this, eventually we must arrive in a situation where either we end up with a central fiber which satisfies $(\stabS_{\ell+1})$ -in which case we are done - or we end up with a central fiber which is still not $V^{\ell+1}$-stable, but such that we have $k\leq \nu(Q)$.  In this latter case, let $kc=\nu(Q)d$ be the least common multiple of $k$ and $\nu(Q)$, so that $d\leq c$. Then let $R'$ be a degree $c$ totally ramified extension of $R$ with uniformizer $\pi'$. Pulling back to $\Spec R'$, we obtain a family $\left((\mathcal{E}_{R'}, \Psi_{R'}), (\pi')^{kc} \right)$. By construction, this family admits a flat quotient $\mathcal{E}_{R'/(\pi'^d)}\to Q_{R'/(\pi'^d)}$ over $\Spec R'/(\pi'^d)$, which is the pullback of the quotient $\mathcal{E}_\xi\to Q$ over the central fiber $\xi\in \Spec R$. Then one has a sequence
				\[\mathcal{E}_{c,R'}\subset\cdots\subset \mathcal{E}_{1,R'}\subset \mathcal{E}_{0,R'}=\mathcal{E}_{R'},\]
				where $\mathcal{E}_{c,R'}= \Ker\left(\mathcal{E}_{R'/(\pi'^d)}\to Q_{R'/(\pi'^d)}\right)$, and $\mathcal{E}_{i,R'}=\mathcal{E}_{c,R'}+(\pi')^i\mathcal{E}_{R'}$. It is not hard to check that this is a sequence of $S$-modifications along $(\stabS_{\ell+1})$-destabilizing subobjects, and that at each step the we have $\mathcal{E}_{i+1,R'}/\mathcal{E}_{i,R'}\simeq Q$. In fact, $\mathcal{E}_{i,\xi'}\simeq S\oplus Q$ for $i=1,\ldots,c-1$. Since $d\leq c$, we can take $\mathcal{E}_{d,R'}$. By the same line of reasoning as before, we have that $\pi^{-d\nu(Q)}\Psi_{R'} $ extends to a trivialization of $\detL(\mathcal{E}_{d,R'})$, and we can arrange for compatibility of the orientations. This gives us the family $\left((\mathcal{E}_{d,R'},\pi^{-d\nu(Q)}\Psi), 1\right)$, which is an extension of the original one $\left((\mathcal{E}_\eta,\Psi_\eta),s\right)$ up to base change and translation by the action of $T$. Moreover, the induced map $\Spec R'\to \mathcal{W}\times\pone $ now clearly factors through $\left(\mathcal{W}\times \pone\right)^s$, so we are done. 
			\end{proof}
		\end{proof}

		\subsection{Identification of fixed loci}\label{subsec:fix}

		We determine the fixed point stack $\mathcal{Z}^{\C^*}$ of the $\C^*$-action on $\mathcal{Z}$. We begin by determining the set-theoretic fixed points. 
		
		\paragraph{Set-theoretic fixed locus.} 
		For an algebraic stack $\mathcal{X}$, we denote by $\abs{\mathcal{X}}$ the set of isomorphism classes of $\C$-points of $\mathcal{X}$. An algebraic $\C^*$-action on $\mathcal{X}$ induces an action on $\abs{\mathcal{X}}$. 
		We have a stratification
		\[\left(\mathcal{W}\times \mathbb{P}^1\right)^s \simeq \mathcal{W}^+\times \{0\} \sqcup \mathcal{W}\times \C^*\sqcup \mathcal{W}^-\times\{\infty\},\]
		which is equivariant both under the $T$-action \eqref{eqnTaction} and under the action \eqref{eqnsecondaction}. After taking the quotient with respect ot the action by $T$, we have $\C^*$-equivariant isomorphisms $ \mathcal{W}\times \C^*/T\simeq \mathcal{W}$ as well as $\mathcal{W}^+\times \{0\}/T\simeq \flagellplusone$ and  $\mathcal{W}^-\times \{\infty\}/T\simeq \flagell$, where the $\C^*$-action on $\flagell$ and $\flagellplusone$ is the trivial one. We see that the fixed point set decomposes as 
		\[ \abs{\mathcal{Z}}^{\C^*}= \abs{\flagell} \sqcup \abs{\flagellplusone} \sqcup  \abs{\mathcal{W}}^{\C^*} .\]
		By Proposition \ref{propstackwprops} \ref{propstackw2}, $ \abs{\mathcal{W}}^{\C^*}$ consists of the set of isomorphism classes of properly polystable objects of $\mathcal{N}^{\ell,\ell+1}$. We will now describe this set in more detail. 
		
		Recall our conventions
		\[N=\dim \globm{E}; \mbox{ and } d =\dim \globm{\mathcal{O}_C(-1)},\]
		where $E$ is any object of $\mathcal{M}^{0,1}(\C)$. 
		Let $(E,V^{\bullet})\in \mathcal{N}^{\ell,\ell+1}(\C)$ be a  $V^{\ell},V^{\ell+1}$-semistable sheaf with flag structure which is properly polystable. By Lemma \ref{lempolystabdefs},we have a splitting $E\simeq S\oplus T$, where $S\simeq \mathcal{O}_C(-1)^{\oplus k}$ for some $k>0$, and a compatible splitting of flags $V^j=V^j_S\oplus V^j_T$.
		We denote by $W_S^{\bullet}$ the flag $V_S^{\bullet}$ with repetitions eliminated, so that $W_S^\bullet$ is a full flag in $\globm{S}$. Similarly, we define a full flag $W_T^\bullet $ in $\globm{T}$ from $V_T^{\bullet}$. Let 
		\[\Lambda:=\left\{j \left\lvert V^{j}_S\neq V^{j+1}_S \right. \right\}\subset \{0,\ldots,N-1\}.\]
		We have $\# \Lambda = kd$, and since $S$ is $(\stabS_{\ell+1})$-destabilizing, we have $\min \Lambda=\ell$. Then the data of $V_S^{\bullet}$ and $V_{T}^\bullet$ is equivalent to the data of $W_{S}^\bullet$ and $W_T^{\bullet}$ together with $\Lambda$. Combining these consideration with Lemma \ref{lemdecpolystables}, we find that the datum of a properly $V^{\ell},V^{\ell+1}$-polystable sheaf $E$ is up to isomorphism equivalent to the following pieces of data:
		\begin{itemize}
			\item a positive integer $k>0$, where we denote $S := \mathcal{O}_C(-1)^{\oplus k}$, 
			\item  a full flag $W_S^\bullet$ inside $\globm{\mathcal{O}_C(-1)^{\oplus k}}$, such that $(S,W_S^1)$ is simple in the sense of Remark \ref{remdestissimple}, 
			\item an object $(T,W_T^{\bullet})\in\mathcal{N}^{\ell}_{\chat-k\chexc}$ up to isomorphism and
			\item  a subset $\Lambda\subset \{0, \ldots, N-1\}$ of size $k d$ and with $\min \Lambda =\ell$. 
		\end{itemize}
		We will now turn this description into a scheme-theoretic description of the fixed locus
		
		\paragraph{Scheme-theoretic fixed locus.}
		Let $\mathcal{X}$ be a quasi-separated Deligne--Mumford stack with a $\C^*$-action, locally of finite type over $\C$. 
		\begin{definition}
			A $\C^*$-equivariant affine \'etale neighborhood of $\mathcal{X}$ consists of an affine \'etale neighborhood $j:\Spec A\to \mathcal{X}$ together with the following pieces of data:
			\begin{enumerate}[label = (\roman*)]
				\item A $\C^*$-action on $\Spec A$,
				\item a non-constant homomorphism $\lambda:\C^*\to \C^*$, and
				\item a $2$-isomorphism fitting in the following diagram 
				\begin{equation*}
					\begin{tikzcd}
						\C^*\times  \Spec A \ar[r]\ar[d,"\lambda\times j"]& \ar[d, "j"] \Spec A \\
						\C^* \times \mathcal{X}\ar[r] \ar[ur, Rightarrow] & \mathcal{X},
					\end{tikzcd}
				\end{equation*}
				where the horizontal maps are the action maps. 
			\end{enumerate}
			Moreover, this this data is required to satisfy a ``higher associativity'' condition \cite[Definition 2.1 (ii)]{Roma}.
		\end{definition}
		\begin{remark}
			This definition may be slightly non-standard, since we allow the character $\lambda$. This is important for us to obtain the right notion of fixed stack. 
		\end{remark}
		
		By Theorem 2.5 in \cite{AlHaRy}, any quasi-separated Deligne--Mumford stack with torus action which is locally of finite type over $\C$  is covered by such equivariant affine \'etale neighborhoods.
		This allows us to define the fixed substack of a $\C^*$ action. A $\C^*$-action on an affine scheme $\Spec A$ corresponds to a $\Z$-grading $A=\oplus_{i\in \Z} A_i$. We denote by $A_{\C^*}$ the quotient of $A$ by the ideal generated by the ideal $(f \mid f\in A_i; i\neq 0)$. Then the fixed locus of $\Spec A$ is given by $(\Spec A)^{\C^*}=\Spec A_{\C^*}$. 
		\begin{proposition}[{cf. \cite[\S 3]{ChKiLi}}]
			There is a unique closed substack $\mathcal{X}^{\C^*}\subset \mathcal{X}$ with the property that for any $\C^*$-equivariant affine \'etale neighborhood $\Spec A\to \mathcal{X}$ we have an equality $\Spec A \times_{\mathcal{X}}\mathcal{X}^{\C^*}= (\Spec A)^{\C^*}$ of closed subschemes of $\Spec A$.   
		\end{proposition}
		\begin{proof}
			We only give an outline. One needs to check the following points
			\begin{enumerate}[label=(\alph*)]
				\item If $\Spec A\to \mathcal{X}$ and $\Spec B\to \mathcal{X}$ are two equivariant \'etale neighborhoods covering a given geometric point $x$ of $X$, then there exists a common refinement covering $x$, i.e. an affine scheme $\Spec C$ with $\C^*$-action and equivariant maps $\Spec C\to\Spec A$ and $\Spec C\to \Spec B$, such that the two resulting maps $\Spec C\to \mathcal{X}$ cover $x$ and are related via a $2$-isomorphism that commutes with the respective equivariant structures induced by the ones on $\Spec A$ and $\Spec B$.\label{prooflista}
				\item If $\Spec B\to \Spec A$ is an \'etale $\C^*$-equivariant morphism, then $(\Spec B)^{\C^*}$ agrees with the base change of $(\Spec A)^{\C^*}$ as closed subschemes of $\Spec B$.   \label{prooflistb}
			\end{enumerate}
			To see \ref{prooflista}, one checks that the fiber product $\Spec A\times_{\mathcal{X}}\Spec B$ has a natural $\C^*$-action, (this needs the fact that any two homomorphisms $\C^*\to \C^*$ commute!), with the right properties. Then one can use Theorem 2.5 in \cite{AlHaRy} again to obtain  $\Spec C \to\Spec A\times_{\mathcal{X}}\Spec B$. For \ref{prooflistb} one can see that the two closed subschemes have the same underlying reduced scheme by using that a $\C$-point of an affine scheme with $\C^*$-action is $\C^*$-fixed if and only if it has finitely many translates. Then by the infinitesimal lifting property for \'etale morphisms the scheme structures also agree. 
			
			Now one may construct the substack $\mathcal{X}$ as follows: Choose a covering family $\{(\Spec A_i\to \mathcal{X})\}$ consisting of equivariant affine \'etale neighborhoods of $\mathcal{X}$. Then from \ref{prooflista} and \ref{prooflistb} it follows that the family of closed subschemes $\left((\Spec A_i)^{\C^* }\subset \Spec A_i\right)$ induces a well-defined closed subscheme $\mathcal{X}^{\C^*}$ of $\mathcal{X}$, which is independent of choice of cover and satisfies the claimed properties.  
		\end{proof}

		We now specialize the discussion to describe the fixed stack $\mathcal{Z}^{\C^*}$ of the master space in detail.
		
		\begin{proposition}\label{propfixeasy}
			We have natural closed immersions  $\mathcal{N}^{\ell}\to \mathcal{Z}$ and $\mathcal{N}^{\ell+1}\to \mathcal{Z}$. Each is an inclusion of connected components of $\mathcal{Z}^{\C^*}$. 
		\end{proposition} 
		
		\begin{proof}
			The closed immersions are induced by the isomorphisms  
			\[\mathcal{N}^{\ell+1} \simeq \left(\mathcal{W}^+\times \{0\}\right)/T\subset \mathcal{Z} \mbox{ and } \mathcal{N}^{\ell}\simeq \left(\mathcal{W}^{-}\times \{\infty\}\right)/T\subset \mathcal{Z}.\] We will only show that $\mathcal{N}^{\ell+1}\to \mathcal{Z}$ is an inclusion of a connected component of $\mathcal{Z}^{\C^*}$, the other case being similar. It is enough to show this \'etale-locally on $\mathcal{Z}$. In fact, we will show the stronger statement that the $\C^*$-invariant open subset $(\mathcal{W}^+\times \aone)/T\subset \mathcal{Z}$ has $\C^*$-fixed locus equal to $(\mathcal{W}^+\times \{0\})/T$. 
			Since the $T$-action respects the projection $\mathcal{W}^{+}\times \pone\to \mathcal{N}^{\ell+1}$, we may work locally over $\mathcal{N}^{\ell+1}$. Since $\mathcal{N}^{\ell+1}$ is a Deligne--Mumford stack, it is covered by \'etale neighborhoods $V$ over which the $\C^*$-bundle $\mathcal{W}^+$ is trivializable. For such a $V$, we may then pick a trivialization $\mathcal{W}^+\mid_{V}\simeq  \mathcal{N}^{\ell+1}\times \C^*$. Then   
			we obtain a $\C^*$-equivariant isomorphism $(\mathcal{W}^+\mid_V\times \aone)/T\simeq (V\times \C^*\times \aone)/T$. Here, on the right-hand side $T$ acts diagonally on $\C^*\times \aone$ and the $\C^*$-action on the quotient is induced from multiplication by $1/t$ on $\aone$. By a change of coordinates on $\C^*\times \aone$ given by $(s,x)\mapsto (s,x/s)$, one obtains a $\C^*$-equivariant isomorphism  $(\mathcal{W}^+\mid_V\times \aone)/T\simeq V\times \aone$, where $\C^*$ acts by $1/t$ on $\aone$. We see that the $\C^*$-fixed locus in this neigborhood is $V\times \{0\}$, which under the chosen isomorphisms is identified with $(\mathcal{W}^+\times \C^*)/T$.
		\end{proof}
		
		It remains to determine the part of $\mathcal{Z}^{\C^*}$ lying in the open subset $\mathcal{W}\times \C^*/T\subset \mathcal{Z}$. Due to the equivariant isomorphism $\mathcal{W}\times \C^*/T\simeq \mathcal{W}$, this is equivalent to determining $\mathcal{W}^{\C^*}$.  
		We fix once and for all an orientation on $\mathcal{O}_C(-1)$, as well as a trivialization $\varphi_e:\detL(\mathcal{O}_C(-1))\simeq \C$. 
		Guided by the set-theoretic descripition of the fixed point set above, we fix the following discrete data:
		\begin{itemize}
			\item A positive integer $k$, and
			\item a subset $\Lambda\subset \{0,1,\ldots,N-1\}$ of cardinality $kd$ with $\min\Lambda =\ell$. 
		\end{itemize}
		Essentially, we want to define a substack $\mathcal{Z}_{k,\Lambda}\subset \mathcal{W}$ parametrizing polystable objects of the form $(E,V^{\bullet})=(S\oplus T,V_S^{\bullet}\oplus V_T^{\bullet})$, such that $S\simeq \mathcal{O}_C(-1)^{\oplus k}$, and such that $\Lambda$ contains exactly those indices $j$ for which $V_S^j\neq V_S^{j+1}$. We will make this precise and show that $\mathcal{Z}_{k,\Lambda}$ in fact defines an open and closed substack of $\mathcal{W}^{\C^*}$.

		Let $\operatorname{Fl}_{\globm{\mathcal{O}_C(-1)^{\oplus k}}}$ denote the full flag variety associated to $\globm{\mathcal{O}_C(-1)^{\oplus k}}$, and let $\operatorname{Fl}_{\globm{\mathcal{O}_C(-1)^{\oplus k}}}^{\circ}$ denote the open subset parametrizing flags $(W^{\bullet}_S)$ such that $W^1_S\subset \globm{\mathcal{O}_C(-1)^{\oplus k}}\simeq \globm{\mathcal{O}_C(-1)}\otimes \C^{\oplus k}$  defines a rank $k$ tensor. Let also $\mathcal{W}^\ell_{\chat-k\chexc}\to \mathcal{N}^\ell_{\chat-k\chexc}$ be the relative moduli stack associating to any $\mathcal{E}_B\in \mathcal{N}^{\ell}_{\chat-k\chexc}(B)$ the set of trivializations $\varphi:\detL(\mathcal{E}_B)\to \mathcal{O}_B$.  
		We consider the morphism 
		\begin{align*}
			\rho_{k,\Lambda}^{\circ}:\operatorname{Fl}_{\globm{\mathcal{O}_C(-1)^{\oplus k}}}\times \mathcal{W}^{\ell}_{\chat-k\chexc}&\to \mathcal{W}\\
			\left(W_S^{\bullet}, (T,W_T^{\bullet},\varphi_T) \right)&\mapsto (T\oplus \mathcal{O}_C(-1)^{\oplus k}, V_T^{\bullet}\oplus V_S^{\bullet}, \varphi_T\otimes \varphi_e^{k}),
		\end{align*}
		Here, the implicit orientations on $\mathcal{O}_C(-1)$ and on $T$ induce an orientation on $T\oplus \mathcal{O}_C(-1)^{\oplus k}$.  The flags $V_{S}^\bullet$ and $V_T^{\bullet}$ respectively are obtained from the flags $W_{S}^\bullet$ and $W_{T}^\bullet$ by repeating certain entries as indicated by $\Lambda$. To be precise, we can set $j_{\Lambda}(i):=\#(\Lambda\cap [0,i-1])$ and similarly $j_{\Lambda^{c}}(i)=\#(\Lambda^c\cap [0,i-1])$ with $\Lambda_c =\{0,1\ldots,N-1\}\setminus \Lambda$. Then $V_S^{i}=W_S^{j_{\Lambda}(i)}$ and $V_T^i= W_T^{j_{\Lambda^c}(i)}$ for $i=0,\ldots, N$. 
		
		We have an action by $\operatorname{GL}(k)\times \C^*$, where $\operatorname{GL}(k)$ acts on $\operatorname{Fl}_{\globm{\mathcal{O}_C(-1)^{\oplus k}}}$ through its action on $\mathcal{O}_C(-1)^{\oplus k}=\mathcal{O}_C(-1)\otimes \C^{\oplus k}$, and $\C^*$ acts on $\mathcal{W}_{\chat-ke}^{\ell}$ in the usual way by scaling $\varphi_T$.  
		We consider the subgroup $G\subset GL(k)\times \C^*$, which consists of elements of the form $(g,(\det g)^{\chi_1})$. We set
		\[\mathcal{Z}_{k,\Lambda}:=[\left(\operatorname{Fl}_{\globm{\mathcal{O}_C(-1)^{\oplus k}}}\times \mathcal{W}^{\ell}_{\chat-k\chexc}\right)/G].\]
		This is our candidate for the closed substack corresponding to the part of the fixed point locus labeled by $k$ and $\Lambda$.  
		
		The $GL(k)$-action on $\operatorname{Fl}_{\globm{\mathcal{O}_C(-1)^{\oplus k}}}$ descends to a free $PGL(k)$-action. Let $\mathfrak{F}_{k}$ denote the quotient. An exercise in linear algebra shows that $\mathfrak{F}_k$ can be constructed as follows: Consider the Grassmannian $\operatorname{Gr}$ of $k$-planes in $\globm{\mathcal{O}_C(-1)}$. The inclusion of the universal subbundle $V\subset \globm{\mathcal{O}_C(-1)}\otimes \mathcal{O}_{\operatorname{Gr}}$ induces a section $\mathcal{O}_{\operatorname{Gr}}\to \globm{\mathcal{O}_C(-1)}\otimes V^{\vee}$. Let $W$ be the cokernel of this section. Then $\mathfrak{F}_k$ is naturally isomorphic to the relative full flag variety of $W$ over $\operatorname{Gr}$, in particular it is a smooth projective variety. 
		
		We then have the following:

		\begin{lemma}\label{lemetalemap}
			There is a natural finite \'etale morphism $\epsilon:\mathcal{Z}_{k,\Lambda}\to \mathfrak{F}_k\times \mathcal{N}^{\ell}_{\chat-k\chexc}$ of degree $1/(k\chi_1)$. In particular, $\mathcal{Z}_{k,\Lambda}$ is a proper Deligne--Mumford stack. 
		\end{lemma}

		\begin{proof}
			The map $\operatorname{Fl}_{\globm{\mathcal{O}_C(-1)^{\oplus k}}}\times \mathcal{W}^{\ell}_{\chat-k\chexc}\to  \mathfrak{F}_k\times \mathcal{N}^\ell_{\chat-k\chexc}$ is a $\operatorname{PGL}(k)\times \C^*$-bundle. The $G$-action on the domain factors through the surjective homomorphism $G\subset GL(k)\times \C^*\to \operatorname{PGL}(k)\times \C^*$. By working \'etale locally over the base $\mathfrak{F}_k\times \mathcal{N}^{\ell}_{\chat-k\chexc}$, we may assume the bundle is trivial. This reduces us to showing that the stack quotient $[PGL(k)\times \C^*/G]$ is finite \'etale of degree $1/k\chi_1$ over $\Spec \C$. But it is in fact isomorphic to the classifying space of the finite group $K=\Ker(G\to \operatorname{PGL}(k)\times \C)$ of order $k\chi_1$  (see Exercise 10.F in \cite{Olso2}). The claim follows from this. 
		\end{proof}
		
		\begin{remark}\label{remequiv}
			The following alternative characterization of $\mathcal{Z}_{k,\Lambda}$ is useful: $\mathcal{Z}_{k,\Lambda}$ parametrizes tuples 
			\[\left((\mathcal{S}, W_S^{\bullet}),(\mathcal{T},W_T^{\bullet}), \operatorname{or}, \Phi \right),\]
			where
			\begin{enumerate}[label=\roman*)]
				\item 	 $(\mathcal{S},W_S^{\bullet})$ is a sheaf with flag-structure, such that $\mathcal{S}$ is \'etale locally isomorphic to $\mathcal{O}_C(-1)^{\oplus k}$, and $(\mathcal{S},W_S^1)$ is a family of simple pairs,
				\item $(\mathcal{T},W_T^{\bullet})$ is a $V_{\ell}$-stable family of sheaves with Chern character $\chat-ke$,
				\item $\operatorname{or}$ is an orientation on $\mathcal{S}\oplus \mathcal{T}$, and
				\item $\Phi:\mathcal{L}(\mathcal{S}\oplus \mathcal{T})\to \mathcal{O}$ is a trivialization.  
			\end{enumerate} 
			We let $\C^*$ act on $\mathcal{Z}_{k,\Lambda}\to \mathcal{W}$ through acting on $\Phi$. Denote this action by $\mu:\C^*\times \mathcal{Z}_{k,\Lambda}\to \mathcal{Z}_{k,\Lambda}$. Then the map $\mathcal{Z}_{k,\Lambda}\to \mathcal{W}$ is in the obvious way $\C^*$-equivariant. It is moreover easy to see that $\mathcal{Z}_{k,\Lambda}$ is $\C^*$-fixed: For this, we want to compare the action $\mu$ with the trivial action on $\mathcal{Z}_{k,\Lambda}$.  Let $\lambda:\C^*\to \C^*, t\mapsto t^{k\chi_1}$ and consider the two outer edges of the following diagram
			\begin{equation}\label{eq:idequiv}
				\begin{tikzcd}
					\C^*\times  \mathcal{Z}_{k,\Lambda} \ar[r,"p_2"]\ar[d,"\lambda\times \operatorname{id}"]& \ar[d,"\operatorname{id}" ] \mathcal{Z}_{k,\Lambda} \\
					\C^* \times \mathcal{Z}_{k,\Lambda}\ar[r,"\mu"] \ar[ur, Rightarrow] & \mathcal{Z}_{k,\Lambda}.
				\end{tikzcd}
			\end{equation}
			We need to give the marked $2$-isomorphism witnessing commutativity of the diagram. Giving this isomorphism is equivalent to giving isomorphisms $a:\mathcal{S}\to \mathcal{S}$ and $b:\mathcal{T}\to \mathcal{T}$ that preserve the flag structures, and such that $a\oplus b$ preserves the orientation on $\mathcal{S}\oplus \mathcal{T}$ and maps $\Phi$ to $t^{k\chi_1} \Phi$. We may take $b$ to be the identity and $a$ to be multiplication by $t$. Then the corresponding $2$-morphism satisfies the higher associativity condition, and this gives   the structure of $\C^*$-equivariant morphism relating the two actions. We see from this that $\mathcal{S}$ and $\mathcal{T}$  are equivariant sheaves on $\mathcal{Z}_{k,\Lambda}$ of weight $-1/k\chi_1$ and weight $0$ respectively.
		\end{remark}
		
		\begin{lemma}
			The morphism $\rho^{\circ}_{k,\Lambda}$ descends to a closed embedding
			\[\rho_{k,\Lambda}: \mathcal{Z}_{k,\Lambda}\to \mathcal{W}.\] 
		\end{lemma}
		\begin{proof}~
			\paragraph{Existence of $\rho_{k,\Lambda}$.} 
			Consider the prestack \[\mathcal{Z}_{k,\Lambda}^{pre}:=\left\{\left(\operatorname{Fl}_{\globm{\mathcal{O}_C(-1)^{\oplus k}}}\times \mathcal{W}^{\ell}_{\chat-k\chexc}\right)/G\right\}\] associated to the $G$-action. It has the same objects as  $\operatorname{Fl}_{\globm{\mathcal{O}_C(-1)^{\oplus k}}}\times \mathcal{W}^{\ell}_{\chat-k\chexc}$ and for a scheme $B$ and objects $y,z\in \mathcal{Z}_{k,\Lambda}^{pre}(B)$ a morphism $y\to z$ in  $\mathcal{Z}_{k,\Lambda}^{pre}(B)$ is a pair $(\gamma,f)$, where $\gamma\in G(B)$, and $f:\gamma z\to z'$ is a morphism in $\operatorname{Fl}_{\globm{\mathcal{O}_C(-1)^{\oplus k}}}\times \mathcal{W}^{\ell}_{\chat-k\chexc}$. 
			To show that $\rho^{\circ}_{k,\Lambda}$ descends to a map $\rho_{k,\Lambda}$, it is enough to show that it induces a morphism $\mathcal{Z}_{k,\Lambda}^{pre}\to \mathcal{W}$. 
			In order to see this, let $B$ be a scheme, let $\gamma\in G(B)$ and let $z \in \mathcal{Z}_{k,\Lambda}^{pre}(B)$. We consider the morphism $(\gamma,\operatorname{id}_{\gamma z}):z\to gz$ of $\mathcal{Z}_{k,\Lambda}^{pre}(B)$. Say $\gamma=(g,(\det g)^{\chi_1})$ and $z=(W_S^{\bullet}, (T,W_T^{\bullet},\varphi_T))$, so that $\gamma z= (gW_S^{\bullet}, (T,W_T^{\bullet},(\det g)^{\chi_1}\varphi_T))$. To the morphism $(\gamma,\operatorname{id}_{\gamma z})$ we associate the morphism in $\mathcal{W}$ given by $\operatorname{id}_T\oplus g:T\oplus \mathcal{O}_C(-1)^{\oplus k}\to T\oplus \mathcal{O}_C(-1)^{\oplus k}$. More precisely, this gives a morphism 
			\begin{equation*}
				\begin{tikzcd}	
					\rho_{k,\lambda}^\circ(z) \ar[r,equals]&(T\oplus \mathcal{O}_C(-1)^{\oplus k},V_T\oplus V_S, \varphi_T\otimes \varphi_{e}^k) \ar[d,"\operatorname{id}_T \oplus g"] \\ \rho_{k,\Lambda}^{\circ}(\gamma z)\ar[r,equals]&(T\oplus \mathcal{O}_C(-1)^{\oplus k},V_T\oplus gV_S,(\det g)^{\chi_1} \varphi_T\otimes \varphi_{e}^k).
				\end{tikzcd}
			\end{equation*}
			This rule extends $\rho^{\circ}_{k,\Lambda}$ in a natural way to give a morphism $\mathcal{Z}^{pre}_{k,\Lambda}\to \mathcal{W}$. Taking the stackification gives the morphism $\rho_{k,\Lambda}$.\\

			\paragraph{Representability.}
			For any $\C$-valued point $z\in \mathcal{Z}_{k,\Lambda}(\C)$, the automorphisms of $z$ are pairs $(\gamma, f)$, where $\gamma\in G(\C)$ and $f:\gamma z\to z$ is an isomorphism in $\operatorname{Fl}_{\globm{\mathcal{O}_C(-1)^{\oplus k}}}\times \mathcal{W}^{\ell}_{\chat-k\chexc}$. Concretely, if $z=(W_S^{\bullet}, (T,W_T^{\bullet},\varphi_T))$, and $\gamma=(g,(\det g)^{\chi_1})$, then $g\in GL(k)$ must be scalar, and $f$ is given by a morphism $f:T\to T$. We have $\rho_{k,\Lambda}(\gamma, f)=f\oplus g$ as an isomorphism of $\rho_{k,Z}(z)$. In particular, only the identity morphism of $z$ gets sent to the identity morphism of $\rho_{k,\Lambda}(z)$. This shows that $\rho_{k,\Lambda}$ is representable.	\\
			\paragraph{$\rho_{k,\Lambda}$ is a monomorphism.}
			Suppose we have two objects \linebreak $z_i=(\mathcal{W}_{S_i},(\mathcal{T}_i,\mathcal{W}_{T_i},\Phi_{T_i} ))$ for $i=1,2$ of $\mathcal{Z}_{k,\Lambda}$ over some base $b:B\to \Spec \C$, which we can assume to be locally Noetherian, such that we have an isomorphism $\rho_{k,\Lambda}(z_{1})\xrightarrow{\sim}\rho_{k,\Lambda}(z_{2})$, i.e. an isomorphism $f: \mathcal{T}_1\oplus b^*\mathcal{O}_C(1)^{\oplus k}\xrightarrow{\sim}\mathcal{T}_2\oplus b^*\mathcal{O}_C(-1)^{\oplus k}$ compatible with the additional data. By composing $f$ with the respective inclusions into and projections out of the direct sum, we obtain maps $g:b^*\mathcal{O}_C(-1)^{\oplus k}\to b^*\mathcal{O}_C(-1)^{\oplus k}$ and $h:\mathcal{T}_1\to \mathcal{T}_2$, which are compatible with the flags $\mathcal{W}_{S_i}^{\bullet}$ and $\mathcal{W}_{T_i}^{\bullet}$ respectively. Since over a geometric point the splitting of a polystable sheaf $E$ into its minimal $S$-destabilizing and $T$-destabilizing subobjects is unique, it follows that $g$ and $h$ are isomorphisms over geometric points of $B$, hence they are isomorphisms by flatness. Since the automorphism group of $\mathcal{O}_C(-1)$ is just $\C^*$, $g$ must correspond to a $B$-valued point of $GL(k)$. Moreover, it is easy to see that $h$ must be compatible with orientations and that we must have $\Phi_{T_1} (\det g)^{\chi_1}=\Phi_{T_2}\circ \detL(h)$. Let $\gamma:=(g,(\det g)^{\chi_1})$.  It follows that $(\gamma,h)$ induces an isomorphism $z_1\to z_2$ in $\mathcal{Z}_{k,\Lambda}(B)$.  Since we have already established representability, this shows that $\rho_{k,\Lambda}$ is a monomorphism. 
			\paragraph{Conclusion.}  Since $\rho_{k,\Lambda}$ is a morphism from a proper to a separated $\C$-stack, it is a proper morphism.
			A closed immersion of stacks is the same as a proper monomorphism. Therefore we are done.  
		\end{proof}

		\begin{proposition}
			The natural map  
			\[\mathcal{N}^{\ell}\sqcup \mathcal{N}^{\ell+1}\sqcup\bigsqcup_{k,\Lambda}\mathcal{Z}_{k,\Lambda}\to \mathcal{Z} \]
			is an equivalence onto the fixed substack $\mathcal{Z}^{\C^*}$.                     
		\end{proposition} 
		
		\begin{proof}
			We already proved part of this in Proposition \ref{propfixeasy}.
			It remains to show that the $\C^*$-fixed stack of $\mathcal{W}$ is the union of the images of the $\rho_{k,\Lambda}$. 
			Let $W:=\Spec A\to \mathcal{W}$ be a $\C^*$-equivariant \'etale neighborhood with respect to some homomorphism $\lambda:s\mapsto s^j$. In particular, $W$ is an affine scheme with $\C^*$-action $\mu:\C^*\times W\to W$, and has a family $(\mathcal{E},\Phi)$ together with an isomorphism $f:\mu^*(\mathcal{E},\Phi)\to (\pi_W^*\mathcal{E},s^j\pi_W^*\Phi)$, satisfying an additional compatibility condition. Now suppose $Y\subset W$ is a closed subscheme that factors through $W^{\C^*}$. Then, for every $s\in \Gamma(Y,\mathcal{O}_Y)^{\times}$, we have isomorphisms $f_s:\mathcal{E}_Y\to \mathcal{E}_Y$, such that $f_s$ maps $\Phi_Y$ to $s^j \Phi_Y$ and the compatibility translates to  $f_s\circ f_t=f_{st}$.  Pick any $s\in \C^*$ for which $s^j\neq 1$ and let $f:=f_s$. Since it induces a nontrivial action on $\Phi$, we have that $f$ is not a scalar multiple of the identity over any closed point $y\in Y$. By Lemma \ref{lempolystabfamilies}, it follows that over each connected component of $Y$, we have a decomposition $\mathcal{E}\simeq \mathcal{T}\oplus \mathcal{S}$ and a compatible decomposition of flags. Let $Y_0$ be any such component. Then on each small enough open  $U\subset Y_0$, we may choose a trivialization $\mathcal{S}\simeq \pi_{X}^*\mathcal{O}_C(-1)^{\oplus k}$ for some $k$. Such an isomorphism induces a lift of the map $U\to \mathcal{W}$ to a map $U\to  \operatorname{Fl}^{\circ}_{\globm{\mathcal{O}_C(-1)^{\oplus k}}}\times \mathcal{W}_{\chat -k\chexc}^{\ell}$ along $\rho_{k,\Lambda}$ for some $\Lambda$. This shows that locally $Y_0\to \mathcal{W}$ factors through $\mathcal{Z}_{k,\Lambda}$, therefore it factors through $\mathcal{Z}_{k,\Lambda}$. Since $Y_0$ was an arbitrary component of $Y$, we conclude that $Y$ factors trough the union of all $\mathcal{Z}_{k,\Lambda}$ for varying $k$ and $\Lambda$. 
		\end{proof}

		\section{Perfect obstruction theories}\label{sec:perfob}
		This section is dedicated to constructing perfect obstruction theories on the various moduli spaces considered in Sections \ref{sec:perverse} and \ref{sec:master}. As the fundamental technical tool to construct obstruction theories on moduli spaces of coherent sheaves we will use the Atiyah class. The necessary results are summarized in Appendix \ref{app:atiyah}.
		For an algebraic stack $U$, we denote by $D(U):=\dbcoh(U)$ the bounded derived category of complexes with coherent cohomology. We denote by $\mathcal{H}^i$ the functor taking the $i$-th cohomology sheaf of a complex. For the smooth surface $\mathcal{X}$, we let $\omega_{X}$ denote the dualizing \emph{complex}, i.e. the sheaf $\Omega_{X}^2$ placed in cohomological degree $-2$ and similarly for $X$ replaced with $\Xhat$.

		\subsection{Obstruction theories and virtual classes}\label{subsec:obthy}
		We will use the following definition of obstruction theory for algebraic stacks, which agrees with the usual notion in the case of Deligne--Mumford stacks. By abuse of notation, we frequently denote obstruction theories as $E^{\bullet}$

		\begin{definition}
			Let $f:\mathcal{X}\to \mathcal{Y}$ be a morphism of quasi-separated, locally finite type stacks over $\C$.
			\begin{enumerate}[label=\arabic*)]
				\item A \emph{relative obstruction theory} for $\mathcal{X}$ over $\mathcal{Y}$ is given by an object $E^{\bullet}\in D^{\leq 1}(\mathcal{X})$ together with an homomorphism $\phi:E^{\bullet}\to L_{\mathcal{X}/\mathcal{Y}}$, such that $\mathcal{H}^i{\phi}$ is an isomorphism form $i=0,1$ and surjective for $i=-1$. We also drop the reference to $\phi$ if it is understood from the context. 
				\item An \emph{obstruction theory} is a relative obstruction theory over $\mathcal{Y}=\Spec \C$.
				\item  A (relative) obstruction theory is called \emph{perfect}, if smooth-locally on $\mathcal{X}$ the object $E^{\bullet}$ is isomorphic to a complex of locally free sheaves $E^{-1}\to E^0\to E^1$. 
			\end{enumerate} 
		\end{definition}

		See \cite[\S 2.3.1]{Moch} for a short overview of the cotangent complex on general algebraic stacks and \cite{Olso1} as the basic reference. 
		
		For a (relative) perfect obstruction theory the rank $\operatorname{rk}E^{\bullet}:=\operatorname{rk}(E^0)-\operatorname{rk}(E^1)-\operatorname{rk}(E^2)$ is always a locally constant function on $\mathcal{X}$. We say that \emph{$E^{\bullet}$ has rank $n$}, if the rank function is in fact constant with value $n$.

		\begin{remark}\label{remcheckperfect}
			The perfectness condition is equivalent to $E^{\bullet}$ having tor-amplitude in $[-1,1]$. 
		\end{remark}
		
		If $\mathcal{X}$ is a Deligne--Mumford stack with perfect obstruction theory $(E^{\bullet},\varphi)$ of rank $n$, there is an associated virtual fundamental class $[\mathcal{X}]^{\operatorname{vir}}\in A_{n}(\mathcal{X})$  by the construction of Behrend--Fantecchi \cite[Section 5]{BeFa} (an alternative approach to the virtual fundamental class is due to Li--Tian \cite{LiTi}).  A more recent treatment has been given by Manolache \cite{Mano} in her study of virtual pullbacks, which we will make use of.

		\paragraph{Virtual pullback.}
		Suppose that $F:\mathcal{X}\to \mathcal{Y}$ is a Deligne--Mumford morphism and that we are given a relative perfect obstruction theory $E^{\bullet} \to L_{\mathcal{X}/\mathcal{Y}}$ of rank $n$. This data induces a \emph{virtual pullback} map on Chow groups $F_{E^{\bullet}}^!:A_{*}(\mathcal{Y})\to A_{*+n}(\mathcal{X})$, which has been constructed by Manolache \cite[Construction 3.6]{Mano}.
		This construction generalizes the one by Behrend and Fantechi, due to the following:

		\begin{proposition}
			Suppose that $\mathcal{Y}$ is an equidimensional algebraic stack and that $\mathcal{X}$ is a Deligne--Mumford stack. Then $F_{E^{\bullet}}^![\mathcal{Y}]$ coincides with the virtual fundamental class of $\mathcal{X}$ in the sense of Behrend--Fantechi. 
		\end{proposition}

		\begin{proof}
			This is Corollary 3.12 in \cite{Mano}.
		\end{proof}

		\begin{example}[cf. {\cite[Remark 3.10]{Mano}}]\label{exvirtualpullbackflat}
			If $F:\mathcal{X}\to \mathcal{Y}$ is a smooth Deligne--Mumford morphism of relative dimension $n$, we may endow $\mathcal{X}$ with the trivial perfect obstruction theory, given by the identity $\Omega_{F}\to \Omega_F$ in degree zero. The induced virtual pullback $F^!:A_*(\mathcal{Y})\to A^n_{*+n}(\mathcal{X})$ is just the usual flat pullback on the Chow groups. 
		\end{example}

		\paragraph{Virtual pullback and relative obstruction theories.}
		
		Let $F:\mathcal{X}\to \mathcal{Y}$ be  a morphism of algebraic stacks and  $E_1^{\bullet}\to L_{\mathcal{Y}}$ an obstruction theory.  Suppose that we have a commutative diagram in $D(\mathcal{X})$:
		
		\begin{equation}\label{eqrelobthy}
			\begin{tikzcd}
				F^*E_1^{\bullet}\ar[r]\ar[d]& E_2^{\bullet}\ar[r]\ar[d]&E_3^{\bullet}\ar[d]\ar[r,"+1"]& \,\\
				F^*L_{\mathcal{Y}}\ar[r]&L_{\mathcal{X}}\ar[r]& L_{\mathcal{X}/\mathcal{Y}}\ar[r,"+1"] & \, 
			\end{tikzcd}
		\end{equation}
		with rows forming exact triangles. 
		
		\begin{lemma}\label{lemobfromrelob}
			If either of $E_2^{\bullet}$ or $E_3^{\bullet}$ is an obstruction theory, so is the other one. In this case, if moreover $E_1^{\bullet}$ and $E_3^{\bullet}$ are perfect, then so is $E_2^{\bullet}$.
		\end{lemma}

		\begin{proof}
			The first statement follows from the usual variants of the five-lemma by taking the long exact sequences in cohomology associated to the triangles. The second statement also follows from considering long exact sequences in view of \autoref{remcheckperfect}.
		\end{proof}
		
		We will frequently use Lemma \ref{lemobfromrelob} in the following situation:
		Suppose we are given an obstruction theory $E_1^{\bullet}$ on $\mathcal{Y}$ and a relative obstruction theory $E_3^{\bullet}$ for $\mathcal{X}$ over $\mathcal{Y}$. We say that the obstruction theories are \emph{compatible} if there exists a morphism $\alpha:E_3^{\bullet}[-1]\to F^*E_1^{\bullet}$ in $D(\mathcal{X})$ making the following diagram commute
		\begin{equation}\label{eqrelob}
			\begin{tikzcd}
				E_3^{\bullet}[-1]\ar[r,"\alpha"]\ar[d] & F^*E_1^{\bullet}\ar[d] \\
				L_{\mathcal{X}/\mathcal{Y}}\ar[r]& F^*L_{\mathcal{Y}}.
			\end{tikzcd}
		\end{equation}
		\begin{construction}\label{constrrelob}
			In the situation, of diagram \eqref{eqrelob}, let $E_2^{\bullet}$ be a cone of $-\alpha$, by which we mean an object together with morphisms fitting into an exact triangle 
			\[E_3^{\bullet}[-1]\xrightarrow{-\alpha}F^*E_1^{\bullet}\to E_2^{\bullet}\xrightarrow{+1}.\] By the axioms of triangulated categories, we can find a morphism $E_2^{\bullet}\to L_{\mathcal{X}}$ such that we are in situation of \eqref{eqrelobthy}, where the connecting homomorphism is given by $\alpha[1]$. By Lemma \ref{lemobfromrelob}, we find that $E_2^{\bullet}$ is an obstruction theory and that it is perfect if $E_1^{\bullet}$ and $E_3^{\bullet}$ are perfect. 
		\end{construction}
		
		Note, that while $E_2^{\bullet}$ in Construction \ref{constrrelob} depends up to isomorphism only on $\alpha$, the map $E_2^{\bullet}\to L_{\mathcal{X}}$ is not uniquely determined. In fact, for a fixed choice of $E_2^{\bullet}$, suppose $\phi_0,\phi_1:E_2^{\bullet}\to L_{\mathcal{X}}$ are maps making \eqref{eqrelobthy} commute. Then for any $t\in \C$, the map $\phi_t:=(1-t)\phi_0+t\phi_1$ also makes $\eqref{eqrelobthy}$ commute. In fact, it is straightforward to construct an obstruction theory on $\mathcal{X}\times \aone$ whose restriction to $t\in \aone(\C)$ is $\phi_t$. Using this, one can show that if $\mathcal{X}$ is a Deligne--Mumford stack, then the virtual class defined by $E_2^{\bullet}$ is independent of the choice of map in Construction \ref{constrrelob}. 
		\begin{remark}
			Let us say, that two obstruction theories $\phi_0,\phi_1:E^{\bullet}\to L_{\mathcal{X}}$ are \emph{homotopic}, if $(1-t)\phi_0+t\phi_1$ is again an obstruction theory for all $t\in \C^*$. Then the construction of the virtual fundamental class depends on the obstruction theory only up to homotopy. Moreover, we will use Construction \ref{constrrelob} in situations where $E_1^{\bullet}\to L_{\mathcal{Y}}$ may only be defined up to homotopy. Then the obstruction theory $E_2^{\bullet}\to L_{\mathcal{X}}$ is still defined up to homotopy, and we still get a well-defined virtual fundamental class.
		\end{remark} 
		
		In the case of Deligne--Mumford stacks, we have the following result relating Construction \ref{constrrelob} to virtual pullbacks:
		\begin{proposition}\label{prop:virpullfunc}
			Suppose that we have a diagram \eqref{eqrelobthy} with all vertical arrows perfect obstruction theories. If $\mathcal{X}$ and $\mathcal{Y}$ are both Deligne--Mumford stacks, then we have  $[\mathcal{X}]^{\operatorname{vir}}=F_{E_3^{\bullet}}^! [\mathcal{Y}]^{\operatorname{vir}}$. In particular, $[\mathcal{X}]^{\vir}$ depends only on $E_1^{\bullet}\to L_{\mathcal{Y}}$ and $E_3^{\bullet}\to L_{\mathcal{X}/\mathcal{Y}}$.   
		\end{proposition} 
		\begin{proof}
			This is a special case of the functoriality property for virtual pullbacks \cite[Thm. 4.8]{Mano}. 
		\end{proof}
		
		Given a $\C^*$-equivariant obstruction theory, we obtain an obstruction theory on the stack quotient by the following
		\begin{construction}\label{constcstarobstr}
			Suppose that $\mathcal{X}$ has a $\C^*$-action and that $\mathcal{Y}$ is the quotient of this action. Suppose moreover, that we have a $\C^*$-equivariant obstruction theory $E_2^{\bullet}\to L_{\mathcal{X}}$ on $\mathcal{X}$. We then have a natural $\C^*$-equivariant isomorphism $\Omega_{\mathcal{X}/\mathcal{Y}}\simeq \mathcal{O}_{X}$, and the following commutative diagram in the equivariant derived category of $\mathcal{X}$:
			\begin{equation*}
				\begin{tikzcd}
					\,&E_2\ar[r]\ar[d]& \mathcal{O}_{\mathcal{X}}\ar[d,"\sim"] \\
					F^*L_{\mathcal{Y}}\ar[r]&L_{\mathcal{X}}\ar[r] & \Omega_{\mathcal{X}/\mathcal{Y}}.
				\end{tikzcd}
			\end{equation*} 
			
			Using the axioms of triangulated categories, we can find an equivariant object $E_{1,\mathcal{X}}^{\bullet}$ that fits into a distinguished triangle  $E_{1,\mathcal{X}}^{\bullet}\to E_2^{\bullet}\to \mathcal{O}_{\mathcal{X}}\xrightarrow{+1}$, and obtain an equivariant morphism $E_{1,\mathcal{X}}^{\bullet}\to F^*L_{\mathcal{Y}}$ which makes the diagram a morphism of exact triangles. The morphism $E_{1,\mathcal{X}}^{\bullet}\to F^*L_{\mathcal{Y}}$ in the equivariant category of $\mathcal{X}$ induces a morphism $E_1^{\bullet}\to L_{\mathcal{Y}}$ in $D(\mathcal{Y})$.
			
		\end{construction}
		Similarly to Lemma \ref{lemobfromrelob}, one shows
		\begin{lemma} \label{lemcstarobstr}
			The morphism $E_1^{\bullet}\to L_{\mathcal{Y}}$ is an obstruction theory on $\mathcal{Y}$. It is perfect if and only if the obstruction theory $E_2^{\bullet}$ is perfect. 
		\end{lemma}

		\subsection{Perfect obstruction theory for sheaf moduli spaces}\label{subsecobstrbasic}
		In this subsection, we recall the construction of a perfect obstruction theory on moduli stacks of coherent sheaves on a smooth projective surface $X$ and of various relative obstruction theories. We then apply these results to the moduli stacks of our interest. 
		
		Let $U$ be a quasicompact and finite type algebraic stack over $\C$ and let $\mathcal{E}$ be a $U$-flat sheaf on $U\times X$. The Atiyah class supplies a natural morphism  (see \ref{subsec:atdef})
		\[R(\pi_X)_*(\mathcal{E}\otimes \mathcal{E}^{\vee}\otimes \omega_X)[-1]\to L_{U}.\]
		We will denote this morphism by $\ob(\mathcal{E}):\Ob(\mathcal{E})\to L_U$.
		
		As described in \ref{subsec:atdef},  $\ob(\mathcal{E})$ is an obstruction theory whenever $U$ is an open subset of the moduli stack of coherent sheaves on $X$ and $\mathcal{E}$ is the universal sheaf. Since $X$ is a smooth projective surface, we always have

		\begin{proposition}
			The complex $\Ob(\mathcal{E})$ is perfect of amplitude in $[-1,1]$. 
		\end{proposition}
		\begin{proof} 
			It is enough to show that $\Ob(\mathcal{E})$ has Tor-amplitude in $[-1,1]$. Equivalently, we show that for any map $f:T\to U$, where $T$ is a scheme, the base change $f^*\Ob(\mathcal{E})\simeq \Ob(f^*\mathcal{E})$ is concentrated in degrees $[-1,1]$. But $\Ob(f^*\mathcal{E})=R\Sheafhom_X(\mathcal{E},\mathcal{E}\otimes \omega_X)$, so this follows from relative Serre duality for the morphism $X\times T\to T$. 
		\end{proof}
		
		In what follows we will use the following suggestive notation: If $U$ is an algebraic stack and we are implicitly or explicitly given a coherent sheaf on $X\times U$ (for example if $U$ is a moduli stack of decorated sheaves), we will denote this sheaf by $\mathcal{E}_{U}$.
		
		\subsubsection{Obstruction theory for orientations} 
		We now construct an obstruction theory on moduli spaces on oriented sheaves. Let $\widetilde{\mathcal{M}}$ be any quasicompact open substack of the moduli stack of coherent sheaves on $X$ with fixed rank $r$ and Chern classes $c_1,c_2$. Assume that $r\neq 0$. Then the complex $\mathcal{E}_{\widetilde{\mathcal{M}}}\otimes \mathcal{E}_{\widetilde{\mathcal{M}}}^{\vee}= \mathcal{O}_T\oplus R\Sheafhom(\mathcal{E}_{\widetilde{\mathcal{M}}},\mathcal{E}_{\widetilde{\mathcal{M}}})_0$ splits canonically into its diagonal and trace-free part. 
		This induces a splitting of the obstruction theory $\Ob(\mathcal{E}_{\widetilde{\mathcal{M}}})=\Ob_d(\mathcal{E}_{\widetilde{\mathcal{M}}})\oplus \Ob_0(\mathcal{E}_{\widetilde{\mathcal{M}}})$ into diagonal and trace-free part respectively. Moreover, we have a canonical isomorphism $\Ob_d(\mathcal{E}_{\widetilde{\mathcal{M}}})\simeq R\Gamma(\omega_X)\otimes \mathcal{O}_{M}[-1]$. As in \S \ref{subsecmodspaces}, we let $\pi_{\operatorname{or}}:\mathcal{M}\to \widetilde{\mathcal{M}}$ denote the corresponding moduli space of oriented sheaves and consider the cartesian diagram
		\begin{equation}\label{eqcartesian1}
			\begin{tikzcd}
				\mathcal{M}\ar[r,"\det"]\ar[d]& \Pic_{X}\ar[d] \\
				\widetilde{\mathcal{M}}\ar[r,"\widetilde{\det}"] & \Picstack_X.
			\end{tikzcd}
		\end{equation}    
		Let $\mathcal{L}_X$ denote the universal line bundle on $\Picstack_X\times X$
		\begin{proposition}\label{propperfobor}
			The natural map $\pi_{\operatorname{or}}^*\Ob_0(\mathcal{E}_{\widetilde{\mathcal{M}}})\oplus \det^*\Omega_{\Pic_X}\to L_{\mathcal{M}}$ is a perfect obstruction theory on $\mathcal{M}$. Denote it by $\ob_{\mathcal{M}}:\Ob_{\mathcal{M}}\to L_{\mathcal{M}}$. It is compatible with the obstruction theory on $\widetilde{\mathcal{M}}$ in the following sense: We have a natural map $\pi_{\operatorname{or}}^*\Ob(\mathcal{E}_{\widetilde{\mathcal{M}}})\to \Ob_{\mathcal{M}}$ making the following diagram commute
			\begin{equation*}
				\begin{tikzcd}
					\pi_{\operatorname{or}}^*\Ob(\mathcal{E}_{\widetilde{\mathcal{M}}})\ar[r]\ar[d]& \Ob_{\mathcal{M}}\ar[d] \\
					\pi_{\operatorname{or}}^*L_{\widetilde{\mathcal{M}}}\ar[r] &L_{\mathcal{M}}.
				\end{tikzcd}
			\end{equation*} 
		\end{proposition}
		\begin{proof}
			First, note that we have a morphism of distinguished triangles
			\begin{equation}\label{eqeasydist}
				\begin{tikzcd}
					\pi_{\operatorname{or}}^*\Ob_0(\mathcal{E}_{\widetilde{\mathcal{M}}})\ar[r]\ar[d]&\operatorname{or}^*\Ob_0(\mathcal{E}_{\widetilde{\mathcal{M}}})\ar[d]\ar[r]& 0 \ar[r,"+1"]\ar[d]& \,\\
					\pi_{\operatorname{or}}^*{L_{\widetilde{\mathcal{M}}}}\ar[r]& L_{\mathcal{M}}\ar[r] & L_{\widetilde{\mathcal{M}}/\mathcal{M}}\ar[r,"+1"]& \,.
				\end{tikzcd}
			\end{equation}
			
			Next, by the compatibility of the Atiyah class with the determinant and with pullback (see \ref{subsec:atatiyah}), we have the commutative square
			\begin{equation*}
				\begin{tikzcd}
					{\widetilde{\det}}^*\Ob(\mathcal{L}_X)\ar[r]\ar[d]& \Ob(\mathcal{E}_{\widetilde{\mathcal{M}}})\ar[d] \\
					{\widetilde{\det}}^*L_{\Picstack_X}\ar[r] & L_{\widetilde{\mathcal{M}}}.
				\end{tikzcd}
			\end{equation*}
			Here, the upper horizontal map identifies $\widetilde{\det}^*\Ob(\mathcal{L}_X)$ with the diagonal part of $\Ob(\mathcal{E}_{\widetilde{\mathcal{M}}})$ via $\widetilde{\det}^*\Ob(\mathcal{L}_X)\simeq \Ob(\det \mathcal{E}_{\widetilde{\mathcal{M}}})\to \Ob(\mathcal{E}_{\widetilde{\mathcal{M}}})$.
			We obtain a map of distinguished triangles
			\begin{equation}\label{eqordiag1}
				\begin{tikzcd}
					\pi_{\operatorname{or}}^*\Ob(\mathcal{L}_X)\ar[r]\ar[d]&\Omega_{\Pic_X}\ar[r,dashed]\ar[d,equals]& E^{\bullet} \ar[r,dashed,"+1"]\ar[d,dashed,"\phi"]& \,\\
					\pi_{\operatorname{or}}^*L_{\Picstack_X}\ar[r]& \Omega_{\Pic_X}\ar[r] & L_{\Pic_X/\Picstack_X}\ar[r,"+1"]& \,.
				\end{tikzcd}
			\end{equation}
			Here the solid arrows are canonical and we use the axioms of the derived category to construct $E^{\bullet}$ and the dashed arrows. By Lemma \ref{lemobfromrelob}, the map $\phi$ is a relative obstruction theory. On the other hand, the diagram \eqref{eqcartesian1} induces the morphism of triangles
			
			\begin{equation}\label{eqordiag2}
				\begin{tikzcd}		
					\det^*\pi_{\operatorname{or}}^*L_{\Picstack_X}\ar[r]\ar[d]& \det^*\Omega_{\Pic_X}\ar[r]\ar[d] & \det^*L_{\Picstack_X/\Pic_X}\ar[r,"+1"]\ar[d,"\sim"]& \,\\
					\pi_{\operatorname{or}}^* L_{\widetilde{\mathcal{M}}}\ar[r]&L_{\mathcal{M}}\ar[r]& L_{\mathcal{M}/\widetilde{\mathcal{M}}} \ar[r,"+1"]& \,.
				\end{tikzcd}
			\end{equation}
			Putting \eqref{eqordiag1} and \eqref{eqordiag2} together, we have the morphism of triangles:
			\begin{equation}\label{eqharddist}
				\begin{tikzcd}
					\det^*\pi_{\operatorname{or}}^*\Ob(\mathcal{L}) \ar[r]\ar[d]&\det^*\Omega_{\Pic_X}\ar[r,dashed]\ar[d,equals]&\det^*E^{\bullet} \ar[r,dashed,"+1"]\ar[d,dashed,"\phi'"]& \,\\
					\pi_{\operatorname{or}}^*L_{\widetilde{\mathcal{M}}}\ar[r]& L_{\mathcal{M}}\ar[r] & L_{\mathcal{M}/\widetilde{\mathcal{M}}}\ar[r,"+1"]& \,.
				\end{tikzcd}
			\end{equation}
			Here the map $\phi'$ is just $\det^*\phi$ composed with an isomorphism, and therefore an obstruction theory. Under the isomorphism $\widetilde{\det}^*\Ob(\mathcal{L}_X)\simeq \Ob_d(\mathcal{E}_{\widetilde{\mathcal{M}}})$, the left vertical arrow is identified with the pullback of the composition $\Ob_d(\mathcal{E}_{\widetilde{\mathcal{M}}})\to \Ob(\mathcal{E}_{\widetilde{\mathcal{M}}})\xrightarrow{\ob(\mathcal{E}_{\widetilde{\mathcal{M}}})}L_{\widetilde{\mathcal{M}}}$ by the map $\pi_{\operatorname{or}}$.
			With these remarks, we see that summing the upper rows of \eqref{eqeasydist} and \eqref{eqharddist} gives a morphism of triangles 
			\begin{equation}\label{eqefinaldist}
				\begin{tikzcd}
					\pi_{\operatorname{or}}^*\Ob(\mathcal{E}_{\widetilde{\mathcal{M}}})\ar[r]\ar[d,"{\pi_{\operatorname{or}}^*\ob(\mathcal{E}_{\widetilde{\mathcal{M}}})}"]&{\pi_{\operatorname{or}}}^*\Ob_0(\mathcal{E}_{\widetilde{\mathcal{M}}})\oplus \det^*\Omega_{\Pic_X}\ar[d]\ar[r,dashed]& \det^*E^{\bullet} \ar[r,dashed,"+1"]\ar[d,dashed,"\phi'"]& \,\\
					\pi_{\operatorname{or}}^*{L_{\widetilde{\mathcal{M}}}}\ar[r]& L_{\mathcal{M}}\ar[r] & L_{\widetilde{\mathcal{M}}/\mathcal{M}}\ar[r,"+1"]& \, , 
				\end{tikzcd}
			\end{equation}
			where the solid maps are the natural ones. The outer vertical maps are obstruction theories, thus so is the middle one by Lemma \ref{lemobfromrelob}. Moreover, it is clearly perfect. This finishes the proof. 
		\end{proof}
		
		\begin{corollary}
			We have a canonical virtual fundamental class on each of the moduli spaces $\mathcal{M}_{X,c}$, $\mathcal{M}_{\Xhat, \chat}$ and  $\mathcal{M}^m_{\chat}$ for $m\in \Z$. 
		\end{corollary}
		
		\begin{remark}	
			By a similar argument, one also has a map of distinguished triangles 
			\begin{equation*}
				\begin{tikzcd}
					\det^*\Omega_{\Pic_X}\ar[r]\ar[d,equals]& \det^*\Omega_{\Pic_X}\oplus \pi_{\operatorname{or}}^*\Ob_0(\mathcal{E}_{\widetilde{\mathcal{M}}})\ar[d]\ar[r] & \pi_{\operatorname{or}}^*\Ob_0(\mathcal{E}_{\widetilde{\mathcal{M}}})\ar[d]\ar[r,"+1"] & \, \\
					\det^*\Omega_{\Pic_X}\ar[r] &L_{\mathcal{M}} \ar[r] & L_{\mathcal{M}/\Pic_X} \ar[r,"+1"]&\, ,
				\end{tikzcd}
			\end{equation*}
			where the maps in the upper triangle are just inclusion into and projection out of the direct sum, and all the vertical maps are the canonical ones. We see that  $\pi_{\operatorname{or}}^*\Ob_0(\mathcal{E}_{\widetilde{\mathcal{M}}})\to L_{\mathcal{M}/\Pic_X}$ is a relative perfect obstruction theory. In particular, if $\mathcal{M}$ is a Deligne--Mumford stack, its virtual fundamental class is equal to the virtual pullback of the usual fundamental class $[\Pic_X]$ defined by the relative obstruction theory $\pi_{\operatorname{or}}^*\Ob_0(\mathcal{E}_{\widetilde{\mathcal{M}}})\to L_{\mathcal{M}/\Pic_X}$.
		\end{remark}

		\subsubsection{Comparison of $\mathcal{M}_X(c)$ and $\mathcal{M}^0(p^*c)$.}
		As one may hope, the perfect obstruction theories on the isomorphic moduli spaces $\mathcal{M}_X(c)$ and $\mathcal{M}^0(p^*c)$ can be identified. 
		\begin{proposition}
			Let $\rho:\mathcal{M}^0(p^*c)\to \mathcal{M}_X(c)$ be the isomorphism given by pushforward of sheaves along $p:\Xhat\to X$ (with inverse given by pullback along $p$). Then the induced obstruction theory $\rho^*\ob_{\mathcal{M}_X(c)}$ is naturally isomorphic to $\ob_{\mathcal{M}^0(p^*c)}$. In particular, the virtual fundamental classes on both spaces agree.
		\end{proposition}
		
		\begin{proof}
			Let $\mathcal{E}$ be a universal sheaf on $\mathcal{M}_X(c)$. Then, using the identification along $\rho$, the sheaf $p^*\mathcal{E}$ is naturally isomorphic to a universal sheaf over $\mathcal{M}^0(p^*c)$. by Lemma \ref{lempushpullp} below, we have an isomorphism $\Ob(\mathcal{E})_0\xrightarrow{\sim}  \Ob(p^*\mathcal{E})_0$ compatible with the morphisms to $L_{\mathcal{M}^0(p^*c)}$. For example, by using that pullback along $p$ is compatible with determinants, we also have the commutative diagram 
			\begin{equation*}
				\begin{tikzcd}
					\rho^*\det_{\mathcal{E}}^*\Omega_{\Pic_X} \ar[r]\ar[d]& \det_{p^*\mathcal{E}}^*\Omega_{\Pic_{\Xhat}} \ar[d]\\
					\rho^*L_{\mathcal{M}_X(c)} \ar[r]& L_{\mathcal{M}^0(p^*c)}. \\
				\end{tikzcd}
			\end{equation*}
			Taking these together, we obtain the desired isomorphism of obstruction theories
			\[\Ob(\mathcal{E})_0\oplus {\det}^*\Omega_{\Pic_X} \to \Ob(p^*\mathcal{E})_0\oplus {\det}^*\Omega_{\Pic_{\Xhat}} .\] 
		\end{proof}
		
		\begin{remark}
			With a little more work, one can also show that for any $j\geq 0$, the obstruction theories on the base and the target of the map $\rho:\mathcal{M}^0(p^*c-je)\to \mathcal{M}_X(c)$ are compatible and are also compatible with the relative obstruction theory for the relative Quot-scheme $\rho$. This can be used to get explicit expressions for virtual integrals on $\mathcal{M}^0(p^*c-je)$ in terms of virtual integrals over $\mathcal{M}_X(c)$.   
		\end{remark}
		
		\begin{lemma}\label{lempushpullp}
			Let $U$ be an algebraic stack over $\C$, and let $\mathcal{E}$ be a family of torsion free sheaves on $X$ over $U$. Then we have a natural commutative diagram
			\begin{equation*}
				\begin{tikzcd}
					\Ob(\mathcal{E})\ar[rr,"\sim"]\ar[dr,"\ob(\mathcal{E})"']&[-15pt] &[-15pt] \Ob(p^*\mathcal{E})\ar[dl,"\ob(p^*\mathcal{E})"] \\[-5pt]
					&L_{U}. & 
				\end{tikzcd}
			\end{equation*}
			If moreover $\mathcal{E}$ has everywhere nonzero rank, then the isomorphism $\Ob(\mathcal{E})\simeq \Ob(p^*\mathcal{E})$ respects the decomposition into diagonal and trace-free part. 
		\end{lemma}

		\begin{proof}
			Let $\pi_{U}:U\times X\to U$ and $\widehat{\pi}_U:U\times \widehat{X}\to U$ denote the respective projections onto the second factor. 
			As in \ref{subsec:atdef}, the map $\ob(p^*\mathcal{E})$ is obtained as the composition
			\[ R(\widehat{\pi}_U)_*(p^*\mathcal{E}\otimes p^*\mathcal{E}^\vee\otimes \omega_{\Xhat})[-1]\xrightarrow{R(\widehat{\pi}_{U})_*(\operatorname{at}'_{p^*\mathcal{E}}\otimes \omega_{\Xhat})[-1]}R (\widehat{\pi}_U)_*(L_{\Xhat\times U/\Xhat}\otimes\omega_{\Xhat})\to L_U. \]
			By the pullback compatibility of the Atiyah class, we have $\at'_{p^*\mathcal{E}}=p^*\at'_\mathcal{E}$ upon identifying $L_{\Xhat\times U/\Xhat}\simeq L_{X\times U/X}$.
			We have a natural isomorphism $\omega_{\Xhat}\simeq p^*\omega_X\otimes \omega_{\Xhat/X}$ which is compatible with the trace maps 
			\begin{equation}\label{eq:trcommutes}
				\begin{tikzcd}
					R\Gamma(\omega_{\Xhat})\ar[r,"\operatorname{tr}_{\Xhat}"]\ar[d]& \mathcal{O}_{\Spec \C} \\
					R\Gamma(\omega_X \otimes Rp_*(\omega_{\Xhat/X}))\ar[r,"R\Gamma(\operatorname{tr}_{\Xhat/X})"] & R\Gamma(\omega_X).\ar[u,"\operatorname{tr}_X"']
				\end{tikzcd}
			\end{equation}
			
			By the projection formula, we can identify $Rp_*(\at'_{p^*\mathcal{E}}\otimes \omega_{\Xhat})$ with 
			\[\at'_{\mathcal{E}}\otimes \omega_X \otimes Rp_*\omega_{\Xhat/X}.\] 
			We therefore have the following diagram, in which at least the square commutes: 
			\begin{equation*}
				\begin{tikzcd}
					R(\pi_U)_*(\mathcal{E}\otimes \mathcal{E}^\vee\otimes \omega_X\otimes Rp_*\omega_{\Xhat/X})[-1]\ar[r]\ar[d]&L_{U}\otimes R(\pi_U)_*(\omega_X\otimes Rp_*\omega_{\Xhat/X})\ar[r]\ar[d]& L_U . \\
					R(\pi_U)_*(\mathcal{E}\otimes \mathcal{E}^\vee\otimes \omega_X)[-1]\ar[r]&L_U\otimes R(\pi_U)_*(\omega_X)\ar[ur] & \,
				\end{tikzcd}
			\end{equation*}
			Here, the upper set of rightwards pointing arrows composes to $\ob(p^*\mathcal{E})$ and the lower set composes to $\ob(\mathcal{E})$. The vertical maps are induced from the trace $Rp_*\omega_{\Xhat/X}\to \mathcal{O}_{X}$, which is an isomorphism, and they are therefore iso morphisms. We are done if we can show that the triangle commutes. In this triangle, the horizontal map is induced from the trace map for $\omega_{\Xhat}$, the vertical map from the trace map for $\omega_{\Xhat/X}$, and the diagonal map from the trace map for $\omega_X$. This commutativity follows from \eqref{eq:trcommutes}.
			  
			To see the final statement about the decomposition into diagonal and trace free part, it is enough to note that this decomposition is preserved under pullbacks and behaves well with respect to the projection formula. 
		\end{proof}

		\subsection{Construction of relative obstruction theories}\label{subsec:relobs}
		\subsubsection{Relative obstruction theory for flag spaces}
		Let $U$ be a quasicompact and finite type algebraic stack over $\C$ and let $\mathcal{E}$ be a $U$-flat sheaf on $U\times X$. Suppose that the fibers of $\mathcal{E}$ over $\C$-points of $U$ all have the same rank and Chern classes. 
		Let $L$ be an ample line bundle on $X$ and $m$ large enough, so that for each fiber $\mathcal{E}_u, u\in U(\C)$, the twist $\mathcal{E}_u\otimes L^{\otimes m}$ is globally generated with vanishing higher cohomology. Then $\mathcal{V}:=\pi_*(\mathcal{E}\otimes L^{\otimes m})$ is a vector bundle on $U$, and we have a surjective map $\pi^*\mathcal{V}\to \mathcal{E}\otimes L^{\otimes m}$. Let $\pi_F:F(\mathcal{V}) \to U$ be the relative flag space of $\mathcal{V}$. We consider $F(\mathcal{V})$ with its relative perfect obstruction theory over $U$ given by $\Omega_{F(\mathcal{V})/U}\xrightarrow{=}\Omega_{F(\mathcal{V})/U}$. 
		\begin{lemma}\label{lem:relobflags}
			There is a natural dashed arrow making the following diagram of solid arrows commute:
			\begin{equation*}
				\begin{tikzcd}
					\Omega_{F(\mathcal{V})/U}[-1]\ar[r, dashed]\ar[d,equals]& \pi_F^*\Ob(\mathcal{E})\ar[d] \\
					\Omega_{F(\mathcal{V})/U}[-1]\ar[r] & \pi_F^*L_{U} .
				\end{tikzcd}
			\end{equation*}
		\end{lemma}
		\begin{proof}
			Let $R:=\operatorname{rk}\mathcal{V}$, and let $\Phi_{\mathcal{V}}:U\to BGL_R$ denote the classifying map and $\mathcal{V}^u$ a universal rank $R$ vector bundle on $BGL_R$. 
			By the pullback, tensor product and pushforward compatibility of the Atiyah class, we have the following commuting diagram on $U$:
			\begin{equation}\label{eqflagdiag}
				\begin{tikzcd}
					\Phi_{\mathcal{V}}^*(\mathcal{V}^u\otimes {\mathcal{V}^u}^{\vee})\ar[r,"\sim"]\ar[d,"\sim"]&\mathcal{V}\otimes \mathcal{V}^{\vee}[-1]\ar[r]\ar[dr]& \Ob(\mathcal{E}\otimes L^{\otimes m})\ar[d] \ar[r,"\sim"]&\Ob(\mathcal{E})\ar[dl]\\
					\Phi_{\mathcal{V}}^*L_{BGL_R}\ar[rr] & \,&L_{U} . &\,
				\end{tikzcd}
			\end{equation}
			To show the lemma, it suffices to show that the connecting map 
			\[\Omega_{\mathcal{F}(\mathcal{V})/U}[-1] \to \pi_F^* L_U\]
			naturally factors through $\Phi_{\mathcal{V}}^*L_{BGL_R}$. To see this, consider the cartesian diagram
			
			\begin{equation*}
				\begin{tikzcd}
					\mathcal{F}(\mathcal{V})\ar[r,"\overline{\Phi}_{\mathcal{V}}"]\ar[d, "\pi_F"]& \mathcal{F}(\mathcal{V}^u) \ar[d] \\
					U\ar[r,"\Phi_{\mathcal{V}}"] & BGL_R.
				\end{tikzcd}
			\end{equation*}
			It induces the following commutative diagram of cotangent complexes, proving the needed statement:
			\begin{equation*}
				\begin{tikzcd}
					\overline{\Phi}_{\mathcal{V}}^*L_{\mathcal{F}(\mathcal{V}^u)/BGL_R}[-1] \ar[r]\ar[d,"\sim"]& \pi_F^*\Phi_{\mathcal{V}}^*L_{BGL_R}\ar[d] \\
					\Phi_{\mathcal{V}}^*L_{\mathcal{F}(\mathcal{V})/U}[-1] \ar[r] & \pi_F^*L_U.
				\end{tikzcd}
			\end{equation*}
		\end{proof} 
		
		By applying Lemma \ref{lem:relobflags} in the situation of \S \ref{subsecunderlyingstack}, i.e. with $U=\mathcal{M}^{0,1}$, $\mathcal{E}=\mathcal{E}_{\mathcal{M}^{0,1}}$ and $m,L$ as chosen there,  and composing with the map $\Ob(\mathcal{E}_{\mathcal{M}^{0,1}})\to \Ob_{\mathcal{M}^{0,1}}$ of Proposition \ref{propperfobor}, we get a relative obstruction theory on $\mathcal{N}$  (defined in \eqref{eqNdef}) over $\mathcal{M}^{0,1}$, which is naturally compatible with $\Ob_{\mathcal{M}^{0,1}}$.  Applying Construction \ref{constrrelob}, we obtain
		
		\begin{corollary}\label{cor:obnstack}
			The stack $\mathcal{N}$ defined by \eqref{eqNdef} has a perfect obstruction theory $\ob_{\mathcal{N}}:\Ob_{\mathcal{N}}\to L_{\mathcal{N}}$, which is canonical up to homotopy. It comes with a map $\Ob(\mathcal{E}_{\mathcal{N}})\to \Ob_{\mathcal{N}}$, compatible with the maps to the cotangent complexes. 
		\end{corollary}
		By restricting $\ob_{\mathcal{N}}$, we obtain perfect obstruction theories on its open substacks $\mathcal{N}^{\ell}$ and $\mathcal{N}^{\ell,\ell+1}$. We will denote these by $(\Ob_{\mathcal{N}^{\ell}},\ob_{\mathcal{N}^{\ell}})$ and $(\Ob_{\mathcal{N}^{\ell, \ell+1}},\ob_{\mathcal{N}^{\ell,\ell+1}})$ respectively. 
		\begin{corollary}
			Each of the stacks $\mathcal{N}^{\ell}$ has a canonical virtual fundamental class. Moreover, $[\mathcal{N}^{0}]^{\vir}$ is the flat pullback of $[\mathcal{M}^0]^{\vir}$, and $[\mathcal{N}^{N}]^{\vir}$ is the flat pullback of $[\mathcal{M}^{1}]^{\vir}$. 
		\end{corollary}

		\subsubsection{Relative obstruction theory for the stack $\mathcal{W}$}

		Now consider the case that $\Xhat$ is a blowup of the smooth projective surface $X$, i.e. we work with $\Xhat$ in place of $X$ in the above constructions. Let $U$ be a quasicompact algebraic $\C$-stack and $\mathcal{E}$ a $U$-flat coherent sheaf on $\Xhat\times U$. We assume that the rank and Chern classes of the fibers of $\mathcal{E}$ over $U$ are constant. We define $\detL(\mathcal{E})$ as in Construction \ref{constrdetl}. 
		Let $\pi_W:\mathcal{W}(\mathcal{E})\to U$ denote the total space of $\detL(\mathcal{E})^\vee$ with the tautological obstruction theory $\Omega_{\mathcal{W}(\mathcal{E})/U}\xrightarrow{=}\Omega_{\mathcal{W} (\mathcal{E})/U}$.  
		\begin{lemma}\label{lemnatW}
			There is a natural dashed arrow making the following diagram of solid arrows in the equivariant derived category on $\mathcal{W}$ commute:
			\begin{equation*}
				\begin{tikzcd}
					\Omega_{\mathcal{W}(\mathcal{E})/U}\ar[r, dashed]\ar[d,equals]& \pi_W^*\Ob(\mathcal{E})\ar[d] \\
					\Omega_{\mathcal{W}(\mathcal{E})/U}\ar[r] & \pi_W^*L_{U} .
				\end{tikzcd}
			\end{equation*}
		\end{lemma}
		\begin{proof}
			Let $B_1=p^*\mathcal{O}_X(M)$ and $B_2=p^*\mathcal{O}_X(M)\otimes \mathcal{O}_{\Xhat}(-C)$. Let also $\mathcal{L}_i:=\det R(\pi_U)_*B_i$ for $i=1,2$.  
			By compatibility of the Atiyah class with pushforward, compatibility with tensor product, and compatibility with determinants as in Apendix \ref{subsec:atatiyah}, we have natural diagrams
			\begin{equation}\label{eqBdiag}
				\begin{tikzcd}
					\mathcal{L}_i\otimes \mathcal{L}_i^{\vee}[-1]\ar[r]\ar[d,"\sim"]& \Ob(\mathcal{E})\ar[d] \\
					\Phi_{\mathcal{L}_i}^*L_{B\mathbb{G}_m}\ar[r] & L_{U}
				\end{tikzcd}
			\end{equation}
			for $i=1,2$. Let $\mu:B\mathbb{G}_m\times B\mathbb{G}_m\to B\mathbb{G}_m$ be the morphism given by $(L_1,L_2)\mapsto L_1^{\chi_2}\otimes L_2^{-\chi_1}$. It induces a natural map
			$\Phi_{\detL(\mathcal{E})}^*L_{B\mathbb{G}_m}\to \Phi_{\mathcal{L}_1}^*L_{B\mathbb{G}_m}\oplus \Phi_{\mathcal{L}_2}^*L_{B\mathbb{G}_m}$. Thus by summing the left hand sides of \eqref{eqBdiag}, we can form the diagram
			\begin{equation*}
				\begin{tikzcd}
					&\mathcal{L}_1\otimes \mathcal{L}_1^{\vee}[-1]\oplus \mathcal{L}_2\otimes \mathcal{L}_2^{\vee}[-1]\ar[r]\ar[d,"\sim"]& Ob(\mathcal{E})\ar[d] \\
					\Phi^*_{\detL(\mathcal{E})}L_{B\mathbb{G}_m}\ar[r]& \Phi_{\mathcal{L}_1}^*L_{B\mathbb{G}_m}\oplus \Phi_{\mathcal{L}_2}^*L_{B\mathbb{G}_m}\ar[r] & L_{U} .
				\end{tikzcd}
			\end{equation*}
			From this, we see that natural map $\Phi^*_{\detL(\mathcal{E})}L_{B\mathbb{G}_m}\to L_{U}$ has a natural lift to $\Ob(\mathcal{E})$. The lemma now follows from  the observation that that the map $\Omega_{\mathcal{W}(\mathcal{E})/U}[-1]\to \pi_{\mathcal{W}(\mathcal{E})}^* L_{U}$ is the composition of $\Phi^*_{\detL(\mathcal{E})}\to L_{U}$ with the natural isomorphism $\Omega_{\mathcal{W}(\mathcal{E})/U}[-1]\xrightarrow{\sim}\pi_{W}^*\Phi^*_{\detL(\mathcal{E})}L_{B\mathbb{G}_m}$. 
		\end{proof}
		
		By applying Lemma \ref{lemnatW} in the situation of \S \ref{subsecmasterspace} to $U=\mathcal{N}^{\ell,\ell+1}$ and $\mathcal{E}=\mathcal{E}_{\mathcal{N}^{\ell,\ell+1}}$, we obtain an equivariant relative perfect obstruction theory for the stack $\mathcal{W}$ defined in $\eqref{eqnWdef}$ over $\mathcal{N}^{\ell,\ell+1}$, which is naturally compatible with $\Ob(\mathcal{E}_{\mathcal{N}^{\ell,\ell+1}})$. By composing with the pullback of the map $\Ob(\mathcal{E}_{\mathcal{N}^{\ell,\ell+1}})\to \Ob_{\mathcal{N}^{\ell,\ell+1}}$ of Corollary \ref{cor:obnstack}, we see that the relative obstruction theory is naturally compatible with $\Ob_{\mathcal{N}^{\ell,\ell+1}}$. Applying the equivariant analogue of Construction \ref{constrrelob}, we obtain an equivariant perfect obstruction theory $\ob_{\mathcal{W}}:\Ob_{\mathcal{W}}\to L_{\mathcal{W}}$ on $\mathcal{W}$, which is canonical up to homotopy.  
		We have therefore shown the following: 
		\begin{corollary}\label{cor:obW}
			The stack $\mathcal{W}$ defined by \eqref{eqnWdef} has a $\C^*$-equivariant perfect obstruction theory $\ob_{\mathcal{W}}:\Ob_{\mathcal{W}}\to L_{\mathcal{W}}$, which is canonical up to homotopy and therefore induces a canonical virtual fundamental class. It comes with a map $\Ob(\mathcal{E}_{\mathcal{W}})\to \Ob_{\mathcal{W}}$, compatible with the maps to the cotangent complexes.
		\end{corollary}

		\subsection{Obstruction theory on the master space} \label{subsec:obstmaster}
		
		Let $\mathcal{W}\to \mathcal{N}^{\ell,\ell+1}$ and $\mathcal{Z}$ be as in \S \ref{subsecmasterspace}. We consider $\mathcal{W}$ with the obstruction theory $\ob_{\mathcal{W}}:\Ob_{\mathcal{W}}\to L_{\mathcal{W}}$ of Corollary \ref{cor:obW}. By taking the trivial obstruction theory on $\mathbb{P}^1$, we obtain in a canonical way an obstruction theory on $\mathcal{W}\times \pone$, which is equivariant for both $\C^*$-actions \eqref{eqnTaction} and \eqref{eqnsecondaction}. By Construction \ref{constcstarobstr} and Lemma \ref{lemcstarobstr}, we obtain an obstruction theory 
		
		\[Ob_{\mathcal{Z}}\to L_{\mathcal{Z}}\]
		on the master space $\mathcal{Z}$. Since the two $\C^*$-actions on $\mathcal{W}\times \pone$ commute, it is equivariant with respect to the $\C^*$-action on $\mathcal{Z}$. By construction, it is also natural up to homotopy in the equivariant derived category. 
		
		\begin{proposition}
			The master space $\mathcal{Z}$ has a $\C^*$-equivariant obstruction theory, natural up to homotopy. In particular, it carries a virtual fundamental class $[\mathcal{Z}]^{\vir}$.  
		\end{proposition}

		\subsection{Fixed obstruction theories and virtual normal bundles}
		\label{subsec:obstrfix}
		We consider the fixed point loci of the $\C^*$-action on the master space, and describe the induced perfect obstruction theory and virtual normal bundle on each component. We let $\teq$ denote a trivial line bundle with $\C^*$-action of weight $1$ and $t=c_1(\teq)$ the corresponding generator of the equivariant Chow ring of a point. 
		
		We recall the decomposition of the fixed point stack determined in \S \ref{subsec:fix}:
		
		\[\mathcal{Z}^{\C^*}=\mathcal{N}^{\ell}\sqcup \mathcal{N}^{\ell+1}\sqcup \bigsqcup_{k,\Lambda} \mathcal{Z}_{k,\Lambda}. \]
		
		\begin{proposition}\label{propfixobeasy}
			\begin{enumerate}[label=\arabic*)]
				\item 	The perfect obstruction theories induced on $\mathcal{N}^{\ell}$ and $\mathcal{N}^{\ell+1}$ as components of $\mathcal{Z}^{\C^*}$ agree with the usual ones.\label{propfixobeasy1}
				\item The virtual normal bundle on $\mathcal{N}^{\ell}$ is given by $\teq$.  \label{propfixobeasy2}
				\item The virtual normal bundle on $\mathcal{N}^{\ell+1}$ is given by $\teq^{\vee}$. \label{propfixobeasy3}
			\end{enumerate}
		\end{proposition}
		\begin{proof}
			We determine the fixed obstruction theory and virtual normal bundle on $\mathcal{N}^{\ell+1}$. This gives \ref{propfixobeasy3} and half of \ref{propfixobeasy1}. The other statements follow in an analogous manner. 
			We may work in the open substack $\mathcal{Z}^+:=\mathcal{W}^+\times \aone/T\subset \mathcal{Z}$ containing the fixed locus corresponding to $\mathcal{N}^{\ell+1}$. 
			We consider the natural map $L_{\mathcal{W}^{+}}\oplus  \Omega_{\mathbb{A}^1}\simeq L_{\mathcal{W}^+\times \aone}\to \Omega_{\mathcal{W}^{+}\times \mathbb{A}^1/\mathcal{Z}^+}$ and its restrictions $L_{\mathcal{W}^+}\to \Omega_{\mathcal{W}^+\times\aone/\mathcal{Z}^+}$ and $\Omega_{\aone}\to \Omega_{\mathcal{W}^+\times \aone/\mathcal{Z}^+}$.  
			\begin{lemma}
				The map   $L_{\mathcal{W}^{+}}\to \Omega_{\mathcal{W}^{+}\times \mathbb{A}^1/\mathcal{Z}}$ factors through $\Omega_{\mathcal{W}^+/\mathcal{N}^{\ell+1}}\to \Omega_{\mathcal{W}^+\times \aone/\mathcal{Z}}$.  
				The map $\Omega_{\aone}\to \Omega_{\mathcal{W}^+\times \aone/\mathcal{Z}^+}$ vanishes when restricted to $\mathcal{W}^{+}\times \{0\}$. 
			\end{lemma}
			\begin{proof}
				The first statement follows from the usual properties of the cotangent complex in view of the following $2$-commutative diagram:
				\begin{equation*}
					\begin{tikzcd}[cramped,sep=small]
						&\mathcal{W}^+\times \aone\ar[dl]\ar[dr] & \\
						\mathcal{W}^+\ar[dr]& & \mathcal{Z}^{+}\ar[dl] \\
						&\mathcal{N}^{\ell+1} & \,
					\end{tikzcd}
				\end{equation*}
				For the second statement, we may work \'etale-locally over $\mathcal{N}^{\ell+1}$. We may then choose a local trivialization of the $\C^* $-bundle $\mathcal{W}^+$ over $\mathcal{N}^{\ell+1}$ and the statement reduces to the following claim:  Let $G:\C^*\times \aone\to \C^*, (t,s)\mapsto s/t$. Then the induced map on K\"ahler differentials $\Omega_{\C^*}\oplus  \Omega_{\aone }\to \Omega_{G}$ vanishes on $\Omega_{\aone}$ when restricted to $\C^*\times \{0\}$. This follows from the formula $d(s/t)=ds/t- sdt/t^2$ and specializing to $s=0$. 
			\end{proof}
			
			It follows from this that on $\mathcal{W}^{+}\times \{0\}$, the map $L_{\mathcal{W}^+}\oplus \Omega_{\aone}\to \Omega_{\mathcal{W}^+\times \aone/\mathcal{Z}}$ factors through a map $\iota^* \Omega_{\mathcal{W}^+/\mathcal{N}^{\ell+1}}\to \iota^*\Omega_{\mathcal{W}^+\times \aone/\mathcal{Z}}$, which is moreover an isomorphism. Indeed, the map $L_{\mathcal{W}^+\times \aone}\to \Omega_{\mathcal{W}^+\times \aone/\mathcal{Z}}$ induces an epimorphism on zero-th cohomology sheaves, and an epimorphism between line bundles is an isomorphism.  We conclude that we have an exact triangle 
			\[\iota^*\left(\pi^* L_{\mathcal{N}^{\ell+1}}\oplus \Omega_{\aone} \right)\to \iota^*L_{\mathcal{W}^+\times \aone }\to \iota^*\Omega_{\mathcal{W}^+\times \aone/\mathcal{Z}}\xrightarrow{+1}.\]
			From the construction of the obstruction theory on $\mathcal{W}$, we obtain in the same way a triangle
			
			\[\iota^*\left(\pi^* \Ob_{\mathcal{N}^{\ell+1}}\oplus \Omega_{\aone}\right)\to \iota^*\left(\Ob_{\mathcal{W}^+} \oplus \Omega_{\aone} \right)\to \iota^*\Omega_{\mathcal{W}^+\times \aone/\mathcal{Z}}\xrightarrow{+1}.\]
			Moreover, the maps associated to the obstruction theories give a morphism between these triangles. 
			It follows, that the restriction of the obstruction theory to the fixed point locus on $\mathcal{Z}$ is given by  $\operatorname{ob}_{\mathcal{N}^{\ell+1}}\oplus {\operatorname{id}}_{\Omega_{\aone}}$. We find that the fixed part recovers the usual obstruction theory on $\mathcal{N}^{\ell+1}$. The moving part of the obstruction theory is the pullback of $\Omega_{\aone}\mid_0$ with the action induced by \eqref{eqnsecondaction}. By our conventions, the standard $\C^*$-action on $\aone$ induces a weight minus one action on $\Omega_{\mathbb{A}^1}|_{0}$. In our case, the action on the master space induced from \eqref{eqnsecondaction} corresponds to the dual of the standard $\C^*$-action on $\aone$, so the moving part is a trivial line bundle with $\C^*$-weight $1$. To get the virtual normal bundle, we need to take the dual again, so we find that the virtual normal bundle is $\teq^{\vee}$.  
		\end{proof}
		
		We now describe the virtual obstruction theory on the remaining fixed loci $\mathcal{Z}_{k,\Lambda}$. 
		\begin{lemma}
			The isomorphism $\mathcal{W}\simeq \mathcal{W}\times \C^*/T$ described in Subsection \ref{subsec:fix} is compatible with obstruction theories.
		\end{lemma}
		\begin{proof}
			Indeed, this isomorphism is given by the composition 
			\begin{equation*}
				\begin{tikzcd}
					\mathcal{W} \ar[rr,"{w\mapsto (w,1)}"] \ar[drr] &\, &\mathcal{W}\times \C^*  \ar[d,"\pi_{\mathcal{Z}}"]&\\
					\,&\, & \mathcal{Z}^{\circ}\ar[hookrightarrow,r]& \mathcal{Z}.
				\end{tikzcd}
			\end{equation*}
		Here $\mathcal{Z}^{\circ}$ is just $\mathcal{W}\times \C^*/T$.
			By construction, the obstruction theory of $\mathcal{Z}$ fits into an exact triangle
			\[\pi_{\mathcal{Z}}^*\Ob_{\mathcal{Z}}\to p_1^*\Ob_{\mathcal{W}}\oplus p_2^*\Omega_{\C^*}\to \Omega_{\mathcal{W}\times \C^*/\mathcal{Z}^{\circ}}.\]
			By composing the first map with projection onto $\Ob_{\mathcal{W}}$, we get the upper horizontal map in the following diagram on $\mathcal{W}\times \{1\}$. The lower horizontal map is obtained in the analogous way from the natural maps of the cotangent complexes. 
			\begin{equation*}
				\begin{tikzcd}
					\pi_{\mathcal{Z}}^*\Ob_{\mathcal{Z}}\mid_{\mathcal{W}\times \{1\}}\ar[r]\ar[d]& \Ob_{\mathcal{W}}\ar[d] \\
					\pi_{\mathcal{Z}}^*L_{\mathcal{Z}}\mid_{\mathcal{W}\times \{1\}}\ar[r] &L_{\mathcal{W}} \,
				\end{tikzcd}
			\end{equation*}
			We claim that the horizontal arrows are isomorphisms. This follows from the statement that over $\mathcal{W}\times \{1\}$, the map $L_{\mathcal{W}}\oplus \Omega_{\C^*}\to \Omega_{\mathcal{W}\times \C^*/\mathcal{Z}}$ restricts to an isomorphism $\Omega_{\C^*}\to \Omega_{\mathcal{W}\times \C^*/\mathcal{Z} }$. The last statement reduces to a corresponding elementary statement for the map $\C^*\times \C^*\to \C^*; (t,s)\mapsto s/t$, and is readily checked. 
		\end{proof}
		
		Now consider the closed embedding $\rho_{k,\Lambda}:\mathcal{Z}_{k,\Lambda}\to \mathcal{W}$. We determine the fixed and moving part of the $\Ob_{\mathcal{W}}$ restricted to $\mathcal{Z}_{k,\Lambda}$.

		Recall that we have built the obstruction theory on $\mathcal{W}$ from the following pieces:
		\begin{enumerate}[label=(\arabic*)]
			\item $\Ob_0(\mathcal{E})=R\pi_*(R\Sheafhom(\mathcal{E},\mathcal{E})_0\otimes \omega_X)[-1]$ \label{piece1},
			\item $\Omega_{\Pic_{\Xhat}}$ \label{piece3},
			\item $\Omega_{Flag}=\operatorname{Coker}(\bigoplus_{i=0}^{N-1} \Hom(\globm{\mathcal{E}}/\mathcal{V}^{i+1},\mathcal{V}^i)  \to \bigoplus_{i=0}^N\Hom(\globm{\mathcal{E}}/\mathcal{V}^i,\mathcal{V}^i)),$ \label{piece2}
			\item $R\Sheafhom(\detL(\mathcal{E}),\detL(\mathcal{E}))$. \label{piece4}
		\end{enumerate}
		More precisely, all of these objects naturally pull back to objects in the equivariant derived category of $\mathcal{W}$, and we have certain natural equivariant maps between them. We may regard $\Ob_{\mathcal{W}}$ as being constructed from this data in the equivariant derived category on $\mathcal{W}$ by successively completing maps to exact triangles. In particular, the weight decomposition of $\rho_{k,\Lambda}^*\Ob_{\mathcal{W}}$ is induced from the decomposition of each summand. For \ref{piece3}, we observe that the equivariant structure is obtained from pullback along the classifying map $\mathcal{W}\to \Pic_{\Xhat}$, which we may regard as equivariant with respect to the trivial action on $\Pic_{\Xhat}$. Since $\Pic_{\Xhat}$ is a scheme, this shows that the pullback of $\Omega_{\Pic}$ to $\mathcal{Z}_{k,\Lambda}$ is $\C^*$-fixed. For the other parts, we use that $\mathcal{E}_{\mathcal{Z}_{k,\Lambda}}$ splits as a direct sum $\mathcal{S}\oplus \mathcal{T}$, with induced splitting of flags $\rho_{k,\Lambda}^*\mathcal{V}^{\bullet}=\mathcal{V}_S^{\bullet}\oplus \mathcal{V}_T^{\bullet}$. By Remark \ref{remequiv}, the sheaf $\mathcal{S}$ is equivariant of weight $-1/k\chi_1$, while $\mathcal{T}$ is equivariant of weight $0$. Let $\bchi$ denote an equivariant line bundle of weight $1/k\chi_1$. We will abuse notation, and from here on out consider $\mathcal{S}$ with the trivial $\C^*$-action and make the action explicit by tensoring with $\bchi^{\vee}$.

		We can now determine the fixed and moving part of the remaining pieces: 
		\begin{enumerate}[label=(\arabic*)]
			\item  The equivariant decomposition $\mathcal{E}_{\mathcal{Z}_{k,\Lambda}}=\mathcal{S}\otimes \bchi^{\vee}\oplus \mathcal{T}$ induces a decomposition 
			\begin{align*}
				\Ob(\mathcal{E}_{\mathcal{Z}_{k,\Lambda}})&=R\pi_*(R\Sheafhom(\mathcal{S},\mathcal{S})\otimes \omega_X)[-1]\oplus R\pi_*(R\Sheafhom(\mathcal{T},\mathcal{T})\otimes \omega_X)[-1]\oplus\\
				&R\pi_*(R\Sheafhom(\mathcal{S},\mathcal{T})\otimes \omega_X)[-1]\otimes \bchi \oplus R\pi_*(R\Sheafhom(\mathcal{T},\mathcal{S})\otimes \omega_X)[-1]\otimes \bchi^{\vee}\\
				&=\Ob(\mathcal{S})\oplus \Ob(\mathcal{T}) \oplus\\ 
				& R\pi_*(R\Sheafhom(\mathcal{S},\mathcal{T})\otimes \omega_X)[-1]\otimes \bchi \oplus R\pi_*(R\Sheafhom(\mathcal{T},\mathcal{S})\otimes \omega_X)[-1]\otimes \bchi^{\vee}
			\end{align*}
			Moreover, the decomposition of $\Ob(\mathcal{E}_{\mathcal{Z}_{k,\Lambda}})$ into diagonal and trace-free part induces a decomposition
			\[\Ob(\mathcal{S})\oplus \Ob(\mathcal{T})=\left(\Ob(\mathcal{S})\oplus \Ob(\mathcal{T})\right)_0\oplus \mathcal{O}_{\mathcal{Z}_{k,\Lambda }}\]
			So the $\C^*$-fixed part of $\Ob_0(\mathcal{E}_{\mathcal{Z}_{k,\Lambda}})$ is the summand 
			$\left(\Ob(\mathcal{S})\oplus \Ob(\mathcal{T})\right)_0$, 
			while the moving part is the summand
			\[ R\pi_*(R\Sheafhom(\mathcal{S},\mathcal{T})\otimes \omega_X)[-1]\otimes \bchi \oplus R\pi_*(R\Sheafhom(\mathcal{T},\mathcal{S})\otimes \omega_X)[-1]\otimes \bchi^{\vee}.\]
			\item For ease of notation, we write  
			\[N(B,A):=\operatorname{Coker}\left(\bigoplus_{i=0}^{N-1} \Hom(\globm{\mathcal{E}}/\mathcal{V}_A^{i+1},\mathcal{V}_B^{i}) \to \bigoplus_{i=0}^N \Hom(\globm{\mathcal{E}}/\mathcal{V}_A^i,\mathcal{V}_B^i)\right)^\vee,\]
			where each of $A,B$ may stand for either $S$ or $T$.
			Then we have an equivariant decomposition 
			\[\rho^*_{k,\Lambda}\Omega_{\operatorname{Flag}}= N(S,S)^{\vee}\oplus N(T,T)^{\vee}\oplus N(S,T)^{\vee}\otimes \bchi^{\vee}\oplus N(T,S)^{\vee}\otimes \bchi \]
			
			So the fixed part is 
			$N(S,S)^{\vee}\oplus N(T,T)^{\vee}$, 
			and the moving part is $N(S,T)^{\vee}\otimes \bchi^{\vee}\oplus N(T,S)^{\vee}\otimes \bchi$.  \addtocounter{enumi}{1}
			\item One may check from the definition that $\detL(\mathcal{E})$ has weight $1$. 
			In any case, we only need that it is equivariant of pure weight and conclude that
			\[\Hom(\detL(\mathcal{E}),\detL(\mathcal{E}))\]
			is $\C^*$-fixed.
		\end{enumerate}
	
		For any two sheaves $\mathcal{F},\mathcal{G}$ on $X\times U$, where $U$ is an algebraic stack, set 
	\[\mathfrak{N}(\mathcal{F},\mathcal{G}):= - [R\pi_*(R\Sheafhom(\mathcal{F},\mathcal{G})\otimes \omega_X)^{\vee}].\]
		Then we obtain the $K$-theory class of the virtual normal bundle along $\mathcal{Z}_{k,\Lambda}$ as the dual of the moving part of the obstruction theory:
		\begin{equation}\label{eq:virnormal2}
		N_{k,\Lambda}:=\mathfrak{N}(\mathcal{S}\otimes \bchi^{\vee},\mathcal{T})+ \mathfrak{N}(\mathcal{T},\mathcal{S}\otimes \mathbb{\bchi}^{\vee})+[N(S,T) \otimes \bchi]+[N(T,S) \otimes \bchi^{\vee}].
	\end{equation}
		
		For the fixed obstruction theory on $\mathcal{Z}_{k,\Lambda}$ we have
		\begin{proposition}\label{propsamevclass}
			The fixed point obstruction theory on $\mathcal{Z}_{k,\Lambda}$ induces the same fundamental class as the one induced from the pullback along the map in Lemma \ref{lemetalemap}, where we take the usual obstruction theory on $\mathcal{N}_{\widehat{c}-ke}^{\ell}$, and the tautological one on $\mathcal{F}$. 
		\end{proposition}
		We break this down into steps. 
		First, we note that there is an -- up to homotopy natural -- obstruction theory $\Ob_{\mathcal{Z}_{k,\Lambda}}$ on $\mathcal{Z}_{k,\Lambda}$ coming from its moduli theoretic description of Remark \ref{remequiv}. The construction of this obstruction theory is an easy modification of the construction of the obstruction theory on $\mathcal{W}$ and we will not repeat the steps. The relative obstruction theories and maps between them are exactly the fixed parts of the restrictions of the relative obstruction theories used in the construction of $\Ob_\mathcal{W}$. This shows:
		\begin{lemma}\label{lem:obcomp1}
			The fixed obstruction theory on $\mathcal{Z}_{k,\Lambda}$ agrees with $\Ob_{\mathcal{Z}_{k,\Lambda}}$ up to homotopy. 
		\end{lemma}
		
		We now consider $\mathcal{Z}_{k,\Lambda}'$ to be the moduli stack parametrizing the same data as $\mathcal{Z}_{k,\Lambda}$, but with the orientation being on $\mathcal{T}$ instead of on $\mathcal{S}\oplus \mathcal{T}$. Since the determinant of $\mathcal{O}_C(-1)$ is canonically trivial, is a natural isomorphism $\mathcal{Z}_{k,\Lambda}\to \mathcal{Z}_{k,\Lambda}'$. Given by replacing the orientation on $\mathcal{S}\oplus \mathcal{T}$ by the induced orientation on $\mathcal{T}$. Again, we construct an obstruction theory $\Ob_{\mathcal{Z}_{k,\Lambda}'}$ on $\mathcal{Z}_{k,\Lambda}'$ which is natural up to homotopy, and we denote its pullback to $\mathcal{Z}_{k,\Lambda}$ by $\Ob'_{\mathcal{Z}_{k,\Lambda}}$. 
		\begin{lemma}\label{lem:obcomp2}
			The obstruction theories $\Ob_{\mathcal{Z}_{k,\Lambda}}$ and $\Ob'_{\mathcal{Z}_{k,\Lambda}}$ induce the same virtual class.  
		\end{lemma}
		
		\begin{proof}
			We will describe exactly in what way the construction of the two obstruction theories differs and why they are still homotopic. Let $\widetilde{\mathcal{M}}_{S,T}$ be the moduli stack for pairs $(\mathcal{S},\mathcal{T})$ and denote by $\operatorname{or}:\mathcal{M}_{S,T}\to\widetilde{\mathcal{M}}_{S,T} $ the relative moduli stack of orientations on $\mathcal{S}\oplus \mathcal{T}$ (which is canonically isomorphic to the stack of orientations on $\mathcal{T}$). The obstruction theory on $\widetilde{\mathcal{M}}_{S,T}$ is given by $\Ob(\mathcal{S}_{\widetilde{\mathcal{M}}_{S,T}})\oplus \Ob(\mathcal{T}_{\widetilde{\mathcal{M}}_{S,T}})$. Similar to Proposition \ref{propperfobor} we want to construct an obstruction theory $\Ob_{\mathcal{M}_{S,T}}$ that fits into a commutative diagram
			
			\begin{equation*}
				\begin{tikzcd}
					\operatorname{or}^*\left(\Ob(\mathcal{S}_{\widetilde{\mathcal{M}}_{S,T}})\oplus \Ob(\mathcal{T}_{\widetilde{\mathcal{M}}_{S,T}})\right)  \ar[r]\ar[d]& \Ob_{\mathcal{M}_{S,T}}\ar[d] \\
					\operatorname{or}^*L_{\widetilde{\mathcal{M}}_{S,T}}\ar[r] &L_{\mathcal{M}_{S,T}}.
				\end{tikzcd} 
			\end{equation*} 
			From this data, we then construct the obstruction theory on $\mathcal{Z}_{k,\Lambda}$. In the proof of Proposition \ref{propperfobor}, one needs to choose a dotted map making the following diagram commute as well as a splitting of this map:
			
			\begin{equation*}
				\begin{tikzcd}
					{\det}^*\Ob(\mathcal{L})\ar[r,dotted]\ar[d]& \Ob(\mathcal{S}_{\widetilde{\mathcal{M}}_{S,T}})\oplus \Ob(\mathcal{T}_{\widetilde{\mathcal{M}}_{S,T}})\ar[d] \\
					{\det}^*L_{\Picstack_X}\ar[r] & L_{\widetilde{\mathcal{M}}_{S,T}}.
				\end{tikzcd}
			\end{equation*}
			Here is where the difference arises: When we consider orientations on $\mathcal{S}\oplus \mathcal{T}$, the dotted map is induced by $\alpha_0=(d,d):\mathcal{O}\to R\Sheafhom(\mathcal{S},\mathcal{S})\oplus R\Sheafhom(\mathcal{T},\mathcal{T})$ with splitting $\beta_0=(\operatorname{tr}/r,\operatorname{tr}/r)^t:R\Sheafhom(\mathcal{S},\mathcal{S})\oplus R\Sheafhom(\mathcal{T},\mathcal{T})\to \mathcal{O}$, where $d$ is the respective diagonal and $\operatorname{tr}$ the trace map. On the other hand, when we consider orientations on $\mathcal{T}$, the dotted map is given by $\alpha_1=(0,d)$ and its splitting by $(0,\operatorname{tr}/r)^ t$. Now defining $\alpha_t$ and $\beta_t$ for $t\in \C$ by interpolation, we obtain a family over $\aone$ of dotted maps together with a splitting (that $\beta_t$ defines a splitting follows easily from the formula $\operatorname{tr}\circ d=0$ for the diagonal and trace map associated to the rank zero sheaf $\mathcal{S}$). Then all constructions go through relative to $\aone$ and one obtains that $\Ob_{\mathcal{Z}_{k,\Lambda}}$ and $\Ob_{\mathcal{Z}'_{k,\Lambda}}$ arise as fibers of an obstruction theory over $\mathcal{Z}_{k,\Lambda}\times \aone$. Thus they define the same virtual classes. 
		\end{proof}
		
		Let $\widetilde{\mathcal{N}}_S$ denote the moduli space of pairs $(\mathcal{S},W_S^{\bullet})$ such that $\mathcal{S}$ is a coherent sheaf on $\Xhat$ that is \'etale-locally isomorphic to $\mathcal{O}_C(-1)^{\oplus k}$ and such that $(\mathcal{S},W_S^1)$ is simple. Then we have a commutative diagram
		
		\begin{equation*}
			\begin{tikzcd}
				\mathcal{Z}_{k,\Lambda}\ar[dr,"\epsilon"]\ar[d,"B"]& \, \\
				\mathcal{N}_S\times \mathcal{N}^{\ell}_{\chat-ke}\ar[r,"G\times \operatorname{id}"] & \mathfrak{F}_k\times \mathcal{N}^{\ell}_{\chat-ke}.
			\end{tikzcd}
		\end{equation*}
		The vertical map $B$ is a $\C^*$-bundle and the horizontal map $G$ is a $\mathbb{G}_m$-gerbe. On $\mathcal{N}_S$ we have an obstruction theory $\Ob_{\mathcal{N}_S}$ for unoriented sheaves with flag structure. By construction, the obstruction theory $\Ob'_{\mathcal{Z}_{k,\Lambda}}$ fits into the following map of distinguished triangles:
		
		\begin{equation*}
			\begin{tikzcd}
				B^*\left(\Ob_{\widetilde{\mathcal{N}}_S}\oplus \Ob_{\mathcal{N}^{\ell}_{\chat-ke}}\right)\ar[r]\ar[d]&\Ob'_{\mathcal{Z}_{k,\Lambda}}\ar[r] \ar[d] &\Omega_B\ar[r,"+1"]\ar[d] &\,\\
				B^*\left(L_{\widetilde{\mathcal{N}}_S}\oplus L_{\mathcal{N}^{\ell}_{\chat-ke}}\right)\ar[r] & L_{\mathcal{Z}_{k,\Lambda}}\ar[r] &\Omega_{B} \ar[r,"+1"] & . \,
			\end{tikzcd}
		\end{equation*}
		Since $\widetilde{\mathcal{N}}_S$ is a $\mathbb{G}_m$-gerbe over a Deligne--Mumford stack, we have that $h^{1}(L_{\widetilde{\mathcal{N}}_S})$ is a line bundle. It follows that $h^1(B^*(\Ob_{\widetilde{\mathcal{N}}_S}\oplus \Ob_{\mathcal{N}^{\ell}_{\chat-ke}}))$ is also a line bundle. Therefore the connecting map 
		\[h^0(\Omega_B)\to h^1(B^*(\Ob_{\widetilde{\mathcal{N}}_S}\oplus \Ob_{\mathcal{N}^{\ell}_{\chat-ke}}))\]
		is an isomorphism. It follows that the map 
		\[B^*\left(\Ob_{\widetilde{\mathcal{N}}_S}\oplus \Ob_{\mathcal{N}^{\ell}_{\chat-ke}}\right)\to \Ob'_{\mathcal{Z}_{k,\Lambda}}\]
		induces an isomorphism on cohomology sheaves in degrees $\leq 0$. 
		
		On the other hand, we have that $\widetilde{\mathcal{N}}_{\mathcal{S}}$ is unobstructed, hence it is smooth and $\Ob_{\widetilde{\mathcal{N}}_S}\to L_{\widetilde{\mathcal{N}}_S}$ is a quasi-isomorphism. Therefore the map $G^*:\Omega_{\mathfrak{F}_k}\to L_{\widetilde{\mathcal{N}}_S}$ factors through $\Ob_{\widetilde{\mathcal{N}}_S}$. Similar to before, we find that the map $G^*:\Omega_{\mathfrak{F}_k}\to \Ob_{\widetilde{\mathcal{N}}_S}$ induces isomorphisms in cohomology in degrees $\leq 0$, and that remains true for its pullback along $B$. Putting these things together, we have shown that we have a commutative diagram
		\begin{equation*}
			\begin{tikzcd}
				\epsilon^*\left(\Omega_{\mathfrak{F}_k}\oplus \Ob_{\mathcal{N}^{\ell}_{\chat-ke}}\right) \ar[r]\ar[d]& \Ob'_{\mathcal{Z}_{k,\Lambda}}\ar[d] \\
				\epsilon^*\left(\Omega_{\mathfrak{F}_k}\oplus L_{\mathcal{N}^{\ell}_{\chat-ke}}\right)\ar[r] & L_{\mathcal{Z}_{k,\Lambda}}
			\end{tikzcd}
		\end{equation*}
		and that the upper horizontal map is an isomorphism in cohomology in degrees $\leq 0$. Since both complexes are concentrated in non-positive degrees, it is therefore a quasi-isomorphism. We thus have shown:  
		\begin{lemma} \label{lem:obcomp3}
			The obstruction theory $\Ob'_{\mathcal{Z}_{k,\Lambda}}$ and the obstruction theory on $\mathfrak{F}_k\times \mathcal{N}^{\ell}_{\chat-ke}$ are compatible along $\epsilon$ with respect to the trivial relative obstruction theory.  
		\end{lemma}  
		
		By combining Lemmas \ref{lem:obcomp1}, \ref{lem:obcomp2} and \ref{lem:obcomp3} we obtain Proposition \ref{propsamevclass}. 
		
		\subsection{Obstruction theory on relative Quot-schemes}\label{subsec:obstrquot}
		We consider the situation of \S \ref{subsec:cohsys}. Recall that we have a diagram
		
		\begin{equation*}
			\begin{tikzcd}
				& \mathcal{Q}(\chat,j)\ar[dr,"q_2"]\ar[dl,"q_1"']& \\
				\mathcal{M}^1(\chat)& &\mathcal{M}^1(\chat-je) 
			\end{tikzcd}
		\end{equation*}
		and that $q_1$ is identified with the relative Quot-scheme  $\operatorname{Quot}(\mathcal{E},j\chexc)\to \mathcal{M}^1(\chat)$, while $q_2$ is identified with the relative Grassmann bundle $Gr(j,\Ext^1(\mathcal{O}_C(-1),\mathcal{F})\to \mathcal{M}^1(\chat-je)$. 
		By Proposition \ref{prop:atrelobquot}, the reduced Atiyah class provides a relative obstruction theory $(\Ob_{q_1},\ob_{q_1})$ on the relative Quot-scheme $q_1$. 
		We will construct a natural obstruction theory on $\mathcal{Q}(\chat,j)$ and show that it is compatible with the ones on $\mathcal{M}^1(\chat)$ and $\mathcal{M}^1(\chat-j e)$.
		
		\begin{proposition}\label{prop:obtheoryq}
			There is a natural obstruction theory $\Ob_{\mathcal{Q}}\to L_{\mathcal{Q}(\chat,j)}$, which has the following properties:
			
			\begin{enumerate}[label=\arabic*)]
				\item It is compatible with the obstruction theory on $\mathcal{M}^1(\chat)$ and the relative obstruction theory for the relative Quot-scheme $\mathcal{Q}(\chat,j)\to\mathcal{M}^1(\chat)$. More precisely, we have a natural map of distinguished triangles 
				\begin{equation}\label{eq:obstrdiagq}
					\begin{tikzcd}
						q_1^*\Ob_{\mathcal{M}^1(\chat)}\ar[r]\ar[d]&\Ob_{\mathcal{Q}}\ar[r]\ar[d] & \Ob_{q_1}\ar[r,"+1"]\ar[d]&\, \\
						q_1^*L_{\mathcal{M}^1(\chat)}\ar[r] & L_{\mathcal{Q}(\chat,j)}\ar[r]& L_{q_1}\ar[r,"+1"] & . \,
					\end{tikzcd}
				\end{equation} 
				\item It is compatible with the obstruction theory on $\mathcal{M}^1_{\chat-je}$, i.e. we have a commutative diagram 
				\begin{equation} \label{eq:obstrdiagq2}
					\begin{tikzcd}
						q_2^*\Ob_{\mathcal{M}^1(\chat-je)}\ar[r]\ar[d]& \Ob_{\mathcal{Q}}\ar[d] \\
						q_2^*L_{\mathcal{M}^1(\chat-je)}\ar[r] & L_{\mathcal{Q}} .
					\end{tikzcd}
				\end{equation}
			\end{enumerate} 
		\end{proposition}
		\begin{proof}
			Let $\mathcal{Q}:=\mathcal{Q}(\chat,j)$.
			We will construct $\ob_{Q}:\Ob_{\mathcal{Q}}\to L_{\mathcal{Q}} $ and show that it fits into the diagram \eqref{eq:obstrdiagq}. It then follows from Lemma \ref{lemobfromrelob} that $\Ob_{\mathcal{Q}}$ is indeed a perfect obstruction theory. To begin, we work on the product $\mathcal{Q}\times \Xhat$, where we have the universal exact sequence $\underline{\mathcal{E}}_{\mathcal{Q}}$ given by $0\to \mathcal{F}_{\mathcal{Q}}\to \mathcal{E}_{\mathcal{Q}}\to \mathcal{G}:=\mathcal{O}_C(-1)\otimes V_\mathcal{Q}\to 0$. By \ref{subsec:atclexseq}, there is a natural morphism \[\at_{\underline{\mathcal{E}}_{\mathcal{Q}}}:\frac{\mathcal{E}\otimes \mathcal{E}^{\vee}}{\mathcal{F}\otimes \mathcal{G}^{\vee}}[-1]\to L_{\mathcal{Q}}. \]
			Here we have a canonical choice of homotopy quotient (i.e. up to canonical isomorphism), as pointed out in \ref{subsec:atclexseq}.
			By Proposition \ref{prop:atclexseqcomp}, this map fits into a natural morphism of distinguished triangles
			\begin{equation}\label{eq:obstrdiagqprel}
				\begin{tikzcd}
					q_1^*\mathcal{E}\otimes \mathcal{E}^{\vee}[-1]\ar[r]\ar[d,"q_1^*\at_{\mathcal{E}}"]&\frac{\mathcal{E}\otimes \mathcal{E}^{\vee}}{\mathcal{F}\otimes \mathcal{G}^{\vee}}[-1]\ar[r]\ar[d] & \mathcal{F}\otimes \mathcal{G}^{\vee}\ar[r,"+1"]\ar[d,"\at'_{\mathcal{E}}"]&\, \\
					q_1^*L_{\mathcal{M}^1(\chat)\times \Xhat/\Xhat}\ar[r] & L_{Q\times \Xhat/\Xhat}\ar[r]& L_{q_1\times \Xhat}\ar[r,"+1"]&\,.
				\end{tikzcd}
			\end{equation}

			Since the composition $\mathcal{F}\otimes \mathcal{G}^{\vee}\to \mathcal{E}\otimes \mathcal{E}^{\vee}\xrightarrow{tr} \mathcal{O}$ is zero, we have a canonical factorization $\mathcal{F}\otimes \mathcal{G}^{\vee}\to (\mathcal{E}\otimes \mathcal{E}^{\vee})_0$. We also have the natural splitting
			
			\[\frac{\mathcal{E}\otimes \mathcal{E}^{\vee}}{\mathcal{F}\otimes \mathcal{G}^{\vee}} \simeq \frac{\left(\mathcal{E}\otimes \mathcal{E}^{\vee}\right)_0}{\mathcal{F}\otimes \mathcal{G}^{\vee}} \oplus \mathcal{O}.\]  In conclusion, the triangle in $\eqref{eq:obstrdiagqprel}$ splits off the diagonal part of $\mathcal{E}\otimes \mathcal{E}^{\vee}$ and we obtain a direct summand
			\[(\mathcal{E}\otimes \mathcal{E}^{\vee})_0[-1]\to  \frac{(\mathcal{E}\otimes \mathcal{E}^{\vee})_0}{\mathcal{F}\otimes \mathcal{G}^{\vee}}[-1]\to \mathcal{F}\otimes G^{\vee}\xrightarrow{+1}.\]

			We define $\ob_{\mathcal{Q}}$ as the natural map 
			
			\[R\pi_*\left(\frac{(\mathcal{E}\otimes \mathcal{E}^{\vee})_0}{\mathcal{F}\otimes \mathcal{G}^{\vee}}[-1]\otimes \omega_{\Xhat} \right)\oplus \operatorname{det}_{\mathcal{E}}^*\Omega_{\Pic}\to L_{\mathcal{Q}}.\]
			It follows that $\ob_{Q}$ fits into a map of exact triangles as in \eqref{eq:obstrdiagq}.
			
			Now we address the second point. By \ref{prop:atclexseqcomp2}, we have the following commutative diagram on $\mathcal{Q}\times \Xhat$:
			
			\begin{equation*}
				\begin{tikzcd}
					q_2^*\mathcal{F}\otimes \mathcal{F}^{\vee}[-1]\ar[r]\ar[d]& \frac{\mathcal{E}\otimes \mathcal{E}^{\vee}}{\mathcal{F}\otimes \mathcal{G}^{\vee}}[-1]\ar[d] \\
					q_2^*L_{\mathcal{M}^1(\chat-je)\times \Xhat/\Xhat}\ar[r] & L_{\mathcal{Q}\times \Xhat/\Xhat}.
				\end{tikzcd}
			\end{equation*}
			The upper horizontal map in this square respects the respective decompositions into trace-free and diagonal part. Moreover, we have a commutative square
			\begin{equation*}
				\begin{tikzcd}
					\det_{\mathcal{F}}^*\Omega_{\Pic}\ar[r]\ar[d]& \det_{\mathcal{E}}^*\Omega_{\Pic}\ar[d] \\
					L_{\mathcal{Q}}\ar[r] & L_{\mathcal{Q}} ,
				\end{tikzcd}
			\end{equation*}
			which comes from the canonical isomorphism $\det \mathcal{E}=\det \mathcal{F}\otimes \det(\mathcal{O}_C(-1)\otimes \mathcal{V})=\det \mathcal{F}\otimes \mathcal{O}(C)^{\otimes j}$. 
			Part 2) follows from these observations in view of the construction of the obstruction theories on $\mathcal{M}^1(\chat-je)$ and $\mathcal{Q}$ respectively. 
		\end{proof}
		
		From Proposition \ref{prop:obtheoryq}, we conclude that the virtual class of $\mathcal{Q}(\chat,j)$  can be related to the virtual classes of $\mathcal{M}^1_{\chat}$ and $\mathcal{M}^1_{\chat-j\chexc}$.
		\begin{corollary}
			\begin{enumerate}[label=\arabic*)]
				\item 
				We have $[\mathcal{Q}(\chat,j)]^{\vir}=(q_1)^!_{\Ob_{q_1}}[\mathcal{M}^1(\chat)]^{\vir}$, where $\Ob_{q_1}$ is the relative obstruction theory for Quot-schemes. 
				\item We have $[\mathcal{Q}(\chat,j)]^{\vir}=q_2^*[\mathcal{M}^1(\chat-j\chexc)]^{\vir}$, where $q_2^*$ is the usual flat pullback on Chow groups associated to $q_2$. 
			\end{enumerate}
		\end{corollary}
		\begin{proof}
			The first assertion follows immediately from the proposition. For the second assertion, we complete the commutative diagram \eqref{eq:obstrdiagq2} horizontally to a map of exact triangles. We obtain a map $\ob_{q_2}:\Ob_{q_2}\to L_{q_2}$. By Lemma \ref{lemobfromrelob} it is an obstruction theory. Since $q_2$ is a smooth representable map, $L_{q_2}$ is concentrated in degree zero. In order to prove the assertion, we need to show that $\Ob_{q_2}$ is perfect with trivial obstruction part. But by construction, we have that $\Ob_{q_2}\simeq R\pi_{\mathcal{M}^1_{p^*c},*}(\mathcal{K}\otimes \mathcal{E}^{\vee}\otimes \omega_{\Xhat})[-1]$ for $\mathcal{K}=\mathcal{O}_C(-1)\otimes \mathcal{V}$. Thus, for any map from a scheme $f:U\to \mathcal{Q}$, we have $f^*U\Ob_{q_2}\simeq R\Sheafhom_{\Xhat}(f^*\mathcal{E},f^*\mathcal{K})^{\vee}[-1]$. But by $1$-stability we have that $R\Sheafhom_{\Xhat}(\mathcal{E},\mathcal{O}_C(-1))=\Sheafext^1_{\Xhat}(\mathcal{E},\mathcal{O}_C(-1))[1]$ is the shift of a vector bundle, from which the result follows. 
		\end{proof}
		
		\section{Proofs of Theorems \ref{th:mainthwallcrossing} and \ref{th:compthm}}\label{sec:technicalresults}
		\subsection{Proof of Theorem \ref{th:mainthwallcrossing} (Wall-crossing formula for the blowup)} \label{subsec:wallcr}

		Let $\alpha$ be a Chow cohomology class on the space $\mathcal{M}^{0,1}=\mathcal{M}^{0,1}_{\chat}$ for some fixed admissible Chern character $\chat$. In this section, we determine an explicit form for the wall-crossing term 
		\[\delta_{\chat}(\alpha):=\int_{ \mathcal{M}^1}\alpha -\int_{ \mathcal{M}^0}\alpha.\]
		We first fix $\ell$ and apply the virtual localization formula to the master space relating the spaces $\mathcal{N}^{\ell}$ and $\mathcal{N}^{\ell+1}$ constructed in \S \ref{subsecmasterspace}, and then we sum up the resulting terms over different values of $\ell$.  
		By taking $\alpha = \Phi(\mathcal{E}^{0,1}_{\chat})$ in the resulting expression, we obtain Theorem \ref{th:mainthwallcrossing} for $m=0$, and one obtains the result for arbitrary $m$ by twisting with $\mathcal{O}_{\Xhat}(C)$.
		
		\paragraph{Virtual localization.}
		We recall the virtual localization formula by Graber and Pandharipande \cite{GrPa}. Suppose $Y$ is a  Deligne--Mumford stack with a $\C^*$-action and let $\Ob_Y$ be a $\C^*$-equivariant perfect obstruction theory on $Y$ which has a global resolution by locally free sheaves. Let $Y_i$ be the connected components of the $\C^*$-fixed stack of $Y$. We consider each component $Y_i$ with the perfect obstruction theory, given by the fixed part of $\Ob_Y$ at $Y_i$. Let $N_i^{\operatorname{vir}}$ denote the virtual normal bundle of $Y_i$ in $Y$. Then the virtual localization formula states
		\[[Y]^{\operatorname{vir}}=\iota_*\frac{[Y_i]^{\vir}}{\operatorname{Eu}(N_i^{\operatorname{vir}})},\]
		where $\operatorname{Eu}$ denotes the equivariant Euler class. 
		In our situation, we will take $Y=\mathcal{Z}$, the master space for the wall-crossing. By \S \ref{subsec:fix} and \S \ref{subsec:obstrfix}, this gives us:
		\begin{equation}\label{eq:virloc}
			[\mathcal{Z}]^{\vir} = \frac{[\mathcal{N}^{\ell+1}]^{\vir}}{-t}+\frac{[\mathcal{N}^{\ell}]^{\vir}}{t} +\sum_{k,\Lambda} \frac{[\mathcal{Z}_{k,\Lambda}]}{\Eu(N_{k,\Lambda})}
		\end{equation}

		\begin{remark}
			In \cite{GrPa}, it is assumed that $Y$ possesses a closed embedding into a smooth Deligne--Mumford stack. This hypothesis has since been removed by Chang--Kiem--Li \cite{ChKiLi}. (It would also be possible to check this hypothesis in our situation for the cost of making the setup somewhat more complicated.) The assumption on global resolutions holds in our situation due to the general result by Totaro \cite{Tota}. 
		\end{remark}
		
		As a first step, we reformulate \eqref{eq:virloc} in the form of a wall-crossing result for $\mathcal{N}^{\ell}$ and $\mathcal{N}^{\ell+1}$:
		Let $\beta\in A^{d}(\mathcal{N}^{\ell,\ell+1})$ be a covariant Chow class. It induces an equivariant class $\beta_{\mathcal{W}}$ on $\mathcal{W}$ and thus an equivariant class $\beta_{\mathcal{Z}}\in A^d_{\C^*}(\mathcal{Z})$ on the master space by first pulling back to $(\mathcal{W}\times \mathbb{P}^1)^s$ and then descending along the $T$-action \eqref{eqnTaction}. Consider the element of $A^T_*(\Spec \C)[1/t]$ obtained by pairing either side of \eqref{eq:virloc} with the equivariant class $t\beta_{\mathcal{Z}}$ and pushing forward to a point. By considering the left hand side, we see that the resulting class has vanishing constant $t$-coefficient. Thus, the same is true when we apply this operation to the right hand side, giving 
		\begin{equation}\label{eqwallcross1}
			\int_{\mathcal{N}^{\ell+1}} \beta\mid_{\mathcal{N}^{\ell+1}}  - \int_{\flagell} \beta\mid_{\mathcal{N}^{\ell}} = \sum_{k,\Lambda} \int_{\mathcal{Z}_{k,\Lambda}}  \underset{t=0}{\operatorname{Res}} \frac{\rho_{k,\Lambda}^*\beta_{\mathcal{W}}}{e(N_{k,\Lambda})}.
		\end{equation}
		Here $\operatorname{Res}_{t=0}$ denotes the operation of taking the coefficient of $t^{-1}$ of a Laurent series in $t$. 
		
		\paragraph{Flag bundle structure.}
		Let $T_{flag}$ be the relative tangent bundle of the map $\pi:\mathcal{N}\to \mathcal{M}^{0,1}$. We will also use $T_{flag}$ to denote various pullbacks to stacks mapping to $\mathcal{N}$.
		We are really interested in comparing integrals on $\mathcal{M}^0$ and $\mathcal{M}^1$ rather than on the spaces $\mathcal{N}^\ell$. For $\alpha\in A^*(\mathcal{M}^{0,1})$, they are related by the following formulas
		\[\int_{\mathcal{M}^0}\alpha= \frac{1}{N!}\int_{\flagzero{} } \pi^* \alpha   \cdot e(T_{flag}),\]
		\[\int_{\mathcal{M}^1}\alpha= \frac{1}{N!}\int_{\flagN } \pi^* \alpha   \cdot e(T_{flag}).\]
		This follows from the projection formula and the fact that the fibers of $\pi$, being full flag spaces in an $N$-dimensional vector space, have topological Euler characteristic $N!$.
		In this light, we define the wall crossing term for $\alpha$ from $\mathcal{N}^{\ell}$ and $\mathcal{N}^{\ell+1}$ as
		\[\delta^{\ell,\ell+1}_{\chat}(\alpha):=\frac{1}{N!}\left(\int_{\mathcal{N}^{\ell+1}} \pi^*\alpha\cdot e(T^{flag})  - \int_{\flagell} \pi^*\alpha\cdot e(T^{flag})  \right).\]
		With this definition we have
		\[\delta_{\chat}(\alpha)=\sum_{\ell=0}^{N-1} \delta_{\chat}^{\ell,\ell+1}(\alpha).\]
		
		Let $\alpha_{\mathcal{W}}$ denote the equivariant class induced on $\mathcal{W}$ by $\pi^*\alpha$ and consider the pullback of $T_{flag}$ to $\mathcal{W}$ as an equivariant bundle. By \eqref{eqwallcross1}, we have
		\begin{equation}\label{eqwallcross2}
			\delta^{\ell,\ell+1}_{\chat}(\alpha)=  \frac{1}{N!}\sum_{k,\Lambda} \int_{\mathcal{Z}_{k,\Lambda}}  \underset{t=0}{\operatorname{Res}} \frac{\rho_{k,\Lambda}^*\alpha_{\mathcal{W}} \cdot \Eu(T_{flag})}{\Eu(N_{k,\Lambda})}.
		\end{equation}
		
		We make the right hand side of equation \eqref{eqwallcross2} more explicit. 
		As in \ref{subsec:obstrfix}, for a fixed choice of $k$ and $\Lambda$ on $\mathcal{Z}_{k,\Lambda}$ we have the equivariant decomposition 
		\begin{align*}
			T_{flag}& = T_{flag}^{fix}\oplus T_{flag}^{mov}\\
			&= N(S,S)\oplus N(T,T)\oplus N(S,T)\otimes \bchi \oplus N(T,S)\otimes \bchi^{\vee}.
		\end{align*}
		
		Then as in \eqref{eq:virnormal2}, we have
		\[\Eu(N_{k,\Lambda})=\Eu(T_{flag}^{mov})\Eu(\mathfrak{N}(\mathcal{S}\otimes \bchi^{\vee},\mathcal{T}))\Eu(\mathfrak{N}(\mathcal{T},\mathcal{S}\otimes \bchi^{\vee})).\]
		To abbreviate, in what follows we set 
		\[\vmor{\mathcal{F}}{\mathcal{G}} = \vnor{\mathcal{F}}{\mathcal{G}}\oplus \vnor{\mathcal{G}}{\mathcal{F}}\]
		We can cancel $\Eu(T_{flag}^{mov})$ on the right hand side of \eqref{eqwallcross2}, and obtain
		\[\frac{1}{N!}\sum_{k,\Lambda}\int_{ \mathcal{Z}_{k,\Lambda}} e(N(S,S))\cdot e(N(T,T)) \cdot\underset{t=0}{\operatorname{Res}} \frac{\rho_{k,\Lambda}^*\alpha_{\mathcal{W}}}{\Eu(\vmor{\mathcal{S}\otimes \bchi^{\vee}}{\mathcal{T}}) }.\]
		
		\begin{convention}
		 Let $\alpha\in A^*(\mathcal{M}^{0,1})$ and let $\mathcal{E}$ be a family of $0,1$-semistable sheaves of Chern character $\chat$ on $\Xhat\times U$, where $U$ is an algebraic stack. We will write $\alpha(\mathcal{E})$ for the pullback of $\alpha$ along the map $U\to \mathcal{M}^{0,1}$ induced by $\mathcal{E}$. If $\mathcal{E}$ is of the form $\mathcal{F}\otimes e^t\oplus \mathcal{G}$, where $t$ is an equivariant parameter associated to a one-dimensional torus $T$, we take $\alpha(\mathcal{F}\otimes e^t\oplus \mathcal{G})\in A^*_T(U)=A^*(B\mathbb{G}_m\times U)$ to be the pullback along the composition $B\mathbb{G}_m\times U\to \mathcal{M}^{0,1}$, where the second map is defined by $\mathcal{F}\otimes \mathcal{L}^u\oplus \mathcal{G}$, where $\mathcal{L}^u$ is the universal line bundle on $B\mathbb{G}_m$. Since $A^*_T(U)=A^*(U)[t]$, for a rational number $a$, we take $\alpha(\mathcal{F}\otimes e^{at}\oplus \mathcal{G})$ as the expression obtained from $\alpha(\mathcal{F}\otimes e^t\oplus \mathcal{G})$ by replacing every occurence of $t$ by $at$.
		\end{convention}
		 
		\begin{lemma}\label{lemdeltasimp}
			\begin{enumerate}[label=\roman*)]
				\item For $k>1$, the term 
				\[\int_{ \mathcal{Z}_{k,\Lambda}} e(N(S,S))\cdot e(N(T,T)) \cdot\underset{t=0}{\operatorname{Res}} \frac{\rho_{k,\Lambda}^*\alpha_{\mathcal{W}}}{\Eu(\vmor{\mathcal{S}\otimes \bchi^{\vee}}{\mathcal{T}})}\]
				vanishes. 
				\item 
				For $k=1$, we have an equality
				\begin{gather*}
					\int_{ \mathcal{Z}_{k,\Lambda}} e(N(S,S))\cdot e(N(T,T)) \cdot\underset{t=0}{\operatorname{Res}} \frac{\rho_{k,\Lambda}^*\alpha_{\mathcal{W}}}{\Eu(\vmor{\mathcal{S}\otimes \bchi^{\vee}}{\mathcal{T}})} \\
					= d! \int_{ \mathcal{N}^{\ell}_{\chat-e}} e(T_{\mathcal{N}^{\ell}_{\chat-e} ,flag})  \cdot\underset{t=0}{\operatorname{Res}} \frac{\alpha(\mathcal{O}_C(-1)\otimes \mathfrak{t}^{\vee}\oplus \mathcal{E}_{\mathcal{N}^{\ell}_{\chat-e}})}{\Eu(\vmor{\mathcal{O}_C(-1)\otimes \mathfrak{t}^{\vee}}{\mathcal{E}_{\mathcal{N}^{\ell}_{\chat-e}}})}. 
				\end{gather*}
				
			\end{enumerate}
		\end{lemma}
		
		\begin{proof}

				Let $(\mathcal{S},W_S)$ and $(\mathcal{T},W_T)$ be the universal sheaves with flag structures on $\mathcal{Z}_{k,\Lambda}$ and consider the map $\epsilon:\mathcal{Z}_{k,\Lambda}\to \mathfrak{F}_k\times \mathcal{N}^{\ell}_{\chat-ke}$ of \autoref{lemetalemap}. Note that $(\mathcal{T},W_T)$ is pulled back from $\mathcal{N}^{\ell}_{\chat-ke}$, and therefore the same is true for $N(T,T)$.
				By construction, $\mathfrak{F}_k$ parametrizes pairs $(V,(F^{i})_{2\leq i\leq d})$, where $V\subset \globm{\mathcal{O}_C(-1)}$ is a $k$-dimensional subspace, and such that $(F^{i})_{0\leq i\leq d}$ is a full flag in $\globm{\mathcal{O}_C(-1)}\otimes V^{\vee}$, where we take $F^0=0$ and $F^1$ to be the image of the natural inclusion $\mathcal{O}_{\mathfrak{F}_k}\to \globm{\mathcal{O}_C(-1)}\otimes V^{\vee}$. Then $N(S,S)$ is isomorphic to the pullback along $\mathcal{Z}_{k,\Lambda}\to \mathfrak{F}_k\times \mathcal{N}^{\ell}_{\chat-ke}\to \mathfrak{F}_k$ of the vector bundle       
				\[N_F:=\Ker\left( \begin{tikzcd}
					\bigoplus_{i=0}^d \Hom(F^i,\globm{\mathcal{O}_C(-1)})\otimes V^{\vee}/F^i)\ar[d]\\ \bigoplus_{i=0}^{d-1}\Hom(F^i,\globm{\mathcal{O}_C(-1)}\otimes V^{\vee}/F^{i+1}) 
				\end{tikzcd}
				\right).\]
				We find that for an arbitrary cohomology class $\beta$ on $\mathcal{Z}_{k,\Lambda}$, we have
				\begin{gather*}
					\int_{ \mathcal{Z}_{k,\Lambda}} e(N(S,S))\cdot e(N(T,T))\cdot \beta \\= \int_{\mathfrak{F}_k\times \mathcal{N}^{\ell}_{\chat-ke}} \beta\cdot\epsilon^*(e(N_F)\cdot [\mathfrak{F}_k]\times e(N(T,T))\cdot [\mathcal{N}^{\ell}_{\chat-ke}]^{\vir}).
				\end{gather*}
				
				\begin{enumerate}[label=\roman*)]
					\item 	
				The bundle $N_F$ has rank $kd(kd-1)/2$, while $\mathfrak{F}_k$ has dimension $k(d-k)+(kd-1)(kd-2)/2 = \operatorname{rk} N_F+ 1-k$. Thus the integral vanishes when $k>1$. 
				
				\item For $k=1$, we have that $\mathfrak{F}_1$ is in fact the variety of full flags in $\globm{\mathcal{O}_C(-1)}$, and $N_F$ is isomorphic to the tangent bundle of $\mathfrak{F}_1$. We get that $e(N_F)\cdot [\mathfrak{F}_1]=d![pt]\in A_0(\mathfrak{F}_1)$. 
			 We have the following diagram in which the solid arrows commute 
				\begin{equation*}
					\begin{tikzcd}
						\mathcal{Z}_{1,\Lambda}\ar[r, "\operatorname{id}"]\ar[d]& \mathcal{Z}_{1,\Lambda}\ar[dr,"\epsilon"]\ar[d,"B"]& \, \\
						B\mathbb{G}_m\times \mathcal{Z}_{1,\Lambda}\ar[r,"\nu=\operatorname{id}\times \epsilon"]& B\mathbb{G}_m\times \mathfrak{F}_1 \times \mathcal{N}^{\ell}_{\chat-e}\ar[r, "G"]\ar[r, leftarrow, bend right, "s", dashed] & \mathfrak{F}_1\times \mathcal{N}^{\ell}_{\chat-e}.
					\end{tikzcd}
				\end{equation*}
				Here, the map $B$ is a $\C^*$-bundle and is given by forgetting $\varphi$, the map $G$ is the projection onto the last two factors and $s$ is the section of $B\mathbb{G}_m$ corresponding to the trivial line bundle. We consider the identity map on $\mathcal{Z}_{k,\Lambda}$ as equivariant using the diagram \eqref{eq:idequiv}. The induced map on quotient stacks is naturally identified with $\nu$.
				Let $\mathcal{L}^u$ denote the universal line bundle on $B\mathbb{G}_m$. Let $p_i$ denote projection onto the $i$-th factor of $B\mathbb{G}_m\times \mathfrak{F}_1 \times \mathcal{N}^{\ell}_{\chat-e}$. We set
				\[\alpha_{\mathcal{N}^\ell_{\chat-e}} :=\alpha(\mathcal{O}_C(-1)\otimes \mathcal{L}^u\oplus \mathcal{E}_{\mathcal{N}^{\ell}_{\chat-e}}).\]
				This is a cohomology class on $B\mathbb{G}_m\times \mathcal{N}^{\ell}_{\chat-e}$, and $p_{13}^*\alpha_{\mathcal{N}^{\ell}_{\chat-e}}$ corresponds to the equivariant class $\rho_{k,\Lambda}^*\alpha_{\mathcal{W}}$ on $\mathcal{Z}_{k,\Lambda}$.   Write $\alpha_{\mathcal{N}^{\ell}_{\chat}}(t)=\sum_{n\geq 0} t^n \gamma_n$. We want to compute $\nu^*p_{13}^*\alpha_{\mathcal{N}^{\ell}_{\chat-e}}$. First we note that the composition
				\[B\mathbb{G}_m\times \mathcal{Z}_{k,\Lambda}\to B\mathbb{G}_m\times \mathfrak{F}_1 \times \mathcal{N}^{\ell}_{\chat-e} \to B\mathbb{G}_m\]
				is the map corresponding to the line bundle $(\mathcal{L}^u)^{\vee}\otimes \mathcal{L}_S$, where $\mathcal{S}=\mathcal{O}_C(-1)\otimes \mathcal{L}_S$. Therefore $\nu^*p_1^*t=t+c_1(\mathcal{L}_S) = - t\times 1+ 1\times 1/\chi_1 c_1(\mathcal{L}(\mathcal{E}_{\mathcal{N}^{\ell}_{\chat-e}}))$. On the other hand, $(p_{23}\circ \nu)^*\gamma = 1\times\epsilon^*\gamma$ for any $\gamma\in A^*(\mathfrak{F}_1\times \mathcal{N}^{\ell}_{\chat-e})$. Since we allowed a degree $\chi_1$ base change on the $\C^*$ action, we have that $t$ on $B\mathbb{G}_m$ corresponds to the equivariant class $t/\chi_1$ on $\mathcal{Z}_{k,\Lambda}$ with the trivial action. We find 
				\[\rho_{\mathcal{Z}_{k,\Lambda}}^*\alpha_{\mathcal{W}} = \sum_{n\geq 0}\left(\frac{\nu^* p_3^*c_1(\detL(\mathcal{E}_{\mathcal{N}^{\ell}_{\chat-e}}))-t}{\chi_1}\right)^n\nu^* p_3^*\gamma_n. \]
			 	Let $\omega:=c_1(\detL(\mathcal{E}_{\mathcal{N}^{\ell}_{\chat-e}}))$. Then this is a pullback of the $\C^*$-equivariant class
				\[\alpha_{\mathcal{N}^{\ell}_{\chat-e}}(\frac{\omega-t}{\chi_1})=\sum_{n\geq 0}\left(\frac{c_1(\detL(\mathcal{E}_{\mathcal{N}^{\ell}_{\chat-e}}))-t}{\chi_1}\right)^n \gamma_n \]
				 on $\mathcal{N}^{\ell}_{\chat-e}$. So, we have
				
				\begin{gather*}
					\left(e(N(S,S))\cdot e(N(T,T))\cdot \rho_{\mathcal{Z}_{k,\Lambda}}^* \alpha_{\mathcal{W}}\right)\cap[\mathcal{Z}_{1,\Lambda}]^{\vir}\\
					= d! \, \epsilon^*\left( [pt] \times (e(N(T,T)\cdot \alpha_{\mathcal{N}^{\ell}_{\chat-e}}(\frac{\omega-t}{\chi_1})\cap[\mathcal{N}^{\ell}_{\chat-e}]^{\vir})\right).
				\end{gather*}
				
				Thus
				\begin{gather}
					\int_{ \mathcal{Z}_{k,\Lambda}} e(N(S,S))\cdot e(N(T,T)) \cdot\underset{t=0}{\operatorname{Res}} \frac{\alpha_{\mathcal{Z}_{k,\Lambda}}}{\Eu(\vmor{\mathcal{S}\otimes \bchi^{\vee}}{\mathcal{T} })} \\
					=d!\frac{1}{\chi_1} \int_{ \mathcal{N}^{\ell}_{\chat-e}} e(T_{flag,\chat-e}) \underset{t=0}{\operatorname{Res}}\frac{\alpha_{\mathcal{N}^{\ell}_{\chat-e}}((\omega-t)/\chi_1)}{\Eu(\vmor{\mathcal{O}_C(-1)\otimes e^{(\omega-t)/\chi_1}}{\mathcal{E}_{\mathcal{N}^{\ell}_{\chat-e}} })}\label{eqwallstepx}
				\end{gather}
				
				We can remove $\omega$ and $\chi_1$ from this expression by following the Method in Mochizuki, \cite[Proof of Theorem 7.2.4]{Moch}: One can check that the term in \eqref{eqwallstepx} over which the residue is naturally an expression in $(\omega-t)/\chi_1$ and can be expanded as 
				\[\sum_{n\in \Z} \gamma_n \left(\frac{\omega-t}{\chi_1}\right)^n,\]
				with $\gamma_n\in A^*(\mathcal{N}^{\ell}_{\chat-e})$.
				Moreover, one has that 
				\[\underset{t=0}{\operatorname{Res}}(t-\omega)^n=\begin{cases}
					1; n=-1 ,\\
					0; n\neq -1.
				\end{cases}\]
				Therefore we have 
				\[\underset{t=0}{\operatorname{Res}}\sum_{n\in \Z} \gamma_n \left(\frac{\omega-t}{\chi_1}\right)^n =\gamma_{-1} -\chi_1 = \chi_1\cdot \underset{t=0}{\operatorname{Res}} \sum_{n\in \Z} \gamma_n (-t)^n .\]
				It follows from this that \eqref{eqwallstepx} is equal to 
				\[ d!\frac{\chi_1}{\chi_1} \int_{ \mathcal{N}^{\ell}_{\chat-e}} e(T_{flag,\chat-e}) \underset{t=0}{\operatorname{Res}}\frac{\alpha_{\mathcal{N}^{\ell}_{\chat-e}}(-t)}{\Eu(\vmor{\mathcal{O}_C(-1)\otimes e^{-t}} {\mathcal{E}_{\mathcal{N}^{\ell}_{\chat-e}}} )},\]
				which finishes the proof. 
			\end{enumerate}	
		\end{proof}
	
	\paragraph{Summing up.}
		Using \autoref{lemdeltasimp}, we find that 
		\begin{align*}
			\delta_{\chat}^{\ell,\ell+1}(\alpha)&=\frac{d!}{N!}\sum_{\Lambda} \int_{ \mathcal{N}^{\ell}_{\chat-e}} e(T_{\mathcal{N}^{\ell}_{\chat-e}}^{flag})  \cdot\underset{t=0}{\operatorname{Res}} \frac{\alpha(\mathcal{O}_C(-1)\otimes \mathfrak{t}^{\vee}\oplus \mathcal{E}_{\mathcal{N}^{\ell}_{\chat-e}})}{\Eu(\vmor{\mathcal{O}_C(-1)\otimes \mathfrak{t}^{\vee}}{\mathcal{E}_{\mathcal{N}^{\ell}_{\chat-e}}})}
		\end{align*}
		
		In this sum, $\Lambda$ ranges over all subsets of $\{\ell,\cdots,N-1\}$ of cardinality $d$ with $\min \Lambda =\ell$. 
		To simplify the notation, we use the following abbreviation: For any $\chat$ and any Chow class $\alpha$ on $\mathcal{M}_{\chat}^{0,1}$, define a Chow class $\Omega\alpha$ on $\mathcal{M}_{\chat-e}^{0,1}$ by
		\[\Omega\alpha:= \underset{t=0}{\operatorname{Res}} \frac{\alpha(\mathcal{O}_C(-1)\otimes \mathfrak{t}^{\vee}\oplus \mathcal{E}_{\mathcal{N}^{\ell}_{\chat-e}})}{\Eu(\vmor{\mathcal{O}_C(-1)\otimes \mathfrak{t}^{\vee}}{\mathcal{E}_{\mathcal{N}^{\ell}_{\chat-e}}})}.\]
		 Summing over all values of $\ell\in \{0,\ldots, N-1\}$, we have	
		\begin{equation}\label{eq:totaltermsum}
			\delta_{\chat}(\alpha)= \sum_{\ell=0}^{N-1} \delta_{\chat}^{\ell,\ell+1}(\alpha) = \frac{d!}{N!} \sum_{\Lambda, \ell=\min(\Lambda)}\int_{\mathcal{N}^{\ell}_{\chat-\chexc}} \Omega\alpha(\mathcal{E}^{\ell}_{\chat-e}) e(T^{flag}_{\chat-e}).
		\end{equation}
		Here we now let $\Lambda$ range over all subsets of $\{0,1,\ldots, N-1\}$ of cardinality $d$.
		
		Observe that 
		\[\delta_{\chat-e}^{\ell,\ell+1}(\Omega\alpha) = \frac{1}{(N-d)!}\left(\int_{ \mathcal{N}^{\ell+1}_{\chat-e}}e(T^{flag}_{\chat-e}) \Omega\alpha - \int_{ \mathcal{N}^{\ell}_{\chat-e}} e(T^{flag})\Omega\alpha\right).\]
		Thus, each of the terms appearing in the sum on the right hand side of \eqref{eq:totaltermsum} may be written as:
		\[\int_{\mathcal{N}^{0}_{\chat-\chexc}} \Omega\alpha(\mathcal{E}^{0}_{\chat-e}) e(T^{flag}_{\chat-e}) + (N-d)!\sum_{\ell_1=0}^{\ell-1} \delta^{\ell_1,\ell_1+1}_{\chat-e}(\Omega\alpha).\]
		This gives
		\begin{align*}
			\delta_{\chat}(\alpha)&=\frac{d!}{N!}\sum_{\Lambda}\int_{\mathcal{N}^{0}_{\chat-\chexc}} \Omega\alpha(\mathcal{E}^{0}_{\chat-e}) e(T^{flag}_{\chat-e}) + \frac{d!(N-d)!}{N!}\sum_{\Lambda}\sum_{\ell_1=0}^{\min\Lambda-1} \delta^{\ell_1,\ell_1+1}_{\chat-e}(\Omega\alpha)\\
			&=\frac{d!}{N!}\sum_{\Lambda_1}\int_{\mathcal{N}^{0}_{\chat-\chexc}} \Omega\alpha(\mathcal{E}^{0}_{\chat-e}) e(T^{flag}_{\chat-e} )\\
			& \qquad  + \frac{d!(N-d)!}{N!}\frac{d!}{N-d!}\sum_{\Lambda_1}\sum_{\substack{\Lambda_2 \\ \ell_2=\min(\Lambda_2)}}\int_{\mathcal{N}^{\ell_2}_{\chat-2e}} e(T^{flag}_{\chat-2e}) \Omega^2\alpha\\
			&=\frac{d!}{N!}\sum_{\Lambda_1}\int_{\mathcal{N}^{0}_{\chat-\chexc}} \Omega\alpha(\mathcal{E}^{0}_{\chat-e}) e(T^{flag}_{\chat-e})+ \frac{d!d!}{N!}\sum_{\substack{\Lambda_1,\Lambda_2 \\ \ell_2 = \min(\Lambda_2)}} \int_{\mathcal{N}^{\ell_2}_{\chat-2e}} e(T^{flag}_{\chat-2e}) \Omega^2\alpha .
		\end{align*}
		Here, $\Lambda_2$ ranges over size $d$ subsets of $\{0,\ldots,N-1\}$  which are disjoint from $\Lambda_1$ and for which $\min(\Lambda_2)<\min(\Lambda_1)$.
		
		We may keep repeating this procedure for $j$ steps until, say, $jd>N$. This gives us
		
		\[\delta_{\chat}(\alpha)=\sum_{j=1} \frac{(d!)^j}{N!} \sum_{\Lambda_1,\ldots,\Lambda_j} \int_{ \mathcal{N}^0_{\chat-je}} \Omega^j\alpha(\mathcal{E}^{0}_{\chat-je}) e(T^{flag}_{\chat-je}).\]
		Here, each $\Lambda_j$ ranges over size $d$ subsets of $\{0,\ldots,N-1\}$  which are disjoint from $\Lambda_1,\ldots,\Lambda_{j-1}$ and for which $\min(\Lambda_j)<\min(\Lambda_{j-1})$.
		In other words, the sum involving $\Lambda_i$'s ranges over all ways to choose $j$ mutually disjoint sets of $d$ elements out of the set $\{0,\cdots,N-1\}$ under the condition $\min \Lambda_1>\min\Lambda_2>\cdots>\min\Lambda_j$. The number of ways to do this, is the multinomial coefficient
		\[\frac{N!}{({d!})^j(N-jd)!}\]
		divided by $j!$ (corresponding to the fact that we are ordering the sets by their minimal element). 
		Thus,
		\[\delta_{\chat}(\alpha) =\sum_{j=1}^{\infty}\frac{1}{j!(N-jd)!}\int_{ \mathcal{N}^0_{\chat-je}}\Omega^j\alpha \cdot e(T^{flag}_{\chat-je}) = \sum_{j=1}^{\infty}\frac{1}{j!}\int_{ \mathcal{M}^0_{\chat-je}}\Omega^j\alpha . \] 
		Writing out $\Omega^j\alpha$ explicitly gives us Theorem \ref{th:mainthwallcrossing} for $m=0$. The general case follows from twisting by $\mathcal{O}_{\Xhat}(C)$.
		
		\subsection{Proof of Theorem \ref{th:compthm} (Push-down formula)}\label{subsec:compthm}
		
		We consider the situation in \S \ref{subsec:cohsys}. There, we had the diagram 
		\begin{equation*}
					\begin{tikzcd}
						& \mathcal{Q}(\chat,j)\ar[dr,"q_2"]\ar[dl,"q_1"']& \\
						\mathcal{M}^1_{\chat}& &\mathcal{M}^1_{p^*c} ,
					\end{tikzcd}
				\end{equation*}
				where $\chat = p^*c+ j e$. Let also $\mathcal{E}$ denote the universal sheaf over $\mathcal{M}^1_{\chat}$ and $\mathcal{F}$ the universal sheaf over $\mathcal{M}^1_{p^*c}$.
		In \S \ref{subsec:obstrquot}, we constructed an obstruction theory $\Ob_{\mathcal{Q}}$ on $\mathcal{Q}(\chat,j)$ and relative obstruction theories $\Ob_{q_1}$ and $\Ob_{q_2}$ and showed that they are compatible. 
		We argue that we can reduce virtual integrals on $\mathcal{M}^1_{\chat}$ to integrals on $\mathcal{M}^1_{p^*c}$ in an explicit way. The main result is: 
		\begin{proposition}\label{prop:pullbackintegral}
			Let $\alpha$ be a Chow cohomology class on $\mathcal{M}^1_{\chat}$. Then we have an equality of virtual integrals 
			\[\int_{\mathcal{M}^1_{\chat}} \alpha = \int _{\mathcal{Q}(\chat,j)} q_1^*\alpha.\]
		\end{proposition}
		
		Once this is established, we are left with the problem of expressing integrals over $\mathcal{Q}(\chat,j)$ in terms of integrals over $\mathcal{M}^1_{\chat-je}$. By Proposition \ref{prop:cohsystem}, the map $q_2$ is identified with the relative Grassmann bundle $Gr(j,\Sheafext^1_{\Xhat}(\mathcal{F},\mathcal{O}_C(-1))$. The Chow cohomology ring $A^*(\mathcal{Q})$ is therefore generated by the Chern classes of $\mathcal{V}$ over the subring $A^*(\mathcal{M}^1_{p^*c})$ (with the inclusion given by pullback). An alternative set of generators is given by the Chern classes of $-\mathcal{V}$, and these two generating sets are explicitly related by $c(\mathcal{V}) c(-\mathcal{V}) = 1$. For a sequence of variables $x_1,x_2\ldots$ and $\lambda=(\lambda_1,\ldots,\lambda_d)$ a sequence of integers, set
		
		\[\Delta_{\lambda}(x_i) :=\Delta_{\lambda_1,\ldots,\lambda_d}(x_i):= \det (x_{\lambda_i+j-i})_{1\leq i,j\leq d},\]
		where we set $x_0=1$ and $x_i=0$ for $i<0$.
		For a $K$-theory class $W$ set $\Delta_{\lambda}(W):= \Delta_{\lambda}(c_i(W))$. Since in our case $\mathcal{V}$ has rank $j$, it follows that every product of Chern classes of $-\mathcal{V}$ can be written as a linear combination of classes of the form $\Delta_{\lambda}(-\mathcal{V})$, where $\lambda$ ranges over all partitions of size $j$ (cf. \cite[\S 14.5]{Fult}).
		 We then have the following identity (\cite[Example 14.2.2]{Fult}):
		\[\pi_*(\Delta_{\lambda_1,\ldots,\lambda_j}c(-\mathcal{V}))=\Delta_{\lambda_1-r+j,\ldots,\lambda_j-r+j}(c(-\Sheafext^1_{\Xhat}(\mathcal{F},\mathcal{O}_C(-1))).\]
		Thus, assuming that $q_1^*\alpha$ can be expressed in an explicit way as an polynomial in $c_i(\mathcal{V})$ with coefficients in $A^*(\mathcal{M}^1_{p^*c})$, one can explicitly determine a class $\overline{\alpha}$ on $\mathcal{M}^1_{p^*c}$ such that 
		\[\int_{ \mathcal{M}^1_{\chat}}\alpha = \int_{ \mathcal{M}^1}\overline{\alpha}.\]

		We now turn to the proof of Proposition \ref{prop:pullbackintegral}. The relative obstruction theory for the map $q_1$ is given by $\Ob_{q_1}=R(\pi_{\mathcal{M}})_*\left(R\Sheafhom(\mathcal{F},\mathcal{V}\otimes \mathcal{O}_C(-1))^{\vee}\otimes \omega_{\Xhat}\right)$. Since $\mathcal{F}$ is a family of sheaves of Chern character $\chat-je=p^*c$, one can use relative Grothendieck--Riemann--Roch to conclude that the rank of this perfect complex is zero. Therefore, the induced virtual pullback map $(q_1)^!_{\Ob_{q_1}}$ preserves dimension. 
		For a Chow cohomology class $\alpha$ on $ \mathcal{M}^1_{\chat}$, we have 
		\[\int_{\mathcal{Q}} q_1^*\alpha = \deg\left( (q_1)_*(q_1^*\alpha \cap[\mathcal{Q}]^{\vir}) \right)=\deg \left(\alpha\cap (q_1)_*[\mathcal{Q}]^{\vir}\right).\]
		By Proposition \ref{prop:virpullfunc}, the virtual class $[\mathcal{Q}]^{\vir}$ is equal to the virtual pullback of $[\mathcal{M}_{\chat}^1]^{\vir}$ along $q_1$. Therefore Proposition \ref{prop:pullbackintegral} follows from the following: 
		\begin{lemma}
			The map $(q_1)_*(q_1)^!_{\Ob_{q_1}}$ is the identity on $A_*(\mathcal{M}_{\chat}^1)$
		\end{lemma}
		
		\begin{proof}
			Since the stack $\mathcal{M}_{\chat}^1$ is a Deligne--Mumford stack, its rational Chow group is generated by the  cycle classes of integral closed substacks \cite[Theorem 2.1.12]{Kresch}. Let $Z\subset \mathcal{M}^1_{\chat}$ be an arbitrary integral closed substack. It is therefore enough to show that $[Z]=(q_1)_*(q_1)^!_{\Ob_{q_1}}[Z]$. Since $(q_1)^!_{\Ob_{q_1}}$ preserves degree, we have in any case 
			\[(q_1)_*(q_1)^!_{\Ob_{q_1}}[Z] =k[Z]\]
			for some integer $k$. To determine $k$, we may restrict to any dense open subset $U\subset Z$ and compute there. In particular, we may choose $U$ such that the object $R\Sheafhom_{\Xhat}(\mathcal{E}\mid_U, \mathcal{O}_C(-1))$ splits as the direct sum of its (shifted) cohomology sheaves $\Sheafext^i_{\Xhat}(\mathcal{E}\mid_U,\mathcal{O}_C(-1))$, and such that its cohomology sheaves are locally free. Note that $\Sheafext^i_{\Xhat}(\mathcal{E}\mid_U,\mathcal{O}_C(-1))$ vanishes for $i=2$, and therefore has rank ranks are $j+s$ for $i=0$, and $s$ for $i=1$, where $s$ is some nonnegative integer. 
			It follows from Proposition \ref{prop:qiscohsys} and Remark \ref{rem:cohsysbch} that the base change of $q_1$ to $U$ is isomorphic to the Grassmannian $Gr(j, \Sheafhom_{\Xhat}(\mathcal{E}_U, \mathcal{O}_C(-1))$ of relative dimension $s\cdot j$. In particular, $q_1\mid_U$ is smooth. For a smooth morphism, the virtual pullback defined by an obstruction theory is equal to the flat pullback followed by intersection with the top-Chern class of the obstruction bundle. Recall that we have the universal exact sequence 
			\[0\to q_2^*\mathcal{F}_U\to q_1^*\mathcal{E}_U\to \mathcal{O}_C(-1)\otimes \mathcal{V}_U\to 0.\]
			Then the obstruction bundle is $R\pi_{\mathcal{M},*}\left(\Sheafext^1( \mathcal{F}_U,\mathcal{O}_C(-1)\otimes \mathcal{V}_U)^{\vee}\otimes \omega_{\Xhat}\right)^{\vee}$. By relative Grothendieck--Riemann--Roch, its rank chern classes are equal to those of  \[\Sheafext^1_{\Xhat}(q_2^*\mathcal{F}_U,\mathcal{O}_C(-1))\otimes \mathcal{V}_U =\Sheafext^1_{\Xhat}(q_1^*\mathcal{E}_U,\mathcal{O}_C(-1))\otimes \mathcal{V}_U.\]
			Thus the rank of the obstruction bundle is $s\cdot j$, and we have for the top Chern class:
			\begin{multline*}	
			c_{sj}(q_1^*\Sheafext^1_{\Xhat}(\mathcal{E}_U,\mathcal{O}_C(-1))\otimes \mathcal{V}_U) =\\ c_j(\mathcal{V}_U)^{\cdot s}+\sum \left(\substack{\mbox{ terms in $c_i(\mathcal{V})$ }\\ \mbox{ of degree $< sj$ }}\right)\cdot \left(\substack{\mbox{ terms pulled }\\\mbox{ back along $q_1$ }}\right).
			\end{multline*}
	
			By the projection formula, and since $\mathcal{V}_U$ is the dual of the universal subsheaf $\mathcal{S}$ on the Grassmannian, we then have
			\[ (q_{1}|_U)_*c_{sj}(\pi_G^*\Sheafext^1_{\Xhat}(\mathcal{E}_U,\mathcal{O}_C(-1))\otimes \mathcal{V}_U) =(q_1\mid_U)_*(c_j(\mathcal{S}^{\vee})^{\cdot s})=1_{A_*(U)}. \]
			By the explicit description of $(q_1)^!_U$ as a flat pullback followed by intersection with the top Chern class of the obstruction bundle, the result follows.
		\end{proof}

		\section{Proofs of Theorems \ref{th:blowupdon} and \ref{thm:structhm}} \label{sec:MainRes}

		\subsection{Proof of Theorem \ref{th:blowupdon} (Blowup formula for Donaldson--Mochizuki invariants)}\label{subsec:donmoch}

		\paragraph{Integration of cohomology classes.}
		Let $\mathcal{M}$ be a proper Deligne-Mumford stack, $[Z]$ a Chow cycle on $\mathcal{M}$ (which for us will always be a virtual cycle) and $\mathcal{E}$ a coherent sheaf on $\mathcal{M}\times X$. Let also $\sigma_1,\ldots,\sigma_n\in H^*(X,\mathbb{Q})$ be  cohomology classes on $X$. In what follows, we would like to evaluate integrals of the form 
		\begin{equation}\label{eq:inteq}
			\int \pi_{\mathcal{M}*} (\ch_{i_1}^{j_1}(\mathcal{E})\pi_{X}^*\sigma_1)\cup\cdots \cup\pi_{\mathcal{M}*}(\ch_{i_n}^{j_n}(\mathcal{E})\pi_{X}^*\sigma_n)\cap [Z].
		\end{equation}
		Moreover, we would like the wall-crossing formula of Theorem \ref{th:mainthwallcrossing} to hold for such integrals when $Z$ is the virtual fundamental class. Since we worked on the level of Chow groups and Chow cohomology classes, it is not immediately clear why the formula would hold for rational cohomology groups. As remarked by Mochizuki \cite[Remark 7.1.3]{Moch}, this would follow if one had a good theory of cohomology groups on Deligne--Mumford stacks that has well-behaved cycle class maps from the Chow-groups and $\mathbb{G}_m$-localization. To circumvent this issue,  Mochizuki \cite[\S\S 7.1.1--7.1.3]{Moch} developed a formalism that allows one to make sense of expression \ref{eq:inteq} that agrees with the naive definition in case both are defined, and such that our wall-crossing result applies. Essentially, instead of trying to evaluate \eqref{eq:inteq} directly, he first considers the cycle on $\mathcal{M}\times X^n$ given by 
		\[ p_{12}^*\ch_{i_1}^{j_1}(\mathcal{E}) \cdots p_{1, n+1}^*\ch_{i_n}^{j_n}(\mathcal{E} p_{1}^*([Z]).\]
		Pushing forward to $X^n$ one obtains a Chow cycle on a smooth scheme, which yields an associated cycle in homology. One can then integrate the cohomology classes $p_i^* \sigma_i$ against the resulting homology class.
		  
		In what follows, we will work with integrals of the form \eqref{eq:inteq} (or that can be brought into this form in a straightforward way), and will always silently make use of Mochizuki's notation.   
		
		\paragraph{Definition of virtual Donaldson invariants.}

		For a cohomology class $\sigma$ on $X$ define 
		\[\mu(\sigma)=\mu_{\mathcal{E}}(\sigma):=(\pi_\mathcal{M})_*\left((c_2(\mathcal{E})-\frac{1}{4}c_1(\mathcal{E})^2)\cup \pi_{X}^*\sigma\right).\] 
		As outlined above, this is really a formal expression, but morally we have that if $\sigma\in H^i(X,\mathbb{Q})$, then $\mu(\sigma)\in H^i(\mathcal{M},\mathbb{Q})$. When $\mathcal{M}$ is a moduli space, we take $\mathcal{E}$ to be the universal sheaf on $\mathcal{M}\times X$ unless mentioned otherwise. 
		
		\begin{definition} \label{def:DonMochi}
			Suppose $c_1$ is chosen so that all semistable sheaves on $X$ with first Chern class $c_1$ are in fact stable. Let $L$ be any line bundle on $X$ with $c_1(L)=L$, and let $\mathcal{M}_X(2,L,k)\subset \mathcal{M}_X(2,c_1,k)$ denote the closed substack of sheaves whose determinant is isomorphic to $L$. 
			For any collection of cohomology classes $\sigma_{1}, \ldots,\sigma_{n}\in H^*(X,\mathbb{Q})$, we define the \emph{Donaldson--Mochizuki} or \emph{virtual Donaldson invariant} associated to $\sigma_1,\ldots,\sigma_n$ as
			\[D_{X,c_1}(\sigma_1\ldots\sigma_n):=\sum_{k}\int_{\mathcal{M}_X(2,L,k)} \mu(\sigma_1)\cdot \cdots \cdot \mu(\sigma_n).\]
			Here the argument of $D_{X,c_1}$ is understood to live in the graded ring $\operatorname{Sym}(H^*(X,\mathbb{Q}))$.
		\end{definition}
		
		\begin{remark}
			\begin{enumerate}[label=\roman*)]
				\item Each integral in the sum defining the invariant is independent of the choice of $L$. 
				\item One can allow more general invariants by allowing higher rank sheaves and insertions coming from higher Chern classes, and by working with non-fixed-determinant spaces and allowing classes coming from the Picard scheme in the integrand. Since we are interested in the concrete blowup formula for rank $2$ invariants, we do not pursue this here.
			\end{enumerate} 
		\end{remark}
		
		\paragraph{Proof of the blowup formula.}
		We first note that we can rewrite the expression defining $\mu$ as
		\[\mu(\sigma)=(\pi_\mathcal{M})_*\left((\frac{1}{4}c_1(\mathcal{E})^2- \ch_2(\mathcal{E}))\cap \pi_{X}^*\sigma\right).\]
		
		We use the following abbreviation: If $\sigma$ is a cohomology class on $X$ and $c$ a class on $U\times X$, their \emph{slant product} is
		\[c/\sigma:=(\pi_{\mathcal{M}})_*(c\cap \pi^*\sigma)\in H^*(U,\mathbb{Q}). \]
		
		\begin{lemma} \label{lem:mucrossing} We have the following equalities:
			\begin{enumerate}[label = \arabic*)]
				\item $ \mu_{\mathcal{E}'\oplus \mathcal{O}_C(-m)\otimes \mathfrak{t}^{\vee}}([C])=\mu_{\mathcal{E}'}([C])-t -\frac{1}{2}c_1(\mathcal{E}')/[pt] . $
				\item If $\alpha\in H^2(\Xhat,\mathbb{Q})$ with $(\alpha\cdot [C])=0$, then
				$ \mu_{\mathcal{E}'\oplus \mathcal{O}_{C}(-m)\otimes \mathfrak{t}^{\vee}}(\alpha) = \mu_{\mathcal{E}'}(\alpha). $
				\item 
				$\mu_{\mathcal{E}'\oplus \mathcal{O}_{C}(-m)\otimes \mathfrak{t}^{\vee}}([pt])=\mu_{\mathcal{E}'}([pt]).$ 
			\end{enumerate}
		\end{lemma}
		\begin{proof}
			We calculate
			\begin{equation*}
			\begin{split}
				& \frac{1}{4}c_1(\mathcal{E}'\oplus \mathcal{O}_C(-m)\otimes \mathfrak{t}^{\vee})^2- \ch_2(\mathcal{E}'\oplus \mathcal{O}_C(-m)\otimes \mathfrak{t}^{\vee})	\\
				& \qquad  =\frac{1}{4}c_1(\mathcal{E}')^2-\ch_2(\mathcal{E}')+\frac{1}{4}([C]^2+2c_1(\mathcal{E}')[C])+(m-\frac{1}{2})[pt]+t[C].
			\end{split}
			\end{equation*}
			Taking slant products gives the desired statements.
		\end{proof}
		
		\begin{proof}[Proof of Theorem \ref{th:blowupdon}]
			We will only address part \ref{th:blowupdon3} of the theorem, as it is the most interesting one, and the other cases are strictly easier.
			Without loss of generality, we may assume that $\sigma_1,\ldots,\sigma_n$ are homogeneous elements and that their degrees sum up to an even number $2d$, such that $d= 4k - c_1^2-3\chi(\mathcal{O}_X)$ for some integer $k$. Let $\alpha(\mathcal{E}):=\mu_{\mathcal{E}}(\sigma_1)\cdots\mu_{\mathcal{E}}(\sigma_n)$. Then
			\[D_{\Xhat,p^*c_1}(\sigma_1\cdots\sigma_n[C]^4) =
			\int_{\mathcal{M}_{\Xhat}(2,p^*c_1,k+1)}\alpha(\mathcal{E})\mu([C])^4.\]
			
			We may choose an integer $m\gg0$ such that $\mathcal{M}_{\Xhat}(2,p^*c_1,k+1)\simeq \mathcal{M}^{m}(p^*c)$ is isomorphic to the moduli space of $m$-stable sheaves for the appropriate Chern-character $c$ on $X$. We apply Theorem \ref{th:mainthwallcrossing} a number of $m$ times, to relate the integral over $\mathcal{M}^m(p^*c)$ to an integral over $\mathcal{M}^0(p^*c)$ and moduli spaces $\mathcal{M}^{m'}(p^*c-j\chexc_{m'+1})$ of lower dimension. By Lemma \ref{lem:mucrossing}, each wall-crossing term that appears will be of the form 
			\[\int_{\mathcal{M}^{m'}(p^*c-je_{m'+1})} \alpha\cdot \gamma_{m', j }\]
			for some $0\leq m'<m$ and some $j>0$, where $\gamma_{m',j}$ is some cohomology class. Note that we have  \[\deg \alpha = d = \vdim(\mathcal{M}^m)(p^*c)-4.\] 
			We also have 
			\[\vdim \mathcal{M}^{m'}(p^*c-je_{m'+1})=\vdim(\mathcal{M}^m(p^*c)) - (j^2+4j(m'+1/2)).\]
			This number is strictly smaller than $\vdim \mathcal{M}^m(p^*c)-4$ unless $j=1$ and $m'=0$. Therefore, the only nontrivial contribution to the wall-crossing term comes from $\mathcal{M}^0(p^*c-e)$. Using Lemma \ref{lem:mucrossing}, we can calculate the wall-crossing term explicitly.
			We get that the Donaldson invariant on the blowup is a sum of the following terms: 
			\begin{equation}\label{eq:muterm1}
				\int_{\mathcal{M}^0(p^*c)} \alpha \mu(C)^4=0; \text{ and} 
			\end{equation}
			\begin{equation}\label{eq:muterm2}
				\int_{\mathcal{M}^0(p^*c-\chexc)}\alpha\cdot (4\mu(C) -c_1(\mathcal{E})/[pt]+2\ch_2(\mathcal{E})/[C]).
			\end{equation}
			Here the vanishing of \eqref{eq:muterm1} first term follows from the projection formula in view of the fact that $\alpha$ and the universal sheaf on $\mathcal{M}^0(p^*c)\times \Xhat$ are pulled back via $p:\Xhat\to X$.
			It remains to compute \eqref{eq:muterm2}. We have the natural isomorphism $\mathcal{M}^0(p^*c-e)\simeq \mathcal{M}^1((p^*c-e)\cdot \ch(\mathcal{O}_{\Xhat}(C)))$ given by twisting the universal sheaf with $\mathcal{O}_{\Xhat}(C)$. Note that $(p^*c-e)\cdot \ch(\mathcal{O}_{\Xhat}(C)) = p^*c+e+[pt]$.
			Then we may rewrite \eqref{eq:muterm2} in terms of the universal sheaf on $\mathcal{M}^1(p^*c+e+[pt])$ as 
			\begin{gather}
				\int_{\mathcal{M}^1(p^*c+e+[pt])}\alpha\cdot (4\mu_{\mathcal{E}(-C)}([C])-c_1(\mathcal{E}(-C))/[pt]+2\ch_2(\mathcal{E}(-C))/[C])\nonumber\\
				=\int_{\mathcal{M}^1(p^*c+e+[pt])}\alpha\cdot (4\mu([C])+c_1(\mathcal{E})/[pt] + 2\ch_2(\mathcal{E})/[C]). \label{eq:muterm3}
			\end{gather}
		
		We now use the diagram 
		\begin{equation*}
			\begin{tikzcd}
				&\mathcal{Q}\ar[dr,"q_1"]\ar[dl,"q_2"']& \\
				\mathcal{M}^1(p^*c+e+[pt])&&\mathcal{M}^1(p^*c+[pt])
			\end{tikzcd}
		\end{equation*}
		of Theorem \ref{th:compthm}, where we write $\mathcal{Q}:=\mathcal{Q}(p^*c+e+[pt],1)$.  We now fix $\mathcal{E}$ to denote the universal sheaf over $\mathcal{M}^1(p^*c+e+[pt])$ and $\mathcal{F}$ to denote the one on $\mathcal{M}^1(p^*c+[pt])$.  On $\mathcal{Q}$, we have the exact sequence \[0\to q_2^*\mathcal{F}\to q_1^*\mathcal{E}\to \mathcal{O}_C(-1)\otimes \mathcal{L}\to 0,\] 
		where $\mathcal{L}$ is the universal subbundle of $q_2^*\Sheafext_{\pi_{\mathcal{M}}}^1(\mathcal{O}_C(-1),\mathcal{F})$, coming from the identification of $\mathcal{Q}$ with a relative Grassmannian over $\mathcal{M}^1(p^*c+[pt])$.
		By Theorem \ref{th:compthm}, we may evaluate a virtual integral over $\mathcal{M}^1(p^*c+e+[pt])$ after pulling the integrand back to $\mathcal{Q}$. We have 
		\[q_1^*\ch(\mathcal{E})=q_2^*\ch(\mathcal{F})+([C]-\frac{1}{2}[pt])e^{-\xi},\]
		where $\xi=-c_1(\mathcal{L})$, and therefore
		\begin{align*}
			q_1^*\mu_{\mathcal{E}}([C])&=q_2^*\mu_{\mathcal{F}}([C])-\frac{1}{2}q_2^*c_1(\mathcal{F})/[pt]-\xi;	\\
			q_1^*c_1(\mathcal{E})/[pt]&=q_2^*c_1(\mathcal{F})/[pt]; \text{ and} \\
			q_1^*\ch_2(\mathcal{E})/[C]&=q_2^*\ch_2(\mathcal{F})/[C]+\xi.
		\end{align*}

		Moreover, $\alpha(q_1^*\mathcal{E})=\alpha(q_2^*\mathcal{F})$. Therefore \eqref{eq:muterm3} is equal to 
		\[\int_{\mathcal{Q}}q_2^*\alpha(\mathcal{F}) \cdot[-2\xi+  q_2^*(4\mu_{\mathcal{F}}(C)- c_1(\mathcal{F})/[pt] +2\ch_2(\mathcal{F})/[C])].\]
		Due to the relative projective bundle structure of $q_2$, we have that \[(q_2)_*\left(\xi^k\cap[Q]^{\vir}\right)=c_{k-1}(\Ext(\mathcal{O}_C(-1), \mathcal{F})) \cap [\mathcal{M}^1(p^*c+[pt])]^{\vir} .\]
		So only the term involving $\xi$ survives, which gives:
		\begin{align*}
			-2\int_{ \mathcal{Q}}  q_2^*\alpha\cdot \xi&=-2 \int_{\mathcal{M}^1(p^*c+[pt])}\alpha = -2\int_{\mathcal{M}^0(p^*c+[pt])}\alpha =-2\int_{\mathcal{M}_X(2,c_1,k)}\alpha. \end{align*} 
		\end{proof}
		
		\subsection{Proof of Theorem \ref{thm:structhm} (Weak general structure theorem)}\label{subsec:structhm}
		\paragraph{Setup and examples.}
		As in the statement of the theorem, let $\Phi$ be a rule that given a smooth projective surface $X$, a finite type Artin stack $U$ and a $K$-theory class $\mathcal{E}$ (e.g. coming from a universal sheaf) on $X\times U$ assigns a Chow cohomology class $\Phi(\mathcal{E})$ on $U$ in a way compatible with pullback on $U$ and on $X$ and that depends on $\mathcal{E}$ only through the Chern classes of $\mathcal{E}$.

		We make the following assumption on $\Phi$:
		\begin{assumption}\label{assum:goodphi}
			
			\begin{enumerate}[label=\arabic*)]
			\item			For any choice of rank $r,\widehat{c}_1$ and $m$, there exist power series $P(\mathcal{E},t)$ in $t$ and terms of the form  
			\begin{enumerate}[label=\roman*)]
				\item slant products $\ch_i(\mathcal{E})/[pt]$; and \label{enum:slantpoint}
				\item slant products $\ch_i(\mathcal{E})/[C]$ \label{enum:slantC}
			\end{enumerate}
			such that for any finite type algebraic stack $U$ and any $U$-flat family $\mathcal{E}$ on $\Xhat\times U$ parametrizing sheaves on $\Xhat$ of rank $r$ and with first Chern class $\widehat{c}_1$, we have
	
				\[\Phi(\mathcal{E}\oplus \mathcal{O}_C(-m)\otimes \mathfrak{t}^{\vee})=\Phi(\mathcal{E})P(\mathcal{E},t),\]
				Moreover, $P$ only depends on $\Phi,m, (\widehat{c}_1,[C])$ and $r$.\label{assum:goodphi2}
			
			\item There exists a power series $Q$ in terms \ref{enum:slantpoint} and \ref{enum:slantC} depending only on $\Phi, (\widehat{c}_1,[C])$ and $r$ such that
			\[\Phi(\mathcal{E}(C))=\Phi(\mathcal{E})Q(\mathcal{E}).\] \label{assum:goodphi1}
			\item For any $s$, there exists a power series $R(\mathcal{E},(x_i)_{i=1}^s)$ in terms \ref{enum:slantpoint}, \ref{enum:slantC} and $(x_i)_{i=1}^s$, depending only on $\Phi,(\chat_1,[C]), r $ and $s$, such that for any $U$, $\mathcal{E}$ and any rank $s$ vector bundle $\mathcal{V}$ on $U$, we have 
			\[\Phi(\mathcal{E}\oplus \mathcal{O}_C(-1)\otimes \pi_U^* \mathcal{V})=\Phi(\mathcal{E})R(\mathcal{E},(c_i(\mathcal{V}))_{i=1}^s).\] \label{assum:goodphi3}
			\end{enumerate}
		\end{assumption}
		
		\begin{remark}\label{rem:rhomgood}
			The Chern character of the objects $R\Hom_{\Xhat}(\mathcal{E},\mathcal{O}_C(-m))$ and 
			$R\Hom_{\Xhat}(\mathcal{O}_C(-m),\mathcal{E})$ can be expressed using only terms of the form \ref{enum:slantpoint} and \ref{enum:slantC} given in Assumption \ref{assum:goodphi}. If $R(\mathcal{E})$ is a given polynomial in terms of the form \ref{enum:slantpoint} and \ref{enum:slantC}, then so is $R(\mathcal{E}\oplus\mathcal{O}_C(-m)\otimes \mathfrak{t}^{\vee})$.
		\end{remark}

		\begin{example} \label{ex:multclasstang}
			Suppose that $A$ is a multiplicative expression in Chern classes. This implies that for a vector bundle $F$, we have $A(F)=f(x_1)\cdots f(x_r)$, where $f$ is an invertible formal power series and the $x_i$ are the Chern roots of $F$. Then 
			\[\Phi(\mathcal{E})= A(-R\Hom_{X}(\mathcal{E},\mathcal{E}))\]
			satisfies Assumption \ref{assum:goodphi}.
			To see this, we observe that 
			\begin{gather*}
				\Phi(\mathcal{E}\oplus \mathcal{O}_C(-m)\otimes \mathfrak{t}^{\vee})=\\ \Phi(\mathcal{E})\cdot A(-R\Hom_{\Xhat}(\mathcal{E},\mathcal{O}_C(-m))\otimes \mathfrak{t})\cdot A(-R\Hom_{\Xhat}(\mathcal{O}_C(-m),\mathcal{E})\otimes \mathfrak{t}^{\vee})\cdot A(-\mathcal{O}_U),
			\end{gather*}
			and so the claimed statement follows from Remark \ref{rem:rhomgood}. 
		\end{example}

		We have the following special cases of Example \ref{ex:multclasstang}: 
		\begin{examples}\label{ex:insertions}
			\begin{enumerate} [label=\arabic*)]
				\item Let $A(E)=c(E)$ be the total Chern class. Denote  by $\mathcal{M}=\mathcal{M}_{X}(r,c_1,c_2)$ the moduli space of Gieseker semistable sheaves on $X$ with fixed determinant,  and let $\mathcal{E}$ be a universal sheaf. Then we have 
				\[\int_{[\mathcal{M}]^{\vir}}\Phi(\mathcal{E}) = \chi^{\vir}(\mathcal{M})\]
				the \emph{virtual Euler characteristic} of $\mathcal{M}$. 
				\item Let $A(E)=\operatorname{td}(E)$ be the Todd genus. Then for $\mathcal{M}$, $\mathcal{E}$ as in the first point, we have 
				\[\int_{[\mathcal{M}]^{\vir}}\Phi(\mathcal{E}) = \chi(\mathcal{O}^{\vir}_{\mathcal{M}}),\]
				the holomorphic Euler characteristic of the virtual structure sheaf. 
				\item As a common generalization, we take $A(E)=A_{-y}(E)$ to be the multiplicative class associated to the power series 
				\[f_y(x) = \frac{x(1-ye^{-x})}{1-e^{-x}},\]
				where $y$ is a formal variable. 
				We then have that $A_{1}(E)=c_{\operatorname{rk} E}(E)$ and $A_0(E)=\operatorname{td}(E)$. Then we have for $\Phi=\Phi_y$ that
				
				\[\int_{[\mathcal{M}]^{\vir}} \Phi_y(\mathcal{E})=\chi_y^{\vir}(\mathcal{M}),\]
				the virtual $\chi_y$-genus of $\mathcal{M}$ with its perfect obstruction theory. For more details, see \cite[\S 4]{FaGo}. \label{ex:insertions3}
			\end{enumerate}
		\end{examples}
		
		\paragraph{Proof of the structure theorem.}
		We begin by establishing a number of preliminary steps.

		We will use the following notation (more generally, one could replace $\Q$ by another field e.g. when working with the $\chi_y$-genus): 
		\begin{enumerate}[label=\arabic*)]
			\item For a power series $R\in \Lambda_{\Xhat} := \Q[[(\nu_i)_{i\geq 1},(\gamma_i)_{i\geq 1}]]$, and $\mathcal{E}$ a $K$-theory class on $\Xhat\times U$, let $R(\mathcal{E})$ denote the expression obtained by substituting $\nu_i$ with $\ch_i(\mathcal{E})/[pt]$ and $\gamma_i$ with $\ch_i(\mathcal{E})/[C]$.
			\item For $R\in \Lambda_X:=\Q[[(\nu_i)_{i\geq 1}]]$, and $\mathcal{E}$ a $K$-theory class on $X\times U$, let $R(\mathcal{E})$ denote the expression obtained by substituting $\nu_i$ with $\ch_i(\mathcal{E})/[pt]$.
	\end{enumerate}

		\begin{lemma}\label{lem:wallcrossseries}
			Let $\chat=p^*c-je$ be an admissible Chern class and $M\geq 1$. Let $Q\in \Lambda_{\Xhat}$. Then there exist power series $Q_{j_M}\in \Lambda_{\Xhat}$, which depend only on $Q, \Phi, j,M$ and $r$, and with $Q_0=Q$, such that
			\[\int_{\mathcal{M}^M(\chat)}\Phi(\mathcal{E})Q(\mathcal{E})=\sum_{j_M\geq 0}\int_{\mathcal{M}^{M-1}(\chat-j_M e_M) } \Phi(\mathcal{E})Q_{j_M}(\mathcal{E}).\]
		\end{lemma}
		\begin{proof}
			This follows by applying Theorem \ref{th:mainthwallcrossing} to 
			\[\int_{\mathcal{M}^M_{\chat}}\Phi(\mathcal{E})Q(\mathcal{E}),\]
			and by repeated application of Assumption \ref{assum:goodphi} and Remark \ref{rem:rhomgood} to each of the integrals on the moduli spaces $\mathcal{M}^{M-1}_{\chat-j_Me_M}$ appearing in Theorem \ref{th:mainthwallcrossing}.  
		\end{proof}
		
		\begin{lemma}\label{lem:powseriesreduce}
			Let $\chat=p^*c-j e$ be an  admissible Chern character on $\Xhat$ with $j\geq 0$ and let $Q\in \Lambda_{\Xhat}$. 
			\begin{enumerate}[label = \roman*)]
				\item Suppose that $j\geq r$.
				There are power series $P_{j_1}\in \Lambda_{\Xhat}$, such that 
				\[\int_{\mathcal{M}^0(\chat)}\Phi(\mathcal{E})Q(\mathcal{E})= \sum_{j_1\geq 0}\int_{\mathcal{M}^0 (p^*c -(j+j_1-r)e+j[pt])}\Phi(\mathcal{E})P_{j_1}(\mathcal{E}).\]
				The virtual dimension of the moduli spaces appearing in the sum on the right hand side is strictly decreasing in $j_1$ and equals $\vdim(\mathcal{M}^0_{\chat})$ when $j_1=0$.
				\label{lem:powseriesreduce1}
				\item Suppose that $0\leq j<r$. There are power series $Q_{j_1}\in \Lambda_{\Xhat}$, such that
				
				\[\int_{\mathcal{M}^0(\chat)}\Phi(\mathcal{E})Q(\mathcal{E})= \sum_{j_1\geq 0}\int_{\mathcal{M}^0 (p^*c-j_1 e+j[pt])}\Phi(\mathcal{E})Q_{j_1}(\mathcal{E}).\]
				The virtual dimension of the moduli spaces appearing in the sum on the right hand side is strictly decreasing in $j_1$ and strictly smaller than $\vdim(\mathcal{M}^0_{\chat})$ for all $j_1$.\label{lem:powseriesreduce2}
				
				\item Let $k$ be the least non negative integer such that $k r\geq j$. Then we have power series $R_{n}\in \Lambda_{\Xhat}$ in the same set of variables which only depend on $Q,\Phi, r$ and $j$, such that:
				\[\int_{\mathcal{M}^0(\chat)}\Phi(\mathcal{E})Q(\mathcal{E})= \sum_{n\geq kj-\binom{k}{2} r}\int_{\mathcal{M}^0({p^*c+n[pt]})}\Phi(\mathcal{E})R_{n}(\mathcal{E})\] \label{lem:powseriesreduce3}
			\end{enumerate}
			Moreover, in each part, the resulting set of series is given in an effective way and depends only on $\Phi, Q, r$ and $j$.
		\end{lemma}
		
		\begin{proof}
			Note that we have the isomorphism $\mathcal{M}_{\chat}^0\simeq \mathcal{M}^1_{p^*c-(j-r)e+j[pt]}$ given by $\mathcal{E}\mapsto \mathcal{E}(C)$. By Assumption \ref{assum:goodphi}, we have 
			
			\[\int_{\mathcal{M}^0_{\chat}}\Phi(\mathcal{E})Q(\mathcal{E})= \int_{\mathcal{M}^1_{p^*c-(j-r)e+j[pt]}} \Phi(\mathcal{E})\widetilde{Q}(\mathcal{E}).\]
			for some $\widetilde{Q}$.
			
			Then \ref{lem:powseriesreduce1} follows by applying Lemma \ref{lem:wallcrossseries} to the right hand side.
			
			In order to show \ref{lem:powseriesreduce2}, we use Theorem \ref{th:compthm}, which gives
			\begin{align*}
			\int_{\mathcal{M}^1_{p^*c-(j-r)e+j[pt]}} \Phi(\mathcal{E})\widetilde{Q}(\mathcal{E})&=\int_{\mathcal{Q}}\Phi(q_1^*\mathcal{E})\widetilde{Q}(q_1^*\mathcal{E})\\ &=\int_{\mathcal{Q}}\Phi(q_2^*\mathcal{F}\oplus \mathcal{O}_C(-1)\otimes \mathcal{V})\widetilde{Q}(q_2^*\mathcal{F}\oplus\mathcal{O}_C(-1)\otimes \mathcal{V}).
			\end{align*}
			Using Assumption \ref{assum:goodphi} \ref{assum:goodphi3}, this can be rewritten as an expression of the form 
			\[\int_{Q}q_2^*\Phi(\mathcal{F})\widetilde{\widetilde{Q}}(q_2^*\mathcal{F},c_i(\mathcal{V})).\]
			By Theorem \ref{th:compthm}, we can reduce this to an integral on
			${\mathcal{M}^1_{p^*c+j[pt]}}$. Then applying Lemma \ref{lem:wallcrossseries} gives \ref{lem:powseriesreduce2}.
			
			Finally we show \ref{lem:powseriesreduce3}. The series $R_n$ is obtained by successively applying either \ref{lem:powseriesreduce1} or \ref{lem:powseriesreduce2} to integrals over $\mathcal{M}^0_{p^*c-j'e+n'[pt]}$, where $j'>0$ starting with the integrals over those moduli spaces whose virtual dimension is highest, and grouping integrals over the same moduli spaces together after each step.  For each value of the virtual dimension, there will after each step only be a finite number of integrals over moduli spaces of that given dimension. Therefore this is a well-defined procedure and after finitely many steps, the only  integral corresponding to the virtual dimension of $\mathcal{M}^0_{p^*c+n[pt]}$ is over that moduli space itself, hence we may read off $R_n$.
		\end{proof}
		
		\begin{lemma}\label{lem:serieseliminate}
			Let $p^*c$ be an admissible Chern class. Let $Q$ be a power series in $\operatorname{ch}_i(\mathcal{E})/[C]$ and $\operatorname{ch}_i(\mathcal{E})/[pt]$. Then there exist power series $P_{n}\in \Q[[\nu_2,\ldots,\nu_r]]\subset\Lambda_X$, which depend only on $Q, \Phi$ and $r$, such that
			\[\int_{\mathcal{M}^0({p^*c})}\Phi(\mathcal{E})Q(\mathcal{E})=\sum_{n\geq 0}\int_{\mathcal{M}_X(c+n[pt])}\Phi(\mathcal{E})P_{n}(\mathcal{E}).\]
		\end{lemma}
		\begin{proof}
			Recall that we have 
			\[\ch_i(R\Hom_{\Xhat}(\mathcal{E},\mathcal{O}_C(-1)))=(-1)^{i+1}\ch_{i+1}(\mathcal{E})/[C], \]
			\[\ch_i(R\Hom_{\Xhat}(\mathcal{O}_C(-1),\mathcal{E}))= -\ch_{i+1}(\mathcal{E})/[C]-\ch_{i}(\mathcal{E})/[pt].\]
			From Lemma \ref{lem:wallcrossseries}, we have
			\[\int_{\mathcal{M}_{p^*c}^0}\Phi(\mathcal{E}){Q}(\mathcal{E})=\int_{\mathcal{M}_{p^*c}^1}\Phi(\mathcal{E}){Q}(\mathcal{E})-\sum_{j_1\geq 1}\int_{\mathcal{M}_{p^*c-j_1 e}^0}\Phi(\mathcal{E}){Q}_{j_1}(\mathcal{E}).\]
			On $\mathcal{M}^1_{p^*c}$, we have $R\Hom_{\Xhat}(\mathcal{O}_C(-1),\mathcal{E})=\Sheafext^1_{\Xhat}(\mathcal{O}_C(-1),\mathcal{E})[-1]$, which is the shift of a rank $r$ vector bundle. In particular, every class $\ch_i(\Sheafext^1_{\Xhat}(\mathcal{O}_C(-1),\mathcal{E}))$ for $i\geq 1$ can be canonically expressed in terms of just the classes for $1\leq i\leq r$. 
			Doing so gives
			\[ \ch_i(\mathcal{E})/[pt]= - \ch_{i+1}(\mathcal{E})/[C]+ \left(\substack{\mbox{a polynomial in $\ch_i(\mathcal{E})/[pt]$}\\\mbox{ and $\ch_{i+1}(\mathcal{E}/[C])$for $i>r$}}\right).\] 
			Let $\widetilde{Q}$ be the power series obtained from ${Q}$ by replacing every occurence of $\nu_i$ by the resulting expression only involving the $\gamma_i$ and $(\nu_i)_{i=1}^r$. 
			Then we have 
			\[\int_{\mathcal{M}_{p^*c}^1}\Phi(\mathcal{E}){Q}(\mathcal{E}) = \int_{\mathcal{M}_{p^*c}^1}\Phi(\mathcal{E})\widetilde{Q}(\mathcal{E})=\int_{\mathcal{M}_{p^*c}^0}\Phi(\mathcal{E})\widetilde{Q}(\mathcal{E})+\sum_{j_1\geq 1}\int_{\mathcal{M}^0_{p^*c-j_1e}}\widetilde{Q}_{j_1}(\mathcal{E}).\]
			
			However, on $\mathcal{M}^0_{p^*c}$, the object $R\Hom_{\Xhat}(\mathcal{O}_C(-1),\mathcal{E})$ vanishes identically, and therefore $ch_i(\mathcal{E})/[C]=0$ for all $i$. Note also, that since $\det(\mathcal{E})$ is pulled back from the Poincar\'e line bundle, we have $\ch_1(\mathcal{E})/[pt]=0$. Let $P_0$ be the series obtained by replacing all occurences of $\gamma_i, i\geq 1$ and $\nu_1$ in $\widetilde{Q}$ by zero. Then $P_0$ is a power series involving only $\nu_i$ for $2\leq i\leq r$ and we have
			\begin{align*}
				\int_{\mathcal{M}^0_{p^*c}}\Phi(\mathcal{E})Q(\mathcal{E}) &=\int_{\mathcal{M}^0_{p^*c}}\Phi(\mathcal{E})P_0(\mathcal{E})+\sum_{j_1\geq 1} \int_{\mathcal{M}^0_{p^*c-j_1e}}\Phi(\mathcal{E})(\widetilde{Q}_{j_1}-Q_{j_1})(\mathcal{E})  \\
				&=\int_{\mathcal{M}^0_{p^*c}}\Phi(\mathcal{E})P_0(\mathcal{E})+\sum_{n\geq 1} \int_{\mathcal{M}^0_{p^*c+n[pt]}} \Phi(\mathcal{E})Q_{1,n}(\mathcal{E}).
			\end{align*}
			Here, we invoke Lemma \ref{lem:powseriesreduce} \ref{lem:powseriesreduce3} to get the second equality, where $Q_{1,n}\in \Lambda_{\Xhat}$.
			By repeating this procedure with $Q_{1,n}$ in place of $Q$, we obtain $P_1$, and we can proceed in this way to determine all the $P_n$ inductively. 
		\end{proof}
		
		\begin{proof}[Proof of Theorem \ref{thm:structhm}]
			A repeated application of Lemma \ref{lem:wallcrossseries} shows the following: For a given positive integer $M$, let $\vec{j}=(j_1,\ldots,j_M)$ range over the tuples of nonnegative integers, and write $\vec{j}\cdot \vec{e}:=j_1e_1+\cdots+j_Me_M$. Then, there are natural power series $Q_{M,\vec{j}}$, such that 
			\[\int_{\mathcal{M}_{\chat}^M}\Phi(\mathcal{E})=\sum_{\vec{j}}\int_{\mathcal{M}^0_{\chat-\vec{j}\cdot \vec{e}}}\Phi(\mathcal{E})Q_{M,\vec{j}}(\mathcal{E}).\]
			We have for $M'>M$ and $\vec{j'}=(j_1,\ldots,j_M,0,\ldots,0)$ that $Q_{M',\vec{j'}}= Q_{M,\vec{j}}$. Note moreover, that the virtual dimension of $\mathcal{M}_{\chat-e_M}$ is strictly decreasing in $M$ and that the virtual dimension of $\mathcal{M}_{\chat-\vec{j}\cdot \vec{e}}$ is strictly decreasing in $\vec{j}$ (with respect to the partial ordering given by comparing each component). Therefore, in any given situation we can choose $M$ arbitrarily large and only finitely many terms will contribute for dimension reasons. We get that in fact
			\[\int_{\mathcal{M}_{\Xhat}(\chat)}\Phi(\mathcal{E})=\sum_{\vec{j}}\int_{\mathcal{M}^0_{\chat-\vec{j}\cdot\vec{e}}} \Phi(\mathcal{E})Q_{\vec{j}}, \]
			for some power series $Q_{\vec{j}}\in \Lambda_{\Xhat}$, where the sum now ranges over all infinite tuples of non negative integers with eventually vanishing entries. For each $\vec{j}$, we can rewrite $\vec{j}\cdot e=\abs{j}e-\sum_{i\geq1} j_i(i-1)[pt]$ and summarize the resulting integrals to get
			\[\int_{\mathcal{M}_{\Xhat}(\chat)}\Phi(\mathcal{E})=\sum_{j'\geq j,n\geq 0} \int_{\mathcal{M}^0_{p^*c-j'e}}\Phi(\mathcal{E})R_{j',n}(\mathcal{E}).\]
			Now apply Lemma \ref{lem:powseriesreduce} \ref{lem:powseriesreduce3} to each term on the right hand side, in descending order of the virtual dimension of the moduli spaces. Collecting all the terms gives
			\[\int_{\mathcal{M}_{\Xhat}(\chat)}\Phi(\mathcal{E})=\sum_{n\geq kj-\binom{k}{2}r}\int_{\mathcal{M}^0_{p^*c+n[pt]}}\Phi(\mathcal{E}) R_n(\mathcal{E}).\]
			Now apply Lemma \ref{lem:serieseliminate} to each term and collect integrals over identical moduli spaces, to get an expression
			\[\int_{\mathcal{M}_{\Xhat}(\chat)} \Phi(\mathcal{E}) = \sum_{n=0}^{\infty} \int_{\mathcal{M}_X(c+n[pt])} \Phi(\mathcal{E})P_n(\mathcal{E}),\]
			where now each $P_n$ is a power series in $\Lambda_X$ involving only the variables $\nu_2,\ldots,\nu_r$. 
		\end{proof}
		
		\begin{remark}
			In Theorem \ref{thm:structhm}, we can let the sum on the right hand side range only over $n\geq kj+\binom{k}{2}r$, where $k$ is chosen as in Lemma \ref{lem:powseriesreduce} \ref{lem:powseriesreduce3}.
		\end{remark}

		\appendix
	
		 \section{The Atiyah class on an algebraic stack}\label{app:atiyah}
		 
		 	We give a summary of the properties of the Atiyah class and related constructions that we use to construct perfect obstruction theories on the moduli stacks. Proofs of these statements will appear in \cite{Kuhn}. Throughout we consider stacks over an arbitrary base scheme $S$. We let $k$ denote a field.

		 	\subsection{The Atiyah class} \label{subsec:atatiyah}
		Let $f:\mathcal{X}\to \mathcal{Y}$ be a morphism of algebraic stacks and let $E\in D^-_{qcoh}(\mathcal{X})$. The {\it Atiyah class} of $E$ over $\mathcal{Y}$ is  a natural map 
		\[\at_E:E\to L_{\mathcal{X}/\mathcal{Y}}\otimes E[1].\]
		We also use the notation $\at_{E,\mathcal{X}/\mathcal{Y}}$ to emphasize the dependence on $f$. 
		If $E$ is dualizable in the derived category (equivalently, a perfect complex, see \cite{HaRy}), then the data of $\at_E$ is equivalent to that of a map
		\[\at'_E:E\otimes E^{\vee}[-1]\to L_{\mathcal{X}/\mathcal{Y}},\]
		which we also call the Atiyah class.
		
		The Atiyah class has the following fundamental properties: Let $E,F$ be objects of $D^-_{qcoh}(\mathcal{X})$.

		\paragraph{Functoriality.} Given a map $F\to E$ in $D^-_{qcoh}(\mathcal{X})$, the induced diagram 
			\begin{equation*}
				\begin{tikzcd}
					F\ar[r,"\at_F"]\ar[d]& L_{\mathcal{X}/\mathcal{Y}}\otimes F[1]\ar[d] \\
					E\ar[r,"\at_E"] & L_{\mathcal{X}/\mathcal{Y}}\otimes E [1]
				\end{tikzcd}
			\end{equation*}
			commutes.
		\paragraph{Pullback.} Given another morphism $f':\mathcal{X}'\to \mathcal{Y}'$ together with maps $A:\mathcal{X}'\to \mathcal{X}$ and $B:\mathcal{Y}'\to \mathcal{Y}$, and a $2$-isomorphism $B\circ f'\twoheadrightarrow f\circ A$, the induced diagram 
			
			\begin{equation*}
				\begin{tikzcd}
					LA^*E\ar[r,"LA^*\at_E"]\ar[d]& LA^*L_{\mathcal{X}/\mathcal{Y}}\otimes LA^*E[1]\ar[d] \\
					LA^*E\ar[r,"\at_{LA^*E}"] & L_{\mathcal{X}'/\mathcal{Y}'}\otimes LA^*E[1].
				\end{tikzcd}
			\end{equation*}
			commutes. If $E$ is perfect, then equivalently, the diagram
			\begin{equation*}
				\begin{tikzcd}
					E\otimes E^{\vee}[-1]\ar[r,"LA^*\at'_E"]\ar[dr,"\at'_{LA^*E}"']& LA^*L_{\mathcal{X}/\mathcal{Y}}\ar[d] \\
					&L_{\mathcal{X}'/\mathcal{Y}'}
				\end{tikzcd}
			\end{equation*} 
		commutes.
		\paragraph{Tensor products.} Identify $E\otimes L_{\mathcal{X}/\mathcal{Y}}[1]\otimes F\simeq L_{\mathcal{X}/\mathcal{Y}}[1]\otimes E\otimes F$ using the standard symmetry isomorphism of the derived tensor product. Then, up to this identification, we have an equality
		\[\at_{E\otimes F} = \at_E\otimes F +E\otimes {\at_F}.\]
		As a special case of this, if $E$ and $F$ are perfect and $\at_F$ is trivial (e.g. if $F$ is pulled back from $\mathcal{Y}$), then the following diagram commutes:
		\begin{equation*}
			\begin{tikzcd}
				E\otimes E^{\vee}[-1]\ar[r,"\at'_E"]\ar[d]& L_{\mathcal{X}/\mathcal{Y}}\ar[d,equals] \\
				E\otimes F\otimes E^{\vee}\otimes F^{\vee}[-1]\ar[r,"\at'_{E\otimes F}"] & L_{\mathcal{X}/\mathcal{Y}}.
			\end{tikzcd}
		\end{equation*} 
	Here the left vertical map is induced by the diagonal map $\mathcal{O}_X\to F\otimes F^{\vee}$ and the symmetry isomorphisms of the tensor product.
		\paragraph{Determinants.}
		Suppose that $E$ is perfect and consider the natural trace map $\operatorname{tr}:\Hom(E,L_{\mathcal{X}/\mathcal{Y}}[1]\otimes E)\to \Hom(\mathcal{O}_{\mathcal{X}}, L_{\mathcal{X}/\mathcal{Y}}[1])$. Then we have $\at_{\det(E)} = \operatorname{tr}(\at_E)\otimes \det(E)$ as morphisms $\det E\to L_{\mathcal{X}/\mathcal{Y}}\otimes \det E[1]$, at least when $E$ has a global resolution by locally free sheaves. In particular, if this condition holds, the following diagram commutes
		\begin{equation*}
			\begin{tikzcd}			\mathcal{O}_{\mathcal{X}}[-1]\ar[dr,"\at'_{\det E}"]\ar[d]& \\
				E\otimes E^{\vee}[-1]\ar[r,"\at'_E"] & L_{\mathcal{X}/\mathcal{Y}},
			\end{tikzcd}
		\end{equation*}
		where the left horizontal map is induced by the natural diagonal map $\mathcal{O}_{\mathcal{X}}\to E\otimes E^{\vee}$. 
		\paragraph{Pushforward.}
	
		Consider a commutative diagram
		\begin{equation*}
			\begin{tikzcd}
				\mathcal{X}'\ar[r]\ar[d,"p"]& \mathcal{Y}'\ar[d] \\
				\mathcal{X}\ar[r] & \mathcal{Y}
			\end{tikzcd}
		\end{equation*}
		and suppose that $p: \mathcal{X}'\to \mathcal{X}$ is {\it concentrated} in the sense of \cite[Definition 2.4]{HaRy} and has finite cohomological dimension. Let $E\in D^-_{qcoh}(\mathcal{X}')$.  Then we have a commutative diagram
		\begin{equation*}
			\begin{tikzcd}
				Rp_*E\ar[r,"\at_{Rp_*E}"]\ar[d]& L_{\mathcal{X}/\mathcal{Y}}[1]\otimes p_*E\ar[d] \\
				Rp_*E\ar[r,"Rp_*(\at_E)"] & Rp_*(L_{\mathcal{X}' /\mathcal{Y}'}[1]\otimes E).
			\end{tikzcd}
		\end{equation*} 
		Here the right vertical map is the natural one coming from the push-pull adjunction and the projection formula.
		Suppose that moreover the diagram is Cartesian and that the morphisms $\mathcal{X}\to \mathcal{Y}$ and $\mathcal{Y}'\to \mathcal{Y}$ are Tor-independent, and that $E$ is a perfect complex. Then we have a natural identification $p_*(\at_E)=\at_{p_*E}$ as morphisms $p_*E\to L_{\mathcal{X}/\mathcal{Y}}[1]\otimes p_*E$. If moreover $p_*E$ is perfect, this can be restated as commutativity of the following diagram 
		\begin{equation*}
			\begin{tikzcd}
				Rp_*E\otimes (Rp_*E)^{\vee}[-1]\ar[d]\ar[dr]&  \\
				R\Sheafhom_{\mathcal{X}'/\mathcal{X}}(\mathcal{E},\mathcal{E})^{\vee}[-1]\ar[r] &L_{\mathcal{X}/\mathcal{Y}}.
			\end{tikzcd}
		\end{equation*}

		\subsection{The reduced Atiyah class}
		Let $f:\mathcal{X}\to \mathcal{Y}$ be a morphism of algebraic stacks over a stack $Z$, and let now $E$ be a bounded above complex with quasi-coherent cohomology sheaves on $\mathcal{Y}$. We assume that all components are Tor-independent with respect to $f$ (for example if $E$ is a complex of flat $\mathcal{O}_{\mathcal{Y}}$-modules). Let $E_{\mathcal{X}}:=f^*E$, and let 
		\[0\to F\to E_{\mathcal{X}}\to G\to 0\]
		be an exact sequence of complexes with quasi-coherent cohomology sheaves on $\mathcal{X}$.
		
		Then the \emph{reduced Atiyah class} is a natural map
		\[\overline{\at}_{E,\mathcal{X}/\mathcal{Y},G}:F\to L_{\mathcal{X}/\mathcal{Y}}\otimes G.\]
		If $G$ is dualizable, we can rewrite this as a map
		\[\overline{\at}_{E,\mathcal{X}/\mathcal{Y},G}':F\otimes G^{\vee}\to L_{\mathcal{X}/\mathcal{Y}}.\]
		\begin{proposition}
			Assume that $E, F$ and $G$ are dualizable. 
			We have the following compatibility between the reduced Atiyah class and the ordinary Atiyah class of $E_{\mathcal{X}}$: The diagram
			\begin{equation*}
				\begin{tikzcd}
					F\otimes G^{\vee}\ar[r]\ar[d]& f^*E\otimes f^*E^{\vee}\ar[d] \\
					L_{\mathcal{X}/\mathcal{Y}}\ar[r] & f^*L_{\mathcal{Y}/\mathcal{Z}}[1]
				\end{tikzcd}
			\end{equation*}
			anticommutes. 
		\end{proposition}

		\subsection{Atiyah class of an exact sequence}\label{subsec:atclexseq}
		Let $\mathcal{X}\to \mathcal{Z}$ be a morphism of algebraic stacks, and let 
		\[\underline{E}:= 0\to F\to E\to G\to 0\]
		be an exact sequence of bounded above complexes of $\mathcal{O}_{\mathcal{Z}}$ modules. Assume that the images of $F,E,G$ in the derived category of $\mathcal{X}$ are perfect complexes. 
		 Then,  there is a canonical way to complete the map $F\otimes G^{\vee}\to E\otimes E^{\vee}$ to a triangle 
		 \[F\otimes G^{\vee}\to E\otimes E^{\vee}\to E\otimes E^{\vee}/F\otimes G^{\vee}\xrightarrow{+1},\]
		 in particular the object $E\otimes E^{\vee}/F\otimes G^{\vee}$ is canonically defined up to unique isomorphism in $D(\mathcal{X})$. If $E$ has nonzero rank, this triangle respects the decomposition into diagonal and trace-free part, i.e., it has as a canonical direct summand the triangle
		 \[0\to \mathcal{O}_{\mathcal{X}}\to \mathcal{O}_{\mathcal{X}}\xrightarrow{+1}.\]
		 We also have natural maps 
		 \[F\otimes F^{\vee}\to \frac{E\otimes E^{\vee}}{F\otimes G^{\vee}}\,\mbox{ and }\, G\otimes G^{\vee}\to \frac{E\otimes E^{\vee}}{F\otimes G^{\vee}}.\]

 		There is a natural mophism
		 \[\at_{\underline{E}}:\frac{E\otimes E^{\vee}}{F\otimes G^{\vee}}[-1]\to L_{\mathcal{X}/\mathcal{Z}},\]
		 which we call the \emph{Atiyah class of the exact sequence $\underline{E}$}.
		 
		  \begin{proposition}\label{prop:atclexseqcomp2}
		  We have the following commutative diagram on $\mathcal{X}$: 
		     \begin{equation*}
		 		\begin{tikzcd}
		 			F\otimes F^{\vee}[-1] \ar[r]\ar[d]&\frac{E\otimes E^{\vee}}{F\otimes G^{\vee}}[-1]\ar[d] \\	L_{\mathcal{X}/\mathcal{Z}}\ar[r,equals]&L_{\mathcal{X}/\mathcal{Z}}
    	 		\end{tikzcd}
		 	\end{equation*}
		     
		 \end{proposition}
		 
		 \begin{proposition}\label{prop:atclexseqcomp}
		 	Let $\mathcal{X}\xrightarrow{f}\mathcal{Y} \to \mathcal{Z}$ be maps of algebraic stacks with $f$ flat, and let $E$ be a bounded above complex of $\mathcal{O}_{\mathcal{X}}$-modules with quasi-coherent cohomology. Let $E_{\mathcal{X}}:=f^*E$ and suppose we are given an exact sequence $\underline{E}_{\mathcal{X}}$ of the form 
		 	\[0\to F\to E_{\mathcal{X}}\to G\to 0.\]
		 	Then we have a natural morphism of distinguished triangles
		 	\begin{equation*}
		 		\begin{tikzcd}
		 			E_{\mathcal{X}}\otimes E_{\mathcal{X}}^{\vee}[-1] \ar[r]\ar[d]&\frac{E_{\mathcal{X}}\otimes E_{\mathcal{X}}^{\vee}}{F\otimes G^{\vee}}[-1]\ar[r]\ar[d]& F\otimes G^{\vee}\ar[d]\ar[r, "{+1}"]&\, \\
		 	    f^*L_{{\mathcal{Y}}/\mathcal{Z}}\ar[r]&L_{\mathcal{X}/\mathcal{Z}}\ar[r] & L_{\mathcal{X}/\mathcal{Y}}\ar[r,"{+1}"]&\,.
    	 		\end{tikzcd}
		 	\end{equation*}
		 	Here, the first vertical map is the pullback of the Atiyah class of $E$, and the last vertical map is the reduced Atiyah class of $E$ with respect to the quotient $E_{\mathcal{X}}\to G$. The connecting map in the upper row is taken to be \emph{minus} the natural morphism $F\otimes G^{\vee}\to E_{\mathcal{X}}\otimes E_{\mathcal{X}}^{\vee}[-1]$.
		 	\end{proposition}

	 	\subsection{Deformation theoretic properties}\label{subsec:atdef}
	 	We demonstrate two important examples of how the Atiyah class can be used to construct obstruction theories. 	We let $V$ be a smooth projective variety of dimension $d$ over a base field $k$ with dualizing object $\omega_V$ (concentrated in degree $-d$).
	 	\paragraph{Obstruction theory on moduli spaces of sheaves.}
	  Let $\mathcal{X}$ be a stack over $\Spec k$, and let $E\in D^-_{qcoh}(V\times \mathcal{X})$ be perfect. Consider the Atiyah class map of $E$ relative to $V$: 
	 	\[\at_{E}':E\otimes E^{\vee}[-1]\to L_{V\times \mathcal{X}/V}.\]
	 	
	 	We have $L_{V\times \mathcal{X}/V}\simeq \pi_{\mathcal{X}}^*L_{\mathcal{X}}$. We consider the composition
	 	\[R\pi_{\mathcal{X}*}(E\otimes E^{\vee}\otimes \omega_V)[-1] \xrightarrow{R\pi_{\mathcal{X}*}(\at'_E\otimes \omega_V)}\to L_{\mathcal{X}}\otimes R\pi_{\mathcal{X}_*}(\omega_V)\to L_{\mathcal{X}},\]
	 		where we used the projection formula for the first map and the second map is induced from the trace map $R\Gamma(\omega_V)\to k$.
	 	We denote the resulting morphism by 
	 	\[\At_E:R\pi_{\mathcal{X}*}(E\otimes E^{\vee}\otimes \omega_V)[-1]\to L_{\mathcal{X}}.\]
	 	
	 	\begin{proposition}
	 		Suppose that $\mathcal{X}$ is an open substack of the moduli stack of coherent sheaves on $V$. Then $\At_E$ is an obstruction theory.
	 	\end{proposition}	
	 		
	 	\paragraph{Obstruction theory on Quot-schemes.}
	 	Let $\mathcal{Y}$ be an algebraic stack, and let $E$ be a $\mathcal{Y}$-flat coherent sheaf on $V\times \mathcal{Y}$. Let $f:\mathcal{X}\to \mathcal{Y}$ be an open substack of the relative Quot-scheme of $E$ over $\mathcal{Y}$, and let 
	 	\[0\to F\to E_{\mathcal{X}}\to G\to 0\]
	 	be the universal exact sequence on $V\times \mathcal{X}$. We consider the associated reduced Atiyah class $\overline{\at}'_{E}:=\overline{\at}'_{E,V\times \mathcal{X}/V\times \mathcal{Y},G}$ as a map $\overline{\at}'_E:F\otimes G^{\vee}\to \pi_{\mathcal{X}}^* L_{\mathcal{X}/\mathcal{Y}}$ in the derived category. We again consider the composition 
	 	\[\overline{\At}_E:\pi_{\mathcal{X}*}(F\otimes G^{\vee}\otimes \omega_V) \to L_{\mathcal{X}/\mathcal{Y}}\otimes R\pi_{\mathcal{X}*}\omega_V\to L_{\mathcal{X}/\mathcal{Y}}.\]
	 	
	 	\begin{proposition}\label{prop:atrelobquot}
	 		The map $\overline{\At}_E$ is a relative obstruction theory for the Quot-scheme $f:\mathcal{X}\to \mathcal{Y}$.
	 	\end{proposition}


		\medskip

\noindent{\small\sc Department of Mathematics, Stanford University, Stanford, CA 94305, U.S.A. 

\noindent E-mail: {\tt ntkuhn@posteo.net}}

\vskip 8pt

\noindent{\small\sc Department of Mathematics, Faculty of Science, Kyoto University, Kitashirakawa Oiwake-cho, Sakyo-ku, Kyoto 606-8502, Japan

\noindent E-mail: {\tt y-tanka@math.kyoto-u.ac.jp}}
		
		
	\end{document}